\documentclass[11pt]{amsart}
\usepackage[english]{babel}
\usepackage[utf8]{inputenc}

\usepackage[inner=1in, outer=1in, bottom =2.9cm, top =2.9cm]{geometry}

\makeatletter
\newcommand*\Cdot{\mathpalette\Cdot@{.5}}
\newcommand*\Cdot@[2]{\mathbin{\vcenter{\hbox{\scalebox{#2}{$\m@th#1\circled{$1$}$}}}}}
\makeatother

\usepackage{amsmath, amssymb, amsthm}
\usepackage{graphicx, overpic}
\usepackage{mathrsfs}
\usepackage{mathtools}
\usepackage[usenames,dvipsnames,svgnames,table]{xcolor}
\usepackage{enumerate}
\usepackage{enumitem}
\usepackage{lipsum}
\usepackage{upgreek}
\usepackage{color}
\usepackage{quiver}
\usepackage{xkeyval}
\usepackage{xypic}
\usepackage{tikz}
\usepackage{pst-node} 
\usepackage{tikz-cd} 
\usepackage{thm-restate} 
\usepackage{float}
\usepackage{bold-extra}
\usepackage{microtype}
\usepackage{etoolbox}
\usepackage[strict]{changepage} 
\usepackage{esvect}
\usepackage{wrapfig}
\usepackage{caption}
\usepackage{cite}
\usepackage{bbding}
\usepackage{array}
\usepackage{pict2e}
\usepackage{makecell}
\usepackage{stmaryrd}
\usepackage{amsbsy}
\usepackage{esvect}
\usepackage{bbding}
\usepackage{thm-restate}
\usepackage[toc,page, titletoc,title]{appendix}
\graphicspath{ {figures/} }
\usepackage{url}
\usepackage[pagebackref]{hyperref} 

\makeatletter
\newcommand*{\centerfloat}{%
  \parindent \z@
  \leftskip \z@ \@plus 1fil \@minus \textwidth
  \rightskip\leftskip
  \parfillskip \z@skip}
\makeatother

\newtagform{red}{\color{blue}(}{)}

\colorlet{linkequation}{blue}
\newcommand*{\refeqq}[1]{%
  \begingroup
    \hypersetup{
      linkcolor=linkequation,
      linkbordercolor=linkequation,
    }%
    \ref{#1}%
  \endgroup
}

\makeatletter
\newcommand{\colim@}[2]{%
  \vtop{\m@th\ialign{##\cr
    \hfil$#1\operator@font colim$\hfil\cr
    \noalign{\nointerlineskip\kern1.5\ex@}#2\cr
    \noalign{\nointerlineskip\kern-\ex@}\cr}}%
}
\newcommand{\colim}{%
  \mathop{\mathpalette\colim@{\rightarrowfill@\scriptscriptstyle}}\nmlimits@
}
\renewcommand{\varprojlim}{%
  \mathop{\mathpalette\varlim@{\leftarrowfill@\scriptscriptstyle}}\nmlimits@
}
\renewcommand{\varinjlim}{%
  \mathop{\mathpalette\varlim@{\rightarrowfill@\scriptscriptstyle}}\nmlimits@
}
\makeatother

\makeatletter
\providecommand*{\twoheadrightarrowfill@}{%
  \arrowfill@\relbar\relbar\twoheadrightarrow
}
\providecommand*{\twoheadleftarrowfill@}{%
  \arrowfill@\twoheadleftarrow\relbar\relbar
}
\providecommand*{\xtwoheadrightarrow}[2][]{%
  \ext@arrow 0579\twoheadrightarrowfill@{#1}{#2}%
}
\providecommand*{\xtwoheadleftarrow}[2][]{%
  \ext@arrow 5097\twoheadleftarrowfill@{#1}{#2}%
}
\makeatother

\makeatletter
\newcommand*{\relrelbarsep}{.386ex}
\newcommand*{\relrelbar}{%
  \mathrel{%
    \mathpalette\@relrelbar\relrelbarsep
  }%
}
\newcommand*{\@relrelbar}[2]{%
  \raise#2\hbox to 0pt{$\m@th#1\relbar$\hss}%
  \lower#2\hbox{$\m@th#1\relbar$}%
}
\providecommand*{\rightrightarrowsfill@}{%
  \arrowfill@\relrelbar\relrelbar\rightrightarrows
}
\providecommand*{\leftleftarrowsfill@}{%
  \arrowfill@\leftleftarrows\relrelbar\relrelbar
}
\providecommand*{\xrightrightarrows}[2][]{%
  \ext@arrow 0359\rightrightarrowsfill@{#1}{#2}%
}
\providecommand*{\xleftleftarrows}[2][]{%
  \ext@arrow 3095\leftleftarrowsfill@{#1}{#2}%
}
\makeatother

\usetikzlibrary{tikzmark,decorations.pathreplacing}
\usetikzlibrary{shapes,shadows,arrows}
\usetikzlibrary{decorations.markings}
\usetikzlibrary{positioning,calc}
\tikzset{near start abs/.style={xshift=1cm}}

\hypersetup{
	colorlinks=true,
	linkcolor=black,
	filecolor=black,      
	urlcolor=blue,
	citecolor= 	red,
}
\DeclareSymbolFont{extraup}{U}{zavm}{m}{n}
\DeclareMathSymbol{\varheart}{\mathalpha}{extraup}{86}
\DeclareMathSymbol{\vardiamond}{\mathalpha}{extraup}{87}
 \DeclareMathSymbol{\varclub}{\mathalpha}{extraup}{84} 
\DeclareMathSymbol{\varspade}{\mathalpha}{extraup}{85}

\newcommand{\bigslant}[2]{{\raisebox{.2em}{$#1$}\left/\raisebox{-.2em}{$#2$}\right.}}

\theoremstyle{definition}
\newtheorem{thm}{Theorem}[section]
\newtheorem{cor}{Corollary}[thm]
\newtheorem{lem}[thm]{Lemma}
\newtheorem{prop}[thm]{Proposition}

\theoremstyle{definition}

\newtheorem{ex}{Example}[section]

\newtheorem{remark}{Remark}[section]

\newcommand{\tb}{\mathtt{b}}

\newcommand{\tf}{\mathtt{f}}
\newcommand{\tc}{\mathtt{c}}

\newcommand{\ta}{\mathtt{a}}

\newcommand{\bR}{\mathbb{R}}
\newcommand{\bN}{\mathbb{N}}
\newcommand{\bC}{\mathbb{C}}

\newcommand{\bZ}{\mathbb{Z}}

\newcommand{\image}{\operatorname{image}}

\newcommand{\Hom}{\operatorname{Hom}}

\DeclareFontFamily{U}{mathx}{\hyphenchar\font45}
\DeclareFontShape{U}{mathx}{m}{n}{
      <5> <6> <7> <8> <9> <10>
      <10.95> <12> <14.4> <17.28> <20.74> <24.88>
      mathx10
      }{}
\DeclareSymbolFont{mathx}{U}{mathx}{m}{n}
\DeclareFontSubstitution{U}{mathx}{m}{n}
\DeclareMathAccent{\widecheck}{0}{mathx}{"71}
\DeclareMathAccent{\wideparen}{0}{mathx}{"75}

\newcommand{\bT}{\mathbb{T}}

\makeatletter
\newcommand*\bigcdot{\mathpalette\bigcdot@{.5}}
\newcommand*\bigcdot@[2]{\mathbin{\vcenter{\hbox{\scalebox{#2}{$\m@th#1\bullet$}}}}}
\makeatother

\makeatletter
\newcommand{\adjunction}{\@ifstar\named@adjunction\normal@adjunction}
\newcommand{\normal@adjunction}[4]{%
  #1\colon #2%
  \mathrel{\vcenter{%
    \offinterlineskip\m@th
    \ialign{%
      \hfil$##$\hfil\cr
      \longrightharpoonup\cr
      \noalign{\kern-.3ex}
      \smallbot\cr
      \longleftharpoondown\cr
    }%
  }}%
  #3 \noloc #4%
}
\newcommand{\named@adjunction}[4]{%
  #2%
  \mathrel{\vcenter{%
    \offinterlineskip\m@th
    \ialign{%
      \hfil$##$\hfil\cr
      \scriptstyle#1\cr
      \noalign{\kern.1ex}
      \longrightharpoonup\cr
      \noalign{\kern-.3ex}
      \smallbot\cr
      \longleftharpoondown\cr
      \scriptstyle#4\cr
    }%
  }}%
  #3%
}
\newcommand{\longrightharpoonup}{\relbar\joinrel\rightharpoonup}
\newcommand{\longleftharpoondown}{\leftharpoondown\joinrel\relbar}
\newcommand\noloc{%
  \nobreak
  \mspace{6mu plus 1mu}
  {:}
  \nonscript\mkern-\thinmuskip
  \mathpunct{}
  \mspace{2mu}
}
\newcommand{\smallbot}{%
  \begingroup\setlength\unitlength{.15em}%
  \begin{picture}(1,1)
  \roundcap
  \polyline(0,0)(1,0)
  \polyline(0.5,0)(0.5,1)
  \end{picture}%
  \endgroup
}
\makeatother

\newcommand{\leftrarrows}{\mathrel{\raise.75ex\hbox{\oalign{%
  $\scriptstyle\leftarrow$\cr
  \vrule width0pt height.5ex$\hfil\scriptstyle\relbar$\cr}}}}
\newcommand{\lrightarrows}{\mathrel{\raise.75ex\hbox{\oalign{%
  $\scriptstyle\relbar$\hfil\cr
  $\scriptstyle\vrule width0pt height.5ex\smash\rightarrow$\cr}}}}
\newcommand{\Rrelbar}{\mathrel{\raise.75ex\hbox{\oalign{%
  $\scriptstyle\relbar$\cr
  \vrule width0pt height.5ex$\scriptstyle\relbar$}}}}

\makeatletter
\def\leftrightarrowsfill@{\arrowfill@\leftrarrows\Rrelbar\lrightarrows}
\newcommand{\xleftrightarrows}[2][]{\ext@arrow 3399\leftrightarrowsfill@{#1}{#2}}
\makeatother

\newcommand{\la}{\langle}
\newcommand{\ra}{\rangle}

\newcommand{\wt}{\widetilde}

\definecolor{Red}{rgb}{0.8,0,0.2}

\newcommand{\GG}[1]{}

\makeatletter
\def\@footnotecolor{red}
\define@key{Hyp}{footnotecolor}{%
 \HyColor@HyperrefColor{#1}\@footnotecolor%
}
\def\@footnotemark{%
    \leavevmode
    \ifhmode\edef\@x@sf{\the\spacefactor}\nobreak\fi
    \stepcounter{Hfootnote}%
    \global\let\Hy@saved@currentHref\@currentHref
    \hyper@makecurrent{Hfootnote}%
    \global\let\Hy@footnote@currentHref\@currentHref
    \global\let\@currentHref\Hy@saved@currentHref
    \hyper@linkstart{footnote}{\Hy@footnote@currentHref}%
    \@makefnmark
    \hyper@linkend
    \ifhmode\spacefactor\@x@sf\fi
    \relax
  }%
\makeatother

\hypersetup{footnotecolor=blue}

\makeatletter
\patchcmd{\@startsection}
  {\@afterindenttrue}
  {\@afterindentfalse}
  {}{}
\makeatother

\title[Braid Arrangement Bimonoids and the Toric Variety of the Permutohedron]{Braid Arrangement Bimonoids and the Toric Variety\\ of the Permutohedron}
\author{William Norledge}
\address{Pennsylvania State University}
\email{wxn39@psu.edu}

\begin{document}

\usetagform{red}

\renewcommand{\chapterautorefname}{Chapter}
\renewcommand{\sectionautorefname}{Section}
\renewcommand{\subsectionautorefname}{Section}

\renewcommand{\chapterautorefname}{Chapter}
\renewcommand{\sectionautorefname}{Section}
\renewcommand{\subsectionautorefname}{Section}

\begin{abstract}
We show that the toric variety of the permutohedron (=permutohedral space) has the structure of a cocommutative bimonoid in species, with multiplication/comultiplication given by embedding/projecting-onto boundary divisors. In terms of Losev-Manin's description of permutohedral space as a moduli space, multiplication is concatenation of strings of Riemann spheres and comultiplication is forgetting marked points. In this way, the bimonoid structure is an analog of the cyclic operad structure on the moduli space of genus zero marked curves. Covariant/contravariant data on permutohedral space is endowed with the structure of cocommutative/commutative bimonoids by pushing-forward/pulling-back data along the (co)multiplication. Many \hbox{well-known} combinatorial objects index data on permutohedral space. Moreover, combinatorial objects often have the structure of bimonoids, with multiplication/comultiplication given by merging/restricting objects in some way. We prove (for a selection of cases) that the bimonoid structure enjoyed by these indexing combinatorial objects coincides with that induced by the bimonoid structure of permutohedral space. Thus, permutohedral space may be viewed as a fundamental underlying object which geometrically interprets many combinatorial Hopf algebras. Aguiar-Mahajan have shown that classical combinatorial Hopf theory is based on the braid hyperplane arrangement in a crucial way. This paper aims to similarly establish permutohedral space as a central object, providing an even more unified perspective. 

The main motivation for this work concerns Feynman amplitudes in the Schwinger parametrization, which become integrals over permutohedral space if one blows-up everything in the resolution of singularities. Then the Hopf algebra structure of Feynman graphs, first appearing in the work of Connes-Kreimer, coincides with that induced by the bimonoid structure of permutohedral space. 
\end{abstract}


\maketitle

\vspace{-4ex}

\setcounter{tocdepth}{1} 
\hypertarget{foo}{ }
\tableofcontents

\section*{Introduction}

In enumerative combinatorics, the notion of generating function may be categorified as follows. One picks a certain type (`species') of combinatorial object $\text{p}$, for example $\text{p}=\text{binary trees}$. One decides what are the points of these combinatorial objects, for example $\text{points}=\text{vertices}$ of binary trees. Then, for $I$ a finite set, form the set $\text{p}[I]$ of all $I$-labeled combinatorial objects of type $\text{p}$, where by $I$-labeling we mean an (often bijective) function on $I$ into the chosen points. The associated generating function just records the cardinalities of the sets $\text{p}[I]$, however there is more structure. For each bijection $\sigma: J\to I$ we have a relabeling bijection $\text{p}[\sigma]:\text{p}[I]\to \text{p}[J]$, which precomposes \hbox{$I$-labelings} with $\sigma$ to produce $J$-labelings. Thus, we actually have a presheaf on the category $\textsf{finSet}^\times$ of finite sets and bijections
\[
\text{p}: ({\textsf{finSet}^\times})^{\text{op}} \to \textsf{Set}
,\qquad
I \mapsto \text{p}[I], \quad \sigma \mapsto \text{p}[\sigma]
.\]
In general, for $\textsf{C}$ a category, we refer to $\textsf{C}$-valued presheaves on $\textsf{finSet}^\times$ as $\textsf{C}$-valued \emph{Joyal species}, going back to \cite{joyal1981theorie}, \cite{joyal1986foncteurs}. 

Instead of the discrete structures of combinatorics, one can instead consider continuous spaces, for example in species whose components $\text{p}[I]$ are configuration spaces of points, or moduli spaces parametrizing spaces with marked points. In this context, species are more often called \hbox{`S-modules'} or `symmetric sequences/families' in the literature. 

The continuous and discrete settings are often related by letting combinatorial objects index data on spaces. That is, one has a combinatorial species $\text{p}$, a family of spaces $(\mathcal{X}^I)_{I\in \textsf{finSet}}$, some chosen data `$\text{D}$', and a morphism of species
\[
\eta: \text{p} \to \text{D}(\mathcal{X}^{(-)})
.\]
In addition, species appearing in these two contexts often have compatible algebraic structure, i.e. $\eta$ is a homomorphism of the appropriate kind. As a famous example, certain multigraphs index cells of the moduli space of genus zero stable marked curves with nodal singularities, and the indexing is a homomorphism of cyclic operads. More precisely, fundamentally the \emph{space} $\mathcal{X}^{(-)}$ has algebraic structure, manifesting e.g. as factorization of its boundaries, which endows $\text{D}(\mathcal{X}^{(-)})$ with algebraic structure by pushforward/pullback of data. 

The main goal of this paper is to rigorously exhibit several of these indexing constructions, but also to provide a theoretical framework for formalizing and systematizing further such examples. We focus on the case of the toric variety of the permutohedron, called \emph{permutohedral space} following \cite{Kap}, which may be realized as a moduli space parametrizing strings of stable marked Riemann spheres, glued at the poles \cite{losevmanin}, \cite{batyrevblume10}. Our main result is the precise formulation of the indexing of nef line bundles on permutohedral space by generalized permutohedra, or more generally line bundles by Boolean functions, see \autoref{Invertible Sheaves and Boolean Functions}. In particular, in \autoref{thm:mainmain} we realize it as a homomorphism of bimonoids. It turns out we actually have to use invertible sheaves instead of line bundles in order to do this faithfully.

This paper is not comprehensive in giving examples of combinatorial objects which index data on permutohedral space. One of the most interesting examples we do not study is the indexing of degree one cycles of (tropical) permutohedral space by loopless valuated matroids, see e.g. \cite[Section 2.3]{eur}. This could be studied analogously to our treatment of generalized permutohedra and Boolean functions. 

Along the way, we shall require three key elaborations of Joyal species. Firstly, we need a generalization of Joyal species called braid arrangement species, which are required in settings where one does not have `factorization'. The reason we shall not enjoy this factorization property is because data of a product of spaces is not necessarily the product of data, 
\[
\text{D}( \mathcal{X}\times \mathcal{Y})  \neq \text{D}(\mathcal{X})\times \text{D}(\mathcal{Y})
.\]
Braid arrangement species (= Coxeter species of type $A$) are due to Aguiar-Mahajan \cite[Section 2.3]{AM22}, who define species and bimonoids over real reflection hyperplane arrangements, arbitrary real hyperplane arrangements \cite{aguiar2020bimonoids}, or even \hbox{left-regular bands}. A \emph{composition} $F$ of $I$ is a tuple of nonempty sets $F=(S_1,\dots, S_k)$ such that $I=S_1\sqcup \dots \sqcup S_k$. A $\textsf{C}$-valued \emph{braid arrangement species}, or just \emph{species}, $\textsf{p}$ consists of an object 
\[
\textsf{p}[F]\in \textsf{C}
\] 
for each composition $F$, equipped with actions of bijections and permutations, see \autoref{bas}. The name `braid arrangement species' is explained by the fact compositions are in natural one-to-one correspondence with faces of the braid hyperplane arrangement. If $\textsf{C}$ is a monoidal category, associated to a Joyal species $\text{p}$ is the \emph{Joyal-theoretic} braid arrangement species $\textsf{p}$, which has
\[
\textsf{p}[F] := \text{p}[S_1]\otimes \dots \otimes \text{p}[S_k] 
.\] 
See \autoref{sec:Joyal-Theoretic Species}. In this way, species generalize Joyal species by allowing the components $\textsf{p}[F]$ to be more complicated objects than just these tensor products. In this paper, we find crucial to connecting the setting without factorization with the more classical setting where we do have factorization is the notion of \emph{weakly Joyal-theoretic species}, see \autoref{sec:Joyal-Theoretic Species}. An example of a weakly Joyal-theoretic species involves taking external tensor products of sheaves of modules, see \autoref{exten}.

A \emph{bimonoid} in species (when $\textsf{C}=\textsf{Vec}$ we tend to say `bialgebra' instead of bimonoid) is a species $\textsf{h}$ equipped with a pair of maps
\[
\mu_{F,G} : \textsf{h}[F]\to   \textsf{h}[G] \qquad \text{and} \qquad   \Delta_{F,G} : \textsf{h}[G]\to   \textsf{h}[F] 
\]
for each pair of compositions $F,G$ such that $F$ contains $G$ as braid arrangement faces, called the multiplication and comultiplication respectively. The key axiom for bimonoids, called the \emph{bimonoid axiom}, depends crucially on Tits' projection product \cite{Tits74} of braid arrangement faces,\footnote{\ the product for just the braid arrangement appeared independently in \cite[Section 4.1]{epstein1976general}, where it is used to express causal factorization of time-ordered products} see \autoref{Bimonoids in Braid Arrangement Species}.

A bimonoid $\textsf{h}$ is called \emph{Joyal-theoretic} if $\textsf{h}$ is a Joyal-theoretic species and its (co)multiplication similarly factorizes, see \autoref{sec:Joyal-Theoretic Bimonoids}. In this case, setting $F=(S,T)$ and $G=(I)$ recovers more familiar looking (co)multiplication,
\[
\mu_{S,T}= \mu_{(S,T),(I)} : \textsf{h}[S]\otimes   \textsf{h}[T]\to   \textsf{h}[I]
 \qquad \text{and} \qquad  
 \Delta_{S,T}= \Delta_{(S,T),(I)} :    \textsf{h}[I] \to\textsf{h}[S]\otimes   \textsf{h}[T]
.\] 
Bimonoids in species may be formalized as bialgebras over a certain bimonad acting on species. See \cite[Section 2.10]{AM22} for the case $\textsf{C}=\textsf{Vec}$, and see \cite[Section 3.7]{aguiar2020bimonoids} for the case $\textsf{C}=\textsf{Set}$ but in a less structured setting where there is no action of bijections. Aguiar-Mahajan have shown how bimonoids in species provide a deeper and more unified perspective on combinatorial graded Hopf algebras \cite[Chapter 15]{aguiar2010monoidal}, and also how these graded Hopf algebra generalize to non-type $A$ settings \cite[Part II]{AM22}.

Secondly, we require a vertical categorification of species and bimonoids to allow for $\textsf{C}$ to be a $2$-category. This is because the data `$\text{D}$' often forms a category 
\[
\text{D}(\mathcal{X}^{I})\in \textsf{Cat}
\]
where $\textsf{Cat}$ is the $2$-category of $1$-categories. This means we need to develop a theory of species and bimonoids where axioms hold only up to coherent isomorphism. This takes up most of the first part of the paper, and culminates in our definition of a `$2$-bimonoid' in \autoref{Bimonoids in Braid Arrangement Species}. We do not prove any coherence theorems, and in this paper it is only a conjecture that we have all the coherence conditions we should want in the correction definition of a $2$-bimonoid. We leave the conjectural formalization of $2$-bimonoids as bialgebras over a $2$-bimonad to future work.

Thirdly, towards applications to Feynman amplitudes mentioned in the abstract (see e.g. \cite{Brown:2015fyf}, \cite{schultka2018toric}), after completing the indexing of nef invertible sheaves by generalized permutohedra, we would like to construct some kind of `bialgebra of global sections of invertible sheaves'. For example, this should be the bialgebra which appears in the factorization identities for Symanzik polynomials. Also, since permutohedral space is a bimonoid, surely the structure sheaf is somehow a bialgebra (recall global regular functions on permutohedral space are all constant). To realize these, we require the notion of $\textsf{p}$-graded set/vector species, see \autoref{Graded Braid Arrangement Species}. Here, $\textsf{p}$ is a species valued in categories, and one has a set/vector space for each object $\ta\in \textsf{p}[F]$. Similarly, when $\textsf{h}$ is a bimonoid valued in categories, one can define $\textsf{h}$-graded bimonoids/bialgebras, see \autoref{Bimonoids in Graded Braid Arrangement Species}. 

We now give a brief summary of the paper. \autoref{part1}, which consists of \autoref{sec1}, \autoref{sec2} and \autoref{sec3}, deals with developing our three elaborations on Joyal species described above. The reader might wish to skip ahead to \autoref{part2}, and refer back to the theory building of \autoref{part1} as required. 

We begin the application to permutohedral space in \autoref{sec:prelim} by considering \emph{preposets} (=reflexive and transitive relations). The bialgebra of preposets $\textbf{O}$ was defined in \cite[Section 13.1.6]{aguiar2010monoidal}. See also \cite[Section 15.4]{aguiar2017hopf}. Preposets $p$ are significant because they bijectively index cones $\sigma_p$ of the braid arrangement. By cone duality, they also bijectively index cones $\sigma^\circ_p=-\sigma^\vee_p$ of the adjoint braid arrangement which are generated by coroots, see \autoref{sec:cones}. Multiplication of preposets corresponds to taking products of coroot cones, and comultiplication of preposets corresponds to taking faces of coroot cones and factorizing, see \autoref{lem:iso} and \autoref{lem:face}. (See \autoref{rem:braidconein} for an interpretation in terms of cones of the braid arrangement.)

However, the bialgebra $\textbf{O}$ does not have a $\textsf{Set}$-valued bimonoid analog, because the comultiplication of preposets is sometimes `inadmissible' and equals zero. This is a problem for us, especially given the kind of role we want combinatorial objects like preposets to play, mentioned above. The fix is to let the bimonoid of preposets $\textsf{O}_\bullet$ be valued in \emph{pointed} sets instead, or even pointed categories\footnote{\ for us, a \emph{pointed category} $\textsf{D}\in \textsf{Cat}_\bullet$ is a functor $1_{\textsf{Cat}} \to \textsf{D}$ where $1_{\textsf{Cat}}$ is the trivial category} $\textsf{O}_\bullet[I]\in \textsf{Cat}_\bullet$, where morphisms between preposets correspond to inclusion of coroots cones, see \autoref{prob:isbimon}. Note the distinguished object $\bullet\in \textsf{O}_\bullet[I]$ playing the role of zero is not a preposet, rather it is a formal element we add by hand. We make the convention that its corresponding cone is the empty cone $\sigma^\circ_\bullet:=\emptyset$. 

In \autoref{sec:Graded Bialgebra}, we proceed to describe the bimonoid structure of permutohedral space `algebra first' as opposed to `geometry first', in order to fully expose the combinatorial structure. (A geometry first description is straightforward using Losev-Manin's realization of permutohedral space \cite{losevmanin} as a certain moduli space parametrizing stable marked strings of Riemann spheres glued at the poles; the multiplication is concatenation of strings of Riemann spheres and the comultiplication is forgetting marked points followed by stabilization.) We construct the \hbox{$\textsf{O}_\bullet$-graded} bialgebra $\bC\textsf{\textbf{O}}$, where the vector space associated to $p$ is the $\bC$-algebra $\bC[\text{M}_p]$ which has basis the integer points of $\sigma^\circ_p$. The algebraic structure of $\bC\textsf{\textbf{O}}$ largely follows from the geometric interpretation of $\textsf{O}_\bullet$ in terms of coroot cones. The multiplication of $\bC\textsf{\textbf{O}}$ is given by taking products of integer points, and the comultiplication is given by projecting onto faces and factorizing, see \autoref{sec:Graded Bialgebra}. 

The toric variety $\bC\Sigma^I$ of the braid arrangement fan over a finite set $I$ is constructed by taking charts the affine varieties 
\[
U_F:= \Hom_{\textsf{Aff}_\bC}\! \big(\,  \ast \, ,\,   \text{Spec}(\bC[\text{M}_F])\big) = \Hom_{\bC\textsf{Alg}}\! \big (\bC[\text{M}_F],\bC\big )
\]
where $F$ is a total preposet on $I$ (which are equivalent to compositions of $I$), and gluing along the inclusion of coroot cones. In \autoref{sec:permspace1} and \autoref{sec:permspace2}, we get a multiplication/comultiplication on permutohedral space $\bC\Sigma$ by pulling back points $\bC[\text{M}_F]\to \bC$ along the comultiplication/multiplication of $\bC\textsf{\textbf{O}}$. In particular, $\bC\textsf{\textbf{O}}$ was a commutative bialgebra, whereas $\bC\Sigma$ is a cocommutative bimonoid. 

From a geometry first perspective, we may reinterpret $\bC\textsf{\textbf{O}}$ as follows. For each preposet $p$, the toric variety $U_p$ of the restriction of the braid fan to $\sigma_p$ is naturally an open set $U_p\subseteq \bC\Sigma^I$. We have a natural identification 
\[
\bC[\text{M}_p] = \big\{\text{regular functions $U_p\to \bC$}\big\}
.\]
Then the multiplication/comultiplication of $\bC\textsf{\textbf{O}}$ is given by pulling back regular functions along the comultiplication/multiplication of permutohedral space $\bC\Sigma$.

In \autoref{sec:site}, we formalize and generalize this structure as follows. We have the bimonoid of opens $\textsf{Op}^{\text{op}}$, given by
\[
\textsf{Op}^{\text{op}}[F] := \big \{ \text{Zariski open sets of } \bC\Sigma^{S_1}\times \dots \times \bC\Sigma^{S_k} \big \} 
\]
where the multiplication/comultiplication is given by pulling back opens along the comultiplication/multiplication of permutohedral space. In \autoref{indexbyopens}, we show the indexing of opens by preposets may be formalized as a homomorphism of bimonoids
\[
\textsf{O}_\bullet \to  \textsf{Op}^{\text{op}}
,\qquad p\mapsto U_p
.\]
We have the $\textsf{Op}^{\text{op}}$-graded bialgebra $\bC\textbf{\textsf{Op}}^{\text{op}}$, where the vector space associated to an open $U$ is the space of regular functions on $U$, and the multiplication/comultiplication is given by pulling back regular functions along the comultiplication/multiplication of permutohedral space. We may then pullback the grading of $\bC\textbf{\textsf{Op}}^{\text{op}}$ along $\textsf{O}_\bullet \to  \textsf{Op}^{\text{op}}$, which recovers the $\textsf{O}_\bullet$-graded bialgebra $\bC\textsf{\textbf{O}}$. In \autoref{thm}, we show that graded bimonoids/bialgebras can always be pulled-back along homomorphisms in this way.

We have that $\bC\textbf{\textsf{Op}}^{\text{op}}$ is not Joyal-theoretic, whereas $\bC\textsf{\textbf{O}}$ is. Thus, $\bC\textsf{\textbf{O}}$ may be viewed as a kind of `algebraic Joyalization' of $\bC\textbf{\textsf{Op}}^{\text{op}}$ by restricting its structure to that which is seen by preposets. Of course, $\bC\textsf{\textbf{O}}$ may be viewed in isolation as Joyal-theoretic, which is what we did in \autoref{sec:Graded Bialgebra}. However, the origin (or at least interpretation) of its algebraic structure, of pulling back opens and regular functions, is inherently non Joyal-theoretic.\footnote{\ An analogous construction can be found in \cite{lno2019}, where algebraic Joyalization happens by taking a quotient, and once again the algebraic structure (this time a Lie algebra) has its origin in a non Joyal-theoretic setting. The quotient is by the so-called Steinmann relations of axiomatic QFT.}

In \autoref{sec:topos} and \autoref{Invertible Sheaves and Boolean Functions}, we consider an analog of the above but in the setting of the sheaf topos $\textsf{Sh}$ of $\textsf{Op}^{\text{op}}$. We consider sheaves of modules instead of opens, global sections instead of regular functions, and Boolean functions $\textsf{BF}$ (or generalized permutohedra $\textsf{GP}$) instead of preposets. This goes as follows.

Let $\mathcal{O}_F$ denote the structure sheaf of $\bC\Sigma^{S_1}\times \dots \times \bC\Sigma^{S_k}$. We have the bimonoid of modules $\textsf{Mod}$, given by
\[
\textsf{Mod}[F] :=  \big \{\text{sheaves of $\mathcal{O}_F$-modules on } \bC\Sigma^{S_1}\times \dots \times \bC\Sigma^{S_k} \big \} 
\]
where the multiplication/comultiplication is given by pulling back modules along the comultiplication/multiplication of permutohedral space. This is a `$2$-bimonoid', with algebraic structure holding only up to coherent isomorphism. In \autoref{A Homomorphism of Bimonoid1} and \autoref{A Homomorphism of Bimonoids}, we construct a (lax, but strong) homomorphism of bimonoids
\[
\textsf{BF}\to \textsf{Mod}, 
\qquad 
z \mapsto \mathcal{O}_z
\]
with essential image the invertible sheaves of permutohedral space, where $\textsf{BF}$ is the bimonoid of so-called Boolean functions, see \cite[Section 12.1]{aguiar2017hopf}. We prove this is a homomorphism in \autoref{thm:mainmain}. The restriction to generalized permutohedra 
\[
\textsf{GP}\hookrightarrow \textsf{BF}
\]
indexes the nef invertible sheaves. If we want this homomorphism to be injective, so that $\textsf{BF}$ and $\textsf{GP}$ are faithfully interpreted, we have to use invertible sheaves here and \emph{not} line bundles. In addition, the invertible sheaves we construct $\mathcal{O}_z$ have to be `closed' under pullback along the (co)multiplication of permutohedral space. In particular, the usual construction of invertible sheaves on toric varieties as subsheaves of the sheaf of rational functions is not sufficient. 

We have the $\textsf{Mod}$-graded bialgebra $\bC\textbf{\textsf{Mod}}$, where the vector space associated to a module $M$ is the space of global sections $s\in M$, and the multiplication/comultiplication of $\bC\textbf{\textsf{Mod}}$ is given by pulling back global sections along the comultiplication/multiplication of permutohedral space, see \autoref{Global Sections}. We may then pullback the grading of $\bC\textbf{\textsf{Mod}}$ along $\textsf{BF}\to \textsf{Mod}$, to give Joyal-theoretic bialgebras $\bC\textbf{\textsf{BF}}$ and $\bC\textbf{\textsf{GP}}$, see \autoref{The Joyal-Theoretic Bialgebra of Global Sections of Invertible Sheaves}. In particular, $\bC\textbf{\textsf{BF}}$ is the bialgebra of global sections of invertible sheaves. As with $\bC\textbf{\textsf{O}}$, $\bC\textbf{\textsf{BF}}$ may be considered in isolation as Joyal-theoretic, however its algebraic structure has come from an inherently non Joyal-theoretic setting.

The restriction $\bC\textbf{\textsf{GP}}$ of $\bC\textbf{\textsf{BF}}$ to generalized permutohedra has a nice description. The space of global sections of the invertible sheaf $\mathcal{O}_\mathfrak{p}$ of a generalized permutohedron $\mathfrak{p}$ is naturality the vector space with basis the integer points of $\mathfrak{p}$. Then pullback of global sections along the multiplication of permutohedral space corresponds to taking (integer points in) faces of generalized permutohedra and factorizing, and the pullback of global sections along the comultiplication of permutohedral space corresponds to taking products of (integer points in) generalized permutohedra. See \cite[Section 5.1]{aguiar2017hopf} for a description of these operations at the level of the indexing generalized permutohedra. 




\subsection*{Future Work}

The main motivation for this work concerns Euclidean Feynman amplitudes in the Schwinger parametrization, which become integrals over permutohedral space if one `blows-up everything' in the resolution of singularities \cite{Brown:2015fyf}, \cite{schultka2018toric}. Then the Hopf algebra structure of Feynman graphs, first appearing in the work of Connes-Kreimer, coincides with that induced by the bimonoid structure of permutohedral space. After conversations with Franz Herzog, the author understands that a conjectural Hopf algebra structure of Feynman amplitudes in \emph{Minkowski signature} may require non Joyal-theoretic bimonoids throughout. For example, can the forest-like treatment of IR divergences in \cite{MR4112459} be understood in terms of non Joyal-theoretic bimonoids? 

In previous work of the author, the construction of the perturbative S-matrix scheme in causal perturbation theory was formalized in terms of bimonoids in species \cite{norledge2020species}. Study the role permutohedral space may play there, via its interpretations as the space of configurations on a worldline (tropical) or worldsheet (complex).

There are other combinatorial bimonoids which can be geometrically interpreted over permutohedral space. For example, the indexing of degree one cycles of tropical permutohedral space by valuated loopless matroids. These cycles can be pulled back along the (co)multiplication of permutohedral space, which is compatible with the bimonoid structure on loopless matroids as given here \cite[Section 13.8.5]{aguiar2010monoidal}. This suggests the question; can the Grassmannian or Dressian be equipped with bimonoid structure?

The author's understanding is that cyclic operads are \emph{not} operads (if by `operad' we mean a monoid with respect to a plethystic tensor product, and by plethystic tensor product we mean the categorified Fourier transform of a composition of analytic endofunctors). Rather they are \emph{algebras} over a certain monad acting on Joyal species, and are thus analogous to the (bi)monoids of this paper. Using Aguiar-Mahajan's theory of bimonoids with respect to left-regular bands \cite[Section 3.9]{aguiar2020bimonoids}, can the moduli space of genus zero stable marked curves with nodal singularities be realized as a bimonoid? 

The type $A$ full flag variety has analogous bimonoid structure, and this may be interesting to study. Copies of permutohedral space sit inside the full flag variety as \hbox{sub-bimonoids}. 

\subsection*{Acknowledgments.} 

We thank Adrian Ocneanu for his support and useful discussions, and Nick Early for useful discussions. We thank Marcelo Aguiar and Swapneel Mahajan for the sharing of unpublished notes on the topic of $\textsf{p}$-graded species. We thank Penn State maths department for their support. 

\part{Species and Bimonoids Background} \label{part1}

\section{Species}\label{sec1}

\subsection{Compositions of Finite Sets}

Let $I$ be a finite set. For $k\in \bN$, let
\[
(k):=\{1,\dots,k\}
\]
equipped with the ordering \hbox{$1>\dots> k$}. A \emph{composition} $F$ of $I$ of \emph{length} $l(F)=k$ is a surjective function $F:I\to (k)$. The set of all compositions of $I$ is denoted $\Sigma[I]$,
\[  
\Sigma[I]:=  \bigsqcup_{k\in \bN}  \big \{  \text{surjective functions}\ F:I \to (k) \big\}
.\]
We often denote compositions by $k$-tuples 
\[  
F=  (S_1, \dots, S_k)
\]
where $S_i:= F^{-1}(i)$, $1\leq i \leq k$. The $S_i$ are called the \emph{lumps} of $F$. In particular, we have the length one composition $(I)$ for $I\neq \emptyset$, and the length zero composition $(\, )$ which is the unique composition of the empty set. Note the $1$-tuple consisting of the empty set $(\emptyset)$ does not correspond to a composition. For compositions $F,G\in \Sigma[I]$, we write $G\leq F$ if $G$ can be obtained from $F$ by iteratively merging contiguous lumps. 

Given a decomposition $I\! =S\sqcup T$ of $I$ ($S$, $T$ can be empty), for $H=(S_1, \dots , S_{k})$ a composition of $S$ and $K=(T_1,\dots, T_{l})$ a composition of $T$, their \emph{concatenation} $H; K$ is the composition of $I$ given by
\[ 
H;  K: =   ( S_1, \dots , S_{k},  T_1,\dots, T_{l}  ) 
.\]
For $S\subseteq I$ and $H=(S_1, \dots , S_{k}) \in \Sigma[I]$, the \emph{restriction} $H|_S$ of $H$ to $S$ is the composition of $S$ given by
\[
H|_S:=  (  S_1 \cap S, \dots, S_k\cap S )_+
\]
where $(-)_+$ means we delete any sets from the list which are the empty set.


\subsection{Joyal Species}

Before describing braid arrangement species, we first recall the more classical notion of Joyal species, going back to \cite{joyal1981theorie}, \cite{joyal1986foncteurs}. The connection between Joyal species and braid arrangement species is described in \autoref{sec:Joyal-Theoretic Species}.

Let $\textsf{finSet}^\times$ denote the monoidal category with objects finite sets $I$, morphisms bijective functions $\sigma$, and monoidal product disjoint union. For $\textsf{C}$ a category, a \hbox{$\textsf{C}$-valued} \emph{Joyal (co)species} $\text{p}$ is a \hbox{$\textsf{C}$-valued} (co)presheaf on $\textsf{finSet}^\times$,
\[
\text{p} : (\textsf{finSet}^\times)^{\text{(op)}} \to \textsf{C}
,\qquad
I \mapsto \text{p}[I], \quad    \sigma\mapsto  \text{p}[\sigma] 
.\]
In particular, for $\textsf{C}\in \{  \textsf{Set}, \textsf{Vec} \}$, we say \emph{Joyal set (co)species} and \emph{Joyal vector (co)species}. 

This definition allows for $\textsf{C}$ to be a higher category. For example, in this paper we shall often have \hbox{$\textsf{C}=\textsf{Cat}$}, where $\textsf{Cat}$ is the $2$-category of $1$-categories. In general, when $\textsf{C}$ is a \hbox{$2$-category}, a $\textsf{C}$-valued Joyal (co)species should be a $\textsf{C}$-valued \hbox{$2$-(co)presheaf}\footnote{\ i.e. pseudofunctor, a.k.a. strong $2$-functor, see e.g. \cite[Definition B1.2.1]{MR2063092}} on $\textsf{finSet}^\times$. Such objects might be called \hbox{`$2$-(co)species'}, and are also known as `stuff types' when $\textsf{C}=\textsf{Grpd}$. We will just say `$\textsf{C}$-valued (co)species', where the fact it is a $2$-(co)species is understood from the fact $\textsf{C}$ is a $2$-category. For simplicity, in this paper we add the condition that $\textsf{C}$-valued Joyal (co)species are strictly unital, i.e. they preserve the identities of $\textsf{finSet}^\times$ strictly.

If $\textsf{C}$ has the additional structure of a monoidal category $\textsf{C}=(\textsf{C}, \otimes, 1_\textsf{C})$, a $\textsf{C}$-valued Joyal (co)species $\text{p}$ is called \emph{connected} if 
\[
\text{p}[\emptyset]=1_\textsf{C}
.\] 
In this paper, we only consider connected Joyal (co)species.\footnote{\ the generic \hbox{not-necessarily-connected} setting would require a more general theory which uses decompositions in place of compositions, where a \emph{decomposition} $F$ of $I$ is a function $F:I\to (k)$ (not necessarily surjective), although the existence of this more general theory is currently just a conjecture for the author}

In \cite[Chapter 15]{aguiar2010monoidal}, Aguiar-Mahajan showed that a deeper, more explanatory, and more unified perspective on combinatorial ($\bN$-graded) Hopf algebras is that they are decategorifications (namely $\textsf{finSet}^\times\to \bN$) of Hopf monoids constructed internal to vector species, using the Day convolution as the tensor product. Moreover, many of these Hopf monoids are linearizations of certain \hbox{set-theoretic} bimonoids constructed in connected set species. (The set-theoretic bimonoids are not internal bimonoids, rather they are correctly formalized as bialgebras over a certain bimonad.)



\begin{ex}
The mapping $I \mapsto \Sigma[I]$ extends to a connected Joyal set species $\Sigma$ (where $\textsf{Set}$ is a cartesian category) given by 
\[
\Sigma : (\textsf{finSet}^\times)^{\text{op}} \to \textsf{Set}
,\qquad
I \mapsto \Sigma[I], \quad    \sigma\mapsto  \Sigma[\sigma] 
\]  
where $\Sigma[\sigma](F):= F\circ \sigma$. Then $\Sigma$ becomes a set-theoretic bimonoid with multiplication concatenation of compositions and comultiplication restriction of compositions. Its covariant linearization has decategorified image in graded vector spaces the graded Hopf algebra of noncommutative symmetric functions $\textbf{NSym}$, which has been widely studied. As we shall see, the bimonoid structure of $\Sigma$ is a precursor to the bimonoid structure of permutohedral space, which is obtained by replacing $(k)$ with a string of $k$-many Riemann spheres, glued at the poles.  
\end{ex}

\subsection{Braid Arrangement Species} \label{bas}

We now describe a generalization of connected Joyal species called braid arrangement species. As we shall see, braid arrangement species are required in settings where one does not have certain `factorization'. Factorization is implicit in Joyal species. This generalization is due to \hbox{Aguiar-Mahajan} \cite{aguiar2020bimonoids}, \cite{AM22}, who focus on the cases $\textsf{C}\in \{ \textsf{Set}, \textsf{Vec}\}$. In this paper, we require a vertical categorification of Aguiar-Mahajan's work to allow for $\textsf{C}$ to be a $2$-category.

As with Joyal species, one can define braid arrangement species so they take values in any higher category $\textsf{C}$. In this paper, $\textsf{C}$ will either be the \hbox{$2$-category} of (strict) $1$-categories $\textsf{Cat}$\footnote{\ we let $\textsf{Cat}$ be sufficiently large, i.e. `very large', to contain e.g. $\textsf{Set}$} or the \hbox{$2$-category} of pointed categories $\textsf{Cat}_\bullet$\footnote{\ For us, a \emph{pointed category} $\textsf{D}$ is a functor of the form $1_{\textsf{Cat}} \to \textsf{D}$, where $\textsf{D}$ is a (strict) $1$-category and $1_{\textsf{Cat}}$ is the trivial category. The image of $1_{\textsf{Cat}}$ is called the \emph{distinguished object} of $\textsf{D}$, which we often denote by $\bullet$. A functor between pointed categories is required to preserve the distinguished object.}, and occasionally the category of sets $\textsf{Set}$ or the category of vector spaces $\textsf{Vec}$ over a field $\Bbbk$.\footnote{\ $\textsf{p}$-\emph{graded species}, defined in \autoref{Graded Braid Arrangement Species}, are valued exclusively in $\textsf{Set}$ and $\textsf{Vec}$, and really, our braid arrangement species valued in $\textsf{Set}$ and $\textsf{Vec}$ come from quotienting out the grading of a $\textsf{p}$-graded species, see e.g. \autoref{quot}} However $\textsf{Cat}$ and $\textsf{Cat}_\bullet$ are moreover \emph{strict} $2$-categories, meaning unitors and associators are trivial. Therefore we define braid arrangement species in the setting of strict $2$-categories only, since we do not need to go any higher.


When a bijection $\sigma:J\to I$ between finite sets is understood and there is no ambiguity, given a composition $F \in \Sigma[I]$, we let
\[
F'   :=     F \circ \sigma             \,    \in \Sigma[J]
.\]  
Let $\text{Sym}_k$ denote the symmetric group on $(k)=\{ 1,\dots,k \}$. When a permutation $\beta\in \text{Sym}_k$ is understood and there is no ambiguity, given a composition $F\in \Sigma[I]$ of length $k$, we let
\[
\wt{F}:=   \beta \circ F         \,    \in \Sigma[I]
.\] 
When composable bijections $K\xrightarrow{\tau} J \xrightarrow{\sigma} I$ are understood, we let $F'':= F\circ (\sigma\circ\tau)$. Similarly, when permutations $\beta,\gamma\in \text{Sym}_k$ are understood,  we let $\wt{\wt{F}}:=(\gamma\circ \beta) \circ F$. We have the following five basic facts, which tell us that compositions form a bimodule under the action of bijections and permutations. The action of bijections and permutations commute
\[
\wt{(F')}=(\wt{F})'
\]
(from now on let $\wt{F}':=\wt{(F')}=(\wt{F})'$). The action of bijections is a right action,
\[
(F')' = F'' \quad \qquad \text{and}\qquad  \quad   F'=F \quad \text{if}\quad \sigma=\text{id}_I
.\]
The action of permutations is a left action,
\[
\wt{(\wt{F})} = \wt{\wt{F}}  \quad\qquad \text{and}\qquad \quad    \wt{F}=F \quad \text{if}\quad \beta=\text{id}_{(k)}
.\] 

For $\textsf{C}$ a strict $2$-category, a \hbox{$\textsf{C}$-valued} (strong, strictly unital) \emph{braid arrangement species}, or just \emph{species}, $\textsf{p}$ consists of an object $\textsf{p}[F]\in \textsf{C}$ for each composition $F\in \Sigma[I]$, $I\in \textsf{finSet}^\times$, a morphism (or $1$-cell)
\[
\textsf{p}[\sigma,F]  : \textsf{p}[F] \to \textsf{p}[F'] 
\]
for each bijection of finite sets $\sigma: J\to I \, \in \textsf{finSet}^\times$ and composition $F\in \Sigma[I]$, and a morphism
\[
\textsf{p}[F,\beta]  : \textsf{p}[F] \to \textsf{p}[\wt{F}]
\]
for each permutation $\beta\in  \text{Sym}_k$, $k\in \bN$, and composition $F$ of length $k$. When there is no ambiguity, we often abbreviate
\[
\textsf{p}[I]:= \textsf{p}[(I)], \qquad \textsf{p}[\sigma]:= \textsf{p}[\sigma,(I)],\qquad \textsf{p}[\beta]:= \textsf{p}[(I),\beta]
\]
for $I\neq \emptyset$, and even just
\[
\sigma:= \textsf{p}[\sigma,F],  \qquad \beta := \textsf{p}[F,\beta]
.\]
In addition, we require an invertible $2$-cell called $(\sigma,\beta)$-\emph{functoriality}
\[\begin{tikzcd}[row sep=0.9cm]
	{\textsf{p}[F]} & {\textsf{p}[\widetilde{F}]} \\
	{\textsf{p}[F']} & {\textsf{p}[\widetilde{F}']}
	\arrow["\beta"', from=2-1, to=2-2]
	\arrow["\sigma"', from=1-1, to=2-1]
	\arrow["\sigma", from=1-2, to=2-2]
	\arrow["\beta", from=1-1, to=1-2]
	\arrow["{  \textsf{p}[ \sigma, F, \beta ] }"{description}, Rightarrow, from=2-1, to=1-2]
\end{tikzcd}\]
an invertible $2$-cell (contravariant) $\sigma$-\emph{functoriality} and a trivial $2$-cell \emph{strictly $\sigma$-unitary}
\[\begin{tikzcd}[row sep=1cm]
	& {\textsf{p}[F']} \\
	{\textsf{p}[F]} && {\textsf{p}[F'']} && {\textsf{p}[F]} & {\textsf{p}[F]}
	\arrow["\sigma", from=2-1, to=1-2]
	\arrow["\tau", from=1-2, to=2-3]
	\arrow["{\sigma\circ \tau}"', from=2-1, to=2-3]
	\arrow[""{name=a, anchor=center, inner sep=0}, "{\text{id}_{\textsf{p}[F]}}", curve={height=-10pt}, from=2-5, to=2-6]
	\arrow[""{name=b, anchor=center, inner sep=0}, "{\text{id}_I}"', curve={height=10pt}, from=2-5, to=2-6]
	\arrow[""{name=0, anchor=center, inner sep=0}, from=2-1, to=2-3]
	\arrow["{ \textsf{p}[ \sigma,\tau, F ]  }"{description}, shorten >=3pt, Rightarrow, from=1-2, to=0]
\end{tikzcd}\]
and an invertible $2$-cell (covariant) $\beta$-\emph{functoriality} and a trivial $2$-cell \emph{strictly $\beta$-unitary}
\[\begin{tikzcd}[row sep=0.9cm]
	& {\textsf{p}[\widetilde{F}]} \\
	{\textsf{p}[F]} && {\textsf{p}[\widetilde{\widetilde{F}}]} && {\textsf{p}[F]} & {\textsf{p}[F].}
	\arrow["\beta", from=2-1, to=1-2]
	\arrow["\gamma", from=1-2, to=2-3]
	\arrow["\gamma\circ\beta"', from=2-1, to=2-3]
	\arrow[""{name=a, anchor=center, inner sep=0},  "{\text{id}_{\textsf{p}[F]}}", curve={height=-10pt}, from=2-5, to=2-6]
	\arrow[""{name=b, anchor=center, inner sep=0},  "{\text{id}_{(k)}}"', curve={height=10pt}, from=2-5, to=2-6]
	\arrow[""{name=0, anchor=center, inner sep=0}, from=2-1, to=2-3]
	\arrow["{\textsf{p}[ F,\beta,\gamma  ] }"{description}, shorten >=3pt, Rightarrow, from=1-2, to=0]
\end{tikzcd}\]
When there is no ambiguity, we often abbreviate
\[
(\sigma,\beta):= \textsf{p}[ \sigma, F, \beta ], \qquad    (\sigma,\tau):= \textsf{p}[ \sigma,\tau, F ]  ,\qquad     (\beta,\gamma):=\textsf{p}[ F,\beta,\gamma  ] 
.\]
We require such $2$-cells for all meaningful choices of $F,\sigma,\tau,\beta,\gamma$, with the requirement that \hbox{$\sigma,\tau,\beta,\gamma\neq \text{id}$.}\footnote{\ alternatively, one can include these cases in the definition, but then needs to add coherence conditions saying that all the resulting $2$-cells are trivial, e.g. $\textsf{p}[\sigma,F,\text{id}_{(k)}]= \text{id}_{\textsf{p}[\sigma,F]}$} For example, we have a $\sigma$-functoriality $2$-cell $(\sigma,\tau)$ for each pair of composable non-identity bijections $K\xrightarrow{\tau} J \xrightarrow{\sigma} I$ and choice of composition $F\in \Sigma[I]$. If all the $2$-cells are trivial, i.e. the above five diagrams commute on the nose, then the species is called \emph{strict}. If $\textsf{C}$ is a $1$-category, for example $\textsf{C}\in \{ \textsf{Set},\textsf{Vec} \}$, then a $\textsf{C}$-valued species must be strict. 

We impose four coherence conditions on the $2$-cells $(\sigma,\beta)$, $(\sigma,\tau)$, $(\beta,\gamma)$. We require trivial $3$-cells \hbox{\emph{$(\sigma,\tau,\beta)$-coherence}} and \emph{$(\beta,\gamma,\sigma)$-coherence}
\begin{figure}[H]
	\centering
	\begin{tikzcd}[column sep=large, row sep=0.75cm]  
		\centering
		& {\textsf{p}[F']} \\
		{\textsf{p}[F]} && {\textsf{p}[F'']} \\
		& {\textsf{p}[\wt{F}']} \\
		{\textsf{p}[\wt{F}]} && {\textsf{p}[\wt{F}'']} 
		\arrow[""{name=a, anchor=center, inner sep=0}, no head,shorten >=30pt,  shorten <=30pt,  from=1-2, to=3-2]
		\arrow["{(\sigma,\tau)}"{description}, Rightarrow,  curve={height=-14pt},  shorten >=4pt,  shorten <=2pt,  from=1-2,  to=a ]
		\arrow[""{name=b, anchor=center, inner sep=0}, no head,shorten >=30pt,  shorten <=30pt,  from=4-1, to=4-3]
		\arrow["{(\sigma,\tau)}"{description}, Rightarrow,  curve={height=-8pt},  shorten >=4pt,  shorten <=2pt,  from=3-2,  to=b]
		\arrow["\sigma\circ\tau"{description}, from=4-1, to=4-3]
		\arrow["\sigma"{description}, from=4-1, to=3-2]
		\arrow["\tau"{description}, from=3-2, to=4-3]
		\arrow["\sigma"{description}, from=2-1, to=1-2]
		\arrow["{\beta}"{description}, from=2-1, to=4-1]
		\arrow["\tau"{description}, from=1-2, to=2-3]
		\arrow["{\beta}"{description}, from=2-3, to=4-3]
		\arrow[""{name=0, anchor=center, inner sep=0}, shorten <=13pt, shorten >=-2.5pt, from=1-2, to=3-2]
		\arrow["{(\sigma,\beta)}"{description}, curve={height=14pt}, shorten <=6pt, shorten >=6pt, Rightarrow, from=1-2, to=4-1]
		\arrow["{(\tau,\beta)}"{description}, curve={height=-6pt}, Rightarrow, from=2-3, to=3-2]
		\arrow["{(\sigma\circ \tau,\beta)}"', curve={height=17pt},   shorten >=5pt,  shorten <=5pt,  Rightarrow, from=2-3, to=4-1]
		\arrow["{\beta}"', shorten >=5pt, no head, from=1-2, to=0]
		\arrow["\sigma\circ\tau"{description}, no head, shorten >=-5pt, from=2-1, to=0]
		\arrow[from=0, to=2-3]
	\end{tikzcd}
\qquad \quad  
\begin{tikzcd}[column sep=large, row sep=0.75cm]  
		\centering
	& {\textsf{p}[\wt{F}]} \\
	{\textsf{p}[F]} && {\textsf{p}[\wt{\wt{F}}]} \\
	& {\textsf{p}[\wt{F}']} \\
	{\textsf{p}[F']} && {\textsf{p}[\wt{\wt{F}}']} 
		\arrow[	""{name=0, anchor=center, inner sep=0}, no head,shorten >=30pt,  shorten <=30pt,  from=2-1, to=2-3]
	\arrow[""{name=a, anchor=center, inner sep=0}, no head,shorten >=30pt,  shorten <=30pt,  from=1-2, to=3-2]
	\arrow["{(\beta,\gamma)}"{description}, Rightarrow,  curve={height=-14pt},  shorten >=4pt,  shorten <=2pt,  from=1-2,  to=a ]
	\arrow[""{name=b, anchor=center, inner sep=0}, no head,shorten >=30pt,  shorten <=30pt,  from=4-1, to=4-3]
	\arrow["{(\beta,\gamma)}"{description}, Rightarrow,  curve={height=-8pt},  shorten >=6pt,  shorten <=2pt,  from=3-2,  to=b]
	\arrow["\gamma \circ \beta"{description}, from=4-1, to=4-3]
	\arrow["\beta"{description}, from=4-1, to=3-2]
	\arrow["\gamma"{description}, from=3-2, to=4-3]
	\arrow["\beta"{description}, from=2-1, to=1-2]
	\arrow["{\sigma}"{description}, from=2-1, to=4-1]
	\arrow["\gamma"{description}, from=1-2, to=2-3]
	\arrow["{\sigma}"{description}, from=2-3, to=4-3]
	\arrow[shorten <=13pt, shorten >=-2.5pt, from=1-2, to=3-2]
	\arrow["{(\sigma, \beta)}"{description}, curve={height=-14pt}, shorten <=6pt, shorten >=6pt, Rightarrow, from=4-1, to=1-2]
	\arrow["{(\sigma,\gamma)}"{description}, curve={height=6pt}, Rightarrow, from=3-2, to=2-3]
	\arrow["{(\sigma, \gamma\circ \beta)}", curve={height=-18pt},   shorten >=5pt,  shorten <=5pt,  Rightarrow, from=4-1, to=2-3]
	\arrow["{\sigma}"', shorten >=5pt, no head, from=1-2, to=0]
	\arrow["\gamma\circ\beta"{description}, no head, shorten >=-5pt, from=2-1, to=0]
	\arrow[from=0, to=2-3]
\end{tikzcd}
\end{figure}
\noindent and trivial $3$-cells \emph{$(\sigma,\tau,\upsilon)$-coherence} and \emph{$(\beta,\gamma,\delta)$-coherence}
\begin{figure}[H]
		\centering
\begin{tikzcd}[column sep=1.2cm, row sep=0.7cm] 
			\centering
	& {\textsf{p}[F']} \\
	\\
	\textsf{p}[F] && {\textsf{p}[F'']} \\
	& {\textsf{p}[F''']}
		\arrow[""{name=1, anchor=center, inner sep=0}, from=3-1, to=4-2]
			\arrow["{\!  \! \! \! \! (\sigma \circ \tau,\upsilon)}", shorten >=8pt, shorten <=5pt, curve={height=3pt}, Rightarrow, from=3-3, to=1]
		\arrow[""{name=0, anchor=center, inner sep=0}, from=3-1, to=3-3]
		\arrow[""{name=2, anchor=center, inner sep=0}, from=1-2, to=4-2]
	\arrow["\tau", from=1-2, to=3-3]
	\arrow["\upsilon", from=3-3, to=4-2]
	\arrow["{(\sigma,\tau \circ \upsilon)}"{description}, curve={height=9pt}, shorten <=7pt, shorten >=5pt, Rightarrow, from=1-2, to=1]
		\arrow["{\ \, (\sigma,\tau)}", curve={height=8pt}, shorten >=7pt,  shorten <=6pt, Rightarrow, from=1-2, to=0]
		\arrow["\sigma", from=3-1, to=1-2]
			\arrow["{(\tau,\upsilon)}"{description}, shorten >=2pt, curve={height=5pt}, Rightarrow, from=3-3, to=2]
\end{tikzcd}
\qquad \quad 
\begin{tikzcd}[column sep=1.2cm, row sep=0.7cm] 
			\centering
	& {\textsf{p}[\wt{F}]} \\
	\\
	\textsf{p}[F] && {\textsf{p}[\wt{\wt{F}}]} \\
	& {\textsf{p}[\wt{\wt{\wt{F}}}].}
	\arrow[""{name=1, anchor=center, inner sep=0}, from=3-1, to=4-2]
	\arrow["{\!  \! \! \! \! (\gamma \circ \beta,\delta)}", shorten >=8pt, shorten <=5pt, curve={height=3pt}, Rightarrow, from=3-3, to=1]
	\arrow[""{name=0, anchor=center, inner sep=0}, from=3-1, to=3-3]
	\arrow[""{name=2, anchor=center, inner sep=0}, from=1-2, to=4-2]
	\arrow["\gamma", from=1-2, to=3-3]
	\arrow["\delta", from=3-3, to=4-2]
	\arrow["{(\beta,\delta \circ \gamma)}"{description}, curve={height=9pt}, shorten <=7pt, shorten >=5pt, Rightarrow, from=1-2, to=1]
	\arrow["{\ \, (\beta ,\gamma)}", curve={height=8pt}, shorten >=7pt,  shorten <=6pt, Rightarrow, from=1-2, to=0]
	\arrow["\beta", from=3-1, to=1-2]
	\arrow["{(\gamma,\delta)}"{description}, shorten >=2pt, curve={height=5pt}, Rightarrow, from=3-3, to=2]
\end{tikzcd}
\end{figure}
\noindent  for all meaningful choices. For example, we have a $(\sigma,\tau,\beta)$-coherence trivial $3$-cell for each pair of composable non-identity bijections $K\xrightarrow{\tau} J \xrightarrow{\sigma} I$, choice of composition $F\in \Sigma[I]$, and choice of non-identity permutation $\beta\in \text{Sym}_{k}$ where $k=l(F)$.

Let us make these coherence conditions fully explicit in the case $\textsf{C}\in \{\textsf{Cat},\textsf{Cat}_\bullet\}$, at the level of components of the $2$-cells (although in this paper, all $\textsf{Cat}_\bullet$-valued species will in fact be strict). For this, the following notation will be useful. When an appropriate bijection $\sigma$ is understood and there is no ambiguity, given a (generalized) element $\ta\in \textsf{p}[F]$ in some component of a species $\textsf{p}$, we abbreviate
\[
\ta' :=    \sigma(\ta)  \, \in   \textsf{p}[F']
.\]
When composable bijections $\sigma\circ\tau$ are understood, we abbreviate $\ta'':=\sigma\circ \tau(\ta)\in   \textsf{p}[F'']$. Similarly, when an appropriate permutation $\beta$ is understood and there is no ambiguity, we abbreviate
\[
\wt{\ta} :=     \beta(\ta)  \, \in   \textsf{p}[\wt{F}]
.\]
When permutations $\beta,\gamma\in \text{Sym}_k$ are understood, we abbreviate $\tilde{\tilde{\ta}}:=\gamma \circ \beta (\ta)\in   \textsf{p}[\wt{\wt{F}}]$. 

Then, in the case $\textsf{C}\in \{\textsf{Cat},\textsf{Cat}_\bullet\}$, the $2$-cell $(\sigma,\beta)$ consists of a natural isomorphism in the category $\textsf{p}[\wt{F}']$ 
\[
{(\sigma,\beta)}_{\ta} : \wt{(\ta')} \to (\tilde{\ta})'
\]
for each object $\ta\in \textsf{p}[F]$, $(\sigma,\tau)$ consists of a natural isomorphism in $\textsf{p}[F'']$ 
\[
{(\sigma,\tau)}_{\ta} :  (\ta')' \to     \ta''
\]
for each object $\ta\in \textsf{p}[F]$, and $(\beta,\gamma)$ consists of a natural isomorphism in $\textsf{p}[\wt{\wt{F}}]$ 
\[
{(\beta,\gamma)}_{\ta} :  \wt{(\wt{\ta})} \to     \wt{\wt{\ta}}
\]
for each object $\ta\in \textsf{p}[F]$. Recall being natural isomorphisms means that e.g. in the case of ${(\sigma,\beta)}_{\ta}$, we have
\[
{(\sigma,\beta)}_{\tb}   \circ  \wt{(\mathtt{f}')}
= 
(\tilde{\mathtt{f}})'  \circ  {(\sigma,\beta)}_{\ta} 
\]
for all morphisms $\mathtt{f}:\ta \to \tb \in  \textsf{p}[F]$. Then $(\sigma,\tau,\beta)$-coherence says 
\[
{(\sigma \circ \tau, \beta)}_\ta  \circ  { \wt{         (\sigma,\tau)  }  }_{\ta}  =  {(\sigma, \tau)}_{\tilde{\ta}} \circ  (\sigma,\beta)'_{\ta}  \circ  {(\tau, \beta)}_{\ta'}
\] 
$(\beta,\gamma,\sigma)$-coherence says 
\[
 ( \sigma, \gamma\circ \beta   )_\ta \circ   (\beta ,\gamma )_{\ta'} = (\beta ,\gamma )'_{\ta} \circ (\sigma,\gamma)_{\tilde{\ta}} \circ \wt{(\sigma,\beta)}_{\ta}  
\]
$(\sigma,\tau,\upsilon)$-coherence says 
\[
( \sigma \circ \tau , \upsilon )_\ta \circ ( \sigma , \tau )'_\ta = (\sigma, \tau \circ \upsilon)_{\ta} \circ  (\tau , \upsilon)_{\ta'}
\]
and $(\beta,\gamma,\delta)$-coherence says 
\[
( \gamma \circ \beta , \delta )_\ta \circ \wt{( \beta , \gamma )}_\ta = (\beta, \delta \circ \gamma )_{\ta} \circ  (\gamma , \delta)_{\tilde{\ta}}
\]
for all meaningful choices. Note here, each instance of a prime or tilde is referring to a bijection or permutation unambiguously. For example, $(\sigma,\beta)_{\ta}$ is a morphism in the category $\textsf{p}[\wt{F}']$ and so the prime $(\sigma,\beta)'_{\ta}$ appearing in $(\sigma,\tau,\beta)$-coherence must be referring to $\tau$ and not $\sigma$.

\begin{remark}
	The name `braid arrangement species' comes from the fact that compositions $F$ are in natural one-to-one correspondence with the chambers of the braid hyperplane arrangement. Therefore we might say that $\textsf{p}$ consists of an object $\textsf{p}[F]$ for each chamber of the braid arrangement. More generally, we have the remarkable fact that species (equipped with actions of the $\beta$'s but not $\sigma$'s), and bimonoids in species, make sense with respect to any real hyperplane arrangement, or even any left-regular band \cite{aguiar2020bimonoids}. There is also a more structured theory involving reflection hyperplane arrangements \cite{AM22}.
\end{remark}

\begin{remark}\label{functors}
We may formalize $\textsf{C}$-valued braid arrangement species as certain strictly unitary pseudofunctors, a.k.a. strictly unitary strong $2$-functors, on the action groupoid for the action of bijections $\sigma$ and permutations $\beta$ on compositions. See \cite[Section 2.1.1]{AM22} for an explicit definition of this groupoid. This is similar to Joyal species, where the action groupoid $\textsf{finSet}^\times$ was for the action of just bijections on finite sets.
\end{remark} 



\begin{ex}\label{ex:exp}
The primordial $\textsf{Set}$-valued species is the \emph{exponential species} $\textsf{E}$, given by
\[
\textsf{E}[F] := \{ F  \} 
, \qquad 
\textsf{E}[\sigma, F](F) : = F'
, \qquad 
\textsf{E}[F, \beta ](F) : =  \wt{F}
.\]
Up to isomorphism, the exponential species is the unique species $\textsf{p}$ such that each component $\textsf{p}[F]$ is a singleton set. The name `exponential species' comes from the connection between species and formal power series. 
\end{ex}

\begin{ex}
We have the $\textsf{Set}$-valued species of compositions $\mathsf{\Sigma}$, given by 
\[
\mathsf{\Sigma}\big [(\, )\big] :=   1_{\textsf{Set}}
,\qquad
\mathsf{\Sigma}  [F] := \big  \{    (H_1,\dots, H_k) \ \big | \   H_i\in \Sigma[S_i]   \big   \}   
\]
where $F=(S_1,\dots, S_k)$, and
\[
\mathsf{\Sigma}[\sigma, F](H_1,\dots, H_k) := (H'_1,\dots, H'_k)
, \qquad 
\mathsf{\Sigma}[F, \beta ](H_1,\dots, H_k) :=   (H_{\beta^{-1}(1)},\dots, H_{\beta^{-1}(k)}) 
.\]
Here, $H'_i=H_i \circ \sigma|_{S'_i}$ where $S'_i= \sigma^{-1}(S_i)$. We also have a $\textsf{Cat}$-valued version of $\mathsf{\Sigma}$ by including a single morphism
\[
(K_1,\dots, K_k) \to (H_1,\dots, H_k)
\]
whenever $K_i\leq H_i$ for all $1\leq i \leq k$. See \autoref{ex:conect} for the connection to the Joyal species of compositions $\Sigma$. 
\end{ex}


Dually, we may define $\textsf{C}$-valued \emph{cospecies} $\textsf{p}$, which consist of objects $\textsf{p}[F]$, morphisms $\textsf{p}[\sigma,F]$, $\textsf{p}[F,\beta]$, and $2$-cells $\textsf{p}[\sigma,F,\beta]$, $\textsf{p}[\sigma,\tau,F]$, $\textsf{p}[F,\beta,\gamma]$ as with species, however the source and target of $\textsf{p}[\sigma,F]$ and $\textsf{p}[F,\beta]$ are switched. Thus, cospecies are covariant in the bijections $\sigma$ and contravariant in the permutations $\beta$. The coherence conditions we impose are the analogous four trivial $3$-cells.

We make this explicit. In the setting of cospecies, we adopt the following alternative notation. For $\sigma:J\to I$, $\beta\in \text{Sym}_{k}$, and $F\in \Sigma[J]$ where $l(F)=k$, we instead let
\[
F'   :=     F \circ \sigma^{-1}             
\qquad \text{and} \qquad
\wt{F}   :=     \beta^{-1}  \circ F      
.\]  
For $\textsf{C}$ a strict $2$-category, a $\textsf{C}$-valued \emph{cospecies} consists of an object $\textsf{p}[F]\in \textsf{C}$ for each composition $F$, a morphism
\[
\textsf{p}[\sigma,F']  : \textsf{p}[F] \to \textsf{p}[F'] 
\]
for each bijection of finite sets $\sigma: J\to I$ and composition $F\in \Sigma[J]$, and a morphism
\[
\textsf{p}[\wt{F},\beta]  : \textsf{p}[F] \to \textsf{p}[\wt{F}]
\]
for each permutation $\beta\in  \text{Sym}_k$, $k\in \bN$, and composition $F$ of length $k$. Then $(\sigma,\beta)$-functoriality is an invertible $2$-cell as before, however the $\sigma$-functoriality and $\beta$-functoriality $2$-cells now have the form $\textsf{p}[ \sigma,F ]\circ  \textsf{p}[ \tau,F ] \Rightarrow   \textsf{p}[   \sigma\circ \tau,F ]$ and $\textsf{p}[ \beta,F ]\circ  \textsf{p}[ \gamma,F ] \Rightarrow   \textsf{p}[ \gamma\circ\beta,F ]$ respectively, i.e. the order of composition of the image is switched. After constructing the analogous trivial $3$-cells, for $\textsf{C}\in \{\textsf{Cat},\textsf{Cat}_\bullet\}$ and at the level of components of the $2$-cells, we find that $(\sigma,\tau,\beta)$-coherence now says 
\[
{(\sigma \circ \tau, \beta)}_\ta  \circ  { \wt{         (\sigma,\tau)  }  }_{\ta}  =  {(\sigma, \tau)}_{\tilde{\ta}} \circ 
\underbrace{ (\tau,\beta)'_{\ta}  \circ  {(\sigma, \beta)}_{\ta'} }_{  \text{  $\sigma$ and $\tau$ switched  } }
\] 
and $(\beta,\gamma,\sigma)$-coherence now says  
\[
( \sigma, \gamma\circ \beta   )_\ta \circ   (\beta ,\gamma )_{\ta'} = (\beta ,\gamma )'_{\ta} \circ
\underbrace{ (\sigma,\beta)_{\tilde{\ta}} \circ \wt{(\sigma,\gamma)}_{\ta}   }_{  \text{  $\beta$ and $\gamma$ switched  } }
 .\]
Whereas $(\sigma,\tau,\upsilon)$-coherence and $(\beta,\gamma,\delta)$-coherence are unchanged.



As we shall see, cospecies arise by taking contravariant data on a species $\textsf{p}$, with the action of bijections and permutations being given by pulling back data along the corresponding actions of $\textsf{p}$. We have the following basic, discrete example of this; of contravariantly linearizing set species to obtain vector cospecies.


\begin{ex}
	Let $\textsf{p}$ be a $\textsf{Set}$-valued species with finite discrete components $\textsf{p}[F]$, probably sets of labeled combinatorial objects of a certain kind. Then we have the $\textsf{Vec}$-valued cospecies $\textbf{\textsf{p}}^\ast$ given by
	\[
	\textbf{\textsf{p}}^\ast [F] := \big \{    \text{functions} \ f:\textsf{p}[F]\to \Bbbk  \big \}
	\]
	\[
	\textbf{\textsf{p}}^\ast [\sigma, F] (f) :=    f\circ  \textsf{p}[\sigma, F]
	,\qquad 
	\textbf{\textsf{p}}^\ast [F,\beta] (f) :=    f\circ  \textsf{p}[F,\beta]
	.\] 
	Note the fact $\textbf{\textsf{p}}^\ast$ is a $\textsf{Vec}$-valued cospecies, i.e. the five functoriality $2$-cells commute, follows immediately from the fact that $\text{p}$ is a species, and that pullback of functions is functorial in the function one pulls backs along. 
\end{ex}

We shall emulate this construction, but in a continuous setting, several times in this paper. If the pullback of data is functorial only up to isomorphism, our cospecies will need to have non-trivial $2$-cells.


\subsection{Lax Morphisms of Species} 

We continue to let $\textsf{C}$ be a strict $2$-category. A \emph{lax morphism}, or just \emph{morphism}, of $\textsf{C}$-valued species $\eta: \textsf{p}\to \textsf{x}$ consists of a morphism
\[
\eta_F : \textsf{p}[F]  \to  \textsf{x}[F]  
\]
for each composition $F\in \Sigma[I]$, $I\in \textsf{finSet}^\times$, together with $2$-cells 
\[\begin{tikzcd}
	{\textsf{p}[F]} & {\textsf{p}[F']} && {\textsf{p}[F]} & {\textsf{p}[\wt{F}]} \\
	{\textsf{x}[F]} & {\textsf{x}[F']} && {\textsf{x}[F]} & {\textsf{x}[\wt{F}]}
	\arrow["\beta", from=1-4, to=1-5]
	\arrow["{\eta_F}"', from=1-4, to=2-4]
	\arrow["\beta"', from=2-4, to=2-5]
	\arrow["{\eta_{\wt{F}}}", from=1-5, to=2-5]
	\arrow["\sigma", from=1-1, to=1-2]
	\arrow["{\eta_F}"', from=1-1, to=2-1]
	\arrow["\sigma"', from=2-1, to=2-2]
	\arrow["{\eta_{F'}}", from=1-2, to=2-2]
	\arrow["\eta_{\sigma,F}", Rightarrow, from=2-1, to=1-2]
	\arrow["\eta_{F,\beta}", Rightarrow, from=2-4, to=1-5]
\end{tikzcd}\]
for each non-identity bijection $\sigma:J\to I$ and non-identity permutation $\beta\in \text{Sym}_{l(F)}$. We often abbreviate
\[
\sigma_F :=  \eta_{\sigma,F} \qquad \text{and} \qquad \beta_F:=   \eta_{F,\beta}
\]
and in the totally lumped case
\[
\eta_I:= \eta_{(I)} \qquad \text{and} \qquad  \sigma_I :=  \eta_{\sigma,(I)}\qquad \text{and} \qquad  \beta_I:=   \eta_{(I),\beta}
.\]
In addition, we require a trivial $3$-cell \emph{$(\sigma,\beta)$-coherence}
\[\begin{tikzcd}[column sep=0.9cm, row sep=0.6cm] 
	& {\textsf{p}[F]} && {\textsf{p}[\widetilde{F}]} \\
	{\textsf{p}[F']} && {\textsf{p}[\widetilde{F}']} \\
	& {\textsf{x}[F]} && {\textsf{x}[\widetilde{F}]} \\
	{\textsf{x}[F']} && {\textsf{x}[\widetilde{F}']}
		\arrow["{\beta_F}"{description}, curve={height=-20pt}, shorten <=6pt, shorten >=6pt, Rightarrow, from=3-2, to=1-4]
								\arrow["{ (\sigma,\beta)  }"{description}, curve={height=3pt}, shorten >=6pt, shorten <=15pt, Rightarrow, from=4-1, to=3-4]
	\arrow["\beta"{description}, from=4-1, to=4-3]
	\arrow["\sigma"{description}, from=1-2, to=2-1]
	\arrow["\sigma"{description}, from=1-4, to=2-3]
	\arrow["\sigma"{description}, from=3-4, to=4-3]
	\arrow["\sigma"{description}, from=3-2, to=4-1]
	\arrow["{\eta_{\widetilde{F}}}"{description}, from=1-4, to=3-4]
	\arrow["\beta"{description}, from=1-2, to=1-4]
	\arrow[""{name=0, anchor=center, inner sep=0}, shorten <=13pt, from=2-3, to=4-3]
	\arrow[""{name=1, anchor=center, inner sep=0}, shorten <=13pt, from=1-2, to=3-2]
	\arrow[""{name=2, anchor=center, inner sep=0}, shorten <=19pt, from=3-2, to=3-4]
	\arrow[""{name=3, anchor=center, inner sep=0}, shorten <=18pt, from=2-1, to=2-3]
	\arrow["{\sigma_F}"{description}, curve={height=7pt}, Rightarrow, from=3-2, to=2-1]
	\arrow["{\sigma_{\widetilde{F}}}"{description},   curve={height=7pt}, Rightarrow, from=3-4, to=2-3]
	\arrow["{\eta_{\widetilde{F}'}}"{description}, shorten >=6pt, no head, from=2-3, to=0]
	\arrow["{\eta_F}", shorten >=5pt, no head, from=1-2, to=1]
	\arrow["\beta"{description}, shorten >=9pt, no head, from=3-2, to=2]
	\arrow["\beta"{description}, shorten >=9pt, no head, from=2-1, to=3]
	\arrow["{\eta_{F'}}"{description}, from=2-1, to=4-1]
		\arrow["{\beta_{F'}}"{description}, curve={height=18pt},  Rightarrow, from=4-1, to=2-3]
			\arrow["{ (\sigma,\beta)  }"{description}, curve={height=-9pt}, shorten >=18pt, Rightarrow, from=2-1, to=1-4]
\end{tikzcd}\]
and trivial $3$-cells \emph{$\sigma$-coherence} and \emph{$\beta$-coherence}
\begin{figure}[H]
	\centering
	\begin{tikzcd}[column sep=1.3cm, row sep=0.8cm] 
		\centering
	& {\textsf{p}[F']} \\
{\textsf{p}[F]} && {\textsf{p}[F'']}  \\
& {\textsf{x}[F']} \\
{\textsf{x}[F]} && {\textsf{x}[F'']} 
\arrow["\sigma\circ\tau"', from=4-1, to=4-3]
\arrow["\sigma"{description}, from=4-1, to=3-2]
\arrow["\tau"{description}, from=3-2, to=4-3]
\arrow["\sigma"{description}, from=2-1, to=1-2]
\arrow["{\eta_F}"{description}, from=2-1, to=4-1]
\arrow["\tau"{description}, from=1-2, to=2-3]
\arrow["{\eta_{F''}}"{description}, from=2-3, to=4-3]
\arrow[""{name=0, anchor=center, inner sep=0}, shorten <=13pt, shorten >=-1pt, from=1-2, to=3-2]
\arrow[""{name=a, anchor=center, inner sep=0}, shorten <=13pt, from=4-1, to=4-3]
	\arrow["{ (\sigma,\tau)  }"{description}, curve={height=-12pt}, shorten >=3pt,  shorten <=0pt, Rightarrow, from=1-2, to=0]
			\arrow["{ (\sigma,\tau)  }"{description}, curve={height=-12pt}, shorten >=3pt,  shorten <=0pt, Rightarrow, from=3-2, to=a]
\arrow["{\sigma_F}"{description}, curve={height=-11pt}, shorten <=6pt, shorten >=6pt, Rightarrow, from=4-1, to=1-2]
\arrow["{\tau_{F'}}"{description}, curve={height=6pt}, Rightarrow, from=3-2, to=2-3]
\arrow["{(\sigma\circ \tau)_{F}}", curve={height=-18pt},   shorten >=5pt,  shorten <=5pt,  Rightarrow, from=4-1, to=2-3]
\arrow["{\eta_{F'}}"', shorten >=5pt, no head, from=1-2, to=0]
\arrow["\sigma\circ\tau"{description}, no head, shorten >=-5pt, from=2-1, to=0]
\arrow[from=0, to=2-3]
	\end{tikzcd}
	\qquad \quad  
	\begin{tikzcd}[column sep=1.3cm, row sep=0.8cm]    
		\centering
	& {\textsf{p}[\wt{F}]} \\
{\textsf{p}[F]} && {\textsf{p}[\wt{\wt{F}}]} \\
& {\textsf{x}[\wt{F}]} \\
{\textsf{x}[F]} && {\textsf{x}[\wt{\wt{F}}]}
\arrow["{\gamma \circ \beta}"', from=4-1, to=4-3]
\arrow["\beta"{description}, from=4-1, to=3-2]
\arrow["\gamma"{description}, from=3-2, to=4-3]
\arrow["\beta"{description}, from=2-1, to=1-2]
\arrow["{\eta_F}"{description}, from=2-1, to=4-1]
\arrow["\gamma"{description}, from=1-2, to=2-3]
\arrow["{\eta_{\wt{\wt{F}}}}"{description}, from=2-3, to=4-3]
\arrow[""{name=0, anchor=center, inner sep=0}, shorten <=13pt, from=2-1, to=2-3]
\arrow[""{name=a, anchor=center, inner sep=0}, shorten <=13pt, from=4-1, to=4-3]
	\arrow["{ (\beta,\gamma)  }"{description}, curve={height=-12pt}, shorten >=3pt,  shorten <=0pt, Rightarrow, from=1-2, to=0]
		\arrow["{ (\beta,\gamma)  }"{description}, curve={height=-12pt}, shorten >=3pt,  shorten <=0pt, Rightarrow, from=3-2, to=a]
\arrow["{\eta_{\wt{F}}}"', shorten >=-35pt,  from=1-2, to=0]
\arrow["{\beta_F}"{description}, curve={height=-12pt}, shorten <=6pt, shorten >=6pt, Rightarrow, from=4-1, to=1-2]
\arrow["{\gamma_{\wt{F}}}"{description}, curve={height=6pt}, Rightarrow, from=3-2, to=2-3]
\arrow["{(\gamma\circ\beta)_{F}}", curve={height=-20pt},   shorten >=5pt,  shorten <=5pt,  Rightarrow, from=4-1, to=2-3]
\arrow["{\gamma \circ \beta}"{description}, from=2-1,no head, to=0]
\arrow[from=0, to=2-3]
	\end{tikzcd}
\end{figure}
\noindent for all meaningful choices. We have a well-defined composition of morphisms by composing \hbox{$2$-cells}. A morphism is called \emph{strong}, respectively \emph{strict}, if all the $2$-cells $\sigma_F$ and $\beta_F$ are isomorphisms, respectively identities. If $\textsf{C}$ is a $1$-category, then all morphisms are strict.  

In the case $\textsf{C}\in \{\textsf{Cat} ,  \textsf{Cat}_\bullet\}$, let us make morphisms explicit at the level of components of the $2$-cells. We have that $\sigma_F$ consists of a natural morphism in $\textsf{x}[F']$
\[
{\sigma}_\ta : \big ( \eta_F(\ta)\big )'   \to   \eta_F(\ta')    
\] 
for each object $\ta\in \textsf{p}[F]$, and $\beta_F$ consists of a natural morphism in $\textsf{x}[\wt{F}]$
\[
\beta_\ta   :  \wt{\eta_F(\ta)}   \to   \eta_F(\tilde{\ta})  
\]
for each object $\ta\in \textsf{p}[F]$. Then $(\sigma,\beta)$-coherence says
\[
 \eta_{\wt{F}'} \big ( (\sigma,\beta)_{\ta} \big ) \circ  \beta_{\ta'}   \circ \wt{\sigma}_\ta  
 =
{\sigma}_{\tilde{\ta}}  \circ \beta'_\ta  \circ  (\sigma,\beta)_{\eta_F(\ta)}
\]
$\sigma$-coherence says
\[
 {{(\sigma \circ \tau)}}_\ta \circ (\sigma,\tau)_{\eta_{F}(\ta)} =   \eta_{F''}\big ((\sigma,\tau)_{\ta}\big )  \circ  {\tau}_{\ta'}   \circ  {\sigma}'_\ta 
\]
and $\beta$-coherence says
\[
 {{(\gamma \circ \beta )}}_\ta   \circ (\beta,\gamma)_{\eta_{F}(\ta)}  =\eta_{\wt{\wt{F}}}\big ((\beta,\gamma)_{\ta}\big )  \circ  {\gamma  }_{\tilde{\ta}}   \circ  \wt{\beta}_\ta 
.\]

	
\begin{remark}
If we formalize species as functors as in \autoref{functors}, lax morphisms of species are lax natural transformations of functors.
\end{remark}

The definition of morphisms for cospecies is analogous, with the same $2$-cells. Then \hbox{$(\sigma,\beta)$-coherence} is unchanged, however $\sigma$-coherence becomes
\[
(\sigma \circ \tau)_\ta \circ (\sigma,\tau)_{\eta_F(\ta)} =   \eta_{F''}\big ( (\sigma,\tau)_{\ta}\big)  \circ   \! \! \! \! \! \! \!  \underbrace{ \sigma_{\ta'} \circ \tau'_{\ta} }_{ \text{$\sigma$ and $\tau$ switched}  }
\]
and $\beta$-coherence becomes
\[
 {{(\gamma \circ \beta )}}_\ta   \circ (\beta,\gamma)_{\eta_{F}(\ta)}  =\eta_{\wt{\wt{F}}}\big ((\beta,\gamma)_{\ta}\big )  \circ  \! \! \! \! \! \! \!    \underbrace{ {\beta  }_{\tilde{\ta}}   \circ  \wt{\gamma}_\ta }_{ \text{$\beta$ and $\gamma$ switched}  }
  \! \! \! \! \! \! \! .\]

\subsection{Joyal-Theoretic Species}  \label{sec:Joyal-Theoretic Species}


In this section, we describe the connection between connected Joyal species and braid arrangement species. In particular, we show the sense in which braid arrangement species are a generalization of connected Joyal species. 

We now suppose that $\textsf{C}$ has the additional structure of a symmetric monoidal \hbox{$2$-category} $\textsf{C}=(\textsf{C},\otimes,1_\textsf{C})$. In this paper, recall we shall have
\[
\textsf{C} \in \{  \textsf{Set}, \textsf{Vec}, \textsf{Cat}, \textsf{Cat}_\bullet      \}
.\]
These are all symmetric monoidal categories as follows. We let $\textsf{Set}$ and $\textsf{Cat}$ be cartesian categories, $\textsf{Vec}$ is equipped with the tensor product of vector spaces, and $\textsf{Cat}_\bullet$ is equipped with the smash product for the cartesian product of categories, see e.g. \cite[Construction 4.19]{MR2558315}. Therefore for the purposes of this paper, we may assume that the monoidal product $\otimes : \textsf{C}\times \textsf{C}\to \textsf{C}$ is a strict $2$-functor, and that the associators, unitors and braiding are all strict 2-natural transformations. 

Given a connected $\textsf{C}$-valued Joyal species $\text{p}$, we have an associated $\textsf{C}$-valued species $\textsf{p}=\textsf{Br}(\text{p})$ given by 
\[
\textsf{p} \big  [(\, )\big]: =  1_\textsf{C} 
, \qquad
\textsf{p} \big  [\text{id}_{\emptyset},(\, )\big ]: =  \text{id}_{1_\textsf{C}}, 
\]
\[
\textsf{p}  [F]: = \text{p}[S_1] \otimes \dots \otimes \text{p}[S_k]
,\qquad 
\textsf{p} [\sigma, F]: = \text{p}[\sigma|_{S'_1}] \otimes \dots \otimes \text{p}[\sigma|_{S'_k}]
\]
and the action of permutations $\beta$, which are then morphisms of the form
\[
\textsf{p} [F,\beta]   :   \text{p}[S_1] \otimes \dots \otimes \text{p}[S_k]\,  \to   \,  \text{p}[S_{\beta^{-1}  (1) }] \otimes \dots \otimes \text{p}[S_{ \beta^{-1}  (k) }]
,\] 
are the isomorphisms given by the braiding of $\textsf{C}$, i.e. permuting tensor factors. If $\textsf{C}$ is just a $1$-category, we are done. Otherwise we need functoriality $2$-cells, which are given by
\[
\textsf{p}[\sigma , F, \beta ] := \text{id}
,\qquad 
\textsf{p}[\sigma ,\tau, F ] :=    \text{p}[\sigma|_{S'_1} ,\tau|_{S''_1}] \otimes  \dots \otimes  \text{p}[\sigma|_{S'_k} ,\tau|_{S''_k}]
,\qquad 
\textsf{p}[F, \beta ,\gamma] :=  \text{id}
\]
where $\text{p}[\sigma, \tau]$ is the compositor $2$-cell $\text{p}[\tau] \circ \text{p}[\sigma] \Rightarrow \text{p}[\sigma \circ \tau]$ of $\text{p}$. 

Going in the opposite direction, given a $\textsf{C}$-valued species $\textsf{p}$, we have the connected $\textsf{C}$-valued Joyal species $\text{p}=\text{Joy}(\textsf{p})$ obtained by restricting its structure to the totally lumped components $F=(I)$, and thus throwing away the action of the $\beta$'s. 

\begin{remark}
Both $\textsf{Br}$ and $\text{Joy}$ extend to functors between Joyal species and braid arrangement species, however they will not feature in this paper.
\end{remark}

The \emph{Joyalization} $\breve{\textsf{p}}:=  \textsf{Br} (\text{Joy}  (\textsf{p}))$ of a species $\textsf{p}$ is given by 
\[
\breve{\textsf{p}} \big  [(\, )\big]: =  1_\textsf{C} 
, \qquad
\breve{\textsf{p}} \big  [\text{id}_{\emptyset},(\, )\big ]: =  \text{id}_{1_\textsf{C}}, 
\]
\[
\breve{\textsf{p}}  [F]: = \textsf{p}[S_1] \otimes \dots \otimes \textsf{p}[S_k]
,\qquad 
\breve{\textsf{p}} [\sigma, F]: = \textsf{p}[\sigma|_{S'_1}] \otimes \dots \otimes \textsf{p}[\sigma|_{S'_k}]
\]
\[
\breve{\textsf{p}}[\sigma , F, \beta ] := \text{id}
,\qquad 
\breve{\textsf{p}}[\sigma ,\tau, F ] :=    \textsf{p}[\sigma|_{S'_1} ,\tau|_{S''_1},  (S_1)] \otimes  \dots \otimes  \textsf{p}[\sigma|_{S'_k} ,\tau|_{S''_k}, (S_k)]
,\qquad 
\breve{\textsf{p}}[F, \beta ,\gamma] :=  \text{id}
\]
with the action of the $\beta$ given by permuting tensor factors. A \emph{weakly Joyal-theoretic species} is a species $\textsf{p}$ equipped with a morphism `$\text{prod}$' of the form
\[
\text{prod}:\breve{\textsf{p}}   \to \textsf{p}
\]
such that $\text{prod}_{I}= \text{id}_{\textsf{p}[I]}$. As we shall see, weakly \hbox{Joyal-theoretic} species provide a bridge between connected Joyal species and braid arrangement species. If $\text{prod}$ is an isomorphism, we say $\textsf{p}$ is \emph{strongly Joyal-theoretic}. If $\text{prod}$ is an identity map, we say $\textsf{p}$ is \emph{strictly Joyal-theoretic}, abbreviated \emph{Joyal-theoretic} in this paper.\footnote{\ this breaks the principle of equivalence, but will nevertheless be useful} (Equivalently, a Joyal-theoretic species is a species which is in the strict image of the functor $\textsf{Br}$, equivalently a species which is equal to its own Joyalization.)

\begin{remark}
	The components of $\text{prod}$ are morphisms of the form
	\[
	\text{prod}_F : \textsf{p}[S_1] \otimes \dots \otimes   \textsf{p}[S_k] \to \textsf{p}[F]
	.\]
	As we shall see, $\textsf{p}[F]$ is often some intractable, ambient object, and the morphism $\text{prod}_F$ is picking out those `semisimple' elements which factorize.
\end{remark}

\begin{ex}\label{ex:conect}
	The $\textsf{Cat}$-valued species of compositions $\mathsf{\Sigma}$ is Joyal-theoretic since
	\[
	\mathsf{\Sigma} = \textsf{Br}(\Sigma)
	.\]
	In particular, 
	\[
		\mathsf{\Sigma}[F] = \Sigma[S_1] \times \dots \times \Sigma[S_k] 
		.\]
\end{ex}

A morphism $\eta:\textsf{p}\to \textsf{q}$ of Joyal-theoretic species is called \emph{Joyal-theoretic} if it similarly factorizes, that is
\[
\eta_{(\, )} = \text{id}_{1_\textsf{C}}
,\qquad
\eta_F =  \eta_{(S_1)}\otimes \dots \otimes \eta_{(S_k)},
\]
\[
\sigma_F =  \sigma_{(S_1)} \otimes \dots \otimes   \sigma_{(S_k)}
,\qquad
\beta_F = \text{id}
.\]
If $\eta$ is Joyal-theoretic, then to check coherence it is enough to check $\sigma$-coherence in the totally lumped case $F=(I)$ only. We show this in the case $\textsf{C}=\textsf{Cat}$ at the level of components of the $2$-cells, with the general case following similarly. For $\sigma$-coherence, given $\ta=(\ta_1,\dots, \ta_k)\in \textsf{p}[F]$, we have
\[
{{(\sigma \circ \tau)}}_{\ta} \circ (\sigma,\tau)_{\eta_{F}(\ta)} 
=
\Big( \bigtimes_{1\leq i \leq k} {{(\sigma|_{S'_i} \circ \tau|_{S''_i})}}_{\ta_i}     \Big)  \circ   \Big( \bigtimes_{1\leq i \leq k}   (\sigma|_{S'_i},\tau|_{S''_i})_{\eta_{(S_i)}(\ta_i)}      \Big) 
\]
\[
=
 \underbrace{\bigtimes_{1\leq i \leq k} \Big (  {{(\sigma|_{S'_i} \circ \tau|_{S''_i})}}_{\ta_i}    \circ     (\sigma|_{S'_i},\tau|_{S''_i})_{\eta_{(S_i)}(\ta_i)}     \Big )
 =
\bigtimes_{1\leq i \leq k}        \Big (   
\eta_{(S''_i)}\big ((\sigma|_{S'_i},\tau|_{S''_i})_{\ta_i}\big )  \circ  {\tau|_{S''_i}}_{\ta'_i}   \circ ( {\sigma|_{S'_i}}_{\ta_i})' 
  \Big )}_{\text{by $\sigma$-coherence in the case $F=(I)$}}
  \]
  \[
=
\Big (\bigtimes_{1\leq i \leq k} \eta_{(S''_i)}\big ((\sigma|_{S'_i},\tau|_{S''_i})_{\ta_i}\big )   \Big ) 
 \circ 
 \Big ( \bigtimes_{1\leq i \leq k}  {\tau|_{S''_i}}_{\ta'_i}      \Big ) 
 \circ
  \Big ( \bigtimes_{1\leq i \leq k}  ({\sigma|_{S'_i}}_{\ta_i})'   \Big ) 
=
 \eta_{F''}\big ((\sigma,\tau)_{\ta}\big )  \circ  {\tau}_{\ta'}   \circ  {\sigma}'_\ta 
.\]
Whereas $(\sigma,\beta)$-coherence and $\beta$-coherence are trivially satisfied.

A morphism $\eta:\textsf{p}\to \textsf{q}$ of strongly Joyal-theoretic species is called \emph{strongly \hbox{Joyal-theoretic}} if the morphism
\[
\text{prod}_\textsf{q}^{-1}\circ \eta \circ\text{prod}_\textsf{p} : \breve{\textsf{p}} \to  \breve{\textsf{q}} 
\]
is Joyal-theoretic, i.e. if $\eta$ is Joyal-theoretic up to specified isomorphism.

We leave implicit the constructions and definitions for cospecies which are analogous to those of this section.

\section{Bimonoids}\label{sec2}

\subsection{Bimonoids in Braid Arrangement Species}\label{Bimonoids in Braid Arrangement Species}

Given compositions $F=(S_1, \dots , S_{k})$ and $G=(T_1,\dots, T_{l})$ of the same finite set $I$, their \emph{Tits product} $FG$ is the composition of $I$ obtained by restricting $G$ to each of the lumps of $F$, and then concatenating,
\[ 
F G:=  G|_{S_1} ;\,  \cdots \,  ; G|_{S_k}
.\]
This gives the set $\Sigma[I]$ the structure of a monoid. The identity is $(I)$ when $I\neq \emptyset$, and $(\, )$ when $I=\emptyset$. The theory of bimonoids in braid arrangement species depends crucially on this monoidal structure. If $G\leq F$, then for a given permutation $\beta\in \text{Sym}_{l}$ we let $\hat{\beta}\in \text{Sym}_k$ such that
\[
\hat{\beta} \circ F =\widetilde{G}    F 
\]
where $\wt{G}=\beta \circ G$ and $\widetilde{G}    F$ is the Tits product of $\widetilde{G}$ with $F$ ($\widetilde{G}    F$ and $F$ do indeed have the same lumps because $G\leq F$).
 
Let $\textsf{C}$ be a strict $2$-category. A $\textsf{C}$-valued (strong, strictly (co)unital) \emph{bimonoid in species} $\textsf{h}$ is a $\textsf{C}$-valued species such that for each pair of compositions $F,G$ with $G\leq F$, we have a pair of morphisms
\[
\mu_{F,G} : \textsf{h}[F] \to  \textsf{h}[G]
\qquad \text{and} \qquad
\Delta_{F,G} : \textsf{h}[G] \to  \textsf{h}[F]
\]  
called the multiplication and comultiplication respectively. We often abbreviate
\begin{equation} \label{eq:F} 
	\mu_{F}:=  \mu_{F,(I)} \qquad \text{and} \qquad \Delta_{F}:=  \Delta_{F,(I)}.
\end{equation}
We require an invertible $2$-cell \emph{multiplication $\sigma$-naturality} and an invertible $2$-cell \emph{multiplication $\beta$-naturality} 
\[\begin{tikzcd}[row sep=large]
	{\textsf{h}[F]} & {\textsf{h}[F']} && {\textsf{h}[F]} & {\textsf{h}[\widetilde{G}    F]} \\
	{\textsf{h}[G]} & {\textsf{h}[G']} && {\textsf{h}[G]} & {\textsf{h}[\widetilde{G}]}
	\arrow["\sigma", from=1-1, to=1-2]
	\arrow["{\mu_{F,G}}"', from=1-1, to=2-1]
	\arrow["\sigma"', from=2-1, to=2-2]
	\arrow["{\mu_{F',G'}}", from=1-2, to=2-2]
	\arrow["{\hat{\beta}}", from=1-4, to=1-5]
	\arrow["{\mu_{F,G}}"', from=1-4, to=2-4]
	\arrow["{\mu_{\widetilde{G}    F,\widetilde{G}}}", from=1-5, to=2-5]
	\arrow["\beta"', from=2-4, to=2-5]
	\arrow["{\mu_{F,G,\sigma}}"{description}, Rightarrow, from=2-1, to=1-2]
	\arrow["{\mu_{F,G,\beta}}"{description},  Rightarrow, from=2-4, to=1-5]
\end{tikzcd}\]
an invertible $2$-cell \emph{associativity} and a trivial $2$-cell \emph{strict unitality}  
\[\begin{tikzcd}[row sep=large]
	& {\textsf{h}[G]} \\
	{\textsf{h}[F]} && {\textsf{h}[E]} && {\textsf{h}[F]} & {\textsf{h}[F]}
	\arrow["{\mu_{F,G}}", from=2-1, to=1-2]
	\arrow["{\mu_{G,E}}", from=1-2, to=2-3]
	\arrow[""{name=0, anchor=center, inner sep=0}, "{\mu_{F,E}}"', from=2-1, to=2-3]
	\arrow["{\text{id}}", curve={height=-10pt}, from=2-5, to=2-6]
	\arrow["{\mu_{F,F}}"', curve={height=10pt}, from=2-5, to=2-6]
	\arrow["{\mu_{F,G,E}}"{description}, Rightarrow, shorten >=4pt, shorten <=3pt, from=1-2, to=0]
\end{tikzcd}\]
an invertible $2$-cell \emph{comultiplication $\sigma$-naturality} and an invertible $2$-cell \emph{comultiplication} \hbox{\emph{$\beta$-naturality}} 
\[\begin{tikzcd}[row sep=large]
	{\textsf{h}[F]} & {\textsf{h}[F']} && {\textsf{h}[F]} & {\textsf{h}[\widetilde{G}    F]} \\
	{\textsf{h}[G]} & {\textsf{h}[G']} && {\textsf{h}[G]} & {\textsf{h}[\widetilde{G}]}
	\arrow["\sigma", from=1-1, to=1-2]
	\arrow["{\Delta_{F,G}}", from=2-1, to=1-1]
	\arrow["\sigma"', from=2-1, to=2-2]
	\arrow["{\Delta_{F',G'}}"', from=2-2, to=1-2]
	\arrow["{\hat{\beta}}", from=1-4, to=1-5]
	\arrow["{\Delta_{F,G}}", from=2-4, to=1-4]
	\arrow["{\Delta_{\widetilde{G}    F,\widetilde{G}}}"', from=2-5, to=1-5]
	\arrow["\beta"', from=2-4, to=2-5]
		\arrow["{\Delta_{F,G,\sigma}}"{description}, Rightarrow, from=1-1, to=2-2]
		\arrow["{\Delta_{F,G,\beta}}"{description}, Rightarrow, from=1-4, to=2-5]
\end{tikzcd}\]
an invertible $2$-cell \emph{coassociativity} and a trivial $2$-cell \emph{strict counitality}
\[\begin{tikzcd}[row sep=large]
	& {\textsf{h}[G]} \\
	{\textsf{h}[F]} && {\textsf{h}[E]} && {\textsf{h}[F]} & {\textsf{h}[F]}
	\arrow["{\Delta_{F,G}}"', from=1-2, to=2-1]
	\arrow["{\Delta_{G,E}}"', from=2-3, to=1-2]
	\arrow[     ""{name=0, anchor=center, inner sep=0},       "{\Delta_{F,E}}", from=2-3, to=2-1]
	\arrow["{\text{id}}"', curve={height=10pt}, from=2-6, to=2-5]
	\arrow["{\Delta_{F,F}}", curve={height=-10pt}, from=2-6, to=2-5]
		\arrow["{\Delta_{F,G,E}}"{description}, Rightarrow, shorten >=4pt, shorten <=3pt, from=1-2, to=0]
\end{tikzcd}\]
and an invertible $2$-cell \emph{bimonoid axiom}
\[\begin{tikzcd} [row sep=large]
	{\textsf{h}[F]} & {\textsf{h}[A]} & {\textsf{h}[G]} \\
\textsf{h}[F  G] && \textsf{h}[G  F]
	\arrow["{\mu_{F,A}}", from=1-1, to=1-2]
	\arrow["{\Delta_{G,A}}", from=1-2, to=1-3]
	\arrow["{\Delta_{F   G,F}}"', from=1-1, to=2-1]
	\arrow[ ""{name=0, anchor=center, inner sep=0}, "{\beta}"', from=2-1, to=2-3]
	\arrow["{\mu_{{G   F},G}}"', from=2-3, to=1-3]
	\arrow["{\text{B}_{F,G,A}}"{description}, shorten >=1pt, shorten <=5pt, Rightarrow, from=0, to=1-2]
\end{tikzcd}\]
for all meaningful choices, with the requirement that $F\neq G$, $\sigma\neq \text{id}_I$, \hbox{$\beta\neq \text{id}_{(k)}$} in (co)multiplication naturality and (co)associativity, and $F,G\neq A$ in the bimonoid axiom (but $F=G$ is allowed).\footnote{\ as with species, we could allow these choices and then add the coherence conditions that the resulting $2$-cells are all trivial} Thus, we have a bimonoid axiom $2$-cell $\text{B}_{F,G,A}$ for each choice of compositions $F,G,A$ such that $A< F,G$. Notice $FG$ and $GF$ always have the same lumps. If all the $2$-cells are trivial, i.e. the above nine diagrams commute on the nose, then the bimonoid is called \emph{strict}. If $\textsf{C}$ is a $1$-category, then a $\textsf{C}$-valued bimonoid must be strict.

\begin{ex}\label{ex:expbimon}
	The exponential species $\textsf{E}$ is a strict bimonoid, with multiplication and comultiplication given by
	\[
	\mu_{F,G} (F) =G
	\qquad \text{and} \qquad
	\Delta_{F,G} (G)= F
	\]
	respectively.
\end{ex}

We impose eleven coherence conditions on these $2$-cells, which are as follows. We require trivial $3$-cells \emph{associativity \hbox{$\sigma$-coherence}} and \emph{associativity $\beta$-coherence}
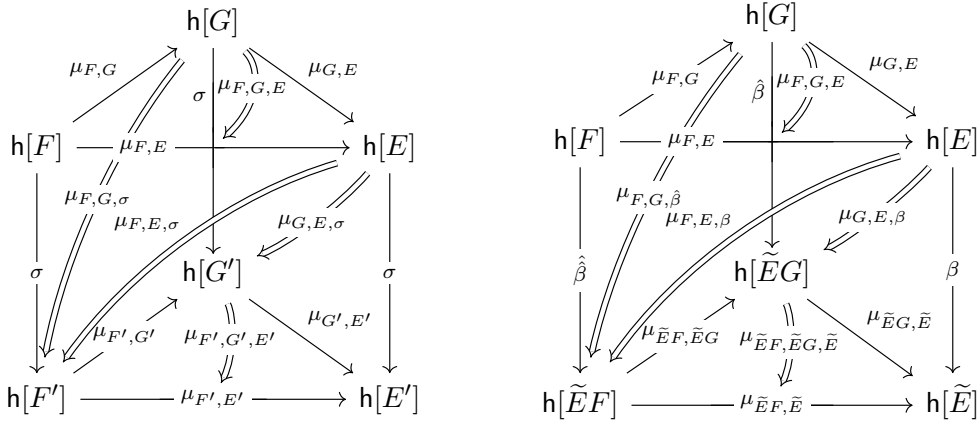
\begin{figure}[H]
	\centering
	\begin{tikzcd}[column sep=1.2cm, row sep=1cm]    
		\centering
		& {\textsf{h}[G]} \\
		{\textsf{h}[F]} && {\textsf{h}[E]} \\
		& {\textsf{h}[G']} \\
		{\textsf{h}[F']} && {\textsf{h}[E']} 
		\arrow[""{name=a, anchor=center, inner sep=0}, no head,shorten >=30pt,  shorten <=30pt,  from=1-2, to=3-2]
		\arrow["{\mu_{F,G,E}}"{description}, Rightarrow,  curve={height=-18pt},  shorten >=4pt,  shorten <=2pt,  from=1-2,  to=a ]
		\arrow[""{name=b, anchor=center, inner sep=0}, no head,shorten >=30pt,  shorten <=30pt,  from=4-1, to=4-3]
		\arrow["{ \mu_{F',G',E'} }"{description}, Rightarrow,  curve={height=-8pt},  shorten >=6pt,  shorten <=2pt,  from=3-2,  to=b]
		\arrow["\mu_{F',E'}"{description}, from=4-1, to=4-3]
		\arrow["\mu_{F',G'}"{description}, from=4-1, to=3-2]
		\arrow["\mu_{G',E'}", from=3-2, to=4-3]
		\arrow["\mu_{F,G}", from=2-1, to=1-2]
		\arrow["{\sigma}"{description}, from=2-1, to=4-1]
		\arrow["\mu_{G,E}", from=1-2, to=2-3]
		\arrow["{\sigma}"{description}, from=2-3, to=4-3]
		\arrow[""{name=0, anchor=center, inner sep=0}, shorten <=13pt, shorten >=-2.5pt, from=1-2, to=3-2]
		\arrow["{\mu_{F,G,\sigma}}"{description}, curve={height=14pt}, shorten <=6pt, shorten >=6pt, Rightarrow, from=1-2, to=4-1]
		\arrow["{\mu_{G,E,\sigma}}"{description}, curve={height=-7pt}, Rightarrow, from=2-3, to=3-2]
		\arrow["{  \mu_{F,E,\sigma} }"', curve={height=20pt},   shorten >=5pt,  shorten <=5pt,  Rightarrow, from=2-3, to=4-1]
		\arrow["{\sigma}"', shorten >=5pt, no head, from=1-2, to=0]
		\arrow["\mu_{F,E}"{description}, no head, shorten >=-5pt, from=2-1, to=0]
		\arrow[from=0, to=2-3]
	\end{tikzcd}
	\qquad \quad  
	\begin{tikzcd}[column sep=1.2cm, row sep=1cm]    
		\centering
		& {\textsf{h}[G]} \\
		{\textsf{h}[F]} && {\textsf{h}[E]} \\
		& {\textsf{h}[\wt{E}G]} \\
		{\textsf{h}[\wt{E}F]} && {\textsf{h}[\wt{E}]} 
		\arrow[	""{name=0, anchor=center, inner sep=0}, no head,shorten >=30pt,  shorten <=30pt,  from=2-1, to=2-3]
		\arrow[""{name=a, anchor=center, inner sep=0}, no head,shorten >=30pt,  shorten <=30pt,  from=1-2, to=3-2]
		\arrow["{\mu_{F,G,E}}"{description}, Rightarrow,  curve={height=-18pt},  shorten >=4pt,  shorten <=2pt,  from=1-2,  to=a ]
		\arrow[""{name=b, anchor=center, inner sep=0}, no head,shorten >=30pt,  shorten <=30pt,  from=4-1, to=4-3]
		\arrow["{\mu_{\wt{E}F,\wt{E}G,\wt{E}}}"{description}, Rightarrow,  curve={height=-8pt},  shorten >=7pt,  shorten <=2pt,  from=3-2,  to=b]
		\arrow["\mu_{\wt{E}F,\wt{E}}"{description}, from=4-1, to=4-3]
		\arrow["\mu_{\wt{E}F,\wt{E}G}"{description}, from=4-1, to=3-2]
		\arrow["\mu_{\wt{E}G,\wt{E}}", from=3-2, to=4-3]
		\arrow["\mu_{F,G}"{description}, from=2-1, to=1-2]
		\arrow["{\hat{\hat{\beta}}}"{description}, from=2-1, to=4-1]
		\arrow["\mu_{G,E}", from=1-2, to=2-3]
		\arrow["{\beta}"{description}, from=2-3, to=4-3]
		\arrow[shorten <=13pt, shorten >=-2.5pt, from=1-2, to=3-2]
		\arrow["{\mu_{F,G,\hat{\beta}}}"{description}, curve={height=14pt}, shorten <=6pt, shorten >=6pt, Rightarrow, from=1-2, to=4-1]
		\arrow["{\mu_{G,E,\beta}}"{description}, curve={height=-6pt}, Rightarrow, from=2-3, to=3-2]
		\arrow["{\mu_{F,E,\beta}}"', curve={height=22pt},   shorten >=5pt,  shorten <=5pt,  Rightarrow, from=2-3, to=4-1]
		\arrow["{\hat{\beta}}"', shorten >=5pt, no head, from=1-2, to=a]
		\arrow["\mu_{F,E}"{description}, no head, shorten >=-5pt, from=2-1, to=0]
		\arrow[from=0, to=2-3]
	\end{tikzcd}
\end{figure}
\noindent trivial $3$-cells \emph{$\sigma$-functoriality multiplication coherence} and \emph{$\beta$-functoriality multiplication coherence}
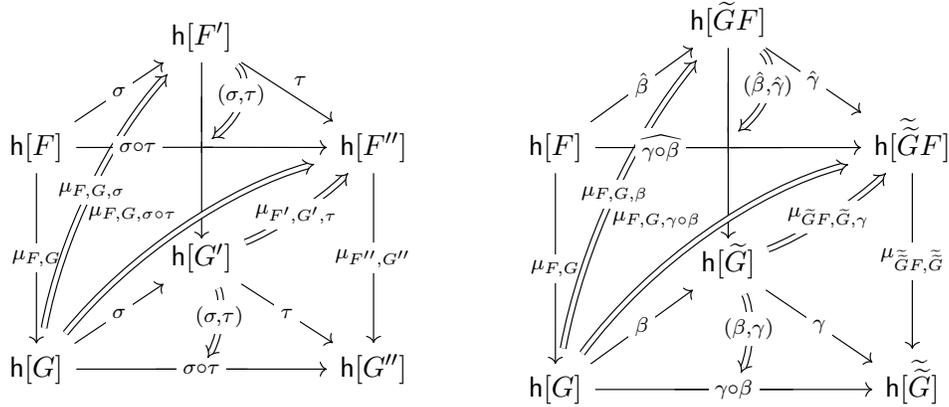
\begin{figure}[H]
	\centering
	\begin{tikzcd}[column sep=1.1cm, row sep=0.8cm]  
		\centering
		& {\textsf{h}[F']} \\
		{\textsf{h}[F]} && {\textsf{h}[F'']} \\
		& {\textsf{h}[G']} \\
		{\textsf{h}[G]} && {\textsf{h}[G'']} 
		\arrow[""{name=a, anchor=center, inner sep=0}, no head,shorten >=30pt,  shorten <=30pt,  from=1-2, to=3-2]
		\arrow["{(\sigma,\tau)}"{description}, Rightarrow,  curve={height=-18pt},  shorten >=4pt,  shorten <=2pt,  from=1-2,  to=a ]
		\arrow[""{name=b, anchor=center, inner sep=0}, no head,shorten >=30pt,  shorten <=30pt,  from=4-1, to=4-3]
		\arrow["{(\sigma,\tau)}"{description}, Rightarrow,  curve={height=-8pt},  shorten >=4pt,  shorten <=2pt,  from=3-2,  to=b]
		\arrow["\sigma\circ\tau"{description}, from=4-1, to=4-3]
		\arrow["\sigma"{description}, from=4-1, to=3-2]
		\arrow["\tau"{description}, from=3-2, to=4-3]
		\arrow["\sigma"{description}, from=2-1, to=1-2]
		\arrow["{\mu_{F,G}}"{description}, from=2-1, to=4-1]
		\arrow["\tau", from=1-2, to=2-3]
		\arrow["{\mu_{F'',G''}}"{description}, from=2-3, to=4-3]
		\arrow[""{name=0, anchor=center, inner sep=0}, shorten <=13pt, shorten >=-2.5pt, from=1-2, to=3-2]
		\arrow["{\mu_{F,G,\sigma}}"{description}, curve={height=-14pt}, shorten <=6pt, shorten >=6pt, Rightarrow, from=4-1, to=1-2]
		\arrow["{\mu_{F',G',\tau}}"{description}, curve={height=6pt}, Rightarrow, from=3-2, to=2-3]
		\arrow["{\mu_{F,G,\sigma\circ \tau}}", curve={height=-17pt},   shorten >=5pt,  shorten <=5pt,  Rightarrow, from=4-1, to=2-3]
		\arrow[shorten >=5pt, no head, from=1-2, to=0]
		\arrow["\sigma\circ\tau"{description}, no head, shorten >=-5pt, from=2-1, to=0]
		\arrow[from=0, to=2-3]
	\end{tikzcd}
	\qquad \quad  
	\begin{tikzcd}[column sep=1.1cm, row sep=0.8cm]    
		\centering
		& {\textsf{h}[\wt{G}F]} \\
		{\textsf{h}[F]} && {\textsf{h}[\wt{\wt{G}}F]} \\
		& {\textsf{h}[\wt{G}]} \\
		{\textsf{h}[G]} && {\textsf{h}[\wt{\wt{G}}]} 
		\arrow[	""{name=0, anchor=center, inner sep=0}, no head,shorten >=30pt,  shorten <=30pt,  from=2-1, to=2-3]
		\arrow[""{name=a, anchor=center, inner sep=0}, no head,shorten >=30pt,  shorten <=30pt,  from=1-2, to=3-2]
		\arrow["{(\hat{\beta},\hat{\gamma})}"{description}, Rightarrow,  curve={height=-18pt},  shorten >=4pt,  shorten <=2pt,  from=1-2,  to=a ]
		\arrow[""{name=b, anchor=center, inner sep=0}, no head,shorten >=30pt,  shorten <=30pt,  from=4-1, to=4-3]
		\arrow["{(\beta,\gamma)}"{description}, Rightarrow,  curve={height=-8pt},  shorten >=6pt,  shorten <=2pt,  from=3-2,  to=b]
		\arrow["\gamma \circ \beta"{description}, from=4-1, to=4-3]
		\arrow["\beta"{description}, from=4-1, to=3-2]
		\arrow["\gamma"{description}, from=3-2, to=4-3]
		\arrow["\hat{\beta}"{description}, from=2-1, to=1-2]
		\arrow["{\mu_{F,G}}"{description}, from=2-1, to=4-1]
		\arrow["\hat{\gamma}"{description}, from=1-2, to=2-3]
		\arrow["{\mu_{\wt{\wt{G} } F,\wt{\wt{G}}}}"{description}, from=2-3, to=4-3]
		\arrow[shorten <=13pt, shorten >=-2.5pt, from=1-2, to=3-2]
		\arrow["{\mu_{F,G,\beta}}"{description}, curve={height=-14pt}, shorten <=6pt, shorten >=6pt, Rightarrow, from=4-1, to=1-2]
		\arrow["{  \mu_{\wt{G}F,\wt{G}, \gamma}   }"{description}, curve={height=6pt}, Rightarrow, from=3-2, to=2-3]
		\arrow["{      \mu_{F,G,\gamma\circ \beta}     }", curve={height=-20pt},   shorten >=5pt,  shorten <=5pt,  Rightarrow, from=4-1, to=2-3]
		\arrow[shorten >=5pt, no head, from=1-2, to=0]
		\arrow["\widehat{\gamma\circ\beta}"{description}, no head, shorten >=-5pt, from=2-1, to=0]
		\arrow[from=0, to=2-3]
	\end{tikzcd}
\end{figure}
\noindent trivial $3$-cells \emph{coassociativity $\sigma$-coherence} and \emph{coassociativity $\beta$-coherence}
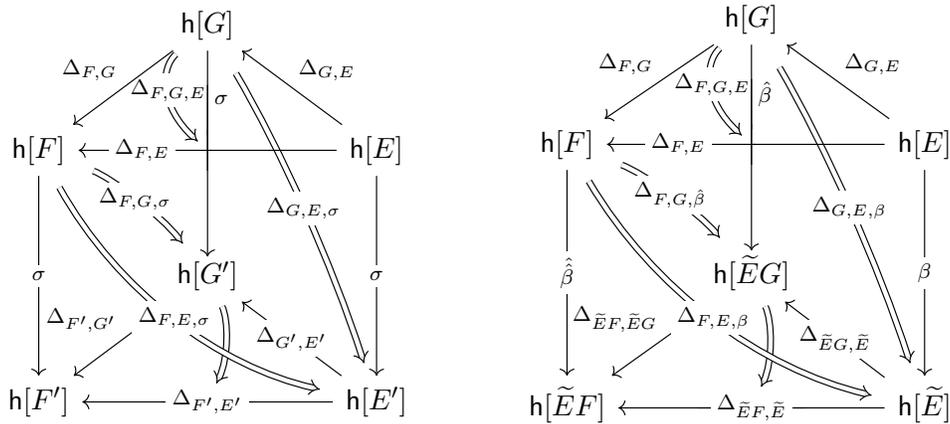
\begin{figure}[H]
	\centering
	\begin{tikzcd}[column sep=1.1cm, row sep=1cm]  
		\centering
		& {\textsf{h}[G]} \\
		{\textsf{h}[F]} && {\textsf{h}[E]} \\
		& {\textsf{h}[G']} \\
		{\textsf{h}[F']} && {\textsf{h}[E']} 
		\arrow[""{name=a, anchor=center, inner sep=0}, no head,shorten >=30pt,  shorten <=30pt,  from=1-2, to=3-2]
		\arrow["{\Delta_{F,G,E}}"{description}, Rightarrow,  curve={height=18pt},  shorten >=4pt,  shorten <=2pt,  from=1-2,  to=a ]
		\arrow[""{name=b, anchor=center, inner sep=0}, no head,shorten >=30pt,  shorten <=30pt,  from=4-1, to=4-3]
		\arrow[Rightarrow,  curve={height=-8pt},  shorten >=7pt,  shorten <=2pt,  from=3-2,  to=b]
		\arrow["\Delta_{F',E'}"{description}, to=4-1, from=4-3]
		\arrow["\Delta_{F',G'}"', to=4-1, from=3-2]
		\arrow["\Delta_{G',E'}"{description}, to=3-2, from=4-3]
		\arrow["\Delta_{F,G}"', to=2-1, from=1-2]
		\arrow["{\sigma}"{description}, from=2-1, to=4-1]
		\arrow["\Delta_{G,E}"', to=1-2, from=2-3]
		\arrow["{\sigma}"{description}, from=2-3, to=4-3]
		\arrow[""{name=0, anchor=center, inner sep=0}, shorten <=13pt, shorten >=-2.5pt, from=1-2, to=3-2]
		\arrow["{\Delta_{F,G,\sigma}}"{description}, curve={height=-8pt}, shorten <=6pt, shorten >=3pt, Rightarrow, from=2-1, to=3-2]
		\arrow["{\Delta_{G,E,\sigma}}"{description}, curve={height=-6pt}, shorten <=10pt,  shorten >=5pt, Rightarrow, from=1-2, to=4-3]
		\arrow["{  \Delta_{F,E,\sigma} }"{description}, curve={height=25pt},   shorten >=5pt,  shorten <=5pt,  Rightarrow, from=2-1, to=4-3]
		\arrow["{\sigma}", shorten >=5pt, no head, from=1-2, to=0]
		\arrow["\Delta_{F,E}"{description}, shorten >=0pt, shorten <=-2pt, from=0, to=2-1]
		\arrow[from=0, to=2-3, no head]
	\end{tikzcd}
	\qquad \quad  
	\begin{tikzcd}[column sep=1.1cm, row sep=1cm]  
		\centering
& {\textsf{h}[G]} \\
{\textsf{h}[F]} && {\textsf{h}[E]} \\
& {\textsf{h}[\wt{E}G]} \\
{\textsf{h}[\wt{E}F]} && {\textsf{h}[\wt{E}]} 
\arrow[""{name=a, anchor=center, inner sep=0}, no head,shorten >=30pt,  shorten <=30pt,  from=1-2, to=3-2]
\arrow["{\Delta_{F,G,E}}"{description}, Rightarrow,  curve={height=18pt},  shorten >=4pt,  shorten <=2pt,  from=1-2,  to=a ]
\arrow[""{name=b, anchor=center, inner sep=0}, no head,shorten >=30pt,  shorten <=30pt,  from=4-1, to=4-3]
\arrow[Rightarrow,  curve={height=-8pt},  shorten >=8pt,  shorten <=2pt,  from=3-2,  to=b]
\arrow["\Delta_{\wt{E}F,\wt{E}}"{description}, to=4-1, from=4-3]
\arrow["\Delta_{\wt{E}F,\wt{E}G}"', to=4-1, from=3-2]
\arrow["\Delta_{\wt{E}G,\wt{E}}"{description}, to=3-2, from=4-3]
\arrow["\Delta_{F,G}"', to=2-1, from=1-2]
\arrow["{\hat{\hat{\beta}}}"{description}, from=2-1, to=4-1]
\arrow["\Delta_{G,E}"', to=1-2, from=2-3]
\arrow["{\beta}"{description}, from=2-3, to=4-3]
\arrow[""{name=0, anchor=center, inner sep=0}, shorten <=13pt, shorten >=-2.5pt, from=1-2, to=3-2]
\arrow["{\Delta_{F,G,\hat{\beta}}}"{description}, curve={height=-8pt}, shorten <=6pt, shorten >=3pt, Rightarrow, from=2-1, to=3-2]
\arrow["{\Delta_{G,E,\beta}}"{description}, curve={height=-6pt}, shorten <=10pt,  shorten >=5pt, Rightarrow, from=1-2, to=4-3]
\arrow["{  \Delta_{F,E,\beta} }"{description}, curve={height=25pt},   shorten >=5pt,  shorten <=5pt,  Rightarrow, from=2-1, to=4-3]
\arrow["{\hat{\beta}}", shorten >=5pt, no head, from=1-2, to=0]
\arrow["\Delta_{F,E}"{description}, shorten >=0pt, shorten <=-2pt, from=0, to=2-1]
\arrow[from=0, to=2-3, no head]
	\end{tikzcd}
\end{figure}
\noindent trivial $3$-cells \emph{$\sigma$-functoriality comultiplication coherence} and \emph{$\beta$-functoriality comultiplication coherence}
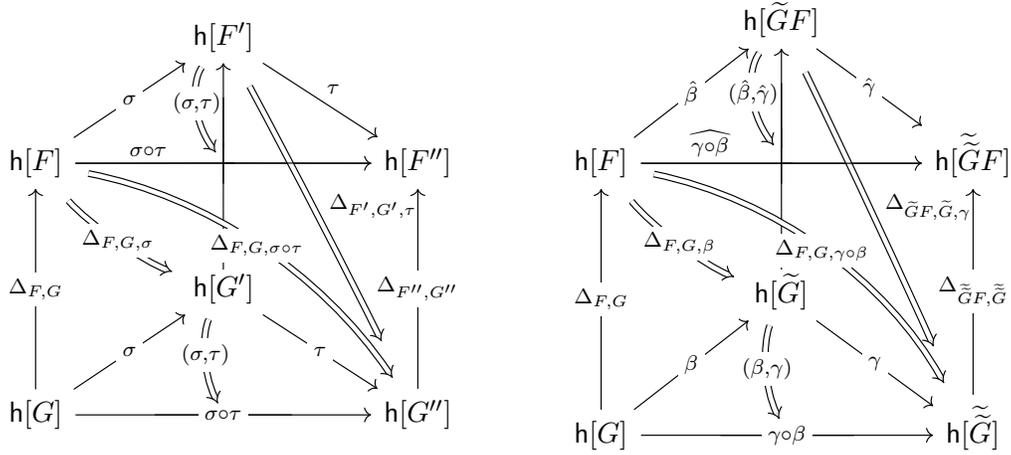
\begin{figure}[H]
	\centering
	\begin{tikzcd}[column sep=1.4cm, row sep=1cm]    
		\centering
		& {\textsf{h}[F']} \\
		{\textsf{h}[F]} && {\textsf{h}[F'']} \\
		& {\textsf{h}[G']} \\
		{\textsf{h}[G]} && {\textsf{h}[G'']} 
		\arrow[""{name=a, anchor=center, inner sep=0}, no head,shorten >=30pt,  shorten <=30pt,  from=1-2, to=3-2]
				\arrow[  from=2-1, to=2-3]
		\arrow["{(\sigma,\tau)}"{description}, Rightarrow,  curve={height=12pt},  shorten >=4pt,  shorten <=2pt,  from=1-2,  to=a ]
		\arrow[""{name=b, anchor=center, inner sep=0}, no head,shorten >=30pt,  shorten <=30pt,  from=4-1, to=4-3]
		\arrow["{(\sigma,\tau)}"{description}, Rightarrow,  curve={height=8pt},  shorten >=4pt,  shorten <=2pt,  from=3-2,  to=b]
		\arrow["\sigma\circ\tau"{description}, from=4-1, to=4-3]
		\arrow["\sigma"{description}, from=4-1, to=3-2]
		\arrow["\tau"{description}, from=3-2, to=4-3]
		\arrow["\sigma"{description}, from=2-1, to=1-2]
		\arrow["{\Delta_{F,G}}"{description}, to=2-1, from=4-1]
		\arrow["\tau", from=1-2, to=2-3]
		\arrow["{\Delta_{F'',G''}}"{description}, to=2-3, from=4-3]
		\arrow[""{name=0, anchor=center, inner sep=0}, no head, shorten <=13pt, shorten >=-2.5pt, from=1-2, to=3-2]
			\arrow["{\Delta_{F,G,\sigma}}"{description}, curve={height=8pt}, shorten <=6pt, shorten >=3pt, Rightarrow, from=2-1, to=3-2]
	\arrow["{\Delta_{F',G',\tau}}", curve={height=-1pt}, shorten <=10pt,  shorten >=20pt, Rightarrow, from=1-2, to=4-3]
	\arrow["{  \Delta_{F,G,\sigma\circ \tau} }"{description}, curve={height=-25pt},   shorten >=5pt,  shorten <=5pt,  Rightarrow, from=2-1, to=4-3]
		\arrow[shorten <=5pt, to=1-2, from=0]
		\arrow["\sigma\circ\tau", no head, shorten >=-5pt, from=2-1, to=0]
		\arrow[from=0, to=2-3, no head]
	\end{tikzcd}
	\qquad \quad  
	\begin{tikzcd}[column sep=1.2cm, row sep=1cm]   
		\centering
		& {\textsf{h}[\wt{G}F]} \\
		{\textsf{h}[F]} && {\textsf{h}[\wt{\wt{G}}F]} \\
		& {\textsf{h}[\wt{G}]} \\
		{\textsf{h}[G]} && {\textsf{h}[\wt{\wt{G}}]} 
		\arrow[	""{name=0, anchor=center, inner sep=0},shorten >=0pt,  shorten <=0pt,  from=2-1, to=2-3]
		\arrow[""{name=a, anchor=center, inner sep=0}, no head,shorten >=30pt,  shorten <=30pt,  from=1-2, to=3-2]
		\arrow["{(\hat{\beta},\hat{\gamma})}"{description}, Rightarrow,  curve={height=13pt},  shorten >=4pt,  shorten <=2pt,  from=1-2,  to=a ]
		\arrow[""{name=b, anchor=center, inner sep=0}, no head,shorten >=30pt,  shorten <=30pt,  from=4-1, to=4-3]
		\arrow["{(\beta,\gamma)}"{description}, Rightarrow,  curve={height=8pt},  shorten >=6pt,  shorten <=2pt,  from=3-2,  to=b]
		\arrow["\gamma \circ \beta"{description}, from=4-1, to=4-3]
		\arrow["\beta"{description}, from=4-1, to=3-2]
		\arrow["\gamma"{description}, from=3-2, to=4-3]
		\arrow["\hat{\beta}"{description}, from=2-1, to=1-2]
		\arrow["{\Delta_{F,G}}"{description}, to=2-1, from=4-1]
		\arrow["\hat{\gamma}"{description}, from=1-2, to=2-3]
		\arrow["{\Delta_{\wt{\wt{G}}F,\wt{\wt{G}}}}"{description}, to=2-3, from=4-3]
		\arrow[shorten <=13pt, shorten >=-2.5pt, from=1-2, to=3-2, no head]
		\arrow["{\Delta_{F,G,\beta}}"{description}, curve={height=8pt}, shorten <=6pt, shorten >=3pt, Rightarrow, from=2-1, to=3-2]
\arrow["{\Delta_{\wt{G}F,\wt{G},\gamma}}", curve={height=-1pt}, shorten <=10pt,  shorten >=20pt, Rightarrow, from=1-2, to=4-3]
\arrow["{  \Delta_{F,G,\gamma\circ \beta} }"{description}, curve={height=-28pt},   shorten >=5pt,  shorten <=5pt,  Rightarrow, from=2-1, to=4-3]
		\arrow[shorten <=5pt, to=1-2, from=0]
		\arrow["\widehat{\gamma\circ\beta}", no head, shorten >=-5pt, from=2-1, to=0]
		\arrow[from=0, to=2-3, no head]
	\end{tikzcd}
\end{figure}
\noindent trivial $3$-cells \emph{associativity coherence} and \emph{coassociativity coherence}
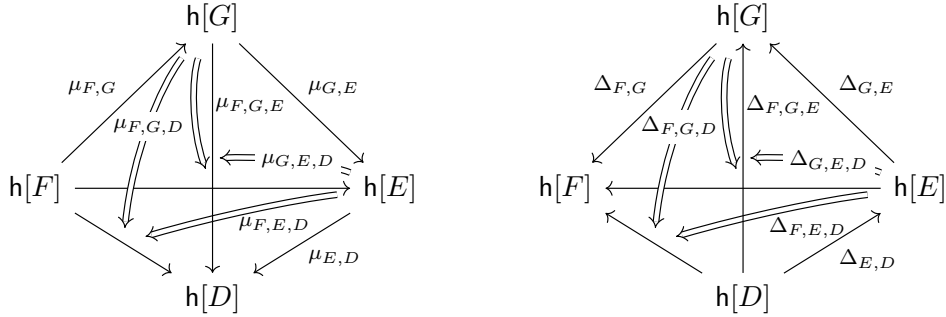
\begin{figure}[H]
	\centering
	\begin{tikzcd}[column sep=1.3cm, row sep=0.8cm]  
		\centering
		& {\textsf{h}[G]} \\
		\\
		\textsf{h}[F] && {\textsf{h}[E]} \\
		& {\textsf{h}[D]}
		\arrow[""{name=1, anchor=center, inner sep=0}, from=3-1, to=4-2]
		\arrow["{\!  \! \! \! \! \mu_{F,E,D}}", shorten >=8pt, shorten <=5pt, curve={height=3pt}, Rightarrow, from=3-3, to=1]
		\arrow[""{name=0, anchor=center, inner sep=0}, from=3-1, to=3-3]
		\arrow[""{name=2, anchor=center, inner sep=0}, from=1-2, to=4-2]
		\arrow["{ \mu_{G,E} }", from=1-2, to=3-3]
		\arrow["{  \mu_{E,D} }", from=3-3, to=4-2]
		\arrow["{  \mu_{F,G,D}   }"{description}, curve={height=9pt}, shorten <=7pt, shorten >=5pt, Rightarrow, from=1-2, to=1]
		\arrow["{\ \,  \mu_{F,G,E}  }", curve={height=8pt}, shorten >=7pt,  shorten <=6pt, Rightarrow, from=1-2, to=0]
		\arrow["{  \mu_{F,G} }", from=3-1, to=1-2]
		\arrow["{ \mu_{G,E,D} }"{description}, shorten >=2pt, curve={height=5pt}, Rightarrow, from=3-3, to=2]
	\end{tikzcd}
	\qquad \quad 
	\begin{tikzcd}[column sep=1.3cm, row sep=0.8cm]  
		\centering
& {\textsf{h}[G]} \\
\\
\textsf{h}[F] && {\textsf{h}[E]} \\
& {\textsf{h}[D]}
\arrow[""{name=1, anchor=center, inner sep=0}, to=3-1, from=4-2]
\arrow["{\!  \! \! \! \! \Delta_{F,E,D}}", shorten >=8pt, shorten <=5pt, curve={height=3pt}, Rightarrow, from=3-3, to=1]
\arrow[""{name=0, anchor=center, inner sep=0}, to=3-1, from=3-3]
\arrow[""{name=2, anchor=center, inner sep=0}, to=1-2, from=4-2]
\arrow["{ \Delta_{G,E} }"', to=1-2, from=3-3]
\arrow["{  \Delta_{E,D} }"', to=3-3, from=4-2]
\arrow["{  \Delta_{F,G,D}   }"{description}, curve={height=9pt}, shorten <=7pt, shorten >=5pt, Rightarrow, from=1-2, to=1]
\arrow["{\ \,  \Delta_{F,G,E}  }", curve={height=8pt}, shorten >=7pt,  shorten <=6pt, Rightarrow, from=1-2, to=0]
\arrow["{  \Delta_{F,G} }"', to=3-1, from=1-2]
\arrow["{ \Delta_{G,E,D} }"{description}, shorten >=2pt, curve={height=5pt}, Rightarrow, from=3-3, to=2]
	\end{tikzcd}
\end{figure}
\noindent and a trivial $3$-cell \emph{bimonoid $\sigma$-coherence} 
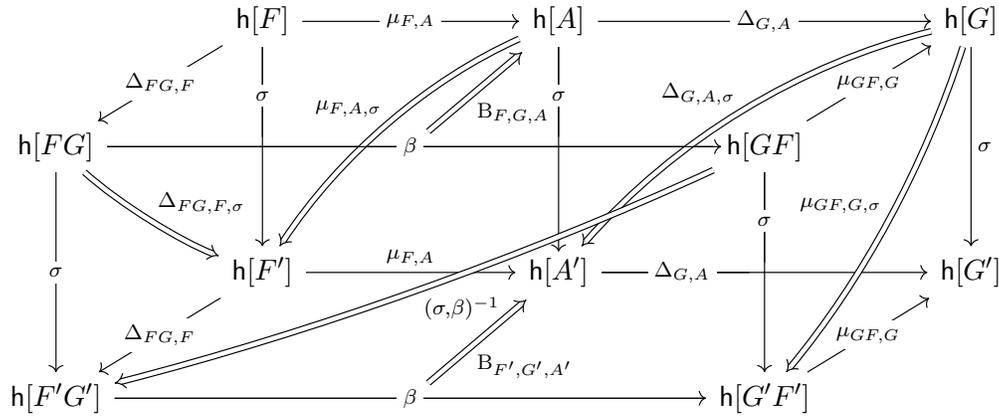
\[\begin{tikzcd} [column sep=1.4cm, row sep=1cm]  
	& {\textsf{h}[F]} && {\textsf{h}[A]} && {\textsf{h}[G]} \\
	{\textsf{h}[FG]} &&&& {\textsf{h}[GF]} \\
	& {\textsf{h}[F']} && {\textsf{h}[A']} && {\textsf{h}[G']} \\
	{\textsf{h}[F'G']} &&&& {\textsf{h}[G'F']}
			\arrow[""{name=a, anchor=center, inner sep=0}, no head, from=2-1, to=2-5]
						\arrow[""{name=b, anchor=center, inner sep=0}, no head, from=4-1, to=4-5]
	\arrow["{\Delta_{FG,F,\sigma}}", curve={height=6pt}, Rightarrow, to=3-2, from=2-1]
	\arrow["{\mu_{F,A,\sigma}}"', curve={height=15pt}, Rightarrow, to=3-2, from=1-4]
	\arrow["{\Delta_{G,A,\sigma}}"', curve={height=20pt}, Rightarrow, to=3-4, from=1-6]
	\arrow["{\mu_{F,A}}"{description}, from=1-2, to=1-4]
	\arrow["{\Delta_{G,A}}"{description}, from=1-4, to=1-6]
	\arrow["{\Delta_{FG,F}}"{description}, from=1-2, to=2-1]
	\arrow["{\beta}"{description}, from=2-1,   shorten >=-3pt,    to=2-5]
	\arrow["{\mu_{GF,G}}"{description}, from=2-5, to=1-6]
	\arrow["{\sigma}"{description}, from=2-1, to=4-1]
	\arrow[""{name=0, anchor=center, inner sep=0}, from=1-2, to=3-2]
	\arrow[""{name=1, anchor=center, inner sep=0}, shorten >=-2pt,   from=1-4, to=3-4]
	\arrow[""{name=2, anchor=center, inner sep=0}, from=2-5, to=4-5]
	\arrow["{\sigma}", from=1-6, to=3-6]
	\arrow["{\mu_{F,A}}", from=3-2, to=3-4]
	\arrow["{\mu_{GF,G}}"{description}, from=4-5, to=3-6]
	\arrow["{\Delta_{FG,F}}"{description}, from=3-2, to=4-1]
	\arrow["{\beta}"{description}, from=4-1, to=4-5]
	\arrow[""{name=3, anchor=center, inner sep=0}, from=3-4, to=3-6]
	\arrow["{\Delta_{G,A}}"{description}, shorten >=5pt, no head, from=3-4, to=3]
	\arrow["{\sigma}"{description}, shorten >=3pt, no head, from=1-2, to=0]
	\arrow["{\sigma}"{description}, shorten >=3pt, no head, from=1-4, to=1]
	\arrow["{\sigma}"{description}, shorten >=3pt, no head, from=2-5, to=2]
	\arrow["{\mu_{GF,G,\sigma}}"', curve={height=-11pt}, Rightarrow, to=4-5, from=1-6]
	\arrow["{(\sigma,\beta)^{-1}}", curve={height=-9pt}, Rightarrow, to=4-1, from=2-5]
		\arrow["{\text{B}_{F,G,A}}"', Rightarrow,   shorten >=3pt,  shorten <=8pt, from=a, to=1-4]
				\arrow["{\text{B}_{F',G',A'}}"', Rightarrow,   shorten >=1pt,  shorten <=8pt, from=b, to=3-4]
\end{tikzcd}\]
\noindent Note that e.g. associativity $\beta$-coherence is well-defined; we have $\wt{E}F=\wt{E}GF$ because $G\leq F$. One can check similar results for the other $3$-cells involving permutations. 

\begin{remark}
	We do not prove any coherence theorems in this paper, and it is only a conjecture that we have all the coherence conditions we should want in the correct definition of a bimonoid valued in a \hbox{$2$-category}. We leave the conjectural formalization of these bimonoids as bialgebras over a \hbox{$2$-bimonad} to future work.
\end{remark}

Let us make these coherence conditions explicit in the case $\textsf{C}\in \{\textsf{Cat},\textsf{Cat}_\bullet\}$, at the level of components of the $2$-cells. The $2$-cell $\mu_{F,G,\sigma}$ consists of a natural isomorphism in the category $\textsf{h}[G']$
\[
{\mu_{F,G,\sigma}}_\ta :  \mu'_{F,G}(\ta) \to    \mu_{F',G'}(\ta') 
\]
for each object $\ta\in \textsf{h}[F]$. The $2$-cell $\mu_{F,G,\beta}$ consists of a natural isomorphism in the category $\textsf{h}[\wt{G}]$
\[
{\mu_{F,G,\beta}}_\ta :  \wt{\mu_{F,G}(\ta)} \to    \mu_{\wt{G}F,\wt{G}}(\tilde{\ta}) 
\]
for each object $\ta\in \textsf{h}[F]$. The $2$-cell $\mu_{F,G,E}$ consists of a natural isomorphism in the category $\textsf{h}[E]$
\[
{\mu_{F,G,E}}_\ta :  \mu_{G,E}\big (\mu_{F,G} (\ta)\big ) \to \mu_{F,E} (\ta)
\]
for each object $\ta\in \textsf{h}[F]$. Similarly for the $2$-cells $\Delta_{F,G,\sigma}$, $\Delta_{F,G,\beta}$ and $\Delta_{F,G,E}$. The $2$-cell $\text{B}_{F,G,A}$ consists of a natural isomorphism in the category $\textsf{h}[G]$
\[
{\text{B}_{F,G,A}}_{\ta}  :   \mu_{GF,G} \big (\wt{\Delta_{FG,F}(\ta)}\big ) \to \mu_{G,A} \big ( \mu_{F,A} (\ta) \big )
\]
for each object $\ta\in \textsf{h}[F]$. Then associativity $\sigma$-coherence says
\[
{\mu_{F',G',E'}}_{\ta'}  \circ   \mu_{G',E'} \big ({\mu_{F,G,\sigma}}_\ta \big )  \circ  {\mu_{G,E,\sigma}}_{\mu_{F,G}(\ta)}  
=
{\mu_{F,E,\sigma}}_{\ta} \circ {\mu'_{F,G,E}}_{\ta}
\]
associativity $\beta$-coherence says
\[
{\mu_{\wt{E}F,\wt{E}G,\wt{E}}}_{\tilde{\ta}}  \circ   \mu_{\wt{E}G,\wt{E}} \big ({\mu_{F,G,\hat{\beta}}}_\ta \big )  \circ  {\mu_{G,E,\beta}}_{\mu_{F,G}(\ta)}  
=
{\mu_{F,E,\beta}}_{\ta} \circ {\wt{\mu_{F,G,E}}}_{\ta}
\]
$\sigma$-functoriality multiplication coherence says
\[
\mu_{F'',G''}\big((\sigma,\tau)_\ta\big) \circ {\mu_{F',G',\tau}}_{\ta'} \circ {\mu'_{F,G,\sigma}}_{\ta}    =    {\mu_{F,G,\sigma \circ \tau}}_\ta  \circ    (\sigma,\tau)_{\mu_{F,G}(\ta)}
\]
$\beta$-functoriality multiplication coherence says
\[
\mu_{ \wt{\wt{G}}F  , \wt{\wt{G}}  }  \big( (\hat{\beta},\hat{\gamma})_\ta \big) \circ {\mu_{\wt{G}F, \wt{G}, \gamma}}_{\tilde{\ta}} \circ {\wt{\mu_{F,G,\beta}}}_{\ta}
=
{\mu_{F,G,\gamma\circ \beta}}_\ta \circ {(\beta,\gamma)}_{  \mu_{F,G} (\ta) }
\]
coassociativity $\sigma$-coherence says
\[
{\Delta_{F',G',E'}}_{\ta'} \circ   \Delta_{F',G'}( {\Delta_{G,E,\sigma}}_\ta) \circ  {\Delta_{F,G,\sigma}}_{\Delta_{G,E}(\ta)}
=
{\Delta_{F,E,\sigma}}_\ta \circ {\Delta'_{F,G,E}}_\ta
\]
coassociativity $\beta$-coherence says
\[
{\Delta_{\wt{E}F,\wt{E}G,\wt{E}}}_{\tilde{\ta}} \circ   \Delta_{\wt{E}F,\wt{E}G}( {\Delta_{G,E,\beta}}_\ta) \circ  {\Delta_{F,G,\beta}}_{\Delta_{G,E}(\ta)}
=
{\Delta_{F,E,\beta}}_\ta \circ {\wt{\Delta_{F,G,E}}}_\ta
\]
$\sigma$-functoriality comultiplication coherence says
\[
\Delta_{F'',G''}\big ((\sigma,\tau)_\ta\big ) \circ {\Delta_{F'G',\tau}}_{\ta'} \circ  {\Delta'_{F,G,\sigma}}_\ta
=
{\Delta_{F,G,\sigma\circ \tau}}_\ta \circ  (\sigma,\tau)_{\Delta_{F,G}(\ta)}
\]
$\beta$-functoriality comultiplication coherence says
\[
\Delta_{\wt{\wt{G}}F ,\wt{\wt{G}}  } \big( (\beta,\gamma)_{\ta}\big ) \circ {\Delta_{\wt{G}F,\wt{G},\gamma}}_{\tilde{\ta}} \circ {\wt{\Delta_{F,G,\beta}}}_{\ta}
=
{\Delta_{F,G,\gamma\circ \beta}}_\ta \circ {(\hat{\beta}, \hat{\gamma})}_{\Delta_{F,G}(\ta)}
\]
associativity coherence says
\[
{\mu_{F,E,D}}_\ta \circ \mu_{E,D}( {{\mu_{F,G,E}}_\ta} )={\mu_{F,G,D}}_\ta \circ  {\mu_{G,E,D}}_{\mu_{F,G}(\ta)} 
\]
coassociativity coherence says
\[
{\Delta_{F,E,D}}_\ta \circ  {{\Delta_{F,G,E}}_{\Delta_{E,D}(\ta)}} ={\Delta_{F,G,D}}_\ta \circ  \Delta_{F,G}({\Delta_{G,E,D}}_{\ta} )
\]
and bimonoid $\sigma$-coherence says
\[
\Delta_{G,A}({\mu_{F,A,\sigma }}_{\ta}) \circ  {\Delta_{G,A,\sigma}}_{\mu_{F,A}(\ta)} \circ {\text{B}'_{F,G,A}}_{\ta}
\]
\[
=
{\text{B}_{F',G',A'}}_{\ta'} \circ  \mu_{GF,G}( { \wt{\Delta_{FG,F,\sigma}}  }_{\ta}) \circ \mu_{GF,G} \big({(\sigma,\beta)}^{-1}_{\Delta_{FG,F}(\ta)}\big ) \circ  {\mu_{GF,G,\sigma}}_{\wt{\Delta_{FG,F}(\ta)}}
.\]

A $\textsf{C}$-valued \emph{bimonoid in cospecies} $\textsf{h}$ is a $\textsf{C}$-valued cospecies $\textsf{h}$ equipped with analogous data, requiring the analogous $3$-cells to commute. The functoriality coherence conditions have to be adjusted as follows; $\sigma$-functoriality multiplication coherence now says
\[
\mu_{F'',G''}\big((\sigma,\tau)_\ta\big) \circ    
 \underbrace{ {\mu_{F',G',\sigma}}_{\ta'} \circ {\mu'_{F,G,\tau}}_{\ta} }_{  \text{$\sigma$ and $\tau$ switched} }  
 =    {\mu_{F,G,\sigma \circ \tau}}_\ta  \circ    (\sigma,\tau)_{\mu_{F,G}(\ta)}
\]
$\beta$-functoriality multiplication coherence now says
\[
\mu_{ \wt{\wt{G}}F  , \wt{\wt{G}}  }  \big( (\hat{\beta},\hat{\gamma})_\ta \big) 
\underbrace{\circ {\mu_{\wt{G}F, \wt{G}, \beta}}_{\tilde{\ta}} \circ {\wt{\mu_{F,G,\gamma}}}_{\ta}  }_{  \text{$\beta$ and $\gamma$ switched} }  
=
{\mu_{F,G,\gamma\circ \beta}}_\ta \circ {(\beta,\gamma)}_{  \mu_{F,G} (\ta) }
\]
$\sigma$-functoriality comultiplication coherence now says
\[
\Delta_{F'',G''}\big ((\sigma,\tau)_\ta\big ) \circ 
\underbrace{ {\Delta_{F'G',\sigma}}_{\ta'} \circ  {\Delta'_{F,G,\tau}}_\ta }_{  \text{$\sigma$ and $\tau$ switched} }  
=
{\Delta_{F,G,\sigma\circ \tau}}_\ta \circ  (\sigma,\tau)_{\Delta_{F,G}(\ta)}
\]
and $\beta$-functoriality comultiplication coherence now says
\[
\Delta_{\wt{\wt{G}}F ,\wt{\wt{G}}  } \big( (\beta,\gamma)_{\ta}\big ) \circ
\underbrace{ {\Delta_{\wt{G}F,\wt{G},\beta}}_{\tilde{\ta}} \circ {\wt{\Delta_{F,G,\gamma}}}_{\ta} }_{  \text{$\beta$ and $\gamma$ switched} }  
=
{\Delta_{F,G,\gamma\circ \beta}}_\ta \circ {(\hat{\beta}, \hat{\gamma})}_{\Delta_{F,G}(\ta)}
.\]
All the other coherence conditions remain unchanged. 



\subsection{Lax Homomorphisms of Bimonoids}


Let $\textsf{C}$ be a strict $2$-category. Let $\textsf{h}$ and $\textsf{x}$ be bimonoids in $\textsf{C}$-valued species. A \emph{lax homomorphism}, or just \emph{homomorphism}, of bimonoids
\[
\varphi: \textsf{h} \to \textsf{x}
\]
consists of a morphism of the underlying species $\varphi=(\varphi_F, {\sigma}_\ta, \beta_\ta )$ together with a pair of $2$-cells called \emph{preservation of multiplication} and \emph{preservation of comultiplication} 
\[\begin{tikzcd}[row sep=large] 
	{\textsf{h}[F]} & {\textsf{h}[G]} && {\textsf{h}[F]} & {\textsf{h}[G]} \\
	{\textsf{x}[F]} & {\textsf{x}[G]} && {\textsf{x}[F]} & {\textsf{x}[G]}
	\arrow["{\mu_{F,G}}", from=1-1, to=1-2]
	\arrow["{\varphi_F}"', from=1-1, to=2-1]
	\arrow["{\mu_{F,G}}"', from=2-1, to=2-2]
	\arrow["{\varphi_G}", from=1-2, to=2-2]
	\arrow["{\Delta_{F,G}}"', from=1-5, to=1-4]
	\arrow["{\varphi_G}", to=2-5, from=1-5]
	\arrow["{\Delta_{F,G}}", from=2-5, to=2-4]
	\arrow["{\varphi_F}"', to=2-4, from=1-4]
	\arrow["{\mathtt{\Phi}_{F,G}}", Rightarrow, from=2-1, to=1-2]
	\arrow["{\mathtt{\Psi}_{F,G}}"', Rightarrow, from=2-5, to=1-4]
\end{tikzcd}\]
for each pair of compositions $F,G$ with $G< F$. We require trivial $3$-cells \emph{multiplication \hbox{$\sigma$-naturality} coherence} and \emph{multiplication $\beta$-naturality coherence}
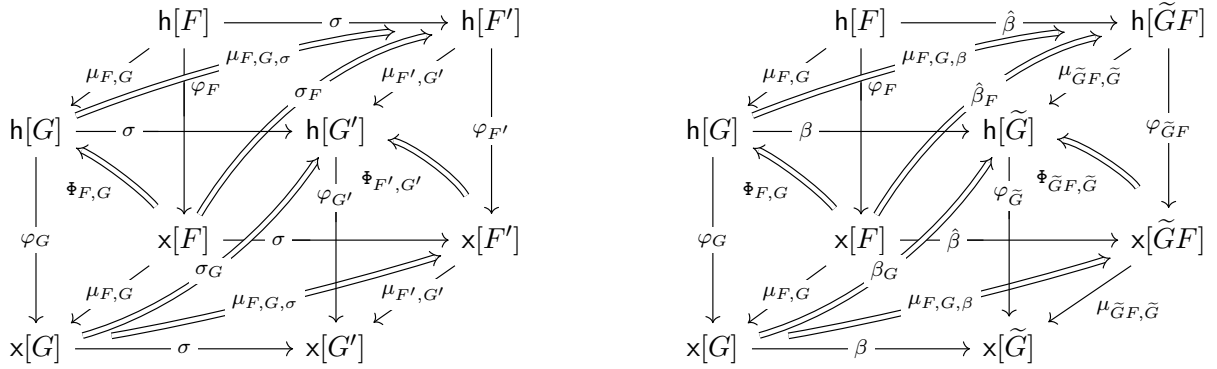
\[\begin{tikzcd}
	& {\textsf{h}[F]} && {\textsf{h}[F']} &&& {\textsf{h}[F]} && {\textsf{h}[\widetilde{G} F]} \\
	{\textsf{h}[G]} && {\textsf{h}[G']} &&& {\textsf{h}[G]} && {\textsf{h}[\widetilde{G}]} \\
	& {\textsf{x}[F]} && {\textsf{x}[F']} &&& {\textsf{x}[F]} && {\textsf{x}[\widetilde{G} F]} \\
	{\textsf{x}[G]} && {\textsf{x}[G']} &&& {\textsf{x}[G]} && {\textsf{x}[\widetilde{G}]}
		\arrow["{\sigma_F}"{description}, curve={height=-22pt}, shorten >=6pt, Rightarrow, from=3-2, to=1-4]
					\arrow["{\hat{\beta}_F}"{description}, curve={height=-22pt}, shorten >=6pt, Rightarrow, from=3-7, to=1-9]
													\arrow["{ \mu_{F,G,\sigma} }"{description}, curve={height=4pt}, shorten >=3pt, shorten <=14pt, Rightarrow, from=4-1, to=3-4]
													\arrow["{ \mu_{F,G,\beta} }"{description}, curve={height=4pt}, shorten >=3pt, shorten <=14pt, Rightarrow, from=4-6, to=3-9]
	\arrow["{\varphi_G}"{description}, from=2-1, to=4-1]
	\arrow["\sigma"{description}, from=4-1, to=4-3]
	\arrow["{\mu_{F,G}}"{description}, from=1-2, to=2-1]
	\arrow["{\mu_{F',G'}}"{description}, from=1-4, to=2-3]
	\arrow["{\mu_{F',G'}}"{description}, from=3-4, to=4-3]
	\arrow["{\mu_{F,G}}"{description}, from=3-2, to=4-1]
	\arrow["{\varphi_{F'}}"{description}, from=1-4, to=3-4]
	\arrow["\sigma"{description}, from=1-2, to=1-4]
	\arrow["{\varphi_G}"{description}, from=2-6, to=4-6]
	\arrow[""{name=0, anchor=center, inner sep=0}, shorten <=18pt, from=2-6, to=2-8]
	\arrow["\beta"{description}, from=4-6, to=4-8]
	\arrow["{\mu_{F,G}}"{description}, from=1-7, to=2-6]
	\arrow["{\mu_{\widetilde{G} F,\widetilde{G}}}"{description}, from=1-9, to=2-8]
	\arrow["{\mu_{\widetilde{G} F,\widetilde{G} }}", from=3-9, to=4-8]
	\arrow["{\varphi_{\widetilde{G} F}}"{description}, from=1-9, to=3-9]
	\arrow["{\mu_{F,G}}"{description}, from=3-7, to=4-6]
	\arrow["{\hat{\beta}}"{description}, from=1-7, to=1-9]
	\arrow["{\mathtt{\Phi}_{F,G}}", curve={height=6pt}, shorten <=3pt, Rightarrow, from=3-7, to=2-6]
	\arrow["{\mathtt{\Phi}_{\widetilde{G} F,\widetilde{G}}}", curve={height=12pt}, shorten <=9pt, shorten >=3pt, Rightarrow, from=3-9, to=2-8]
	\arrow["{\mathtt{\Phi}_{F,G}}", curve={height=6pt}, shorten <=4pt, Rightarrow, from=3-2, to=2-1]
	\arrow["{\mathtt{\Phi}_{F',G'}}", curve={height=12pt}, shorten <=8pt, shorten >=3pt, Rightarrow, from=3-4, to=2-3]
	\arrow[""{name=1, anchor=center, inner sep=0}, shorten <=13pt, from=2-3, to=4-3]
	\arrow[""{name=2, anchor=center, inner sep=0}, shorten <=13pt, from=1-2, to=3-2]
	\arrow[""{name=3, anchor=center, inner sep=0}, shorten <=18pt, from=3-2, to=3-4]
	\arrow[""{name=4, anchor=center, inner sep=0}, shorten <=18pt, from=3-7, to=3-9]
	\arrow[""{name=5, anchor=center, inner sep=0}, shorten <=13pt, from=2-8, to=4-8]
	\arrow[""{name=6, anchor=center, inner sep=0}, shorten <=13pt, from=1-7, to=3-7]
	\arrow[""{name=7, anchor=center, inner sep=0}, shorten <=18pt, from=2-1, to=2-3]
	\arrow["\beta"{description}, shorten >=11pt, no head, from=2-6, to=0]
	\arrow["{\varphi_{G'}}"{description}, shorten >=6pt, no head, from=2-3, to=1]
	\arrow["{\varphi_F}", shorten >=5pt, no head, from=1-2, to=2]
	\arrow["\sigma"{description}, shorten >=9pt, no head, from=3-2, to=3]
	\arrow["\hat\beta"{description}, shorten >=9pt, no head, from=3-7, to=4]
	\arrow["{\varphi_{\widetilde{G}}}"{description}, shorten >=6pt, no head, from=2-8, to=5]
	\arrow["{\varphi_F}", shorten >=11pt, no head, from=1-7, to=6]
	\arrow["\sigma"{description}, shorten >=9pt, no head, from=2-1, to=7]
		\arrow["{\sigma_G}"{description}, curve={height=18pt}, shorten <=3pt, Rightarrow, from=4-1, to=2-3]
					\arrow["{\beta_G}"{description}, curve={height=18pt}, shorten <=3pt, Rightarrow, from=4-6, to=2-8]
							\arrow["{ \mu_{F,G,\sigma} }"{description}, curve={height=-9pt}, shorten >=20pt, Rightarrow, from=2-1, to=1-4]
							\arrow["{ \mu_{F,G,\beta} }"{description}, curve={height=-9pt}, shorten >=20pt, Rightarrow, from=2-6, to=1-9]
\end{tikzcd}\]
a trivial $3$-cell \emph{associativity coherence} 
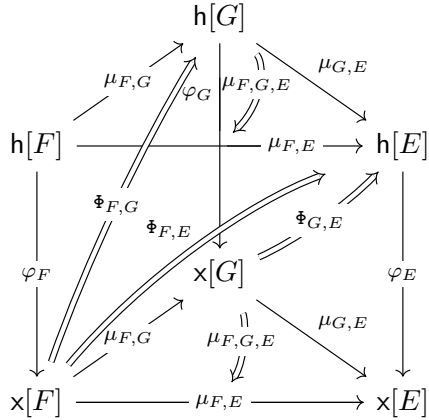
\[\begin{tikzcd}[column sep=large, row sep=large]
	& {\textsf{h}[G]} \\
	{\textsf{h}[F]} && {\textsf{h}[E]} \\
	& {\textsf{x}[G]} \\
	{\textsf{x}[F]} && {\textsf{x}[E]} 
		\arrow[""{name=0, anchor=center, inner sep=0}, shorten >=18pt, no head, from=2-1, to=2-3]
				\arrow[""{name=a, anchor=center, inner sep=0}, no head, from=4-1, to=4-3]
									\arrow["{ \mu_{F,G,E} }"{description}, curve={height=-17pt}, shorten >=6pt, shorten <=6pt, Rightarrow, from=1-2, to=0]
	\arrow["{\mu_{F,E}}"{description}, from=4-1, to=4-3]
	\arrow["{\mu_{F,G}}"{description}, from=4-1, to=3-2]
	\arrow["{\mu_{G,E}}", from=3-2, to=4-3]
	\arrow["{\mu_{F,G}}"{description}, from=2-1, to=1-2]
	\arrow["{\varphi_F}"{description}, from=2-1, to=4-1]
	\arrow["{\mu_{G,E}}", from=1-2, to=2-3]
	\arrow["{\varphi_E}"{description}, from=2-3, to=4-3]
	\arrow["{\mathtt{\Phi}_{F,G}}"{description}, curve={height=-6pt}, shorten <=6pt, shorten >=6pt, Rightarrow, from=4-1, to=1-2]
	\arrow["{\mathtt{\Phi}_{G,E}}"{description}, curve={height=6pt}, Rightarrow, from=3-2, to=2-3]
	\arrow[""{name=1, anchor=center, inner sep=0}, shorten <=13pt, from=1-2, to=3-2]
	\arrow["{\mu_{F,E}}"{description}, shorten <=2pt, from=0, to=2-3]
	\arrow["{\varphi_G}"', shorten >=5pt, no head, from=1-2, to=1]
	\arrow["{\mathtt{\Phi}_{F,E}}", curve={height=-16pt}, shorten <=6pt, shorten >=15pt, Rightarrow, from=4-1, to=2-3]
									\arrow["{ \mu_{F,G,E} }"{description}, curve={height=-10pt}, shorten >=8pt, shorten <=6pt, Rightarrow, from=3-2, to=a]
\end{tikzcd}\]
trivial $3$-cells \emph{comultiplication $\sigma$-naturality coherence} and \emph{comultiplication $\beta$-naturality coherence}
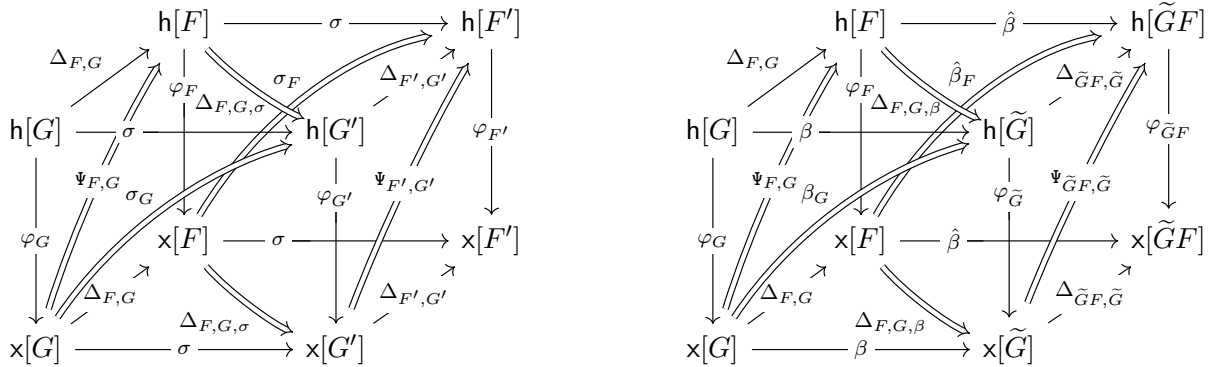
\[\begin{tikzcd}
	& {\textsf{h}[F]} && {\textsf{h}[F']} &&& {\textsf{h}[F]} && {\textsf{h}[\widetilde{G}    F]} \\
	{\textsf{h}[G]} && {\textsf{h}[G']} &&& {\textsf{h}[G]} && {\textsf{h}[\widetilde{G}]} \\
	& {\textsf{x}[F]} && {\textsf{x}[F']} &&& {\textsf{x}[F]} && {\textsf{x}[\widetilde{G}    F]} \\
	{\textsf{x}[G]} && {\textsf{x}[G']} &&& {\textsf{x}[G]} && {\textsf{x}[\widetilde{G}]}
		\arrow["     \Delta_{F,G,\sigma} \! \!  "', curve={height=6pt}, shorten >=0pt, Rightarrow, from=3-2, to=4-3]
				\arrow["     \Delta_{F,G,\beta} \! \!  "', curve={height=6pt}, shorten >=0pt, Rightarrow, from=3-7, to=4-8]
			\arrow["{\sigma_F}", curve={height=-22pt}, shorten >=6pt, Rightarrow, from=3-2, to=1-4]
	\arrow["{\hat{\beta}_F}", curve={height=-23pt}, shorten >=6pt, Rightarrow, from=3-7, to=1-9]
	\arrow["{\varphi_G}"{description}, from=2-1, to=4-1]
	\arrow["\sigma"{description}, from=4-1, to=4-3]
	\arrow["{\Delta_{F,G}}", from=2-1, to=1-2]
	\arrow["{\Delta_{F',G'}}"{description}, from=2-3, to=1-4]
	\arrow["{\Delta_{F',G'}}"{description}, from=4-3, to=3-4]
	\arrow["{\Delta_{F,G}}"{description}, from=4-1, to=3-2]
	\arrow["{\varphi_{F'}}"{description}, from=1-4, to=3-4]
	\arrow["\sigma"{description}, from=1-2, to=1-4]
	\arrow["{\varphi_G}"{description}, from=2-6, to=4-6]
	\arrow[""{name=0, anchor=center, inner sep=0}, shorten <=18pt, from=2-6, to=2-8]
	\arrow["\beta"{description}, from=4-6, to=4-8]
	\arrow["{\Delta_{F,G}}", from=2-6, to=1-7]
	\arrow["{\Delta_{\widetilde{G}    F,\widetilde{G}}}"{description}, from=2-8, to=1-9]
	\arrow["{\Delta_{\widetilde{G}    F,\widetilde{G} }}"{description}, from=4-8, to=3-9]
	\arrow["{\varphi_{\widetilde{G}    F}}"{description}, from=1-9, to=3-9]
	\arrow["{\Delta_{F,G}}"{description}, from=4-6, to=3-7]
	\arrow["{\hat{\beta}}"{description}, from=1-7, to=1-9]
	\arrow[""{name=1, anchor=center, inner sep=0}, shorten <=13pt, from=2-3, to=4-3]
	\arrow[""{name=2, anchor=center, inner sep=0}, shorten <=13pt, from=1-2, to=3-2]
	\arrow[""{name=3, anchor=center, inner sep=0}, shorten <=18pt, from=3-2, to=3-4]
	\arrow[""{name=4, anchor=center, inner sep=0}, shorten <=13pt, from=2-8, to=4-8]
	\arrow[""{name=5, anchor=center, inner sep=0}, shorten <=13pt, from=1-7, to=3-7]
	\arrow[""{name=6, anchor=center, inner sep=0}, shorten <=18pt, from=2-1, to=2-3]
	\arrow["{\mathtt{\Psi}_{F,G}}"'{description}, curve={height=-6pt}, shorten <=6pt, shorten >=6pt, Rightarrow, from=4-1, to=1-2]
	\arrow["{\mathtt{\Psi}_{F',G'}}"'{description}, curve={height=-4pt}, shorten <=6pt, shorten >=6pt, Rightarrow, from=4-3, to=1-4]
	\arrow["{\mathtt{\Psi}_{F,G}}"'{description}, curve={height=-6pt}, shorten <=6pt, shorten >=6pt, Rightarrow, from=4-6, to=1-7]
	\arrow[""{name=7, anchor=center, inner sep=0}, shorten <=18pt, from=3-7, to=3-9]
	\arrow["{\mathtt{\Psi}_{\widetilde{G}    F,\widetilde{G}}}"'{description}, curve={height=-4pt}, shorten <=6pt, shorten >=6pt, Rightarrow, from=4-8, to=1-9]
	\arrow["\beta"{description}, shorten >=11pt, no head, from=2-6, to=0]
	\arrow["{\varphi_{G'}}"{description}, shorten >=6pt, no head, from=2-3, to=1]
	\arrow["{\varphi_F}"{description}, shorten >=5pt, no head, from=1-2, to=2]
	\arrow["\sigma"{description}, shorten >=9pt, no head, from=3-2, to=3]
	\arrow["{\varphi_F}"{description}, shorten >=11pt, no head, from=1-7, to=5]
	\arrow["\sigma"{description}, shorten >=9pt, no head, from=2-1, to=6]
	\arrow["{\varphi_{\widetilde{G}}}"{description}, shorten >=6pt, no head, from=2-8, to=4]
	\arrow["\hat\beta"{description}, shorten >=9pt, no head, from=3-7, to=7]
			\arrow["{\sigma_G}", curve={height=-18pt}, shorten <=3pt, Rightarrow, from=4-1, to=2-3]
	\arrow["{\beta_G}", curve={height=-18pt}, shorten <=3pt, Rightarrow, from=4-6, to=2-8]
					\arrow["   \Delta_{F,G,\sigma}\! \! \! \! \!\!   "', curve={height=5.5pt}, shorten >=-5pt, Rightarrow, from=1-2, to=2-3]
										\arrow["   \Delta_{F,G,\beta}\! \! \! \! \!\!   "', curve={height=5.5pt}, shorten >=-5pt, Rightarrow, from=1-7, to=2-8]
\end{tikzcd}\]
a trivial $3$-cell \emph{coassociativity coherence}
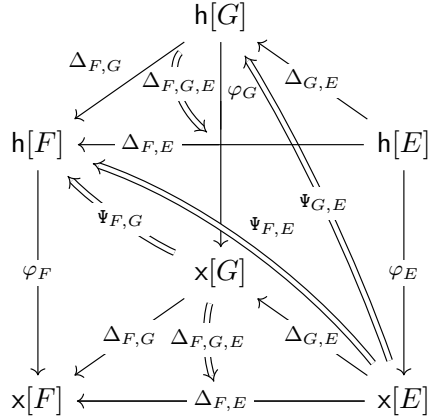
\[\begin{tikzcd}[column sep=large, row sep=large]
	& {\textsf{h}[G]} \\
	{\textsf{h}[F]} && {\textsf{h}[E]}   \\
	& {\textsf{x}[G]}  \\
	{\textsf{x}[F]} && {\textsf{x}[E]}  
		\arrow[""{name=0, anchor=center, inner sep=0}, shorten <=13pt, from=1-2, to=3-2]
				\arrow[""{name=a, anchor=center, inner sep=0}, no head,  from=4-1, to=4-3]
										\arrow["{ \Delta_{F,G,E} }"{description}, curve={height=19pt}, shorten >=6pt, shorten <=6pt, Rightarrow, from=1-2, to=0]
	\arrow["{\Delta_{F,E}}"{description}, from=4-3, to=4-1]
	\arrow["{\Delta_{F,G}}"{description}, from=3-2, to=4-1]
	\arrow["{\Delta_{G,E}}"{description}, from=4-3, to=3-2]
	\arrow["{\Delta_{F,G}}"', from=1-2, to=2-1]
	\arrow["{\varphi_F}"{description}, from=2-1, to=4-1]
	\arrow["{\Delta_{G,E}}"'{description}, from=2-3, to=1-2]
	\arrow["{\varphi_E}"{description}, from=2-3, to=4-3]
	\arrow["{\mathtt{\Psi}_{F,E}}"', curve={height=17pt}, shorten <=6pt, shorten >=6pt, Rightarrow, from=4-3, to=2-1]
	\arrow["{\mathtt{\Psi}_{F,G}}"{description}, curve={height=-6pt}, shorten <=3pt, shorten >=3pt, Rightarrow, from=3-2, to=2-1]
	\arrow["{\mathtt{\Psi}_{G,E}}"{description}, curve={height=6pt}, shorten <=6pt, shorten >=6pt, Rightarrow, from=4-3, to=1-2]
	\arrow[""{name=1, anchor=center, inner sep=0}, shorten >=18pt, no head, from=2-3, to=2-1]
	\arrow["{\varphi_G}", shorten >=5pt, no head, from=1-2, to=0]
	\arrow["{\Delta_{F,E}}"{description}, from=1, to=2-1]
			\arrow["{ \Delta_{F,G,E} }"{description}, curve={height=6pt}, shorten >=7pt, shorten <=2pt, Rightarrow, from=3-2, to=a]
\end{tikzcd}\]
\noindent and a trivial $3$-cell \emph{bimonoid coherence}
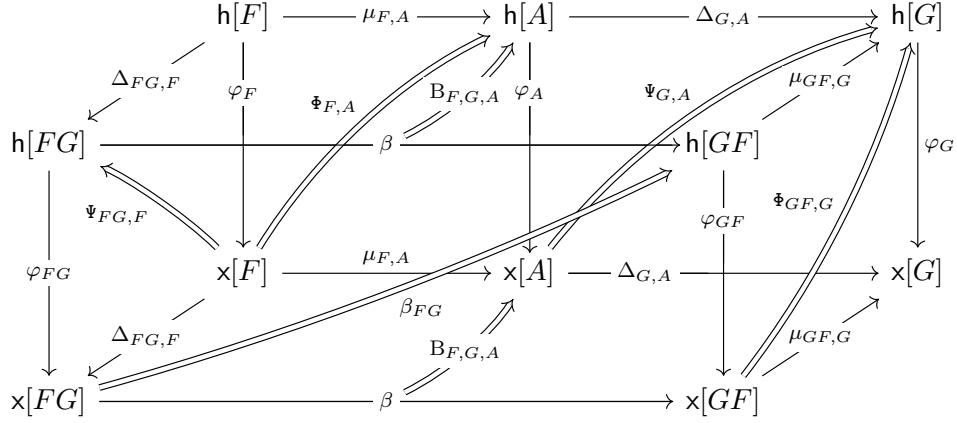
\[\begin{tikzcd} [row sep=large, column sep=large]
	& {\textsf{h}[F]} && {\textsf{h}[A]} && {\textsf{h}[G]} \\
	{\textsf{h}[FG]} &&&& {\textsf{h}[GF]} \\
	& {\textsf{x}[F]} && {\textsf{x}[A]} && {\textsf{x}[G]} \\
	{\textsf{x}[FG]} &&&& {\textsf{x}[GF]}
			\arrow[""{name=a, anchor=center, inner sep=0}, no head, from=2-1, to=2-5]
						\arrow[""{name=b, anchor=center, inner sep=0}, no head, from=4-1, to=4-5]
	\arrow["{\mathtt{\Psi}_{FG,F}}", curve={height=6pt}, Rightarrow, from=3-2, to=2-1]
	\arrow["{\mathtt{\Phi}_{F,A}}", curve={height=-15pt}, Rightarrow, from=3-2, to=1-4]
	\arrow["{\mathtt{\Psi}_{G,A}}", curve={height=-20pt}, Rightarrow, from=3-4, to=1-6]
	\arrow["{\mu_{F,A}}"{description}, from=1-2, to=1-4]
	\arrow["{\Delta_{G,A}}"{description}, from=1-4, to=1-6]
	\arrow["{\Delta_{FG,F}}"{description}, from=1-2, to=2-1]
	\arrow["{\beta}"{description}, from=2-1,   shorten >=-3pt,    to=2-5]
	\arrow["{\mu_{GF,G}}"{description}, from=2-5, to=1-6]
	\arrow["{\varphi_{FG}}"{description}, from=2-1, to=4-1]
	\arrow[""{name=0, anchor=center, inner sep=0}, from=1-2, to=3-2]
	\arrow[""{name=1, anchor=center, inner sep=0}, shorten >=-2pt,   from=1-4, to=3-4]
	\arrow[""{name=2, anchor=center, inner sep=0}, from=2-5, to=4-5]
	\arrow["{\varphi_G}", from=1-6, to=3-6]
	\arrow["{\mu_{F,A}}", from=3-2, to=3-4]
	\arrow["{\mu_{GF,G}}"{description}, from=4-5, to=3-6]
	\arrow["{\Delta_{FG,F}}"{description}, from=3-2, to=4-1]
	\arrow["{\beta}"{description}, from=4-1, to=4-5]
	\arrow[""{name=3, anchor=center, inner sep=0}, from=3-4, to=3-6]
	\arrow["{\Delta_{G,A}}"{description}, shorten >=5pt, no head, from=3-4, to=3]
	\arrow["{\varphi_F}"{description}, shorten >=3pt, no head, from=1-2, to=0]
	\arrow["{\varphi_A}"{description}, shorten >=3pt, no head, from=1-4, to=1]
	\arrow["{\varphi_{GF}}"{description}, shorten >=3pt, no head, from=2-5, to=2]
	\arrow["{\mathtt{\Phi}_{GF,G}}", curve={height=11pt}, Rightarrow, from=4-5, to=1-6]
				\arrow["{\beta_{FG}}"', curve={height=9pt}, Rightarrow, from=4-1, to=2-5]
					\arrow["{\text{B}_{F,G,A}}"{description}, curve={height=7pt}, shorten <=7pt, Rightarrow, from=a, to=1-4]
										\arrow["{\text{B}_{F,G,A}}"{description}, curve={height=7pt}, shorten <=7pt, Rightarrow, from=b, to=3-4]
\end{tikzcd}\]
for all meaningful choices. We have a well-defined composition of homomorphisms by composing $2$-cells. A homomorphism is called \emph{strong}, respectively \emph{strict}, if all the $2$-cells ${\sigma_F}$, ${\beta_F}$, ${\mathtt{\Phi}_{F,G}}$, ${\mathtt{\Psi}_{F,G}}$, are isomorphisms, respectively identities.

Let us make these coherence conditions explicit in the case $\textsf{C}\in \{\textsf{Cat},\textsf{Cat}_\bullet\}$, at the level of components of the $2$-cells. We have that ${\mathtt{\Phi}_{F,G}}$ consists of a natural morphism in $\textsf{x}[G]$
\[
{\mathtt{\Phi}_{F,G}}_\ta : \mu_{F,G}   \big( \varphi_F(\ta)  \big   )  \to    \varphi_G  \big (\mu_{F,G}(\ta) \big )
\] 
for each object $\ta \in \textsf{h}[F]$, and ${\mathtt{\Psi}_{F,G}}$ consists of a natural morphism in $\textsf{x}[F]$
\[
{\mathtt{\Psi}_{F,G}}_\ta  : \Delta_{F,G}   \big   (   \varphi_G(\ta) \big       ) \to   \varphi_F\big ( \Delta_{F,G} ( \ta ) \big )
\]
for each object $\ta \in  \textsf{h}[G]$. Then multiplication \hbox{$\sigma$-naturality} coherence says
\[
\varphi_{G'} ( {\mu_{F,G,\sigma}}_\ta ) \circ {\sigma}_{\mu_{F,G}(\ta)}  \circ  {\mathtt{\Phi}'_{F,G}}_{\ta}  = {\mathtt{\Phi}_{F',G'}}_{\ta'} \circ \mu_{F',G'}({\sigma}_{\ta}) \circ {\mu_{F,G,\sigma}}_{\varphi_F(\ta)}
\]
multiplication $\beta$-naturality coherence says
\[
\varphi_{\wt{G}} ( {\mu_{F,G,\beta}}_\ta ) \circ {\beta}_{\mu_{F,G}(\ta)} \circ \wt{{\mathtt{\Phi}_{F,G}}}_{\ta} 
= 
{\mathtt{\Phi}_{\wt{G}F,\wt{G}}}_{\tilde{\ta}}\circ  \mu_{\wt{G}F, \wt{G}}( {\hat{\beta}}_\ta  ) {\circ \mu_{F,G,\beta}}_{\varphi_F(\ta)}
\]
associativity coherence says
\[
\varphi_E ({\mu_{F,G,E}}_{\ta}) \circ {\mathtt{\Phi}_{G,E}}_{\mu_{F,G}(\ta)}  \circ  \mu_{G,E} (   {\mathtt{\Phi}_{F,G}}_{\ta}  )  
=
{\mathtt{\Phi}_{F,E}}_{\mathtt{a}} \circ {\mu_{F,G,E}}_{\varphi_F(\ta)}
\]
comultiplication $\sigma$-naturality coherence says
\[
\varphi_{F'}({\Delta_{F,G,\sigma}}_\ta)  \circ {\sigma}_{\Delta_{F,G}( \ta )} \circ {\mathtt{\Psi}'_{F,G}}_{\ta}  
= 
{\mathtt{\Psi}_{F',G'}}_{\ta'}\circ  \Delta_{F',G'} (  {\sigma_{G}}_\ta ) \circ{\Delta_{F,G,\sigma}}_{\varphi_G(\ta)}
\]
comultiplication $\beta$-naturality coherence says 
\[
\varphi_{\wt{G}F}({\Delta_{F,G,\beta}}_\ta) \circ {\hat{\beta}}_{\Delta_{F,G}( \ta )} \circ \wt{{\mathtt{\Psi}_{F,G}}}_{\ta} 
= 
{\mathtt{\Psi}_{\wt{G}F,\wt{G}}}_{\tilde{\ta}} \circ \Delta_{ \wt{G}F,\wt{G} } (  {\beta}_{\ta} ) \circ {\Delta_{F,G,\beta}}_{\varphi_G(\ta)}
\]
coassociativity coherence says
\[
\varphi_F ({\Delta_{F,G,E}}_{\ta}) \circ  {\mathtt{\Psi}_{F,G}}_{\Delta_{G,E}(\ta)}  \circ  \Delta_{F,G} (   {\mathtt{\Psi}_{G,E}}_{\ta}  )  
=
{\mathtt{\Psi}_{F,E}}_{\mathtt{a}} \circ {\mu_{F,G,E}}_{\varphi_E(\ta)}
\] 
and bimonoid coherence says   
\[
{\mathtt{\Psi}_{G,A}}_{\mu_{F,A}(\ta)} \circ   \Delta_{G,A}( {\mathtt{\Phi}_{F,A}}_{\ta}  ) \circ  {\text{B}_{F,G,A}}_{\varphi_F(\ta)}
\]
\[
=   
 \varphi_G({\text{B}_{F,G,A}}_{\ta}) \circ{\mathtt{\Phi}_{GF,G}}_{\wt{\Delta_{FG,F}}(\ta)}   
\circ   \mu_{GF,F} ({\beta}_{ \Delta_{FG,F} (\ta)}) \circ
\mu_{GF,G}(\wt{\mathtt{\Psi}_{FG,F}}_\ta) 
.\]
We leave the definition of lax homomorphisms for bimonoids in cospecies implicit. The coherence laws will all be unchanged.



\subsection{Joyal-Theoretic Bimonoids} \label{sec:Joyal-Theoretic Bimonoids}

Throughout this section, we assume all species and bimonoids are strict. One can develop the notion of Joyal-theoretic bimonoids in the non-strict setting, as we did for Joyal-theoretic species, however we shall not need it in this paper. Indeed, all the examples of bimonoids we come across for which the algebraic structure holds only up to coherent isomorphism are in any case not Joyal-theoretic.

Given a Joyal-theoretic species $\textsf{h}$, we can give $\textsf{h}$ the structure of a bimonoid as follows. Let $\text{h}$ be the connected Joyal species corresponding to $\textsf{h}$, so that $\textsf{h}=\textsf{Br}(\text{h})$. Suppose $\text{h}$ is equipped with a pair of maps
\begin{equation} \label{eq:muST}        
	\mu_{S,T}  :    \text{h} [S] \otimes   \text{h} [T]   \to  \textsf{h}   [I   ] 
	\qquad \text{and} \qquad
	\Delta_{S,T}  :  \text{h}   [I  ] \to   \text{h} [S] \otimes   \text{h} [T]    
\end{equation}
for each decomposition $S\sqcup T=I$. When either $S$ or $T$ is the empty set, we require that $\mu_{S,T}$ and $\Delta_{S,T}$ are the unitors of $\textsf{C}$, e.g. $\mu_{\emptyset,T}$ is the right unitor
\[
\mu_{\emptyset,T}  :    \underbrace{\text{h} [\emptyset]}_{=1_\textsf{C}} \otimes   \text{h} [T]  \to \textsf{h}   [T   ] 
.\]
Therefore the actual data here is just a pair of maps $\mu_{S,T}$ and $\Delta_{S,T}$ for each two-lump composition $(S,T)$. Let $[I;2]$ denote the set of two-lump compositions of $I$.

\begin{remark}
	Without giving a full exposition of the (conjectural) generalized theory which uses decompositions in place of compositions, the inclusion of empty sets here may seem ad hoc. One can avoid them (since the data is just a pair of functions for each two-lump composition) but then needs to separately include several instances of the Joyal bimonoid axiom below, see \autoref{ex:bimonoid}.
\end{remark}

If $\textsf{h}$ is to be a bimonoid, we need to define maps $\mu_{F,G}$ and $\Delta_{F,G}$ for compositions $G\leq F$. For $(S,T)\in [I;2]$, put
\[
\mu_{(S,T),(I)}:= \mu_{S,T} 
\qquad \text{and} \qquad 
\Delta_{(S,T),(I)} := \Delta_{S,T}
.\] 
Towards defining $\mu_{F,G}$ and $\Delta_{F,G}$ in general, create a sequence of the form
\[
G =F_L \leq  \dots \leq F_1 = F
\]
where $F_{j-1}$ is obtained from $F_j$ by merging just two lumps, say $S_j$ and $T_j$. This is always possible, but the sequence will not be unique. Put
\[
\mu_{F_{j-1},F_j} :=   \text{id} \otimes\dots \otimes   \text{id}   \otimes   \mu_{ S_j, T_j } \otimes \text{id} \otimes \dots \otimes \text{id}
\]
and
\[
\Delta_{F_{j-1},F_j} :=   \text{id} \otimes\dots \otimes   \text{id}   \otimes   \Delta_{ S_j, T_j } \otimes \text{id} \otimes \dots \otimes \text{id}
.\]
If we want $\textsf{h}$ to satisfy associativity and coassociativity, we must then put
\begin{equation} \label{eq:aso,coaso}        
	\mu_{F,G}: =     \mu_{F_{L-1},F_L}  \circ \dots \circ  \mu_{F_{1},F_2}
	\qquad \text{and} \qquad
	\Delta_{F,G}: =   \Delta_{F_{1},F_2}   \circ \dots \circ    \Delta_{F_{L-1},F_L}.
\end{equation}
The Joyal associativity and Joyal coassociativity axioms, see below, ensure these definitions do not depend upon the choice of sequence $F_1 \leq  \dots \leq F_L$. Recalling the notation \textcolor{blue}{(\refeqq{eq:F})} (e.g. $\mu_F:=\mu_{F,(I)}$), and letting $G=(T_1,\dots, T_l)$, the would-be (co)multiplication factorizes as follows,
\begin{equation}\label{eq:fact}
	\mu_{F,G} =  \mu_{F|_{T_1}}\otimes \dots \otimes   \mu_{F|_{T_l}}      \qquad \text{and} \qquad    \Delta_{F,G} =\Delta_{F|_{T_1}}\otimes \dots \otimes   \Delta_{F|_{T_l}}   
	.      
\end{equation}
We now ask what properties must $\mu_{S,T}$ and $\Delta_{S,T}$ satisfy in order for $\mu_{F,G}$ and $\Delta_{F,G}$ to be the well-defined multiplication and comultiplication of a bimonoid. Consider the diagrams \emph{Joyal multiplication naturality} and \emph{Joyal comultiplication naturality}\footnote{\ juxtaposition, e.g. $ST$ or $\sigma\tau$, abbreviates disjoint union $S\sqcup T$ or $\sigma\sqcup \tau$}
\begin{figure}[H]
	\centering
	\begin{minipage}{.5\textwidth}
		\centering
		\begin{tikzcd}[column sep=large]  
			{\text{h}[S]\otimes \text{h}[T]} & {\text{h}[S']\otimes \text{h}[T']} \\
			{\text{h}[I]} & {\text{h}[J]}
			\arrow["{\mu_{S,T}}"', from=1-1, to=2-1]
			\arrow["{\text{h}[\sigma]\otimes \text{h}[\tau]}", from=1-1, to=1-2]
			\arrow["{\text{h}[\sigma\tau]}"', from=2-1, to=2-2]
			\arrow["{\mu_{S',T'}}", from=1-2, to=2-2]
		\end{tikzcd}
	\end{minipage}%
	\begin{minipage}{.5\textwidth}
		\centering
		\begin{tikzcd}[column sep=huge]  
			& {\text{h}[S]\otimes \text{h}[T]} & {\text{h}[S']\otimes \text{h}[T']} \\
			{} & {\text{h}[I]} & {\text{h}[J]}
			\arrow["{\text{h}[\sigma|_{S'}]\otimes \text{h}[\sigma|_{T'}]}", from=1-2, to=1-3]
			\arrow["{\text{h}[\sigma]}"', from=2-2, to=2-3]
			\arrow["{\Delta_{S',T'}}"', from=2-3, to=1-3]
			\arrow["{\Delta_{S,T}}", from=2-2, to=1-2]
		\end{tikzcd}
	\end{minipage}
\end{figure}

\noindent the diagrams \emph{Joyal associativity} and \emph{Joyal coassociativity}

\begin{figure}[H]
	\centering
	\begin{minipage}{.5\textwidth}
		\centering
		\begin{tikzcd}[column sep=large,row sep=large] 
			\text{h}[S]\otimes \text{h}[T]\otimes \text{h}[U]    \arrow[d, " \mu_{S,T} \otimes  \text{id}"'] \arrow[r,   "\text{id}  \otimes  \mu_{T,U}"]   
			& 
			\text{h}[S]\otimes \text{h}[T U] \arrow[d,   "\mu_{S,T  U} "]     
			\\
			\text{h}[S T]\otimes \text{h}[U]   \arrow[r,   "\mu_{S T ,U}"']       
			& 
			\text{h}[I]
		\end{tikzcd}
	\end{minipage}%
	\begin{minipage}{.5\textwidth}
		\centering
		\begin{tikzcd}[column sep=large,row sep=large] 
			\text{h}[S]\otimes \text{h}[T]\otimes \text{h}[U]    \arrow[d,leftarrow, "\Delta_{S,T} \otimes  \text{id}"'] \arrow[r,leftarrow,   "\text{id}  \otimes  \Delta_{T,U}"]   
			& 
			\text{h}[S]\otimes \text{h}[T U] \arrow[d,leftarrow,  "\Delta_{S,T  U} "]     
			\\
			\text{h}[S T]\otimes \text{h}[U]   \arrow[r, leftarrow,   "\Delta_{S T ,U}"']       
			& 
			\text{h}[I]
		\end{tikzcd}
	\end{minipage}
\end{figure}
\noindent and the diagram \emph{Joyal bimonoid axiom}
\[\begin{tikzcd}[row sep=large]  
	{\text{h}[ST]\otimes\text{h}[UV]} & {\text{h}[I]} & {\text{h}[SU]\otimes\text{h}[TV]} \\
	{\text{h}[S]\otimes\text{h}[T]\otimes\text{h}[U]\otimes\text{h}[V]} && {\text{h}[S]\otimes\text{h}[U]\otimes\text{h}[T]\otimes\text{h}[V]}
	\arrow["{\Delta_{S,T}\otimes \Delta_{U,V}}"', from=1-1, to=2-1]
	\arrow["{\mu_{ST,UV}}", from=1-1, to=1-2]
	\arrow["{\Delta_{SU,TV}}", from=1-2, to=1-3]
	\arrow["{\text{id}\otimes B \otimes \text{id}}"', from=2-1, to=2-3]
	\arrow["{\mu_{S,U}\otimes \mu_{T,V}}"', from=2-3, to=1-3]
\end{tikzcd}\]
where $B$ denotes the braiding $B: \textsf{h}[T]   \otimes  \textsf{h}[U]  \xrightarrow{\sim}  \textsf{h}[U]   \otimes  \textsf{h}[T]$. It is important we allow the sets $S,T,U,V$ to be empty in these diagrams. 

\begin{ex}\label{ex:bimonoid}
	If $U=\emptyset$ and $S,T,V\neq \emptyset$, then the bimonoid axiom is equivalent to the following diagram consisting of only nonempty sets,
	\[\begin{tikzcd}
		& {\text{h}[I]} \\
		{\text{h}[ST]\otimes\text{h}[V]} && {\textsf{h}[S]\otimes\text{h}[TV]} \\
		& {\text{h}[S]\otimes\text{h}[T]\otimes\text{h}[V]}
		\arrow["{\Delta_{S,T}\otimes \text{id}}"', from=2-1, to=3-2]
		\arrow["{\mu_{ST,V}}", from=2-1, to=1-2]
		\arrow["{\Delta_{S,TV}}", from=1-2, to=2-3]
		\arrow["{\text{id}\otimes \mu_{T,V}}"', from=3-2, to=2-3].
	\end{tikzcd}\]
	There are several other diagrams like this, obtained by setting various subsets equal to the empty set. If we did not allow $S,T=\emptyset$ in \textcolor{blue}{(\refeqq{eq:muST})}, then we would have had to separately include each instance of these diagrams in our list of axioms.
\end{ex}

\begin{remark}
To relate this to more familiar notions of bimonoid or bialgebra, one can check that in the case $\textsf{C}=\textsf{Vec}$, these Joyal axioms say exactly that $\mu_{S,T}$ and $\Delta_{S,T}$ give $\text{h}$ the structure of a (connected) bimonoid internal to Joyal vector species with respect to the Day convolution. Such an internal bimonoid decategorifies to give classical connected graded bialgebras in various ways, see \cite[Chapter 15]{aguiar2010monoidal}.

A similar result holds for $\textsf{C}=\textsf{Set}$, but there the bimonoid is a bialgebra over a bimonad, rather than an internal bimonoid. 
\end{remark}


\begin{thm}[Aguiar-Mahajan]\label{thm:joyal}
	The maps $\mu_{F,G}$ and $\Delta_{F,G}$ as defined in \textcolor{blue}{(\refeqq{eq:aso,coaso})} give $\textsf{h}$ the structure of a bimonoid if and only if the diagrams Joyal (co)multiplication naturality, Joyal (co)associativity and the Joyal bimonoid axiom commute.
\end{thm}
\begin{proof}
	This result is one of the key ideas behind Aguiar-Mahajan's generalization of species and bimonoids, see \cite[Proposition 50]{aguiar2013hopf} and \cite[Section 2.12.3]{AM22}. The author cannot find an explicit proof in their published work so far. The only non-trivial part of the proof is the bimonoid axiom, which can be proved using cartesian product notation or sumless Sweedler notation for the cases $\textsf{C}\in \{ \textsf{Set}, \textsf{Vec},\textsf{Cat},\textsf{Cat}_\bullet\}$. For a generic symmetric monoidal \hbox{$1$-category} $\textsf{C}$, the bimonoid axiom can be proved using string diagram notation in $\textsf{C}$, which we sketch below.
	
	Let us deal with the `only if' direction first. For $S,T\neq \emptyset$, Joyal (co)multiplication naturality is a special case of (co)multiplication $\sigma$-naturality, obtained by setting $F=(S,T)$. (Co)multiplication naturality is trivial if $S=\emptyset$ or $T=\emptyset$. For $S,T,U\neq \emptyset$, Joyal associativity is obtained from the associativity diagrams 
	\[\begin{tikzcd}
		& {\textsf{h}\big[(ST,U)\big ]} &&&& {\textsf{h}\big[(S,TU)\big]} \\
		{\textsf{h}\big[(S,T,U)\big]} && {\textsf{h}[I]} && {\textsf{h}\big[(S,T,U)\big)} && {\textsf{h}[I]}
		\arrow["{\mu_{(S,T,U),(I)}}"', from=2-1, to=2-3]
		\arrow["{\mu_{(S,T,U),(ST,U)}}", from=2-1, to=1-2]
		\arrow["{\mu_{(ST,U),(I)}}", from=1-2, to=2-3]
		\arrow["{\mu_{(S,T,U),(I)}}"', from=2-5, to=2-7]
		\arrow["{\mu_{(S,T,U),(S,TU)}}", from=2-5, to=1-6]
		\arrow["{\mu_{(S,TU),(I)}}", from=1-6, to=2-7]
	\end{tikzcd}\]
	by gluing along the morphism $\mu_{(S,T,U),(I)}$. Joyal associativity is trivial if any of $S,T,U$ are empty. Joyal coassociativity follows similarly. The Joyal bimonoid axiom is either trivial, or a case of the bimonoid axiom where $l(F),l(G)\leq 2$. For example, if $S,T,U,V\neq \emptyset$, then we obtained the bimonoid axiom with $F=(ST,UV)$ and $G=(SU,TV)$, and if $U=\emptyset$ as in \autoref{ex:bimonoid}, then we obtain the bimonoid axiom with $F=(ST,V)$ and $G=(S,TV)$.
	
	For the `if' direction, there are nine diagrams to check. For $\sigma$-multiplication naturality, we have
	\[
	\sigma \circ \mu_{F,G} =  \Big (\bigotimes_{1\leq j \leq l} \sigma|_{T'_j}\Big ) \circ \Big ( \bigotimes_{1\leq j \leq l} \mu_{F|_{T_j}}\Big )  
	= \underbrace{\bigotimes_{1\leq j \leq l} \sigma|_{T'_j} \circ \mu_{F|_{T_j}}  =    \bigotimes_{1\leq j \leq l} \Big (  \mu_{F'|_{T'_j}} \circ \bigotimes_{\mathtt{i}_j} \sigma|_{S'_{\mathtt{i}_j}}  \Big )}_{\text{Joyal multiplication naturality}}
	=\mu_{F',G'} \circ \sigma  
	.\]
	Then $\sigma$-comultiplication naturality follows similarly from Joyal comultiplication naturality. We have that $\beta$-(co)multiplication naturality follows from the fact that the braiding of $\textsf{C}$ is a natural transformation. (Co)associativity follows directly from the definition \textcolor{blue}{(\refeqq{eq:aso,coaso})}, with Joyal (co)associativity ensuring that the (co)multiplication is well-defined.

To prove the bimonoid axiom, it is sufficient to consider the totally lumped case $A=(I)$, since the generic case $A=(A_1,\dots, A_a)$ then follows by factorization of the (co)multiplication as follows,
\[
\Delta_{G,A} \circ \mu_{F,A} =
\Big ( \bigotimes_{1\leq m \leq a}  \Delta_{G|_{A_m}} \Big )\circ \Big (   \bigotimes_{1\leq m \leq a}  \mu_{F|_{A_m}} \Big )
\]
\[
=
\underbrace{\bigotimes_{1\leq m \leq a} \big (  \Delta_{G|_{A_m}} \circ  \mu_{F|_{A_m}} \big )
=
\bigotimes_{1\leq m \leq a} \big (     \mu_{G|_{A_m} F|_{A_m},G|_{A_m}} \circ \beta \circ \Delta_{F|_{A_m}G|_{A_m},F|_{A_m}}         \big )}_{\text{bimonoid axiom in the case $A=(I)$}}
=
\mu_{GF,G} \circ \beta \circ \Delta_{FG,F} 
.\]
We use string diagram notation in $\textsf{C}$. Abbreviating $I=\textsf{h}[I]$, $\mu=\mu_{F}$ and $\Delta=\Delta_{F}$, the multiplication of $\textsf{h}$ is then denoted e.g.
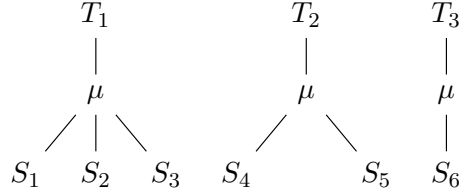
\[\begin{tikzcd}[row sep=0.5cm, column sep=0.2cm]
	& {T_1} &&& {T_2} && {T_3} \\
	& \mu &&& \mu && \mu \\
	{S_1} & {S_2} & {S_3} & {S_4} && {S_5} & {S_6}
	\arrow[no head, from=2-2, to=3-1]
	\arrow[no head, from=2-2, to=3-2]
	\arrow[no head, from=2-2, to=3-3]
	\arrow[no head, from=2-5, to=3-4]
	\arrow[no head, from=2-5, to=3-6]
	\arrow[no head, from=2-7, to=3-7]
	\arrow[no head, from=2-2, to=1-2]
	\arrow[no head, from=1-5, to=2-5]
	\arrow[no head, from=1-7, to=2-7]
\end{tikzcd}\]
where $F=(S_1,\dots, S_6)$ and $G=(  T_1 = S_1\sqcup S_2 \sqcup S_3, T_2= S_4\sqcup S_5, T_3=S_6 )$. Similarly for the comultiplication of $\textsf{h}$. The bimonoid axiom in the case $A=(I)$ is then denoted e.g. 
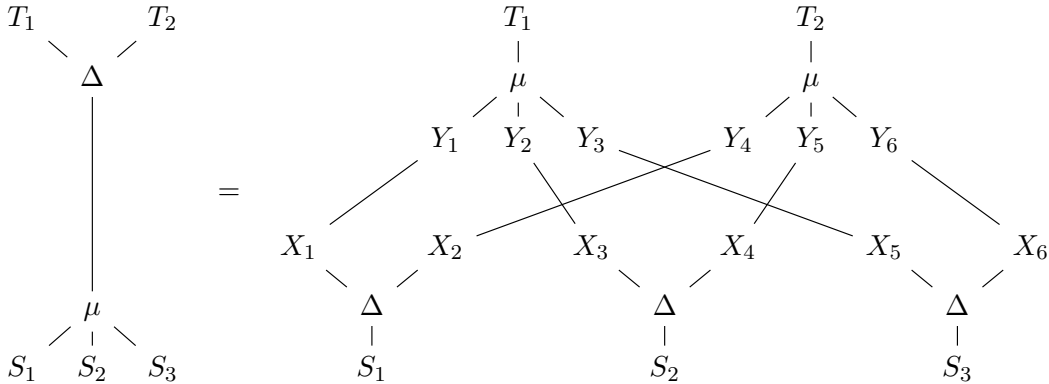
\[\begin{tikzcd}[row sep=0.2cm, column sep=0.2cm]
	{T_1} && {T_2} &&&&& {T_1} &&&& {T_2} \\
	& \Delta &&&&&& \mu &&&& \mu \\
	&&&&&& {Y_1} & {Y_2} & {Y_3} && {Y_4} & {Y_5} & {Y_6} \\
	&&& {=} \\
	&&&& {X_1} && {X_2} && {X_3} && {X_4} && {X_5} && {X_6} \\
	& \mu &&&& \Delta &&&& \Delta &&&& \Delta \\
	{S_1} & {S_2} & {S_3} &&& {S_1} &&&& {S_2} &&&& {S_3}
	\arrow[no head, from=6-2, to=7-1]
	\arrow[no head, from=6-2, to=7-2]
	\arrow[no head, from=6-2, to=7-3]
	\arrow[no head, from=6-2, to=2-2]
	\arrow[no head, from=2-2, to=1-1]
	\arrow[no head, from=2-2, to=1-3]
	\arrow[no head, from=6-6, to=5-5]
	\arrow[no head, from=6-6, to=5-7]
	\arrow[no head, from=6-10, to=5-9]
	\arrow[no head, from=6-10, to=5-11]
	\arrow[no head, from=6-14, to=5-13]
	\arrow[no head, from=6-14, to=5-15]
	\arrow[no head, from=6-6, to=7-6]
	\arrow[no head, from=2-8, to=3-7]
	\arrow[no head, from=2-8, to=3-8]
	\arrow[no head, from=2-8, to=3-9]
	\arrow[no head, from=2-12, to=3-11]
	\arrow[no head, from=2-12, to=3-12]
	\arrow[no head, from=2-12, to=3-13]
	\arrow[no head, from=2-8, to=1-8]
	\arrow[no head, from=2-12, to=1-12]
	\arrow[no head, from=6-10, to=7-10]
	\arrow[no head, from=6-14, to=7-14]
	\arrow[no head, from=5-5, to=3-7]
	\arrow[no head, from=5-7, to=3-11]
	\arrow[no head, from=5-9, to=3-8]
	\arrow[no head, from=5-11, to=3-12]
	\arrow[no head, from=5-13, to=3-9]
	\arrow[no head, from=5-15, to=3-13]
\end{tikzcd}\]
Drawn here is the diagram for the case $l(F)=3$ and $l(G)=2$. In general, $F$ has length $k$ and $G$ has length $l$, each $S_i$ on the right-hand side divides into $l$ many $X$'s, and each $T_j$ is the union of $k$ many $Y$'s. Note some of the $X$'s and $Y$'s may be empty, in which case one can either complete maps by tensoring with unitors, or just remove those nodes from the diagram. 

The Joyal bimonoid axiom is denoted
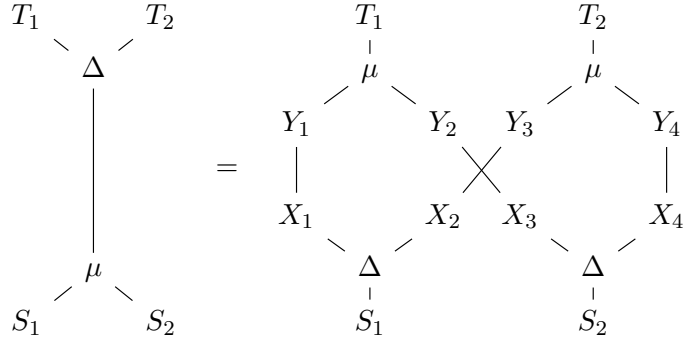
\[\begin{tikzcd}[row sep=0.1cm, column sep=0.2cm]
	{T_1} && {T_2} &&& {T_1} &&& {T_2} \\
	& \Delta &&&& \mu &&& \mu \\
	&&&& {Y_1} && {Y_2} & {Y_3} && {Y_4} \\
	&&& {=} \\
	&&&& {X_1} && {X_2} & {X_3} && {X_4} \\
	& \mu &&&& \Delta &&& \Delta \\
	{S_1} && {S_2} &&& {S_1} &&& {S_2}
	\arrow[no head, from=6-6, to=7-6]
	\arrow[no head, from=6-6, to=5-5]
	\arrow[no head, from=6-6, to=5-7]
	\arrow[no head, from=1-6, to=2-6]
	\arrow[no head, from=2-6, to=3-5]
	\arrow[no head, from=2-6, to=3-7]
	\arrow[no head, from=7-9, to=6-9]
	\arrow[no head, from=6-9, to=5-8]
	\arrow[no head, from=6-9, to=5-10]
	\arrow[no head, from=1-9, to=2-9]
	\arrow[no head, from=2-9, to=3-8]
	\arrow[no head, from=2-9, to=3-10]
	\arrow[no head, from=5-5, to=3-5]
	\arrow[no head, from=5-7, to=3-8]
	\arrow[no head, from=5-8, to=3-7]
	\arrow[no head,  from=5-10, to=3-10]
	\arrow[no head, from=7-3, to=6-2]
	\arrow[no head, from=7-1, to=6-2]
	\arrow[no head, from=6-2, to=2-2]
	\arrow[no head, from=2-2, to=1-3]
	\arrow[no head, from=2-2, to=1-1]
\end{tikzcd}\]
First, on the left-hand side of the bimonoid axiom string digram, use (co)assosiativity to replace the top and bottom nodes with the corresponding left comb trees (or indeed any appropriate binary tree). For example, if $l(F)=3$ and $l(G)=4$, we obtain the string diagram
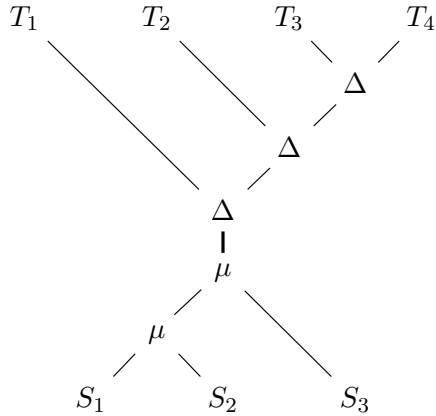
\[\begin{tikzcd}[row sep=0.3cm, column sep=0.15cm]
	{T_1} && {T_2} && {T_3} && {T_4} \\
	&&&&& \Delta \\
	&&& {} & \Delta \\
	&&& \Delta \\
	&&& \mu \\
	&& \mu \\
	& {S_1} && {S_2} && {S_3}
	\arrow[no head, from=5-4, to=6-3]
	\arrow[no head, from=6-3, to=7-2]
	\arrow[very thick, no head, from=5-4, to=4-4]
	\arrow[no head, from=4-4, to=1-1]
	\arrow[no head, from=4-4, to=3-5]
	\arrow[no head, from=3-5, to=2-6]
	\arrow[no head, from=3-5, to=1-3]
	\arrow[no head, from=2-6, to=1-5]
	\arrow[no head, from=2-6, to=1-7]
	\arrow[no head, from=6-3, to=7-4]
	\arrow[no head, from=5-4, to=7-6]
\end{tikzcd}\]
All that appears now is binary (co)multiplication. Next, we repeatedly substitute in the right-hand side of the Joyal bimonoid axiom until no $\mu$ comes before a $\Delta$. This is always possible, since the Joyal bimonoid axiom effectively permutes a $\mu$ occurring before an adjacent $\Delta$, and replaces them with two $\Delta$'s occurring before two $\mu$'s.

For example, in the above string diagram, we have one pair of adjacent $\mu$ and $\Delta$ in the wrong order, marked with a thick line. After permuting them using the Joyal bimonoid axiom, we obtain the string diagram 
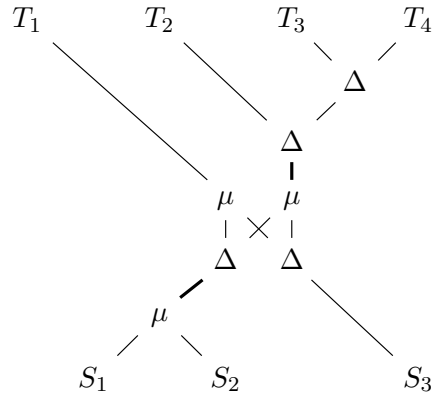
\[\begin{tikzcd}[row sep=0.25cm, column sep=0.15cm]
	{T_1} && {T_2} && {T_3} && {T_4} \\
	&&&&& \Delta \\
	&&& {} & \Delta \\
	&&& \mu & \mu \\
	&&& \Delta & \Delta \\
	&& \mu \\
	& {S_1} && {S_2} &&  &  {S_3}
	\arrow[very thick, no head, from=5-4, to=6-3]
	\arrow[no head, from=6-3, to=7-2]
	\arrow[no head, from=5-4, to=4-4]
	\arrow[no head, from=4-4, to=1-1]
	\arrow[no head, from=3-5, to=2-6]
	\arrow[no head, from=3-5, to=1-3]
	\arrow[no head, from=2-6, to=1-5]
	\arrow[no head, from=2-6, to=1-7]
	\arrow[no head, from=6-3, to=7-4]
	\arrow[no head, from=4-5, to=5-5]
	\arrow[no head, from=4-4, to=5-5]
	\arrow[no head, from=4-5, to=5-4]
	\arrow[very thick, no head, from=4-5, to=3-5]
	\arrow[no head, from=5-5, to=7-7]
\end{tikzcd}\]
We now have two places we can substitute in, marked with two thick lines. We continue to substitute in for the higher of the two thick lines, which after several iterations gives the string diagram
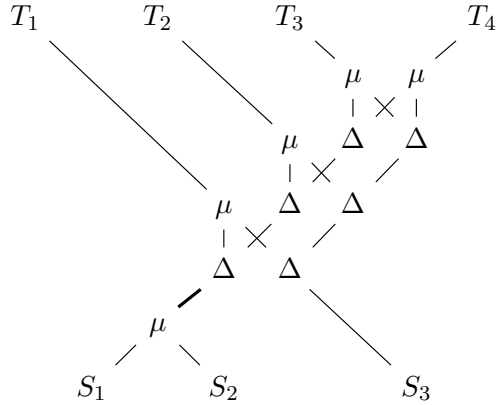
\[\begin{tikzcd}[row sep=0.25cm, column sep=0.15cm]
	{T_1} && {T_2} && {T_3} &&& {T_4} \\
	&&&&& \mu & \mu \\
	&&& {} & \mu & \Delta & \Delta \\
	&&& \mu & \Delta & \Delta \\
	&&& \Delta & \Delta \\
	&& \mu \\
	& {S_1} && {S_2} &&  & {S_3}
	\arrow[very thick, no head, from=5-4, to=6-3]
	\arrow[no head, from=6-3, to=7-2]
	\arrow[no head, from=5-4, to=4-4]
	\arrow[no head, from=4-4, to=1-1]
	\arrow[no head, from=3-5, to=1-3]
	\arrow[no head, from=2-6, to=1-5]
	\arrow[no head, from=6-3, to=7-4]
	\arrow[no head, from=4-4, to=5-5]
	\arrow[no head, from=4-5, to=5-4]
	\arrow[no head, from=4-5, to=3-5]
	\arrow[no head, from=5-5, to=7-7]
	\arrow[no head, from=3-5, to=4-6]
	\arrow[no head, from=3-6, to=4-5]
	\arrow[no head, from=4-6, to=5-5]
	\arrow[no head, from=2-7, to=3-7]
	\arrow[no head, from=3-6, to=2-6]
	\arrow[no head, from=2-6, to=3-7]
	\arrow[no head, from=2-7, to=3-6]
	\arrow[no head, from=4-6, to=3-7]
	\arrow[no head, from=2-7, to=1-8]
\end{tikzcd}\]
We now do the same for the second thick line, after which we obtain a string diagram with all $\mu$'s occurring after the $\Delta$'s, thus
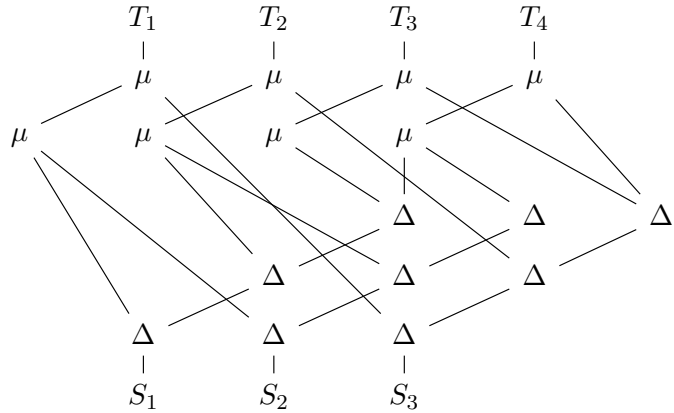
\[\begin{tikzcd}[row sep=0.25cm, column sep=1cm]
	& {T_1} & {T_2} & {T_3} & {T_4} \\
	& \mu & \mu & \mu & \mu \\
	\mu & \mu & \mu & \mu \\
	\\
	&&& \Delta & \Delta & \Delta \\
	&& \Delta & \Delta & \Delta \\
	& \Delta & \Delta & \Delta \\
	& {S_1} & {S_2} & {S_3}
	\arrow[no head, from=2-2, to=1-2]
	\arrow[no head, from=2-3, to=1-3]
	\arrow[no head, from=2-4, to=1-4]
	\arrow[no head, from=7-3, to=8-3]
	\arrow[no head, from=2-2, to=7-4]
	\arrow[no head, from=5-4, to=6-3]
	\arrow[no head, from=7-4, to=8-4]
	\arrow[no head, from=2-3, to=6-5]
	\arrow[no head, from=3-4, to=5-4]
	\arrow[no head, from=6-5, to=7-4]
	\arrow[no head, from=2-5, to=5-6]
	\arrow[no head, from=2-4, to=5-6]
	\arrow[no head, from=2-5, to=3-4]
	\arrow[no head, from=2-5, to=1-5]
	\arrow[no head, from=3-1, to=7-2]
	\arrow[no head, from=3-1, to=7-3]
	\arrow[no head, from=6-3, to=7-2]
	\arrow[no head, from=7-2, to=8-2]
	\arrow[no head, from=3-1, to=2-2]
	\arrow[no head, from=6-5, to=5-6]
	\arrow[no head, from=3-2, to=6-4]
	\arrow[no head, from=3-2, to=6-3]
	\arrow[no head, from=3-2, to=2-3]
	\arrow[no head, from=6-4, to=7-3]
	\arrow[draw=none, from=3-3, to=5-5]
	\arrow[no head, from=3-3, to=5-4]
	\arrow[no head, from=3-4, to=5-5]
	\arrow[no head, from=3-3, to=2-4]
	\arrow[no head, from=5-5, to=6-4]
\end{tikzcd}\]
By (co)multiplication $\beta$-naturality of $\textsf{h}$, we can move the $\mu$ and $\Delta$ nodes around in the interior of the diagram arbitrarily, which we already look the liberty of doing to draw the last diagram. Therefore we can draw the diagram so that all braidings happen after the $\Delta$'s, but before the $\mu$'s. For example, by moving nodes, the above diagram becomes
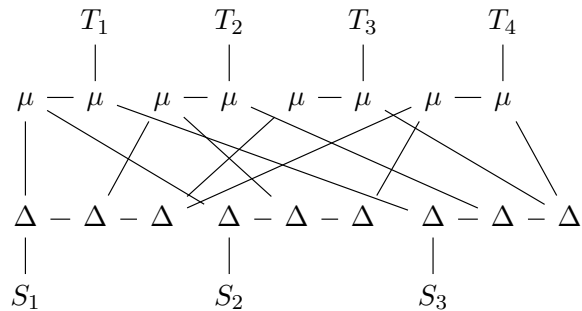
\[\begin{tikzcd}[row sep=0.5cm, column sep=0.2cm]
	& {T_1} && {T_2} && {T_3} && {T_4} \\
	\mu & \mu & \mu & \mu & \mu & \mu & \mu & \mu \\
	\\
	\Delta & \Delta & \Delta & \Delta & \Delta & \Delta & \Delta & \Delta & \Delta \\
	{S_1} &&& {S_2} &&& {S_3}
	\arrow[no head, from=2-2, to=1-2]
	\arrow[no head, from=2-4, to=1-4]
	\arrow[no head, from=2-6, to=1-6]
	\arrow[no head, from=4-4, to=5-4]
	\arrow[no head, from=2-2, to=4-7]
	\arrow[no head, from=4-3, to=4-2]
	\arrow[no head, from=4-7, to=5-7]
	\arrow[no head, from=2-4, to=4-8]
	\arrow[no head, from=2-7, to=4-3]
	\arrow[no head, from=4-8, to=4-7]
	\arrow[no head, from=2-8, to=4-9]
	\arrow[no head, from=2-6, to=4-9]
	\arrow[no head, from=2-8, to=2-7]
	\arrow[no head, from=2-8, to=1-8]
	\arrow[no head, from=2-1, to=4-1]
	\arrow[no head, from=2-1, to=4-4]
	\arrow[no head, from=4-2, to=4-1]
	\arrow[no head, from=4-1, to=5-1]
	\arrow[no head, from=2-1, to=2-2]
	\arrow[no head, from=4-8, to=4-9]
	\arrow[no head, from=2-3, to=4-5]
	\arrow[no head, from=2-3, to=4-2]
	\arrow[""{name=0, anchor=center, inner sep=0}, no head, from=2-3, to=2-4]
	\arrow[no head, from=4-5, to=4-4]
	\arrow[no head, from=2-5, to=4-3]
	\arrow[no head, from=2-7, to=4-6]
	\arrow[no head, from=2-5, to=2-6]
	\arrow[no head, from=4-6, to=4-5]
	\arrow[draw=none, from=0, to=4-6]
\end{tikzcd}\]
Finally, using (co)associativity, we replace consecutive $\Delta$ nodes and $\mu$ nodes with single $\Delta$'s and $\mu$'s, which yields the right-hand side of the bimonoid axiom. The actual braiding that happens in the interior of the diagram is arbitrary by the hexagon identities of $\textsf{C}$.
\end{proof}

 

The analogous construction for bimonoids in cospecies is left implicit. A \emph{strictly Joyal-theoretic bimonoid}, abbreviated \emph{Joyal-theoretic bimonoid}, is a strictly Joyal-theoretic (co)species $\textsf{h}$ equipped with a bimonoid structure given by the construction above. Equivalently, it is a bimonoid whose (co)multiplication factorizes as follows,
\[
\mu_{F,G} =  \mu_{F|_{T_1}}\otimes \dots \otimes   \mu_{F|_{T_l}}      \qquad \text{and} \qquad    \Delta_{F,G} =\Delta_{F|_{T_1}}\otimes \dots \otimes   \Delta_{F|_{T_l}}   
.
\]
A \emph{strongly Joyal-theoretic bimonoid} is a strongly Joyal-theoretic (co)species $\textsf{h}$ equipped with a bimonoid structure which is Joyal-theoretic up to the specified isomorphism. That is, we have factorization
\[
\boldsymbol{\mu}_{F,G} =  \boldsymbol{\mu}_{F|_{T_1}}\otimes \dots \otimes   \boldsymbol{\mu}_{F|_{T_l}}      
\qquad \text{and} \qquad    
\boldsymbol{\Delta}_{F,G} =\boldsymbol{\Delta}_{F|_{T_1}}\otimes \dots \otimes   \boldsymbol{\Delta}_{F|_{T_l}}   
\]
where
\[
\boldsymbol{\mu}_{F,G}= \text{prod}^{-1}_G   \circ \mu_{F,G} \circ  \text{prod}_F
,\qquad 
\boldsymbol{\Delta}_{F,G}= \text{prod}^{-1}_F   \circ \Delta_{F,G} \circ  \text{prod}_G
\]
and $\mu_{F,G}$ and $\Delta_{F,G}$ are the (co)multiplication of $\textsf{h}$.
	

\begin{ex}
The species of compositions $\mathsf{\Sigma}$ is a Joyal-theoretic bimonoid, where $\mu_{S,T}$ is given by concatenation and $\Delta_{S,T}$ is given by restriction,
\[
\mu_{S,T} (H,K) := H;K
\qquad \text{and} \qquad
\Delta_{S,T}(H) = ( H|_S, H|_T)
.\]
Notice that concatenation and restriction both preserve `$\leq$'. Thus $\mathsf{\Sigma}$ has the structure of a bimonoid in $\textsf{Cat}$-valued species. 
\end{ex}

\section{Graded Species and Bimonoids}\label{sec3}

\subsection{Graded Braid Arrangement Species}\label{Graded Braid Arrangement Species}

Recall the category of categories $\textsf{Cat}$ is cartesian, and the category of pointed categories $\textsf{Cat}_\bullet$ is equipped with the smash product of the cartesian product of categories, see e.g. \cite[Construction 4.19]{MR2558315}.

Let $\textsf{C}\in \{  \textsf{Cat}, \textsf{Cat}_\bullet \}$, and let $\textsf{p}$ be a $\textsf{C}$-valued species (we do not assume that $\textsf{p}$ is a strict species). A $\textsf{p}$-\emph{graded set species} $\textsf{P}$ consists of a set $\textsf{P}_{\ta}[F]$ for each composition $F$ and object $\ta \in \textsf{p}[F]$, a function 
\[
\textsf{P}_\ta[ \sigma,F ]  :     \textsf{P}_\ta[ F ] \to \textsf{P}_{\ta'}[ F' ]
\]
for each bijection $\sigma:J\to I$, composition $F\in \Sigma[I]$ and object $\ta \in \textsf{p}[F]$, a function 
\[
\textsf{P}_\ta[ F, \beta ]  :     \textsf{P}_\ta[ F ] \to \textsf{P}_{\tilde{\ta}}[ \wt{F} ]
\]
for each composition $F$, permutation $\beta\in \text{Sym}_{l(F)}$ and object $\ta \in \textsf{p}[F]$, and a function 
\[
\textsf{P}_{\mathtt{f}}[ F,F ]  :     \textsf{P}_\ta[ F ] \to \textsf{P}_{\tb}[ F ]
\]
for each composition $F$ and morphism $\ta\xrightarrow{\mathtt{f}} \tb \in  \textsf{p}[F]$. When there is no ambiguity, we abbreviate
	\begin{equation} \label{eq:abrf}  
\sigma :=\textsf{P}_\ta[ \sigma,F ], \qquad  \beta := \textsf{P}_\ta[ F, \beta ],  \qquad \mathtt{f} :=  \textsf{P}_{\mathtt{f}}[ F,F ] 
.
\end{equation}
We require that the diagrams \emph{graded $(\sigma,\beta)$-functoriality}, \emph{$(\sigma,\mathtt{f})$-functoriality} and \emph{$(\beta,\mathtt{f})$-functoriality}
\[\begin{tikzcd}[column sep=small]
	{\textsf{P}_{\mathtt{a}}[F]} & {\textsf{P}_{\tilde{\mathtt{a}}}[\widetilde{F}]} && {\textsf{P}_{\mathtt{a}}[F]} & {\textsf{P}_{\mathtt{b}}[F]} && {\textsf{P}_{\mathtt{a}}[F]} & {\textsf{P}_{\tilde{\mathtt{a}}}[\widetilde{F}]} \\
	{\textsf{P}_{\mathtt{a}'}[F']} & {\textsf{P}_{   \wt{(\ta')}   }[\widetilde{F}']}\cong {\textsf{P}_{ (\tilde{\ta})' }[\widetilde{F}']} && {\textsf{P}_{\mathtt{a}'}[F']} & {\textsf{P}_{\mathtt{b}'}[F']} && {\textsf{P}_{\mathtt{b}}[F]} & {\textsf{P}_{\tilde{\mathtt{b}}}[\widetilde{F}]}
	\arrow["\beta"', from=2-1, to=2-2]
	\arrow["\sigma"', from=1-1, to=2-1]
	\arrow["\sigma", from=1-2, to=2-2]
	\arrow["\mathtt{f}' "', from=2-4, to=2-5]
	\arrow["\mathtt{f}", from=1-4, to=1-5]
	\arrow["\sigma"', from=1-4, to=2-4]
	\arrow["\sigma", from=1-5, to=2-5]
	\arrow["\beta", from=1-7, to=1-8]
	\arrow["{\mathtt{f}}"', from=1-7, to=2-7]
	\arrow["\beta"', from=2-7, to=2-8]
	\arrow["\tilde{\mathtt{f}}", from=1-8, to=2-8]
	\arrow["\beta", from=1-1, to=1-2]
\end{tikzcd}\]
the diagrams \emph{graded $\sigma$-functoriality} and \emph{graded $\sigma$-unitality}
\[\begin{tikzcd}[row sep=0.4cm, column sep=0.9cm]
	& {\textsf{P}_{\ta'}[F']} \\
	{\textsf{P}_{\ta}[F]} &&{\textsf{P}_{(\ta')'}[F'']}  \cong  {\textsf{P}_{\ta''}[F'']}&& {\textsf{P}_{\ta}[F]} & {\textsf{P}_{\ta}[F]}
	\arrow["\sigma", from=2-1, to=1-2]
	\arrow["\tau", from=1-2, to=2-3]
	\arrow["{\sigma\circ \tau}"', from=2-1, to=2-3]
	\arrow["{\text{id}_{\textsf{P}_{\ta}[F]}}", curve={height=-10pt}, from=2-5, to=2-6]
	\arrow["{\text{id}_I}"', curve={height=10pt}, from=2-5, to=2-6]
\end{tikzcd}\]
the diagrams \emph{graded $\beta$-functoriality} and \emph{graded $\beta$-unitality}
\[\begin{tikzcd}[row sep=0.4cm, column sep=0.9cm]
	& {\textsf{P}_{\tilde{\ta}}[\wt{F}]} \\
	{\textsf{P}_{\ta}[F]} &&{\textsf{P}_{  \wt{(\tilde{\ta})} }[\wt{\wt{F}}]} \cong  {\textsf{P}_{\tilde{\tilde{\ta}}}[\wt{\wt{F}}]}&& {\textsf{P}_{\ta}[F]} & {\textsf{P}_{\ta}[F]}
	\arrow["\beta", from=2-1, to=1-2]
	\arrow["\gamma", from=1-2, to=2-3]
	\arrow["{\gamma \circ \beta}"', from=2-1, to=2-3]
	\arrow["{\text{id}_{\textsf{P}_{\ta}[F]}}", curve={height=-10pt}, from=2-5, to=2-6]
	\arrow["{\text{id}_{(k)}}"', curve={height=10pt}, from=2-5, to=2-6]
\end{tikzcd}\]
and the diagrams \emph{$\mathtt{f}$-functoriality} and \emph{$\mathtt{f}$-unitality}
\[\begin{tikzcd}[row sep=0.4cm, column sep=0.9cm]
	& {\textsf{P}_{\tb}[F]} \\
	{\textsf{P}_{\ta}[F]} && {\textsf{P}_{\tc}[F]} && {\textsf{P}_{\ta}[F]} & {\textsf{P}_{\ta}[F]}
	\arrow["\mathtt{f}", from=2-1, to=1-2]
	\arrow["\mathtt{g}", from=1-2, to=2-3]
	\arrow["{\mathtt{g} \circ \mathtt{f}}"', from=2-1, to=2-3]
	\arrow["{\text{id}_{\textsf{P}_{\ta}[F]}}", curve={height=-10pt}, from=2-5, to=2-6]
	\arrow["{\text{id}_{\ta}}"', curve={height=10pt}, from=2-5, to=2-6]
\end{tikzcd}\]
commute for all meaningful choices. If $\textsf{p}$ is strict, the isomorphisms `$\cong$' in these diagrams can be replaced by equalities `$=$'. Otherwise, these isomorphisms abbreviate the corresponding components of the invertible $2$-cells of $\textsf{p}$. For example, the fully explicit graded \hbox{$(\sigma,\beta)$-functoriality} diagram is
\[\begin{tikzcd}[row sep=0.4cm, column sep=0.5cm]
	{\textsf{P}_{\mathtt{a}}[F]} &   &  {\textsf{P}_{\tilde{\mathtt{a}}}[\widetilde{F}]}    \\
  & & 		{\textsf{P}_{ (\tilde{\mathtt{a}})'    }[\widetilde{F}']}. \\
	{\textsf{P}_{\mathtt{a}'}[F']} &  {\textsf{P}_{   \wt{(\mathtt{a}')}    }[\widetilde{F}']}  & 
	\arrow["\beta"', from=3-1, to=3-2]
	\arrow["\sigma"', from=1-1, to=3-1]
	\arrow["\sigma", from=1-3, to=2-3]
	\arrow["\beta", from=1-1, to=1-3]
		\arrow["{{(\sigma,\beta)}_{\ta}}"', from=3-2, to=2-3]
\end{tikzcd}\]
Similarly, graded $\sigma$-functoriality and graded $\beta$-functoriality are completed by including components of the $2$-cells $(\sigma,\tau)$ and $(\beta,\gamma)$ respectively. 



\begin{remark} \label{rem:formal}
As with species, $\textsf{p}$-graded set species $\textsf{P}$ may be formalized as functors. These functors have the form
\[
\textsf{P} :\textstyle\int\displaystyle \textsf{p} \to \textsf{Set} 
.\]
where $\int \textsf{p}$ is the Grothendieck construction of $\textsf{p}$ (itself formalized as a functor).
\end{remark}

\begin{remark}
Recall the exponential species $\textsf{E}$ from \autoref{ex:exp}, which we can take to be \hbox{$\textsf{Cat}$-valued} by replacing singletons with the trivial category with one object. An $\textsf{E}$-graded set species is the same thing as a set species. In this way, $\textsf{p}$-graded set species generalize set species. 
\end{remark}

Let $\textsf{C}\in \{  \textsf{Cat}, \textsf{Cat}_\bullet \}$, and let $\textsf{p}$ and $\textsf{x}$ be $\textsf{C}$-valued species. Let $\textsf{X}$ be an $\textsf{x}$-graded set species. Given a lax morphism of species 
\[
\eta: \textsf{p} \to \textsf{x}
,\qquad
\eta=(\eta_F,{\sigma}_\ta, {\beta}_\ta)
\]
we have the \emph{pullback} $\textsf{p}$-graded set species $\eta^\ast\textsf{X}$ given by
\[
\eta^\ast\textsf{X}_\ta [F] :=  \textsf{X}_{\eta_F(\ta)} [F],
\]
\[ 
\eta^\ast\textsf{X}_\ta [\sigma,F] :=   {\sigma}_\ta  \circ \textsf{X}_{\eta_F(\ta)} [\sigma,F]
,\qquad 
\eta^\ast\textsf{X}_\ta [F,\beta] :=  {\beta}_\ta  \circ \textsf{X}_{\eta_F(\ta)} [F,\beta]
,\qquad
\eta^\ast\textsf{X}_\tf [F,F] :=  \textsf{X}_{\eta_F(\mathtt{f})} [F,F] 
.\]
Here, e.g. ${\sigma}_\ta$ abbreviates $\textsf{X}_{ {\sigma}_\ta }[F,F]$, according to \textcolor{blue}{(\refeqq{eq:abrf})}. Functoriality of $\eta^\ast\textsf{X}$ then follows directly from functoriality of $\textsf{X}$ and coherence of $\eta$. For example, graded \hbox{$(\sigma,\beta)$-functoriality} of $\eta^\ast\textsf{X}$ is the outer nonagon of the digram (and denoting the isomorphisms `$\cong$' explicitly)
\[\begin{tikzcd}[row sep=0.5cm, column sep=0.9cm]
\textsf{X}_{\eta_F(\ta)} [F]   & \textsf{X}_{\wt{\eta_F(\ta)}} [\wt{F}] &   & \textsf{X}_{\eta_F(\tilde{\ta})} [\wt{F}]    \\
\textsf{X}_{\eta'_F(\ta)} [F'] 	&     &   \textsf{X}_{     \wt{\eta_F(\ta)}'     }   [\wt{F}']	&\textsf{X}_{\eta'_F(\tilde{\ta})} [\wt{F}']  \\
 &  \textsf{X}_{\wt{\eta'_F(\ta)}} [\wt{F}']     &   & \textsf{X}_{\eta_F((\tilde{\ta})')} [\wt{F}']. \\
	\textsf{X}_{\eta_F(\ta')} [F']  & 	\textsf{X}_{\wt{\eta_F(\ta')}} [\wt{F}']  & 	\textsf{X}_{\eta_F(\wt{(\ta')})} [\wt{F}']  &   \\
	\arrow["\beta"', from=4-1, to=4-2]
	\arrow[" \sigma    "', from=1-1, to=2-1]
		\arrow[" \sigma_\ta    "', from=2-1, to=4-1]
	\arrow["\sigma", from=1-4, to=2-4]
	\arrow[" \beta    ", from=1-1, to=1-2]
		\arrow["{\beta}_\ta    ", from=1-2, to=1-4]
	\arrow["{    \eta_F ( {(\sigma,\beta)}_{\ta} )       }"', from=4-3, to=3-4]
			\arrow["    \sigma_\ta  ", from=2-4, to=3-4]
						\arrow["   \beta_\ta  ", from=4-2, to=4-3]
								\arrow["\beta", from=2-1, to=3-2]
										\arrow["\sigma", from=1-2, to=2-3]
				\arrow["{(\sigma,\beta)_\ta}", from=3-2, to=2-3]
					\arrow["{  \beta'_\ta }", from=2-3, to=2-4]
								\arrow["{  \wt{\sigma}_\ta }", from=3-2, to=4-2]
\end{tikzcd}\]
\noindent The top-left pentagon commutes by graded $(\sigma,\beta)$-functoriality of $\textsf{X}$, the top-right square commutes by $(\sigma,\mathtt{f})$-functoriality of $\textsf{X}$, the bottom-right square commutes by $(\beta,\mathtt{f})$-functoriality of $\textsf{X}$, and the bottom-right hexagon commutes by $(\sigma,\beta)$-coherence of $\eta$. The other functoriality diagrams of $\eta^\ast\textsf{X}$ follow similarly. In \autoref{thm}, we elevate this result to the level of bimonoids. 


Continue to let $\textsf{p}$ and $\textsf{x}$ be $\textsf{C}$-valued species for $\textsf{C}\in \{  \textsf{Cat}, \textsf{Cat}_\bullet \}$, equipped with a lax morphism 
\[
\eta: \textsf{p} \to \textsf{x}
.\] 
Let $\textsf{X}$ be an $\textsf{x}$-graded set species and let $\textsf{P}$ be a $\textsf{p}$-graded set species. A \emph{morphism} $\xi: \textsf{P} \to \textsf{X}$ of graded set species over $\eta$ consists of a function 
\[
\xi_{(F,\ta)}: \textsf{P}_\ta[F] \to    \textsf{X}_{\eta_F(\ta)}[F]	
\]
for each composition $F$ and object $\ta \in \textsf{p}[F]$ such that the diagrams
\[
\begin{tikzcd}
	{\textsf{P}_{\mathtt{a}}[F]} & {\textsf{X}_{\eta_F(\ta)}[F]} && 	{\textsf{P}_{\mathtt{a}}[F]} & {\textsf{X}_{\eta_F(\ta)}[F]}   \\
	{\textsf{P}_{\mathtt{a}'}[F']} & {\textsf{X}_{  \eta_{F'}(\mathtt{a}')}[F']}\cong {\textsf{X}_{  \eta'_{F}(\mathtt{a})}[F']}    && 	{\textsf{P}_{\tilde{\mathtt{a}}}[\wt{F}]} & {\textsf{X}_{ \eta_{\wt{F}}(\tilde{\mathtt{a}})}[\wt{F}]}\cong {\textsf{X}_{ \wt{\eta_{F}(\mathtt{a})}}[\wt{F}]}  
	\arrow["\sigma"', from=1-1, to=2-1]
	\arrow["{\xi_{(F,\mathtt{a})}}", from=1-1, to=1-2]
	\arrow["{\xi_{(F',\mathtt{a}')}}"', from=2-1, to=2-2]
	\arrow["\sigma", from=1-2, to=2-2]
	\arrow["{\xi_{(F,\mathtt{a})}}", from=1-4, to=1-5]
	\arrow["{\beta}"', from=1-4, to=2-4]
	\arrow["{\xi_{(\wt{F},\tilde{\mathtt{a}})}}"', from=2-4, to=2-5]
	\arrow["{\beta}", from=1-5, to=2-5]
\end{tikzcd}
\]
\[
\begin{tikzcd}
{\textsf{P}_{\mathtt{a}}[F]} & {\textsf{X}_{\eta_F(\ta)}[F]} \\
 {\textsf{P}_{\mathtt{b}}[F]} & {\textsf{X}_{\eta_F(\mathtt{b})}[F]}
	\arrow["{\mathtt{f}}"', from=1-1, to=2-1]
	\arrow["{\eta_F(\mathtt{f})}", from=1-2, to=2-2]
	\arrow["{\xi_{(F,\mathtt{a})}}", from=1-1, to=1-2]
	\arrow["{\xi_{(F,\mathtt{b})}}"', from=2-1, to=2-2]
\end{tikzcd}
\]
commute for all meaningful choices. Morphism composition is given by composing components. Note the composition of a morphism over $\eta_1$ with a morphism over $\eta_2$ is a morphism over $\eta_2\circ\eta_1$. This defines the category of graded set species. 

\begin{remark}
	In the formalization of graded species as functors, see \autoref{rem:formal}, the pullback $\eta^\ast\textsf{X}$ is given by $\eta^\ast\textsf{X}=\textsf{X} \circ \int \eta$, and a morphism of graded species $\textsf{P}\to \textsf{X}$ over $\eta$ is the same thing as a natural transformation $\textsf{P}\Rightarrow \eta^\ast\textsf{X}$.
\end{remark}


We also have \emph{graded set cospecies}, the only differences being $\textsf{p}$ is now a cospecies, and the source and target of $\textsf{P}_\ta[ \sigma,F]$ and $\textsf{P}_\ta[ F, \beta ]$ are switched. Pullbacks and morphisms defined similarly. Finally, we have $\textsf{p}$-\emph{graded} \emph{vector (co)species} $\textsf{\textbf{P}}$, obtained by replacing sets/functions with vector spaces/linear maps above. Pullbacks and morphisms defined similarly to the set case. 


\subsection{Bimonoids in Graded Braid Arrangement Species}\label{Bimonoids in Graded Braid Arrangement Species}

Let $\textsf{C}\in \{  \textsf{Cat}, \textsf{Cat}_\bullet \}$, and let $\textsf{h}$ be a bimonoid in $\textsf{C}$-valued species (we do not assume that $\textsf{h}$ is a strict bimonoid). An $\textsf{h}$\emph{-graded bimonoid} $\textsf{H}$ is an $\textsf{h}$-graded set species such that for each pair of compositions $G\leq F$ we have functions 
\[
{\mu_{F, G}}_{\ta}  :      \textsf{H}_\ta [F]  \to  \textsf{H}_{\mu_{F,G}(\ta)} [G] 
\qquad \text{and} \qquad
{\Delta_{F, G}}_{\ta}  :      \textsf{H}_\ta [G]  \to  \textsf{H}_{\Delta_{F,G}(\ta)} [F] 
\]
for each object $\ta\in \textsf{h}[F]$ and $\ta\in \textsf{h}[G]$ respectively, such that the diagrams \emph{graded multiplication \hbox{$\sigma$-naturality}} and \emph{graded multiplication $\beta$-naturality}
\[\begin{tikzcd}
	{\textsf{H}_{\mathtt{a}}[F]} & {\textsf{H}_{\mathtt{a}'}[F']} & {\textsf{H}_{\mathtt{a}}[F]} & {\textsf{H}_{\tilde{\mathtt{a}}}[\widetilde{G}    F]}   \\
	{\textsf{H}_{\mu_{F,G}(\mathtt{a})}[G]} & {\textsf{H}_{\mu'_{F,G}(\mathtt{a})}[G']}\cong {\textsf{H}_{\mu_{F',G'}(\mathtt{a}')}[G']} & {\textsf{H}_{\mu_{F,G}(\mathtt{a})}[G]} &  {\textsf{H}_{ \wt{\mu_{F,G}( \mathtt{a})  }}[\widetilde{G}]}\cong {\textsf{H}_{ \mu_{\wt{G}F,\wt{G}}(\tilde{\mathtt{a}})  }[\widetilde{G}]}  
	\arrow["{\hat{\beta}}", from=1-3, to=1-4]
	\arrow["{{\mu_{F,G}}_\ta}"', from=1-3, to=2-3]
	\arrow["\beta"', from=2-3, to=2-4]
	\arrow["{\mu_{\widetilde{G}    F,\widetilde{G}}}", from=1-4, to=2-4]
	\arrow["\sigma", from=1-1, to=1-2]
	\arrow["\sigma"', from=2-1, to=2-2]
	\arrow["{{\mu_{F',G'}}_{\ta'}}", from=1-2, to=2-2]
	\arrow["{{\mu_{F,G}}_\ta}"', from=1-1, to=2-1]
\end{tikzcd}\]
the diagram \emph{multiplication $\mathtt{f}$-naturality}
\[\begin{tikzcd}[column sep=large]
 {\textsf{H}_{\mathtt{a}}[F]} & {\textsf{H}_{\mathtt{b}}[F]} \\
 {\textsf{H}_{\mu_{F,G}(\mathtt{a})}[G]} & {\textsf{H}_{\mu_{F,G}(\mathtt{b})}[G]}
	\arrow["{{\mu_{F,G}}_\ta}"', from=1-1, to=2-1]
	\arrow["{\mathtt{f}}", from=1-1, to=1-2]
	\arrow["{\mu_{F,G}(\mathtt{f})}"', from=2-1, to=2-2]
	\arrow["{{\mu_{F,G}}_\tb}", from=1-2, to=2-2]
\end{tikzcd}\]
the diagrams \emph{graded comultiplication $\sigma$-naturality} and \emph{graded comultiplication $\beta$-naturality}
\[\begin{tikzcd}
	{\textsf{H}_{\Delta_{F,G}(\mathtt{a})}[F]} &
	 {\textsf{H}_{ \Delta'_{F,G}( \mathtt{a})}[F']}\cong {\textsf{H}_{ \Delta_{F',G'}( \mathtt{a}')}[F']}
	  & {\textsf{H}_{\Delta_{F,G}(\mathtt{a})}[F]} & 
	{    \textsf{H}_{  \wt{\Delta_{F,G}(\ta)} }   }[\widetilde{G}F]     \cong   {\textsf{H}_{ \Delta_{\wt{G}F,\wt{G}}( \tilde{\mathtt{a}} )  }[\widetilde{G}    F]}   	\\
	{\textsf{H}_{\mathtt{a}}[G]} & {\textsf{H}_{\mathtt{a}'}[G']} & {\textsf{H}_{\mathtt{a}}[G]} & {\textsf{H}_{ \tilde{\mathtt{a}}  }[\widetilde{G}]}  
	\arrow["{\hat{\beta}}", from=1-3, to=1-4]
	\arrow["{\Delta_{F,G}}", from=2-3, to=1-3]
	\arrow["\beta"', from=2-3, to=2-4]
	\arrow["{\Delta_{\widetilde{G}    F,\widetilde{G}}}"', from=2-4, to=1-4]
	\arrow["\sigma", from=1-1, to=1-2]
	\arrow["\sigma"', from=2-1, to=2-2]
	\arrow["{\Delta_{F',G'}}"', from=2-2, to=1-2]
		\arrow["{\Delta_{F,G}}", from=2-1, to=1-1]
\end{tikzcd}\]
the diagram \emph{comultiplication $\mathtt{f}$-naturality}
\[\begin{tikzcd}[column sep=large]
  {\textsf{H}_{  \Delta_{F,G} (\mathtt{a}) }[F]} & {\textsf{H}_{  \Delta_{F,G}( \mathtt{b})    }[F]} \\
  {\textsf{H}_{ \mathtt{a}}[G]} & {\textsf{H}_{   \mathtt{b}}[G]}
	\arrow["{{\Delta_{F,G}}_\ta}", from=2-1, to=1-1]
	\arrow["{  \Delta_{F,G} (\mathtt{f})}", from=1-1, to=1-2]
	\arrow["{  \mathtt{f}}"', from=2-1, to=2-2]
	\arrow["{{\Delta_{F,G}}_\tb}"', from=2-2, to=1-2]
\end{tikzcd}\]
the diagrams \emph{graded associativity} and \emph{graded unitality}
\[\begin{tikzcd}
	& {\textsf{H}_{\mu_{F,G}(\mathtt{a})}[G]} \\
	{\textsf{H}_{\mathtt{a}}[F]} && 
	\textsf{H}_{\mu_{F,E}(\mathtt{a})}[E]\cong   	\textsf{H}_{\mu_{G,E}(\mu_{F,G}(\mathtt{a})) }  [E]	&& {\textsf{H}_{\mathtt{a}}[F]} & {\textsf{H}_{\mathtt{a}}[F]}
	\arrow["{{\mu_{F,E}}_{\mathtt{a}}}"', from=2-1, to=2-3]
	\arrow["{{\mu_{F,G}}_{\mathtt{a}}}", from=2-1, to=1-2]
	\arrow["{{\mu_{G,E}}_{\mu_{F,G}(\mathtt{a})}}", from=1-2, to=2-3]
	\arrow["{\text{id}}", curve={height=-10pt}, from=2-5, to=2-6]
	\arrow["{{\mu_{F,F}}_{\mathtt{a}}}"', curve={height=10pt}, from=2-5, to=2-6]
\end{tikzcd}\]
the diagrams \emph{graded coassociativity} and \emph{graded counitality}
\[\begin{tikzcd}
	& {\textsf{H}_{\Delta_{G,E}(\mathtt{a})}[G]} \\
	{\textsf{H}_{     \Delta_{F,G} (  \Delta_{G,E} ( \mathtt{a} ))    }[F]}\cong    	{\textsf{H}_{\Delta_{F,E}(\mathtt{a})}[F]}  && {\textsf{H}_{\mathtt{a}}[E]}  && {\textsf{H}_{\mathtt{a}}[F]} & {\textsf{H}_{\mathtt{a}}[F]}
	\arrow["{{\Delta_{F,E}}_{\mathtt{a}}}", from=2-3, to=2-1]
	\arrow["{{\Delta_{F,G}}_{\Delta_{G,E}(\mathtt{a})}}"', from=1-2, to=2-1]
	\arrow["{{\Delta_{G,E}}_{\mathtt{a}}}"', from=2-3, to=1-2]
	\arrow["{\text{id}}"', curve={height=10pt}, from=2-6, to=2-5]
	\arrow["{{\Delta_{F,F}}_{\mathtt{a}}}", curve={height=-10pt}, from=2-6, to=2-5]
\end{tikzcd}\]
and the diagram \emph{graded bimonoid axiom}
\[\begin{tikzcd}[row sep=large, column sep=large]
	{\textsf{H}_\ta[F]} & {\textsf{H}_{\mu_{F,A}(\ta)}[A]} & {\textsf{H}_{ \Delta_{G,A}(\mu_{F,A}(\ta))}[G]}\cong 
	  {\textsf{H}_{    \mu_{GF,G} ( \wt{  \Delta_{FG,F} (\ta)  } )    }[G]}  \\
\textsf{H}_{\Delta_{FG,F}(\ta)}[FG] && \textsf{H}_{\wt{\Delta_{FG,F}}(\ta)}[GF]
	\arrow["{\mu_{F,A}}_{\ta}", from=1-1, to=1-2]
	\arrow["{\Delta_{G,A}}_{\mu_{F,A}(\ta)}", from=1-2, to=1-3]
	\arrow["{\Delta_{FG,F}}_{\ta}"', from=1-1, to=2-1]
	\arrow["{ \beta }"', from=2-1, to=2-3]
	\arrow["{\mu_{GF,G}}_{\wt{\Delta_{FG,F}}(\ta)}"', from=2-3, to=1-3]
\end{tikzcd}\]
commute for all meaningful choices. The isomorphisms `$\cong$' abbreviate components of the invertible $2$-cells of $\textsf{h}$, similar to graded species. 

\begin{remark}
	Recall that $\textsf{E}$ is a bimonoid, see \autoref{ex:expbimon}. An $\textsf{E}$-graded bimonoid is the same thing as a bimonoid in $\textsf{Set}$-valued species. In this way, $\textsf{h}$-graded bimonoids generalize bimonoids in $\textsf{Set}$-valued species. 
\end{remark}

An $\textsf{h}$\emph{-graded bialgebra} $\textbf{\textsf{H}}$ is defined analogously, by replacing sets/functions with vector spaces/linear maps above. 


\begin{thm} \label{thm}
Let $\textsf{x}$ and $\textsf{h}$ be bimonoids in $\textsf{C}$-valued species for $\textsf{C}\in \{  \textsf{Cat}, \textsf{Cat}_\bullet \}$, and let $\textsf{X}$ be an $\textsf{x}$-graded bimonoid (respectively bialgebra). Given a lax homomorphism of bimonoids 
\[
\varphi:\textsf{h}\to \textsf{x}
,\qquad \varphi=(  \varphi_F, {\sigma}_\ta, \beta_\ta, {\mathtt{\Phi}_{F,G}}_\ta ,{\mathtt{\Psi}_{F,G}}_\ta  )
\]
the pullback $\textsf{h}$-graded set species (respectively vector species) $\textsf{H}:=\varphi^\ast  \textsf{X}$ is moreover an \hbox{$\textsf{h}$-graded} bimonoid (respectively $\textsf{h}$-graded bialgebra), with multiplication and comultiplication given by
\[
{\mu_{F,G}}_\ta :=          {\mathtt{\Phi}_{F,G}}_\ta \circ {\mu_{F,G}}_{\varphi(\ta)}
\qquad \text{and} \qquad 
{\Delta_{F,G}}_\ta :=                {    \mathtt{\mathtt{\Psi}}   _{F,G}}_\ta                \circ {\Delta_{F,G}}_{\varphi(\ta)}
\]
respectively. 
\end{thm}
\begin{proof}

We have eleven diagrams to check. Throughout this proof, we let
\[
X_{(-)}=\varphi_F(-)
.\]
We can think of $X_\ta$ as the data which is indexed by the combinatorial object $\ta$, and ${\textsf{X}_{X_\mathtt{a}}[F]}$ is a set or vector space associated to $X_\ta$. We also use the abbreviations `$\cong$' for components of $2$-cells throughout this proof, in the hope that it makes the commutative diagrams easier to comprehend.

Graded $\sigma$-multiplication naturality of $\textsf{H}$ is the outside octagon of the diagram
\begin{center} \centerfloat
	\begin{tikzcd}[row sep=0.7cm, column sep=0.9cm]
		{\textsf{X}_{X_\mathtt{a}}[F]} & {\textsf{X}_{X'_\mathtt{a}}[F']} & {\textsf{X}_{X_{\mathtt{a}'}}[F']} \\
		{\textsf{X}_{\mu_{F,G}(X_{\mathtt{a}})}[G]} & {\textsf{X}_{\mu'_{F,G}(X_{\mathtt{a}})}[G']}\cong {\textsf{X}_{\mu_{F',G'}(X'_{\mathtt{a}})}[G']}  & {\textsf{X}_{\mu_{F',G'}(X_{\mathtt{a}'})}[G']} \\
		{\textsf{X}_{X_{\mu_{F,G}(\mathtt{a})}}[G]} & {\textsf{X}_{X'_{\mu_{F,G}(\mathtt{a})}}[G']} & {\textsf{X}_{X_{\mu'_{F,G}(\mathtt{a})}}[G']}\cong  {\textsf{X}_{X_{\mu_{F',G'}(\mathtt{a}')}}[G']}.
		\arrow["{\mu_{F',G'}({\sigma}_{\mathtt{a}})}", from=2-2, to=2-3]
		\arrow["{{\mathtt{\Phi}_{F',G'}}_{\mathtt{a}'}}", from=2-3, to=3-3]
		\arrow["{{\mathtt{\Phi}'_{F,G}}_{\mathtt{a}}}"', from=2-2, to=3-2]
		\arrow["{{\sigma}_{\mu_{F,G}(\mathtt{a})}}"', from=3-2, to=3-3]
		\arrow["\sigma"', from=3-1, to=3-2]
		\arrow["{{\mathtt{\Phi}_{F,G}}_{\mathtt{a}}}"', from=2-1, to=3-1]
		\arrow["\sigma", from=2-1, to=2-2]
		\arrow["{{\mu_{F,G}}_{X_\mathtt{a}}}"', from=1-1, to=2-1]
		\arrow["\sigma", from=1-1, to=1-2]
		\arrow["{{\mu_{F',G'}}_{X'_\mathtt{a}}}"', from=1-2, to=2-2]
		\arrow["{{\mu_{F',G'}}_{X_{\mathtt{a}'}}}", from=1-3, to=2-3]
		\arrow["{{\sigma}_{\mathtt{a}}}", from=1-2, to=1-3]
	\end{tikzcd}
\end{center}
\noindent The top-left square commutes by graded multiplication $\sigma$-naturality of $\textsf{X}$, the top-right square commutes by multiplication $\mathtt{f}$-naturality of $\textsf{X}$, the bottom-left square commutes by \hbox{$(\sigma,\mathtt{f})$-functoriality} of $\textsf{X}$, and the bottom-right square commutes by multiplication $\sigma$-naturality coherence of $\varphi$. Thus, the outer octagon commutes.

Graded $\beta$-multiplication naturality of $\textsf{H}$ is the outside octagon of the diagram
\begin{center} \centerfloat
	\begin{tikzcd}[row sep=0.7cm, column sep=0.9cm]
		{\textsf{X}_{X_\mathtt{a}}[F]} & {\textsf{X}_{\widetilde{X}_{\mathtt{a}}}[\widetilde{G}F]} & {\textsf{X}_{X_{\tilde{\mathtt{a}}}}[\widetilde{G}F]} \\
		{\textsf{X}_{\mu_{F,G}(X_{\mathtt{a}})}[G]} & 
		\textsf{X}_{\wt{ \mu_{F,G}(X_\ta)} }[\wt{G}] \cong {\textsf{X}_{\mu_{\widetilde{G}F,\widetilde{G}}(\widetilde{X}_{\mathtt{a}})}[\widetilde{G}]} 
		& {\textsf{X}_{\mu_{\widetilde{G}F,\widetilde{G}}(X_{\tilde{\mathtt{a}}})}[\widetilde{G}]} \\
		{\textsf{X}_{X_{\mu_{F,G}(\mathtt{a})}}[G]} & {\textsf{X}_{\widetilde{X}_{\mu_{F,G}(\mathtt{a})}}[\widetilde{G}]} & 
	{\textsf{X}_{X_{\wt{\mu_{F,G}(\mathtt{a})}}}[\widetilde{G}]} \cong {\textsf{X}_{X_{\mu_{\widetilde{G}F,\widetilde{G}}(\tilde{\mathtt{a}})}}[\widetilde{G}]} .
		\arrow["{{\mu_{F,G}}_{X_\mathtt{a}}}"', from=1-1, to=2-1]
		\arrow["{{\mathtt{\Phi}_{F,G}}_{\mathtt{a}}}"', from=2-1, to=3-1]
		\arrow["{{\mu_{\widetilde{G}F,\widetilde{G}}}_{X_{\tilde{\mathtt{a}}}}}", from=1-3, to=2-3]
		\arrow["{{\mathtt{\Phi}_{\widetilde{G}F,\widetilde{G}}}_{\tilde{\mathtt{a}}}}", from=2-3, to=3-3]
		\arrow["\hat\beta", from=1-1, to=1-2]
		\arrow["\beta", from=2-1, to=2-2]
		\arrow["\beta"', from=3-1, to=3-2]
		\arrow["{{\mu_{\widetilde{G}F,\widetilde{G}}}_{\widetilde{X}_{\mathtt{a}}}}"', from=1-2, to=2-2]
		\arrow["{{\widetilde{\mathtt{\Phi}_{F,G}}_{\mathtt{a}}}}"', from=2-2, to=3-2]
		\arrow["{{\hat{\beta}}_{\mathtt{a}}}", from=1-2, to=1-3]
		\arrow["{\mu_{\widetilde{G}F,\widetilde{G}}({\hat{\beta}}_{\mathtt{a}})}", from=2-2, to=2-3]
		\arrow["{{\beta}_{\mu_{F,G}(\mathtt{a})}}"', from=3-2, to=3-3]
	\end{tikzcd}
\end{center}
\noindent The top-left square commutes by graded multiplication $\beta$-naturality of $\textsf{X}$, the top-right square commutes by multiplication $\mathtt{f}$-naturality of $\textsf{X}$, the bottom-left square commutes by \hbox{$(\beta,\mathtt{f})$-functoriality} of $\textsf{X}$, and the bottom-right square commutes by multiplication $\beta$-naturality coherence of $\varphi$. Thus, the outer octagon commutes. Graded $\sigma$-comultiplication naturality and graded $\beta$-comultiplication naturality of $\textsf{H}$ follow similarly, from graded comultiplication naturality of $\textsf{X}$, comultiplication $\mathtt{f}$-naturality of $\textsf{X}$, functoriality of $\textsf{X}$, and comultiplication naturality coherence of $\varphi$.

Then multiplication $\mathtt{f}$-naturality of $\textsf{H}$ is the outside hexagon of the diagram
\begin{center} \centerfloat
	\begin{tikzcd}[row sep=0.7cm, column sep=0.9cm]
		{\textsf{X}_{X_\mathtt{a}}[F]} & {\textsf{X}_{X_\mathtt{b}}[F]} \\
		{\textsf{X}_{\mu_{F,G}(X_{\mathtt{a}})}[G]} & {\textsf{X}_{\mu_{F,G}(X_{\mathtt{b}})}[G]} \\
		{\textsf{X}_{X_{\mu_{F,G}(\mathtt{a})}}[G]} & {\textsf{X}_{X_{\mu_{F,G}(\mathtt{b})}}[G]}.
		\arrow["{{\mu_{F,G}}_{X_\mathtt{a}}}"', from=1-1, to=2-1]
		\arrow["{{\mathtt{\Phi}_{F,G}}_{\mathtt{a}}}"', from=2-1, to=3-1]
		\arrow["{\mu_{F,G}(\mathtt{f})}"', from=3-1, to=3-2]
		\arrow["{{\mathtt{\Phi}_{F,G}}_{\mathtt{b}}}", from=2-2, to=3-2]
		\arrow["{\mu_{F,G}(X_\mathtt{f})}", from=2-1, to=2-2]
		\arrow["{{\mu_{F,G}}_{X_\mathtt{b}}}", from=1-2, to=2-2]
		\arrow["{\mathtt{f}}", from=1-1, to=1-2]
	\end{tikzcd}
\end{center}
\noindent  The top square commutes by multiplication $\mathtt{f}$-naturality of $\textsf{X}$, and the bottom square commutes by $\mathtt{f}$-functoriality of $\textsf{X}$ and since $\mathtt{\Phi}_{F,G}$ is a natural transformation. Thus, the outer hexagon commutes. Then comultiplication $\mathtt{f}$-naturality of $\textsf{H}$ follows similarly, from comultiplication $\mathtt{f}$-naturality of $\textsf{X}$, $\mathtt{f}$-functoriality of $\textsf{X}$ and since $\mathtt{\Psi}_{F,G}$ is a natural transformation.

Graded associativity of $\textsf{H}$ is the outside hexagon of the diagram
\[\begin{tikzcd}[row sep=0.8cm, column sep=-0.4cm]
	&& {\textsf{X}_{X_{\mu_{F,G}(\mathtt{a})}}[F]} \\
	& {\textsf{X}_{\mu_{F,G}(X_\mathtt{a})}[F]} && {\textsf{X}_{\mu_{G,H}(X_{\mu_{F,G}(\mathtt{a})})}[F]} \\
	{\textsf{X}_{X_\mathtt{a}}[F]} && {\textsf{X}_{\mu_{F,H}(X_\mathtt{a})}[F]}  \cong   {\textsf{X}_{\mu_{G,E}(\mu_{F,G}(X_\mathtt{a}))}[F]} && 
	{\textsf{X}_{X_{\mu_{F,H}(\mathtt{a})}}[F]}\cong  {\textsf{X}_{X_{\mu_{G,E}(\mu_{F,G}(\mathtt{a}))}}[F]} .
	\arrow["{{\mu_{F,H}}_{X_{\mathtt{a}}}}"', from=3-1, to=3-3]
	\arrow["{{\mu_{F,G}}_{X_{\mathtt{a}}}}", from=3-1, to=2-2]
	\arrow["{{\mu_{G,H}}_{ \mu_{F,G}(X_{\mathtt{a}} )  }}"{description}, from=2-2, to=3-3]
	\arrow["{{\mathtt{\Phi}_{F,G}}_{\mathtt{a}}}", from=2-2, to=1-3]
	\arrow["{{\mu_{G,H}}_{ X_{\mu_{F,G}(\mathtt{a})}   }}", from=1-3, to=2-4]
	\arrow["{{\mathtt{\Phi}_{G,H}}_{\mu_{F,G}(\mathtt{a})}}", from=2-4, to=3-5]
	\arrow["{\mu_{G,H}({\mathtt{\Phi}_{F,G}}_{\mathtt{a}})}"{description}, from=3-3, to=2-4]
	\arrow["{{\mathtt{\Phi}_{F,H}}_{\mathtt{a}}}"', from=3-3, to=3-5]
\end{tikzcd}\]
\noindent  The left triangle commutes by graded associativity of $\textsf{X}$. The right triangle commutes by associativity coherence of $\varphi$. The middle square commutes by multiplication $\mathtt{f}$-naturality of $\textsf{X}$. Thus, the outer hexagon commutes. Then graded coassociativity of $\textsf{H}$ follows similarly, from graded coassociativity of $\textsf{X}$, coassociativity coherence of $\varphi$, and comultiplication $\mathtt{f}$-naturality of $\textsf{X}$.

The graded bimonoid axiom of $\textsf{H}$ is the outside decagon of the diagram
\begin{center} \centerfloat
	\begin{tikzcd}[column sep=-1.2cm]
		&& {{\textsf{X}_{X_{\mu_{F,A}(\mathtt{a})}}}[A]} \\
		& {{\textsf{X}_{\mu_{F,A}(X_\mathtt{a})}}[A]} && {{\textsf{X}_{\Delta_{G,F}(X_{\mu_{F,A}(\mathtt{a})})}}[G]} \\
		{{\textsf{X}_{X_\mathtt{a}}}[F]} && {{\textsf{X}_{ \mu_{GF,G} (\wt{\Delta_{FG,F}} (X_\mathtt{a}))     }}[G]}\cong {{\textsf{X}_{\Delta_{G,A}(\mu_{F,A}(X_\mathtt{a}))}}[G]}  && {{\textsf{X}_{X_{\Delta_{G,F}(\mu_{F,A}(\ta))}}}[A] \cong  {{\textsf{X}_{X_{  \mu_{GF,G}( \wt{\Delta_{FG,F} (\ta)}  ) }}}}[A]}   . \\
		& {{\textsf{X}_{\widetilde{\Delta_{FG,G}(X_\mathtt{a})}}}[GF]} && {{\textsf{X}_{ \mu_{GF,G}( \widetilde{{X}_{\Delta_{FG,G}(\mathtt{a})}} )   }}[G]} \\
		{{\textsf{X}_{\Delta_{FG,G}(X_\mathtt{a})}}[FG]} &&&& {{\textsf{X}_{  \mu_{GF,G}( {X}_{ \widetilde{  \Delta_{FG,G}(\mathtt{a}})})}}[G]} \\
		{{\textsf{X}_{{X}_{\Delta_{FG,G}(\mathtt{a})}}}[FG]} && 
		{{\textsf{X}_{\widetilde{{X}_{\Delta_{FG,G}(\mathtt{a})}}}}[GF]} 
		&& {{\textsf{X}_{{X}_{ \widetilde{  \Delta_{FG,G}(\mathtt{a}})}}}[GF]}
		\arrow["\beta"', from=6-1, to=6-3]
		\arrow["{{\beta}_{\mathtt{a}}}"', from=6-3, to=6-5]
		\arrow["{{\mu_{GF,G}}_{  {X}_{ \widetilde{  \Delta_{FG,G}(\mathtt{a}})} }}", from=6-5, to=5-5]
		\arrow["{{\mathtt{\Phi}_{GF,G}}_{  \widetilde{ X_{ \Delta_{FG,G}  (\mathtt{a})    }  } }}"{description}, from=5-5, to=3-5]
		\arrow["{{\Delta_{FG,F}}_{X_\mathtt{a}}}"', from=3-1, to=5-1]
		\arrow["{{\mathtt{\Psi}_{GF,F}}_{\mathtt{a}}}", from=5-1, to=6-1]
		\arrow["{{\mu_{F,A}}_{X_\mathtt{a}}}", from=3-1, to=2-2]
		\arrow["{{\mathtt{\Phi}_{F,A}}_{\mathtt{a}}}", from=2-2, to=1-3]
		\arrow["{{\Delta_{G,A}}_{X_{\mu_{F,A}(\mathtt{a})}}}", from=1-3, to=2-4]
		\arrow["{{\mathtt{\Psi}_{G,A}}_{\mathtt{a}}}", from=2-4, to=3-5]
		\arrow["{\Delta_{G,A}({\mathtt{\Phi}_{F,A}}_{\mathtt{a}})}", from=3-3, to=2-4]
		\arrow["\! \! \! \! \! \! \! \! {{\Delta_{G,A}}_{  \mu_{F,A}  ( X_\mathtt{a} )   }}", from=2-2, to=3-3]
		\arrow["{\mu_{GF,G}(\widetilde{\mathtt{\Psi}_{FG,F}}_{\mathtt{a}})}"', from=3-3, to=4-4]
		\arrow["{\mu_{GF,G}({\beta}_{\mathtt{a}})}"{description}, from=4-4, to=5-5]
		\arrow["\beta", from=5-1, to=4-2]
		\arrow["{\widetilde{\mathtt{\Psi}_{FG,F}}_{\mathtt{a}}}"{description}, from=4-2, to=6-3]
		\arrow["\! \! \! \!{{\mu_{GF,G}}_{\widetilde{\Delta_{FG,G}(X_\mathtt{a})}}}"', from=4-2, to=3-3]
		\arrow["{{\mu_{GF,G}}_{\widetilde{{X}_{\Delta_{FG,G}(\mathtt{a})}} }}"{description}, from=6-3, to=4-4]
	\end{tikzcd}
\end{center}
\noindent  The top-left  pentagon commutes by the graded bimonoid axiom of $\textsf{X}$. The bottom-left square commutes by $(\beta,\mathtt{f})$-functoriality of $\textsf{X}$. The top-middle square commutes by comultiplication \hbox{$\mathtt{f}$-naturality} of $\textsf{X}$. The top-right pentagon commutes by bimonoid coherence of $\varphi$. The \hbox{bottom-middle} square and the bottom-right square commute by multiplication $\mathtt{f}$-naturality of $\textsf{X}$. Thus, the outer decagon commutes.
\end{proof}

We leave implicit the constructions and definitions for cospecies which are analogous to those of this section.

\subsection{Joyal-Theoretic Graded Species and Bimonoids} \label{sec:Joyal-Theoretic Graded Species and Bimonoids}

Similar to \autoref{sec:Joyal-Theoretic Bimonoids}, throughout this section we assume all species $\textsf{p}$ and bimonoids $\textsf{h}$ are strict. One can develop the notion of \hbox{Joyal-theoretic} graded species and bimonoids in the non-strict setting, however we shall not need it in this paper. 

Let $\textsf{P}$ be a $\textsf{p}$-graded set species. We abbreviate
\[
\textsf{P}_{\ta}[ I ] :=     \textsf{P}_{\ta}[ (I) ] 
, \qquad
\textsf{P}_{\ta} [\sigma]:=    \textsf{P}_{\ta} [\sigma,(I)] 
,\qquad 
\textsf{P}_{\mathtt{f}}[ I ] :=     \textsf{P}_{\mathtt{f}}[ (I),(I) ]  
.\]
At the level of objects, let us denote 
\[
\{ 1 \}=1_{\textsf{Cat}}\qquad \text{and} \qquad\{ 1,\bullet \}=1_{\textsf{Cat}_\bullet}
\]
where `$\bullet$' is the distinguished object of $1_{\textsf{Cat}_\bullet}$ as usual. The \emph{Joyalization} $\breve{\textsf{P}}$ of $\textsf{P}$ is the $\breve{\textsf{p}}$-graded species given by
\[
\breve{\textsf{P}}_{1} \big [ (\, ) \big ]: =  1_{\textsf{Set}}
,\qquad 
\big (\text{and} \quad  
\breve{\textsf{P}}_{\bullet} \big [ (\, ) \big ] :=  \emptyset\quad \text{if } \textsf{C}=\textsf{Cat}_\bullet  \big )
,\]
\[
\breve{\textsf{P}}_{(\ta_1, \dots, \ta_k)} [F]: =   \textsf{P}_{\ta_1} [S_1] \times   \dots \times  \textsf{P}_{\ta_k} [S_k]
,\]
\[
\breve{\textsf{P}}_{(\ta_1, \dots, \ta_k)} [\sigma,F] :  =   \textsf{P}_{\ta_1}[\sigma|_{S'_1}] \times \dots \times \textsf{P}_{\ta_k}[\sigma|_{S'_k}] 
,\qquad
\breve{\textsf{P}}_{(\mathtt{f}_1, \dots,  \mathtt{f}_k )  }[ F,F ] : =  \textsf{P}_{\mathtt{f}_1}[ S_1 ] \times \dots \times    \textsf{P}_{\mathtt{f}_k}[ S_k ]  
\]
with the action of permutations $\beta$ given by permuting cartesian product factors. For $\textsf{p}$ a weakly Joyal-theoretic species, a $\textsf{p}$-graded species $\textsf{P}$ is called \emph{weakly Joyal-theoretic} if it is equipped with a morphism
\[
\text{prod} : \breve{\textsf{P}} \to \textsf{P}
\]
over $\text{prod}: \breve{\textsf{p}} \to \textsf{p}$ such that $\text{prod}_{(I,\ta)}=\text{id}$ for all $\ta\in \textsf{p}[I]$. If $\text{prod}$ is an isomorphism, respectively identity, then $\textsf{P}$ is called \emph{strongly}, respectively \emph{strictly}, Joyal-theoretic. We abbreviate strictly Joyal-theoretic to just \emph{Joyal-theoretic}.

Let $\textsf{H}$ be a Joyal-theoretic $\textsf{h}$-graded set species, where $\textsf{h}$ is a Joyal-theoretic bimonoid. Recall $\textsf{h}$ is equipped with functors $\mu_{S,T}$ and $\Delta_{S,T}$ for each decomposition $S\sqcup T=I$, which we often abbreviate 
\[
\ta \tb:= \mu_{S,T}(\ta , \tb)
\qquad \text{and} \qquad
(\ta|_S, \ta|_T) := \Delta_{S,T}(\ta)
.\] 

\begin{ex}\label{ex:boolean}
	If $\textsf{h}$ is $\textsf{Cat}_\bullet$-valued, then $\textsf{h}[\emptyset]= 1_{\textsf{Cat}_\bullet} =\{ 1,\bullet \}$. It follows from the monoidal structure of $\textsf{Cat}_\bullet$ that we always have
	\[
	\mu_{\emptyset,\emptyset} ( 1,1 ) = 1
	,\qquad
	\mu_{\emptyset,\emptyset} ( 1,\bullet ) =  \mu_{\emptyset,\emptyset} ( \bullet, 1 ) =\bullet 
	,\qquad
	\mu_{\emptyset,\emptyset} ( \bullet,\bullet ) = \bullet \, 
	.\]
	More generally, since functors must preserve the distinguished object, we have
	\[
	\mu_{S,T} ( \ta,\bullet )=\mu_{S,T} ( \bullet,\tb )  = \bullet
	\]
	for all objects $\ta\in\textsf{h}[S]$, $\tb\in\textsf{h}[T]$. Similarly, we always have
	\[
	\bullet|_S = \bullet \qquad \text{and} \qquad 	\bullet|_T = \bullet
 	.\]
\end{ex}

We can give $\textsf{H}$ the structure of an $\textsf{h}$-graded bimonoid as follows. Let $\textsf{H}$ take values on the empty set as follows
\[
\textsf{H}_{1}[ \emptyset ] :=  1_{\textsf{Set}}
\qquad \big (\text{and} \qquad 
\textsf{H}_{\bullet}[ \emptyset ] :=\emptyset \quad \text{if } \textsf{C}=\textsf{Cat}_\bullet  \big )
.\]
For each decomposition $S\sqcup T=I$ of a finite set $I$, suppose we have functions
\[
{\mu_{S,T}}_{\ta, \tb}= {\mu}_{\ta, \tb} :   \textsf{H}_\ta[S]\times \textsf{H}_\tb[T] \to \textsf{H}_{\ta \tb}[I]
\qquad \text{and} \qquad
{\Delta_{S,T}}_\ta: \textsf{H}_\ta[I] \to \textsf{H}_{\ta|_S}[S] \times \textsf{H}_{\tb|_T}[T]
\]
where $\ta\in \textsf{h}[S]$, $\tb\in  \textsf{h}[T]$ and $\ta\in \textsf{h}[I]$ respectively. When $S$ or $T$ are empty, we require these functions be the unitors of $\textsf{Set}$ (or the unique function $\emptyset \to \emptyset$ if $\textsf{C}=\textsf{Cat}_\bullet$ and $\ta=\bullet$ or $\tb=\bullet$).


For $(S,T)\in [I;2]$, we now set
\[
\mu_{(\ta,\tb)}:= \mu_{\ta, \tb} 
\qquad \text{and} \qquad 
{\Delta_{(S,T),(I)}}_\ta := {\Delta_{S,T}}_\ta
.\] 
Towards defining ${\mu_{F,G}}_{(\ta_1,\dots, \ta_k)}$ and ${\Delta_{F,G}}_{(\ta_1,\dots, \ta_l)}$ in general, create a sequence of compositions the form
\[
G =F_1 \leq  \dots \leq F_m = F
\]
where $F_{j-1}$ is obtained from $F_j$ by merging just two lumps, say $S_j$ and $T_j$. Put
\[
{\mu_{F_{j},F_{j-1}}}_{\mu_{F,F_{j}}(\ta_1,\dots, \ta_k)} :=   \text{id} \times\dots \times   \text{id}   \times   \mu_{ \tb_j, \tc_j } \times \text{id} \times \dots \times \text{id}
\]
where $\tb_j$ and $\tc_j$ are the $S_j$ and $T_j$ components of $\mu_{F,F_{j}}(\ta_1,\dots, \ta_k)$ respectively. Also put
\[
{\Delta_{F_{j},F_{j-1}}}_{\Delta_{F_{j-1},G}(\ta_1,\dots, \ta_l)} :=   \text{id} \times\dots \times   \text{id}   \times   {\Delta_{S_j,T_j}}_{\mathtt{d}_j} \times \text{id} \times \dots \times \text{id}
\]
where $\mathtt{d}_j$ is the $S_j\sqcup T_j$ component of $\Delta_{F_{j-1},G}(\ta_1,\dots, \ta_l)$. If we want $\textsf{H}$ to satisfy graded associativity and graded coassociativity, we must then put
\[
{\mu_{F,G}}_{(\ta_1,\dots, \ta_k)}: =  \]
\begin{equation}\label{eq:mul}
	{\mu_{F_{2},G}}_{ \mu_{F_,F_2} (\ta_1,\dots, \ta_k) }  \circ
{\mu_{F_{3},F_2}}_{ \mu_{F_,F_3} (\ta_1,\dots, \ta_k) } \circ \dots \circ {\mu_{F_{m-1},F_{m-2}}}_{ \mu_{F,F_{m-1}} (\ta_1,\dots, \ta_k) } 
\circ   {\mu_{F,F_{m-1}}}_{(\ta_1,\dots, \ta_k)}
\end{equation}
and
\[
{\Delta_{F,G}}_{(\ta_1,\dots, \ta_l)}: =  \]
\begin{equation}\label{eq:comul}
 {\Delta_{F,F_{m-1}}}_{\Delta_{F_{m-1},G}(\ta_1,\dots, \ta_l)}   \circ {\Delta_{F_{m-1},F_{m-2}}}_{\Delta_{F_{m-2},G}(\ta_1,\dots, \ta_l)} \circ  \dots \circ {\Delta_{F_{3},F_2}}_{\Delta_{F_2,G}(\ta_1,\dots, \ta_l)} \circ   {\Delta_{F_{2},G}}_{(\ta_1,\dots, \ta_l)}
.
\end{equation}
We can now ask what properties must $\mu_{\ta,\tb}$ and ${\Delta_{S,T}}_\ta$ satisfy in order for ${\mu_{F,G}}_{(\ta_1,\dots, \ta_k)}$ and ${\Delta_{F,G}}_{(\ta_1,\dots, \ta_k)}$ to be the well-defined multiplication and comultiplication of an $\textsf{h}$-graded bimonoid. 

Consider the diagrams \emph{graded Joyal multiplication} $\sigma$-\emph{naturality} and \emph{Joyal multiplication} \hbox{$\mathtt{f}$-\emph{naturality}}
\begin{figure}[H]
	\centering
	\begin{minipage}{.5\textwidth}
		\centering
		\begin{tikzcd}[row sep=large, column sep=large]  
			\textsf{H}_{\ta}[S]\times \textsf{H}_{\tb}[T]   \arrow[d, "\mu_{\ta,\tb}"'] \arrow[r,   "\sigma\times \tau"]   
			& 
			\textsf{H}_{ \ta' }[ S' ]\times \textsf{H}_{\mathtt{b}'}[T'] \arrow[d,   "\mu_{\mathtt{a}' , \mathtt{b}'}"]     
			\\
			\textsf{H}_{\ta \tb}[I]           \arrow[r,   "\sigma \tau"']       
			& 
			\textsf{H}_{\ta'  \mathtt{b}' }[J]
		\end{tikzcd}
	\end{minipage}%
	\begin{minipage}{.5\textwidth}
		\centering
		\begin{tikzcd}[row sep=large, column sep=large]  
			\textsf{H}_{\ta}[S]\times \textsf{H}_{\tb}[T]   \arrow[d, "\mu_{\ta,\tb}"'] \arrow[r,   "\mathtt{f}  \times \mathtt{g}"]   
			& 
			\textsf{H}_{ \tc }[ S ]\times \textsf{H}_{\mathtt{d}}[T]   \arrow[d,   "\mu_{\mathtt{c} , \mathtt{d}}"]     
			\\
			\textsf{H}_{\ta \tb}[I]           \arrow[r,   "\mathtt{f} \mathtt{g}"']       
			& 
			\textsf{H}_{\mathtt{c}  \mathtt{d} }[I]  
		\end{tikzcd}
	\end{minipage}
\end{figure}
\noindent the diagram \emph{graded Joyal associativity} 
\begin{figure}[H]
	\begin{tikzcd}[row sep=large, column sep=large]  
		\textsf{H}_{\ta}[S]\times \textsf{H}_{\tb}[T]\times \textsf{H}_{\tc}[U]    \arrow[d, " \mu_{\ta,\tb} \times  \text{id}"'] \arrow[r,   "\text{id}  \times  \mu_{\tb,\tc}"]   
		& 
		\textsf{H}_{\ta}[S]\times \textsf{H}_{\tb \tc}[T  U] \arrow[d,   "\mu_{\ta,\tb   \tc} "]     
		\\
		\textsf{H}_{\ta \tb}[S  T]\times \textsf{H}_{\tc}[U]   \arrow[r,   "\mu_{\ta  \tb ,\tc}"']       
		& 
		\textsf{H}_{\ta \tb \tc}[I]
	\end{tikzcd}
\end{figure}
\noindent  the diagrams \emph{graded Joyal comultiplication} $\sigma$-\emph{naturality} and \emph{Joyal comultiplication} $\mathtt{f}$-\emph{naturality}
\begin{figure}[H]
	\centering
	\begin{minipage}{.5\textwidth}
		\centering
		\begin{tikzcd}[row sep=large]  
			\textsf{H}_{\ta|_S}[S]\times \textsf{H}_{\ta|_T}[T] \      \arrow[r,   "\sigma|_{S'}  \times  \sigma|_{T'}"]   
			& 
			\  \textsf{H}_{\ta'|_{S'}}[S']\times \textsf{H}_{\ta'|_{T'}}[T']     
			\\
			\textsf{H}_{\ta}[I]  \arrow[u, "{\Delta_{S,T}}_{\ta}"]              \arrow[r,   "\sigma"']   
			& 
			\textsf{H}_{\ta'}[J]      \arrow[u,   "{\Delta_{S',T'}}_{\ta'}"']     
		\end{tikzcd}
	\end{minipage}%
	\begin{minipage}{.5\textwidth}
		\centering
		\begin{tikzcd}[row sep=large]  
			\textsf{H}_{\ta|_S}[S]\times \textsf{H}_{\ta|_T}[T]    \arrow[r, "\mathtt{f}|_S  \times \mathtt{f}|_T"]    
			& 
			\textsf{H}_{\mathtt{b}|_{S}}[S]\times \textsf{H}_{\mathtt{b}|_{T}}[T]       
			\\
			\textsf{H}_{\ta}[I]        \arrow[r,   "\mathtt{f}"']        \arrow[u,   "{\Delta_{S,T}}_{\ta}"]    
			& 
			\textsf{H}_{\mathtt{b}}[I]   \arrow[u,   "{\Delta_{S,T}}_{\mathtt{b}}"']    
		\end{tikzcd}
	\end{minipage}
\end{figure}
\noindent the diagram \emph{graded Joyal coassociativity} 
\begin{figure}[H]
	\begin{tikzcd}[row sep=large, column sep=large]  
		\textsf{H}_{   \ta|_S  }[S]\times \textsf{H}_{  (\ta|_{ST})|_T  }[T]\times \textsf{H}_{\ta|_U}[U]\     \arrow[d,leftarrow, "{\Delta_{S,T}}_{\ta|_{ST}} \times  \text{id}"'] \arrow[r,leftarrow,   "\text{id}  \times  {\Delta_{T,U}}_{\ta|_{TU}}"]   
		& 
		\ \textsf{H}_{\ta|_S}[S]\times \textsf{H}_{\ta|_{TU}}[T  U] \arrow[d,leftarrow,  "{\Delta_{S,T   U}}_{\ta} "]     
		\\
		\textsf{H}_{\ta|_{ST}}[S  T]\times \textsf{H}_{\ta|_U}[U]   \arrow[r, leftarrow,   "{\Delta_{ST,U}}_\ta"']       
		& 
		\textsf{H}_{\ta}[I]
	\end{tikzcd}
\end{figure}
\noindent and the diagram \emph{graded Joyal bimonoid axiom}
\[\begin{tikzcd}[row sep=large, column sep=small]  
	\textsf{H}_{\mathtt{a}}[ST] \times   \textsf{H}_{\mathtt{b}}[UV]     &  \textsf{H}_{\mathtt{a} \mathtt{b}}[I]  &  \textsf{H}_{  \mathtt{a} \mathtt{b}|_{SU}}[SU] \times  
	\textsf{H}_{   \mathtt{a} \mathtt{b}|_{TV}}[TV]     \\
	\textsf{H}_{\mathtt{a}|_S}[S] \times   \textsf{H}_{\mathtt{a}|_T}[T]   \times  \textsf{H}_{\mathtt{b}|_U}[U] \times   \textsf{H}_{\mathtt{b}|_V}[V]    &&  
	\textsf{H}_{\mathtt{a}|_S}[S] \times   \textsf{H}_{\mathtt{b}|_U}[U]   \times  \textsf{H}_{\mathtt{a}|_T}[T] \times   \textsf{H}_{\mathtt{b}|_V}[V] 
	\arrow["\mu_{\ta ,\tb }", from=1-1, to=1-2]
	\arrow["{\Delta_{SU,TV}}_{\ta \tb}", from=1-2, to=1-3]
	\arrow["{\Delta_{S,T}}_{\ta}\times {\Delta_{U,V}}_{\tb} "', from=1-1, to=2-1]
	\arrow[" \text{id} \times B \times \text{id} "', from=2-1, to=2-3]
	\arrow["\mu_{{\ta|}_{S},\tb|_U}\times \mu_{\ta|_T,\tb|_V}"', from=2-3, to=1-3].
\end{tikzcd}\]

\begin{thm}\label{thm:gradedjoyal}
	The maps ${\mu_{F,G}}_{(\ta_1,\dots, \ta_k)}$ and ${\Delta_{F,G}}_{(\ta_1,\dots, \ta_l)}$ as defined in \textcolor{blue}{(\refeqq{eq:mul})} and \textcolor{blue}{(\refeqq{eq:comul})} give $\textsf{H}$ the structure of an $\textsf{h}$-graded bimonoid if and only if the diagrams graded Joyal (co)multiplication $\sigma$-naturality, Joyal (co)multiplication $\mathtt{f}$-naturality, graded Joyal (co)associativity and the graded Joyal bimonoid axiom commute.
\end{thm}
\begin{proof}
The proof follows along the same lines as \autoref{thm:joyal}.
\end{proof}

The \emph{Joyalization} $\breve{\textbf{\textsf{P}}}$ of a $\textsf{p}$-graded vector species $\textbf{\textsf{P}}$ is defined analogously, 
\[
\breve{\textbf{\textsf{P}}}_{1} \big [ (\, ) \big ]: =  \Bbbk
,\qquad 
\big (\text{and} \quad  
\breve{\textbf{\textsf{P}}}_{\bullet} \big [ (\, ) \big ] :=  0 \quad \text{if } \textsf{C}=\textsf{Cat}_\bullet  \big )
,\]
\[
\breve{\textbf{\textsf{P}}}_{(\ta_1, \dots, \ta_k)} [F]: =   \textbf{\textsf{P}}_{\ta_1} [S_1] \otimes   \dots \otimes  \textbf{\textsf{P}}_{\ta_k} [S_k]
,\]
\[
\breve{\textbf{\textsf{P}}}_{(\ta_1, \dots, \ta_k)} [\sigma,F] :  =   \textbf{\textsf{P}}_{\ta_1}[\sigma|_{S'_1}] \otimes \dots \otimes \textbf{\textsf{P}}_{\ta_k}[\sigma|_{S'_k}] 
,\qquad
\breve{\textbf{\textsf{P}}}_{(\mathtt{f}_1, \dots,  \mathtt{f}_k )  }[ F,F ] : =  \textbf{\textsf{P}}_{\mathtt{f}_1}[ S_1 ] \otimes \dots \otimes   \textbf{\textsf{P}}_{\mathtt{f}_k}[ S_k ]  
.\]
The terms \emph{weakly/strongly/strictly Joyal-theoretic} are also defined analogously for $\textsf{p}$-graded vector species. As before, strictly Joyal-theoretic is abbreviated \emph{Joyal-theoretic}.

If $\textbf{\textsf{H}}$ is a Joyal-theoretic $\textsf{h}$-graded vector species, where $\textsf{h}$ is a Joyal-theoretic bimonoid, we can give $\textbf{\textsf{H}}$ the structure of an $\textsf{h}$-graded bialgebra by replacing sets/functions above with vector spaces/linear maps. Indeed, \autoref{thm:gradedjoyal} goes through unchanged.

If the algebraic structure of an $\textsf{h}$-graded bimonoid/bialgebra $\textsf{H}$ is obtained according to the above construction, then $\textsf{H}$ is called \emph{strictly Joyal-theoretic}, abbreviated \emph{Joyal-theoretic}. If $\textsf{h}$ is strongly Joyal-theoretic, then an \hbox{$\textsf{h}$-graded} bimonoid/bialgebra is called \emph{strongly Joyal-theoretic} if it is Joyal-theoretic up to the specified isomorphism.

We leave implicit the constructions and definitions for cospecies which are analogous to those of this section.
 


\part{Permutohedral Space}\label{part2}

We now use the theory we built in \autoref{part1} to study the bimonoid structure of the toric variety of the permutohedron, called permutohedral space.

\section{Preliminaries}\label{sec:prelim}



\subsection{Preposets} \label{sec:proset}

Given a finite set $I$, a \emph{preposet} $p$ of $I$ is a reflexive and transitive relation $\geq_p$ on $I$. Let $[2;I]$ denote the set of ordered pairs $(\text{i}_1,\text{i}_2)$ of distinct elements $\text{i}_1,\text{i}_2\in I$, $\text{i}_1\neq \text{i}_2$. We identify $p$ with the following subset of $[2;I]$,
	\begin{equation} \label{eq:modeling}          
p :=  \big \{ (\text{i}_1,\text{i}_2)\in [2;I] \ \big | \   \text{i}_1\geq_p \text{i}_2   \big  \} .
	\end{equation}
For preposets $p$ and $q$ of $I$, we write 
\begin{equation} \label{eq:ordering} 
	q\leq  p
	\qquad  
	\text{if} \quad p \subseteq q.
\end{equation}
For $p$ a preposet of $I$ and $S\subseteq I$, the \emph{restriction} $p|_S$ is the preposet of $S$ given by 
\[  
(\text{i}_1,\text{i}_2)\in  p|_S  \qquad  \text{if and only if} \qquad(\text{i}_1,\text{i}_2)\in p
.\]
Recall $\textsf{Cat}_\bullet$ denotes the monoidal category of pointed categories, with monoidal product the smash product of the cartesian product. We have the Joyal-theoretic $\textsf{Cat}_\bullet$-valued cospecies of \emph{augmented preposets} $\textsf{O}_\bullet$, given by 
\[
\textsf{O}_\bullet[I]:=  \big\{  p   \ \big | \   \text{$p$ is a preposet of $I$}    \big \}\sqcup\{  \bullet \},
\]
\[
\textsf{O}_\bullet[\sigma] (\bullet) :=  \bullet
,\qquad
\textsf{O}_\bullet[\sigma] (p) :=  \Big \{ \big(\sigma(\text{i}_1),\sigma(\text{i}_2)\big )\in [2;I] \ \Big | \   \text{i}_1\geq_p \text{i}_2   \Big  \}
.\]
We make the convention that $\bullet \leq p$ for all preposets $p$ of $I$, and turn $\textsf{O}_\bullet[I]$ into a category by including a single morphism $q\to p$ whenever $p\leq q$. At the level of sets, we have
\[     
\textsf{O}_\bullet[F]=   \underbrace{  \textsf{O}_\bullet[S_1] \times \dots \times   \textsf{O}_\bullet[S_k]}_{\text{smash product}}
=
\underbrace{  \big (\text{O}[S_1] \times \dots \times  \text{O}[S_k]\big )}_{\text{cartesian product}}   \sqcup\,   \{  \bullet  \}
\]
where $\text{O}[S_i]$ is the set of all preposets of $S_i$.

\subsection{Total Preposets and Compositions}  \label{sec:preposetcomp}

Preposets may be viewed as a generalization of compositions as follows. We say a preposet is \emph{total} if at least one of 
\[
(\text{i}_1,\text{i}_2)\in p \qquad  \text{and} \qquad (\text{i}_2,\text{i}_1)\in p
\] 
is true for all $(\text{i}_1,\text{i}_2)\in [2;I]$. We have the subcospecies of total preposets $\mathsf{\Sigma}^\ast\subset \textsf{O}_\bullet$, given by
\[
\mathsf{\Sigma}^\ast [I] :=  \big \{     p\in \textsf{O}_\bullet[I] \ \big | \ p \ \text{is a total preposet}        \big \}
.\] 
(The distinguished element `$\bullet$' is not contained in $\mathsf{\Sigma}^\ast$, so this is a subcospecies at the level of \hbox{$\textsf{Cat}$-valued} cospecies.) Given a composition $F\in \mathsf{\Sigma}[I]$, when there is no ambiguity we also let $F$ denote its corresponding total preposet $F\in \mathsf{\Sigma}^\ast [I]$ given by
\[
F   :=  
\big \{     (\text{i}_1,\text{i}_2) \in [ 2; I ]  \ \big | \    \text{`the lump containing $\text{i}_1$' } \geq    \text{ `the lump containing $\text{i}_2$'}      \big \}
.\] 
The function
\[
\mathsf{\Sigma}[I] \to \mathsf{\Sigma}^\ast[I]
,\qquad F \mapsto F
\]
is a bijection which preserves `$\leq$'. The set\footnote{\ where $(S,T)$ denotes its corresponding total preposet, and $(S,T) \leq p$ refers to the relation \textcolor{blue}{(\refeqq{eq:ordering})}}
	\begin{equation} \label{eq:modeling1}     
\big \{ (S,T)\in [I;2] \ \big | \     (S,T) \leq p    \big  \}  
\end{equation}
determines $p$, since it records all of $p$'s upward-closed subsets. Indeed, we have $(S,T)\leq p$ if and only if $S$, respectively $T$, is an upward-closed, respectively downward-closed, subset of $I$. 

\begin{remark}
Note the duality between \textcolor{blue}{(\refeqq{eq:modeling})} and \textcolor{blue}{(\refeqq{eq:modeling1})}; \textcolor{blue}{(\refeqq{eq:modeling})} contains all the $2$-probes $2\hookrightarrow I$ of $I$ which are compatible with $p$, whereas \textcolor{blue}{(\refeqq{eq:modeling1})} contains all the $2$-coprobes $I\twoheadrightarrow 2$ of $I$ which are compatible with $p$.
\end{remark}

\subsection{The Bimonoid of Augmented Preposets} \label{sec:proset2}

We now give augmented preposets $\textsf{O}_\bullet$ the structure of a \hbox{Joyal-theoretic} bimonoid. We saw in \autoref{sec:Joyal-Theoretic Bimonoids} that this amounts to defining a pair of functors $\mu_{S,T}$ and $\Delta_{S,T}$ for each two-lump composition $(S,T)$, and then checking the five Joyal diagrams. 

For $(S,T)\in [I;2]$ and elements $p\in \textsf{O}_\bullet[S]$ and $q\in \textsf{O}_\bullet[T]$, let
\[
(p\, |\, q):= 
\begin{cases}
\text{disjoint set union of $p$ and $q$} &\quad  \text{if } p,q\neq  \bullet \\
\bullet \in \textsf{O}_\bullet[I] &\quad \text{otherwise}.
\end{cases}
\]
For $(S,T)\in [I;2]$ and $p\in \textsf{O}_\bullet[I]$, let
\[
p\talloblong_S:=  
\begin{cases}
p|_S &\quad  \text{if $(S,T)\leq p$}\\
\bullet\in \textsf{O}_\bullet[S]  &\quad \text{otherwise}
\end{cases}\qquad \ \ \ \    \text{and} \  \ \ \  \qquad 
p\! \fatslash_{\, T}:=  
\begin{cases}
p|_T &\quad \text{if $(S,T)\leq p$}\\
\bullet\in \textsf{O}_\bullet[T]  &\quad \text{otherwise.}
\end{cases}
\]

\begin{prop}\label{prob:isbimon}
The $\textsf{Cat}_\bullet$-valued cospecies of augmented preposets $\textsf{O}_\bullet$ has the structure of a Joyal-theoretic bimonoid, with multiplication and comultiplication given by 
\[  
\mu_{S,T} (  p ,  q ): =  (p\, |\, q) 
\qquad \text{and} \qquad    
\Delta_{S,T}  ( p) := (  p\talloblong_S , p \!  \fatslash_{\, T} )  
\]
respectively. 
\end{prop}
\begin{proof}
We have five diagrams to check. Joyal multiplication naturality, Joyal comultiplication naturality, and Joyal associativity are immediate. Recall Joyal coassociativity says
\[
\big ( (p\talloblong_{ST})\talloblong_{S}  ,  (p\talloblong_{ST})\!   \fatslash_{\, T} , p \! \!  \fatslash_{\, U}  \big   )
=
\big (  p  \talloblong_{S},   (p\!  \fatslash_{\, TU})\talloblong_{T} ,  (p\!  \fatslash_{\, TU})\! \!  \fatslash_{\, U}     \big )
.\]
If $(S,T,U)\leq p$, then
\[
\big ( (p\talloblong_{ST})\talloblong_{S}  ,  (p\talloblong_{ST})\!   \fatslash_{\, T} , p \! \!  \fatslash_{\, U}  \big   )
=
\big (  p|_S, p|_T, p|_U  \big ) 
=
\big (  p  \talloblong_{S},   (p\!  \fatslash_{\, TU})\talloblong_{T} ,  (p\!  \fatslash_{\, TU})\! \!  \fatslash_{\, U}     \big )
.\]
We claim that if $(S,T,U)\nleq p$, then 
\[
\big ( (p\talloblong_{ST})\talloblong_{S}  ,  (p\talloblong_{ST})\!   \fatslash_{\, T} , p \! \!  \fatslash_{\, U}  \big   )
=
\, \bullet \, 
=
\big (  p  \talloblong_{S},   (p\!  \fatslash_{\, TU})\talloblong_{T} ,  (p\!  \fatslash_{\, TU})\! \!  \fatslash_{\, U}     \big )
.\]
To see this, first note that if $(S,T,U)\nleq p$, then at least one of $(\text{t},\text{s}),(\text{t},\text{u}),(\text{u},\text{s})\in p$ is true for $\text{s}\in S$, $\text{t}\in T$, $\text{u}\in U$. We have
\[
(\text{t},\text{s})\in p\qquad  \implies \qquad   (p\talloblong_{ST})\talloblong_{S} =\bullet \quad \text{and} \quad   p  \talloblong_{S}=\bullet
.\]
We have
\[
(\text{t},\text{u})\in p\qquad  \implies \qquad   p \!   \fatslash_{\, U} =\bullet \quad \text{and} \quad    (p\!  \fatslash_{\, TU})\! \!  \fatslash_{\, U}=\bullet
.\]
We have
\[
(\text{u},\text{s})\in p\qquad  \implies \qquad   p \!   \fatslash_{\, U}  =\bullet  \quad \text{and} \quad   p  \talloblong_{S}=\bullet
.\]
Recall the Joyal bimonoid axiom says
\[
\big ( ( p\,  |\,  q) \talloblong_{SU}    , (p\, |\, q)  \! \!  \!   \fatslash_{\, TV} \big )
  \, = \, 
    \big (  ( p\! \talloblong_{S}    |\,  q\talloblong_{U} ) ,  ( p\! \!  \fatslash_{\, T}  |\,    q\!  \fatslash_{\, V} )  \big )
.\]
We have 
\[
(SU,TV)\leq  (p\, |\, q)    \qquad \iff \qquad  (S,T) \leq p \quad \text{and} \quad (U,V) \leq q
.\]
Therefore, if $(SU,TV)\nleq  (p\, |\, q)$ then
\[
( p\,  |\,  q) \talloblong_{SU} =  \bullet 
\qquad \text{and} \qquad   ( p\! \talloblong_{S}    |\,  q\talloblong_{U} ) = \bullet 
\]
and if $(SU,TV)\leq  (p\, |\, q)$ then
\[
 ( p\,  |\,  q) \talloblong_{SU}=   ( p\,  |\,  q)|_{SU}      = \big  ( p|_{S}   \big  |\,  q|_{U} \big )   = ( p\! \talloblong_{S}    |\,  q\talloblong_{U} )
\]
and
\[
 (p\, |\, q)  \! \!    \fatslash_{\, TV}  
 =   ( p\,  |\,  q)|_{TV}      = \big  ( p|_{T}   \big  |\,  q|_{V} \big )   = 
( p\! \!  \fatslash_{\, T}  |\,    q\!  \fatslash_{\, V} )
.\]
Finally, notice $\mu_{S,T}$ and $\Delta_{S,T}$ preserve `$\leq$' and `$\bullet$', and so do indeed define functors of pointed categories. 
\end{proof}

For the higher comultiplication, we denote e.g.
\[
(p\talloblong_S , p\talloblong_T , p\talloblong_U):= \Delta_{(S,T,U)} (p) 
.\]

\subsection{Root Datum of $\textmd{SL}_I(\bC)$}

For each finite set $I$, we have the lattice $\bZ^I$ of $\bZ$-valued functions on $I$,
\[      
\bZ^I:=\big \{  \text{functions}\     \lambda: I\to \bZ  \big  \}     
.\]
For $S\subseteq I$, let $\lambda_S\in \bZ^I$ denote the characteristic function of $S$,
\[
\lambda_S: I \to \bZ
,\qquad
i \mapsto  
\begin{cases}
1 &\quad \text{if} \quad \text{i}\in S  \\
0 &\quad     \text{if} \quad  \text{i}\notin S.
\end{cases} 
\] 
If we consider functions only up to translations of $\bZ$, we obtain the quotient lattice 
\[  
\text{N}^I:=    \bigslant{ \big  \{  \text{functions}\     \lambda: I\to \bZ   \big  \} }{\bZ}                = \bZ^I/ \bZ \lambda_I  
.\]
The lattice $\text{N}^I$ is called the \emph{weight lattice} of (the Lie algebra of) $\text{SL}_I(\bC)$. We have the Joyal-theoretic species $I \mapsto \text{N}^I$ with the action of bijections given by $\sigma(\lambda):= \lambda \circ  \sigma$.

We also have the lattice $\bZ I$ of formal $\bZ$-linear combinations of elements of $I$,
\[
\bZ I:=\big \{ h \ \big | \   h=(h_\text{i})_{\text{i}\in I},\, h_\text{i}\in \bZ   \big \}
.\] 
Then $\bZ I$ is the dual lattice of $\bZ^I$ via the perfect pairing
\[ 
\la-, - \ra:\bZ I\times \bZ^I\to \bZ
,\qquad   
\la h, \lambda  \ra := \sum_{\text{i}\in I}h_\text{i} \,  \lambda(\text{i})    
.\]
Since $\text{N}^I$ is a quotient of $\bZ^I$, its dual lattice $\text{M}_I=\Hom(\text{N}^I,\bZ)$ is naturally a sublattice of $\bZ I$, consisting of those formal $\bZ$-linear combinations whose coefficients sum to zero
\[
\text{M}_I :=\big \{  h\in \bZ I  \ \big | \ \la h, \lambda_I \ra =0    \big   \}
.\]
The lattice $\text{M}_I$ is called the \emph{coweight lattice} of $\text{SL}_I(\bC)$. We have the Joyal-theoretic cospecies $I \mapsto \text{M}^I$ with the action of bijections given by 
\[
\la\sigma(h),\lambda\ra:= \la    h,\sigma(\lambda)     \ra
,\]
i.e. $(\sigma(h))_\text{i} = h_{\sigma^{-1}(\text{i})}$ for $\text{i}\in I$. We abbreviate $h':=\sigma(h)$ as usual.

For $(S,T)\in [I;2]$, the \emph{fundamental weight} $\lambda_{ST}\in \text{N}^I$ is defined to be the image of $\lambda_S$ in $\text{N}^I$. We have $\lambda_{TS}=-\lambda_{ST}$. For $(\text{i}_1,\text{i}_2)\in [2;I]$, the \emph{coroot} $h_{\text{i}_1 \text{i}_2} \in \text{M}_I$ is given by
\[
\la h_{\text{i}_1 \text{i}_2},\lambda\ra:=\lambda(\text{i}_1)-\lambda(\text{i}_2)
.\]
We have $h_{\text{i}_2 \text{i}_1}=-h_{\text{i}_1 \text{i}_2}$. 

\begin{remark} \label{cyclic}
We have the following interesting isomorphisms. Given a cyclic ordering of $I$ and $\text{i}_1,\text{i}_2\in I$, let $[\text{i}_1,\text{i}_2]\subseteq I$ denote the interval from $\text{i}_1$ inclusive to $\text{i}_2$ exclusive. Then, for each cyclic ordering of $I$, we have the isomorphism given by
\[
\text{M}_I \to \text{N}^I
,\qquad
h_{\text{i}_1 \text{i}_2} \mapsto  \lambda_{[\text{i}_1,\text{i}_2][\text{i}_2,\text{i}_1]}
.\]
These isomorphisms underlie objects such as alcove polytopes \cite{MR2352704} and blades \cite{early2020weighted}. 
\end{remark}


\subsection{Algebraic Tori}

Let $\bC \text{T}^I$ denote the diagonal maximal torus of $\text{PGL}_I(\bC)$, given by
\[
\bC\text{T}^I:=(\bC^\times)^I/\bC^\times=\bigslant{\{\text{functions}\ x:I\to\bC^\times\}}{\bC^\times} \cong \text{N}^I \otimes_{\bZ} \bC^\times
.\] 
Let $\bC \text{T}^\vee_I$ denote the diagonal maximal torus of $\text{SL}_I(\bC)$, given by
\[
\bC \text{T}^\vee_I
:=
\Big  \{     y \ \Big | \  y= (y_i)_{\text{i}\in I} \in \bC^\times I,\  \prod_{\text{i}\in I} y_\text{i} =1     \Big   \}\cong  \text{M}_I \otimes_{\bZ} \bC^\times 
.\]
We have the following natural identifications of (co)character lattices,
\[
\text{N}^I = \Hom(\bC^\times, \bC \text{T}^I)= \Hom(\bC \text{T}^\vee_I, \bC^\times)  \qquad \text{and} \qquad   \text{M}_I    =   \Hom( \bC \text{T}^I,\bC^\times) =  \Hom( \bC^\times, \bC \text{T}^\vee_I)          
.\]
In particular, for each $h\in \text{M}_I$ we have the monomial
\[
  f^h: \bC\text{T}^I \to \bC
,\qquad 
x \mapsto    \prod_{\text{i}\in I} x(\text{i})^{h_\text{i}}
\]
and for each $\lambda \in \text{N}^I$ we have the monomial
\[
 f^\lambda : \bC\text{T}^\vee_I \to \bC
,\qquad
y \mapsto \prod_{\text{i}\in I} y^{\lambda(\text{i})}_\text{i}  
.\]

\begin{remark}
	We may interpret $\bC\text{T}^I$ as the `worldsheet' $\bC^\times$ with $I$-marked points, modulo scaling. Complex permutohedral space will be a compactification of $\bC\text{T}^I$ with boundary divisors in \hbox{one-to-one} correspondence with elements $(S,T)\in [I;2]$, which are approached by moving the points of $S$ to the infinite future, or the points of $T$ to the infinite past. This interpretation is suggested by the role the bimonoid of compositions $\mathsf{\Sigma}$ plays in casual perturbation theory, see \cite{norledge2020species}.
\end{remark}

Recall one can also do algebraic geometry over the \emph{tropical rig} 
\[
\bT:=\bR\sqcup \{ -\infty \}
.\] 
This is the rig with addition given by maximum and multiplication given by addition. Thus, $\bT^\times =\bR$. We have the following tropical analogs of the tori $\bC\text{T}^I$ and $\bC\text{T}^\vee_I$. Let $\text{T}^I=\bT \text{T}^I$ denote the tropical algebraic torus given by 
\[
\text{T}^I:=(\bT^\times)^I/\bT^\times=\bR^I/ \bR \lambda_I= \bigslant{\{\text{functions}\ \lambda:I\to\bR\}}{\bR}  \cong \text{N}^I \otimes_{\bZ} \bT^\times
.\] 
Let $\text{T}^\vee_I=\bT\text{T}^\vee_I$ denote the tropical algebraic torus given by
\[
\text{T}^\vee_I  :=\bigg \{  h  \  \bigg | \  h=(h_\text{i})_{\text{i}\in I}\in \bR I,\    \sum_{\text{i}\in I} h_\text{i}=0    \bigg   \}  \cong \text{M}_I \otimes_{\bZ} \bT^\times
.\]
This is the sum-zero hyperplane of $\bR I$. Similar to $\bC\text{T}^I$, we may interpret $\text{T}^I$ as the worldline $\bR$ with $I$-marked points, modulo \hbox{time-translational} symmetry. 

At the level of real vector spaces, we have the identification $\text{T}^I= \Hom(  \text{T}^\vee_I, \bR  )$ given by the pairing
\[ 
\la-, - \ra:  \text{T}^\vee_I   \times  \text{T}^I  \to \bR
,\qquad   
\la h, \lambda  \ra := \sum_{\text{i}\in I}h_\text{i} \,  \lambda(\text{i})    
.\]
In particular, for each $h\in \text{M}_I$ we have the tropical monomial
\[
  h: \text{T}^I \to \bT
,\qquad 
\lambda \mapsto \la h, \lambda  \ra 
\]
and for each $\lambda \in \text{N}^I$ we have the tropical monomial
\[
 \lambda: \text{T}^\vee_I \to \bT
,\qquad
h \mapsto \la h,\lambda  \ra 
.\]


\subsection{Hyperplane Arrangements and the Monoid of Semisimple Flats} \label{sec:monoid}

For $(\text{i}_1,\text{i}_2)\in [2;I]$, the \emph{reflection hyperplane} \hbox{$\mathcal{H}_{(\text{i}_1|\text{i}_2)}\subset \text{T}^I$} is given by
\[      
\mathcal{H}_{(\text{i}_1|\text{i}_2)}  :=  \big  \{   \lambda\in \text{T}^I \ \big | \   \la h_{\text{i}_1  \text{i}_2},\lambda\ra =0\big       \}
.\]
This is the tropical variety cut out by the tropical polynomial ``$0+h_{\text{i}_1  \text{i}_2}$''. The collection of all reflection hyperplanes in $\text{T}^I$ is called the \emph{braid arrangement} over $I$. For $(S,T)\in [I;2]$, the \emph{special hyperplane} $\mathcal{H}_{(S|T)}\subset \text{T}^\vee_I$ is given by
\[
\mathcal{H}_{(S|T)}:=  \big  \{  h\in \text{T}^\vee_I \ \big | \  \la h,\lambda_{ST}\ra =0\big       \}              
.\]
This is the tropical variety cut out by the tropical polynomial ``$0+\lambda_{ST}$''. The collection of all special hyperplanes in $\text{T}^\vee_I$ is called the \emph{adjoint braid arrangement} over $I$, which is the adjoint of the braid arrangement in the sense of \cite{aguiar2017topics}. 

For $F=(S_1,\dots, S_k)$ a composition of $I$, the \emph{semisimple flat} $\textsf{T}^\vee_F \subseteq \text{T}^\vee_I$ is given by
\[
\textsf{T}^\vee_F:=  \big \{  h\in \text{T}^\vee_I \ \big | \     \la  h,\lambda_{S_i} \ra  =0  \ \text{for} \  1\leq i \leq k      \big \}         
.\]
We have the set cospecies $\textsf{T}^\vee$ given by
\[
\textsf{T}^\vee[F]:= \textsf{T}^\vee_F 
\]
\[
(\sigma(h))_{\text{i}} := h_{\sigma^{-1}(\text{i})}
,\qquad 
\beta=\text{id}
.\]
Moreover, $\textsf{T}^\vee$ is a monoid with multiplication the inclusion of semisimple flats,
\[
\mu_{F,G}  : \textsf{T}^\vee_F \hookrightarrow \textsf{T}^\vee_G 
,\qquad
h\mapsto h
.\]
Let $\text{T}^\vee$ denote the Joyalization of $\textsf{T}^\vee$. A key fact is that semisimple flats factorize, in the sense that $\textsf{T}^\vee$ is strongly Joyal-theoretic; we have the isomorphism
\[
\eta: \text{T}^\vee \to \textsf{T}^\vee 
\]
given by
\[
\eta_{F}:   \text{T}^\vee_{S_1}    \times  \dots \times   \text{T}^\vee_{S_k}   \xrightarrow{\sim}   \textsf{T}^\vee_F 
,\qquad
\big ( (h_\text{i})_{\text{i}\in S_1} ,\dots ,(h_\text{i})_{\text{i} \in S_k} \big )  \mapsto  (h_\text{i})_{\text{i}\in I} 
.\]
We denote the corresponding multiplication of $\text{T}^\vee$ by $\text{m}_{F,G}$. Particularly important for us is the multiplication $\text{m}_{F}=\text{m}_{F,(I)}=\mu_{F,(I)} \circ \eta_{F}$, which is given by
\[
\text{m}_{F}: \text{T}^\vee_{S_1}    \times  \dots \times   \text{T}^\vee_{S_k}  \hookrightarrow \text{T}^\vee_I
,\qquad 
\big ( (h_\text{i})_{\text{i}\in S_1} ,\dots ,(h_\text{i})_{\text{i} \in S_k} \big )  \mapsto  (h_\text{i})_{\text{i}\in I} 
.\]


\begin{remark}
	Dual to $\textsf{T}^\vee$, there is a comonoid $\textsf{T}$, which is the restriction of the comultiplication of tropical permutohedral space to its embedded torus.
\end{remark}

\subsection{Cones}\label{sec:cones}

Given a subset $X$ of $\text{T}^I$ or $\text{T}_I^\vee$, we have the conical subspace $\text{Coni}\, X$ consisting of all non-negative $\bR$-linear combinations of elements of $X$. For $p\in \textsf{O}_\bullet[I]$ a preposet, let
\[
\sigma^\vee_p: =  \text{Coni}\big   \{ h_{\text{i}_1 \text{i}_2}  \ \big | \  (\text{i}_1,\text{i}_2)\in p\big \} 
\qquad \text{and} \qquad 
\sigma^o_p : =-\sigma^\vee_p =    \text{Coni}\big   \{ h_{\text{i}_1 \text{i}_2} \ \big | \ (\text{i}_2,\text{i}_1)\in p\big \} 
.\]
For the distinguished element $\bullet\in \textsf{O}_\bullet[I]$, we put $\sigma^o_\bullet:=\emptyset$. We have
\[
\sigma^o_q  \subseteq   \sigma^o_p\qquad  \iff   \qquad  p\leq q
.\]
The set of coroots given by $p$ is closed under conical combinations:

\begin{prop}
For $p\in \textsf{O}_\bullet[I]$ a preposet, we have 
\[
\sigma^o_p  \    \cap \   \big  \{ h_{\text{i}_1 \text{i}_2} \ \big | \  (\text{i}_1,\text{i}_2)\in [ 2 ;I ]\big \}\  =\      \big  \{ h_{\text{i}_1 \text{i}_2}  \ \big | \  (\text{i}_2,\text{i}_1)\in p \big \} 
.\]
\end{prop}
\begin{proof}
This follows from e.g. \cite[Section 3]{vic}. 
\end{proof}
\begin{cor}
The map $p\mapsto \sigma^o_p$ is a one-to-cone correspondence between preposets and cones which are generated by coroots. 
\end{cor}

The $\sigma^o_p$, for $p$ a preposet, are generalized permutohedral tangent cones, and the $\sigma^o_F$, for $F$ a total preposet, are permutohedral tangent cones. We have the following important result, which describes the cone $\sigma^o_p$ dually using $[I;2]$ instead of $[2;I]$.

\begin{prop}\label{prop:dualdes}
For $p\in \textsf{O}_\bullet[I]$ a preposet, we have
\[
\sigma^o_p = \big \{   h\in \text{T}_I^\vee \ \big |  \  \la h, \lambda_{ST}  \ra       \leq    0 \quad  \text{for all} \quad  (S,T)\leq  p   \big \}
.\]
\end{prop}
\begin{proof}
See e.g. \cite[Proposition 3.2]{norledge2019hopf}.
\end{proof}

\begin{cor}
The polar cone $\sigma_p$ of $\sigma^o_p$ ($=\, $the dual cone of $\sigma_p^\vee$) is given by
\[
\sigma_p =   \text{Coni}\big   \{ \lambda_{ST} \ \big | \  (S,T) \leq p \big \} 
.\]
\end{cor}
\begin{proof}
	This follows directly from the definition of polar (or dual) cones.
\end{proof}

Thus $p\mapsto \sigma_p$ is a one-to-cone correspondence between preposets and cones ($=$ an intersection of halfspaces) of the braid arrangement. The \emph{braid fan} over $I$ is the set of conical subspaces given by
\[
\big \{   \sigma_F : F\in \Sigma[I] \big   \}
.\]

The following gives us a geometric interpretation of the multiplication of the bimonoid $\textsf{O}_\bullet$; it corresponds to taking products of cones.


\begin{prop}\label{lem:iso}
Given preposets $p \in \textsf{O}_\bullet[S]$ and $q \in \textsf{O}_\bullet[T]$, the multiplication $\text{m}_{(S,T)}$ restricts to an isomorphism of cones
\[
\text{m}_{(S,T)}:    \sigma^o_{p}  \times      \sigma^o_{q} \xrightarrow{\sim}   \sigma^o_{(p\, |\, q)} 
.\]

\end{prop}
\begin{proof}
The product cone $ \sigma^o_{p}  \times      \sigma^o_{q}\subseteq \text{T}^\vee_S \times  \text{T}^\vee_T$ is generated by the set
\[
  \big \{  h_{s_1 s_2} \ \big | \   (\text{s}_1, \text{s}_2) \in p   \big \}  \ \sqcup \   \big \{  h_{\text{t}_1 \text{t}_2} \ \big | \   (\text{t}_1, \text{t}_2) \in q   \big \}
.\]
The $\text{m}_{(S,T)}$-image of this generating set is the generating set of $\sigma^o_{(p\, |\, q)} $,
\[
\text{m}_{(S,T)}  
 \Big (             \big \{  h_{\text{s}_1 \text{s}_2} \ \big | \   (\text{s}_1, \text{s}_2) \in p   \big \}  \ \sqcup \   \big \{  h_{\text{t}_1 \text{t}_2} \ \big | \   (\text{t}_1, \text{t}_2) \in q   \big \}   \Big )
=
  \big \{  h_{\text{i}_1 \text{i}_2} \ \big | \   (\text{i}_1, \text{i}_2) \in  (p\, |\, q)  \big \} 
.\]
Therefore $ \sigma^o_{p}  \times      \sigma^o_{q} \xrightarrow{\sim}   \sigma^o_{(p\, |\, q)}$ is surjective. It is injective because it is the restriction of an injective map.
\end{proof}
	
More generally, it follows that the multiplication $\text{m}_{F,G}$ restricts to an isomorphism
\[
\text{m}_{F,G}:    \sigma^o_{p}  \xrightarrow{\sim}   \sigma^o_{\mu_{F,G}(p)} 
\]
where $p=(p_1, \dots, p_k)\in \textsf{O}_\bullet[F]$ and $\sigma^o_{p}:=\sigma^o_{p_1}\times \dots \times \sigma^o_{p_k}$. 

We now attempt something similar for the comultiplication of $\textsf{O}_\bullet$. Let the \emph{face} of the cone $\sigma^o_p$ in the $\lambda_{ST}$-direction be the region of $\sigma^o_p$ on which the function $\lambda_{ST}$ is maximized.

\begin{lem}\label{prop:face}
	If $(S,T)\leq p$, then the face of $\sigma^o_p$ in the $\lambda_{ST}$-direction is given by the intersection
	\[
 \sigma^o_{p}     \cap  \mathcal{H}_{(S|T)} 
	.\] 
	If $(S,T)\nleq p$, then $\sigma^o_p$ does not have a face in the $\lambda_{ST}$-direction.
\end{lem}
\begin{proof}
	Suppose $(S,T)\leq p$. Then $\lambda_{ST}$ is non-positive on $\sigma^o_{p} $ by \autoref{prop:dualdes}, and $\lambda_{ST}$ is $0$ on $\mathcal{H}_{(S,T)}  $ by definition. Finally, $\sigma^o_{p}     \cap  \mathcal{H}_{(S|T)}$ is nonempty because they both contain the origin.
	
	If $(S,T)\nleq p$, then $(\text{t},\text{s})\in p$ for some $\text{s}\in S$ and $\text{t}\in T$. Therefore $h_{\text{s} \text{t}}\in \sigma^o_{p}$, and so $\lambda_{ST}$ obtains arbitrary large values on $\sigma^o_{p}$.
\end{proof}

The following gives us a geometric interpretation of the comultiplication of $\textsf{O}_\bullet$; it corresponds to taking faces of cones.

\begin{prop} \label{lem:face}
Given a preposet $p\in \textsf{O}_\bullet[I]$ and $(S,T)\leq p$, the multiplication $\text{m}_{(S,T)}$ restricts to an embedding of cones
\[
\text{m}_{(S,T)}:              \sigma^o_{ p\talloblong_S}      \times        \sigma^o_{ p \!   \fatslash_{\, T}}         \hookrightarrow     \sigma^o_{p}  
\]
with image the face of $\sigma^o_p$ in the $\lambda_{ST}$-direction.
\end{prop}
\begin{proof}
In light of \autoref{lem:iso} and \autoref{prop:face}, this amounts to showing
\[
 \sigma^o_{p}     \cap  \mathcal{H}_{(S|T)}   =   \sigma^o_{ (p\talloblong_S | p \!   \fatslash_{\, T} )    }   
.\]
We have $ \sigma^o_{p}     \cap  \mathcal{H}_{(S|T)}   \supseteq   \sigma^o_{ (p\talloblong_S | p \!   \fatslash_{\, T} )    }  $ since
\[
(\text{i}_1 , \text{i}_2)\in  \sigma^o_{ p\talloblong_S}  \qquad  \implies \qquad  (\text{i}_1, \text{i}_2) \in p\quad  \text{and}  \quad  \text{i}_1,\text{i}_2\in S
\]
and
\[
(\text{i}_1 , \text{i}_2)\in  \sigma^o_{p \!   \fatslash_{\, T}}  \qquad  \implies \qquad  (\text{i}_1, \text{i}_2) \in p\quad  \text{and}  \quad  \text{i}_1,\text{i}_2\in T
.\]
To show $ \sigma^o_{p}     \cap  \mathcal{H}_{(S|T)}  \subseteq   \sigma^o_{ (p\talloblong_S | p \!   \fatslash_{\, T} )    }$, recall the cone $\sigma^o_{p}$ may be given as the region satisfying
\[ \la h, \lambda_{AB}  \ra \leq 0
,\qquad
\text{for all} \quad  (A,B) \leq p
.\]
The intersection $    \sigma^o_{p}     \cap  \mathcal{H}_{(S|T)}    $ is then given by adding the extra equality
	\[
	\la h, \lambda_{ST}  \ra = 0 
	.\]
	On the other hand, $ \sigma^o_{ p\talloblong_S}$ is the region given by
	\[ \la h, \lambda_{S_1  S_2}  \ra \leq 0
	,\qquad
	\text{for all} \quad  (S_1,S_2)\leq  p \talloblong_S
	\]
and $\sigma^o_{ p \!   \fatslash_{\, T}}  $ is the region given by
	\[ \la h, \lambda_{T_1  T_2}  \ra \leq 0
,\qquad
\text{for all} \quad  (T_1,T_2)\leq    p \! \!   \fatslash_{\, T}
.\]
Let $h\in   \sigma^o_{p}     \cap  \mathcal{H}_{(S|T)}  $. If $(S_1,S_2)\leq  p \talloblong_S$ then $(S_1, S_2\sqcup T)\leq p$, and so
\[
\la h, \lambda_{S_1}  \ra      \leq 0
.\]
Similarly, if $(T_1,T_2)\leq  p \!   \fatslash_{\, T}$ then $(T_1\sqcup S, T_2)\leq p$, and so
\[
 \la h, \lambda_{T_1}  \ra=   \la h, \lambda_{T_1}  \ra+ \underbrace{\la h, \lambda_{S}  \ra}_{=0}  = \la h, \lambda_{T_1\sqcup S}  \ra  \leq 0
.\qedhere \]
\end{proof}	

\begin{remark}\label{rem:braidconein}
What about an interpretation of the (co)multiplication of preposets $p$ in terms of cones of the braid arrangement? One can view braid arrangement cones $\sigma_p$ as non-positive loci $\sigma_p= (U_p)_{\leq 0}\cap \text{T}_I$ of certain open sets of tropical permutohedral space $U_p\subseteq \mathbb{T}\Sigma^I$ (see \autoref{Up} for the definition of $U_p$ in the complex case). Then the (co)multiplication of braid arrangement cones corresponds to the pullback of opens along the (co)multiplication of $\mathbb{T}\Sigma^I$. 
\end{remark}





\section{The Bimonoid of Permutohedral Space}\label{sec:The Bimonoid of Permutohedral Space}


\subsection{The $\mathsf{{\normalfont O}}_0$-Graded Bialgebra $\bC\textbf{\textsf{O}}$} \label{sec:Graded Bialgebra}



For $p\in \textsf{O}_\bullet[I]$, let
\[
\text{M}_p := \sigma^o_p \cap \text{M}_I
\qquad \text{and} \qquad
\bC[\text{M}_p]:= \bigoplus_{h\in \text{M}_p} \bC\cdot f^h
.\]
For preposets $p\neq \bullet$, $\bC[\text{M}_p]$ is a \hbox{$\bC$-algebra} with multiplication given by $f^{h_1} f^{h_2}:=f^{h_1+h_2}$. If $p=\bullet$, we have the zero vector space $\bC[\text{M}_{\bullet}]= 0$. We have the Joyal-theoretic $\textsf{O}_\bullet$-graded vector cospecies $\bC\textbf{\textsf{O}}$, given by
\[
\bC\textbf{\textsf{O}}_{p}[I]:=  \bC[\text{M}_{p	}] 
,\qquad 
\sigma(f^{h}) := f^{h'}     
\]
and for elements $p,q\in \textsf{O}_\bullet[I]$ with $p\leq q$ and $p\neq \bullet$, the action of the morphism $q\to p$ is the natural inclusion $\bC[\text{M}_q]\hookrightarrow  \bC[\text{M}_p]$. If $p=\bullet$, the action of $q\to p$ is the unique map into the zero vector space. 

We now give $\bC\textbf{\textsf{O}}$ the structure of a Joyal-theoretic $\textsf{O}_\bullet$-graded bialgebra. We shall then use the structure of $\bC\textbf{\textsf{O}}$ to give permutohedral space the structure of a bimonoid. Once this is done, we may then reinterpret $\bC\textbf{\textsf{O}}$ as the bialgebra structure which is induced on the sheaf of regular functions, pulled-back as in \autoref{thm} to the opens of permutohedral space which are indexed by preposets.

For \hbox{$(S,T)\in [I;2]$}, $p\in \textsf{O}_\bullet[S]$ and $q\in \textsf{O}_\bullet[T]$, we have the linear map given by
\[      
\Updelta^\ast_{p , q}:   \bC[\text{M}_{p}]   \otimes  \bC[\text{M}_{q}]  \to    \bC[\text{M}_{(p|q)}]  
,\qquad      
f^{h_1}\otimes f^{h_2}\mapsto     f^{ \text{m}_{(S,T)} (h_1,h_2)  }  
.\]
For $p,q\neq \bullet$, this is an isomorphism of $\bC$-algebras by \autoref{lem:iso}. For $(S,T)\in [I;2]$ and $p\in \textsf{O}_\bullet[I]$, we have the linear map given by
\[
{\upmu^\ast_{S,T}}_p: \bC[ \text{M}_p ] \to \bC[\text{M}_{p\talloblong_S}] \otimes \bC[\text{M}_{p\! \fatslash_{\, T}}] 
,\qquad   
f^{h} \mapsto 
\begin{cases}
 f^{h_1}\otimes f^{h_2}   &\quad  \text{if}\quad  h= \text{m}_{S,T}(h_1 , h_2) \in \sigma^o_{(p \talloblong_S|\,    p\! \fatslash_{\, T})}  \\
 0   &\quad \text{otherwise}.
\end{cases} 
\]
In light of \autoref{lem:face}, we see that in the case $(S,T)\leq p$ this map is sending monomials $f^h$ for which $h$ is in the $\lambda_{ST}$-face of $\sigma_p$ to their image in the decomposition of that face as a product, and is sending all other monomials to zero. Thus, we may view ${\upmu^\ast_{S,T}}_p$ as a projection onto the $\lambda_{ST}$-face of $\sigma_p$. 

\begin{prop} The linear maps $\Updelta^\ast_{p , q}$ and ${\upmu^\ast_{S,T}}_p$ give $\bC\textbf{\textsf{O}}$ the structure of a \hbox{Joyal-theoretic} $\textsf{O}_\bullet$-graded bialgebra. 
\end{prop}
\begin{proof}
There are seven diagrams to check. The four (co)multiplication naturality diagrams are immediate. For associativity, we require the map
\[
f^{h_1} \otimes f^{h_2} \otimes f^{h_3} \mapsto f^{   \text{m}_{  S,T } (   h_1,h_2  )   } \otimes f^{h_3} \mapsto   f^{\text{m}_{  ST,U }( \text{m}_{  S,T } (   h_1,h_2  )  , h_3 ))  }
\]
to equal the map
\[
f^{h_1} \otimes f^{h_2} \otimes f^{h_3} \mapsto  f^{h_1} \otimes f^{   \text{m}_{  T,U } (   h_2,h_3  )   }  \mapsto   f^{\text{m}_{  S,TU }(  h_1, \text{m}_{  T,U } (   h_2,h_3 )  )  }
.\]
Indeed, we have 
\[
\text{m}_{  ST,U }\big (   \text{m}_{  S,T } (   h_1,h_2  )  , h_3  \big )   = \text{m}_{  S,TU }\big ( h_1,  \text{m}_{  T,U } (   h_2,h_3  )   \big ) 
\]
by associativity of the monoid $\text{T}^\vee$ (defined in \autoref{sec:monoid}). For coassociativity, both maps send $f^h$ to zero unless
\[
h=\text{m}_{  ST,U }\big (   \text{m}_{  S,T } (   h_1,h_2  )  , h_3  \big )   = \text{m}_{  S,TU }\big ( h_1,  \text{m}_{  T,U } (   h_2,h_3  )     \big )
\] 
for some $h_1,h_2,h_3$. Then the map
\[
f^h   =   f^{\text{m}_{  ST,U }(   \text{m}_{  S,T } (   h_1,h_2  )  , h_3  ) }  \mapsto  f^{  \text{m}_{S,T}( h_1 , h_2) } \otimes f^{  h_3 } \mapsto f^{ h_1  } \otimes f^{h_2} \otimes f^{  h_3 }
\] 
is equal to the map
\[
f^h = f^{ \text{m}_{  S,TU } ( h_1,  \text{m}_{  T,U } (   h_2,h_3  )    )} \mapsto  f^{ h_1  } \otimes f^{ \text{m}_{T,U} (h_2 , h_3) } \mapsto f^{ h_1  } \otimes f^{h_2} \otimes f^{  h_3 }
.\] 
For the bimonoid axiom, both maps send $f^h\otimes f^k$ to zero unless
\[
 h=\text{m}_{S,T} (h_1 , h_2 )  \qquad \text{and} \qquad k= \text{m}_{U,V}(k_1 , k_2 )
\]
for some $h_1,h_2$ and $k_1,k_2$. Then the map
\[
f^{ \text{m}_{S,T} (h_1 , h_2 ) } \otimes f^{  \text{m}_{U,V}(k_1 , k_2 ) } \mapsto   f^{    \text{m}_{ST,UV}(       \text{m}_{S,T} (h_1 , h_2 )   ,   \text{m}_{U,V}(k_1 , k_2)         ) }
   =     f^{    \text{m}_{SU,TV}(       \text{m}_{S,U} (h_1 , k_1 )   ,   \text{m}_{T,V}(h_2 , k_2)         ) }
   \]
   \[
    \mapsto   f^{ \text{m}_{S,U}(  h_1, k_1 )  }  \otimes f^{ \text{m}_{T,V}(  h_2, k_2 )  }
\]
is equal to the map
\[
f^{ \text{m}_{S,T} (h_1 , h_2 ) } \otimes f^{ \text{m}_{U,V}(k_1 , k_2 ) } \mapsto  f^{ h_1 } \otimes  f^{ h_2 } \otimes f^{ k_1 } \otimes f^{ k_2 }
 \mapsto
f^{ h_1 } \otimes  f^{  k_1 } \otimes f^{ h_2 } \otimes f^{ k_2 }   \]
\[
 \mapsto f^{ \text{m}_{S,U}(  h_1, k_1 )  }  \otimes f^{ \text{m}_{T,V}(  h_2, k_2 )  }
. \qedhere \]
\end{proof}

\subsection{The $\mathsf{\Sigma}$-Graded Bimonoid $\wt{\bC\Sigma}$}\label{sec:permspace1}

We now describe the bimonoid structure of permutohedral space on an atlas of affine charts, before gluing them together in the next section. We define a \hbox{$\mathsf{\Sigma}$-graded} bimonoid $\wt{\bC\Sigma}$ which consists of homomorphisms on the $\bC$-algebras of $\bC\textbf{\textsf{O}}$, with structure maps the pullback of homomorphisms along the corresponding structure maps of $\bC\textbf{\textsf{O}}$ (although the comultiplication ${\Updelta_{S,T}}_H$ is not quite just this, see below).

For $H$ a composition, we denote the affine variety which is the set of $\bC$-points of the affine scheme $\text{Spec}(\bC[\text{M}_{H}])$ by $U_{H}$, thus
\[
U_{H}:= \big   \{  \text{$\bC$-algebra homomorphisms }  x:  \bC[\text{M}_{H}] \to \bC \big \} 
.\]
We have the Joyal-theoretic $\mathsf{\Sigma}$-graded set species $\wt{\bC\Sigma}$, given by
\[
\wt{\bC\Sigma}_H[I]: = U_{H}
,\qquad
\wt{\bC\Sigma}_{H}[\sigma](x) := x\circ \bC\textbf{\textsf{O}}_{H}[\sigma] 
,\qquad 
\wt{\bC\Sigma}_{K\to H}[I](x) :=  x \circ  \bC\textbf{\textsf{O}}_{H\to K}[I]  = x|_{\bC[\text{M}_{H}]}
.\]
The functoriality of $\wt{\bC\Sigma}$ then follows directly from the functoriality of $\bC\textbf{\textsf{O}}$ and the fact that pullback of homomorphisms along functions $f$ is functorial in $f$.

We now give $\wt{\bC\Sigma}$ the structure of a Joyal-theoretic $\mathsf{\Sigma}$-graded bimonoid. For $(S,T)\in [I;2]$, $H\in \Sigma[S]$ and $K\in \Sigma[T]$, we have the function\footnote{\ here we use the fact that the (total preposet of the) concatenation $H; K$ satisfies $(H ; K)\talloblong_S =H$ and $(H ; K) \! \fatslash_{\, T} =K$, i.e. the restriction of the comultiplication of $\textsf{O}_\bullet$ to $\mathsf{\Sigma}^\ast$ is deconcatenation}
\[
\upmu_{H,K}:  U_H  \times U_K   \to  U_{H; K}  
,\qquad  \upmu_{H,K}(x_1, x_2):= ( x_1 \otimes  x_2 )   \circ {\upmu^\ast_{S,T}}_{H ; K}
.\]
For $(S,T)\in [I;2]$ and $H\in \Sigma[I]$, we have the function
\[
 {\Updelta_{S,T}}_H:  U_H \to U_{H|_S}   \times U_{H|_T} 
\]
\[
 {\Updelta_{S,T}}_H(x):= 
x \circ
 \big  ( (H|_S\,  | \,  H|_T)\to H\big   ) \circ \Updelta^\ast_{H|_S,H|_T} =( x|_{\bC[\text{M}_{H|_S}]} ,  x|_{\bC[\text{M}_{H|_T}]} )
\]
where $ ( (H|_S\,  | \,  H|_T)\to H  )$ abbreviates $\bC\textbf{\textsf{O}}_{(H|_S\,  | \,  H|_T)\to H  }[I]$. 

\begin{prop}
The maps $\upmu_{H,K}$ and ${\Updelta_{S,T}}_H$ give $\wt{\bC\Sigma}$ the structure of a \hbox{Joyal-theoretic} $\mathsf{\Sigma}$-graded bimonoid.
\end{prop}
\begin{proof}
This follows from the fact that the structure maps of $\wt{\bC\Sigma}$ are just pullbacks along the structure maps of $\bC\textbf{\textsf{O}}$, which we know satisfies the axioms of a bialgebra. This is immediate, except when the comultiplication is involved since this is not quite just the pullback, and we must make use of comultiplication $\mathtt{f}$-naturality of $\bC\textsf{\textbf{O}}$ as well.

We make the bimonoid axiom explicit, with coassociativity following similarly. We have
\[
{\upmu^\ast_{ST,UV}}_{H\ast K} \circ \big ( (H\ast K|_{SU}\,  | \,  H\ast K|_{TV})\to H\ast K\big ) \circ \Updelta^\ast_{H\ast K|_{SU},H\ast K|_{TV}}
\]
\[
=
\underbrace{\big( 
(H|_{S}\,  | \,  H|_{T}) \to H    \otimes  ( K|_{U}\,  | \,  K|_{V}) \to K
\big )\circ   {\upmu^\ast_{ST,UV}}_{(H\ast K)|_{SU}\,  | \,  (H\ast K)|_{TV})}}_{\text{comultiplication $\mathtt{f}$-naturality of $\bC\textsf{\textbf{O}}$}}  \circ \Updelta^\ast_{(H\ast K)|_{SU},(H\ast K)|_{TV}}
\]
\[
=
\big (  (H|_S\,  | \,  H|_T)\to H  \otimes  (K|_U\,  | \,  K|_V)\to K \big )  \circ \underbrace{\big ( \Updelta^\ast_{H|_S,H|_T} \otimes  \Updelta^\ast_{K|_U,K|_V} \big )  \circ B \circ 
\big ( {\upmu^\ast_{S,U}}_{H|_S, K_U} \otimes   {\upmu^\ast_{T,V}}_{H|_T, K|_V} \big )}_{\text{bimonoid axiom of $\bC\textsf{\textbf{O}}$}}
\]
\[
=\big ( (H|_S\,  | \,  H|_T)\to H \circ \Updelta^\ast_{H|_S,H|_T} \otimes(K|_U\,  | \,  K|_V)\to K \circ \Updelta^\ast_{K|_U,K|_V} \big )  \circ B \circ 
\big ( {\upmu^\ast_{S,U}}_{H|_S, K_U} \otimes   {\upmu^\ast_{T,V}}_{H_T, K|_V} \big )
.\]
To write the second line we used that e.g. $(H\ast K)|_S=H|_S$ because $K$ is a composition of $UV$, and similarly e.g. $(H\ast K)|_{ST}=H$. 
\end{proof}



\subsection{The Bimonoid of Complex Permutohedral Space}\label{sec:permspace2}


We call \emph{complex permutohedral space} over $I$ the $\bC$-scheme which is the colimit of the diagram 
\[
\Sigma[I] \to \textsf{Aff}_\bC
,\qquad
H \mapsto \text{Spec}\big (\bC[\text{M}_H]\big ), \quad (K\to H) \mapsto    (\bC\textbf{\textsf{O}}_{H\to K}[I])^{\ast}
.\] 
Let $\bC\Sigma^I$ denote the algebraic variety which is the set of $\bC$-points of complex permutohedral space, thus
\[ 
\bC\Sigma^I : =          \colim_{H\in \Sigma[I]} U_H  =\bigsqcup_{H\in \Sigma[I]} U_H \ \Big /  \sim
\]
where the equivalence is given by
\[
  x\in U_{H}\sim y\in U_{K}  \iff \exists z\in U_{H\cup K}  \text{ with }  z|_{\bC[\text{M}_{H}]} =x  \text{ and }   z|_{\bC[\text{M}_{K}]} =y 
.\]
The algebraic variety $\bC\Sigma^I$ is the toric variety which is the compactification of the torus $\bC \text{T}^I$ with respect to the braid fan $\{\sigma_H\  | \   H\in \Sigma[I]\}$. The affine charts $U_H\subset \bC\Sigma$ are open sets, however as open sets we shall come to index them by total preposets as opposed to compositions. 

We have the Joyal-theoretic set species $\bC\Sigma$, given by
\[
\bC\Sigma[I] := \bC\Sigma^I
,\qquad 
\bC\Sigma[\sigma](x) :=    x\circ \bC\textbf{\textsf{O}}_{H}[\sigma]
\]
where $x\in U_H$. The action of bijections is well-defined by $(\sigma,\mathtt{f})$-functoriality of $\wt{\bC\Sigma}$; if $x\in U_H$ and $y\in U_K$ such that there exists $z\in U_{H\cup K}$ with $z|_{\bC[\text{M}_{H}]} =x$ and $z|_{\bC[\text{M}_{K}]} =y$, then $z'\in U_{H'\cup K'}$ with $z'|_{\bC[\text{M}_{H'}]} =x'$ and $z'|_{\bC[\text{M}_{K'}]} =y'$.

We let
\[
\bC\Sigma^F:=  \bC\Sigma[F] = \bC\Sigma^{S_1}\times \dots \times \bC\Sigma^{S_k}
\]
and denote the projections by
\[
\pi_{S_j} : \bC\Sigma^F \to \bC\Sigma^{S_j}
.\]
We now give $\bC\Sigma$ the structure of a connected bimonoid by `quotienting out' the grading of $\wt{\bC\Sigma}$. Then the multiplication is given by
\[
\upmu_{S,T}  : \bC\Sigma^S \times \bC\Sigma^T \to \bC\Sigma^I, \qquad (x_1, x_2)\mapsto \upmu_{H,K}(x_1,x_2)
\]
where $H\in \Sigma[S]$, $K\in \Sigma[T]$ such that $(x_1,x_2)\in U_H\times U_K$, and the comultiplication is given by
\[
\Updelta_{S,T}  : \bC\Sigma^I \to \bC\Sigma^S \times \bC\Sigma^T
,\qquad 
x \mapsto   {\Updelta_{S,T}}_H (x) 
\]
where $H\in \Sigma[I]$ such that $x\in U_H$.  


\begin{thm}
Equipped with the maps $\upmu_{S,T}$ and $\Delta_{S,T}$ as defined above, the set species of permutohedral space $\bC\Sigma$ is a Joyal-theoretic bimonoid.
\end{thm}
\begin{proof}
The maps $\upmu_{S,T}$ and $\Delta_{S,T}$ are well-defined for the quotient $\bC\Sigma^I  =          \colim_{H\in \Sigma[I]} U_H$ by (co)multiplication $\mathtt{f}$-naturality of $\wt{\bC\Sigma}$, analogous to what we saw above with bijections $\sigma$. The diagrams for a bimonoid commute for $\bC\Sigma$ because they commute for $\wt{\bC \Sigma}$. 

Alternatively, one can use the Losev-Manin realization of permutohedral space as a moduli space to easily prove this theorem, mentioned in the introduction.
\end{proof}

\begin{remark}\label{quot}
Here, we are taking the colimit over the $\mathsf{\Sigma}$-grading of the bimonoid $\wt{\bC\Sigma}$ to give an ungraded bimonoid $\bC\Sigma$, analogous to taking the vector space $\bigoplus_i V_i\in \textsf{Vec}$ given a graded vector space $(V_i)\in \textsf{gVec}$. The functoriality and naturality axioms of a graded bimonoid means this is always possible; one can always `quotient out' the grading of a graded bimonoid to get an ordinary bimonoid. For example, the colimit of $\bC\textbf{\textsf{O}}$ is rather uninteresting; it has the zero vector space in each component.
\end{remark}

\subsection{Torus Orbits}

For each composition $H=(S_1,\dots,S_k)$ of $I$, we have the \emph{torus orbit} \hbox{$V_H\subseteq U_H\subset \bC \Sigma^I$} given by
\[
V_H :=  \big  \{   x\in U_H\  \big | \    x(f^h)=0 \quad \text{for all} \quad  h\notin \sigma^o_{(  S_1 | \dots | S_k )}       \big \}
\]
where $(  S_1 | \dots | S_k )$ denotes the preposet $p$ given by $(\text{i}_1,\text{i}_2)\in p$ if $\text{i}_1,\text{i}_2\in S_j$ for some $1\leq j \leq k$. We have
\[
\bC\Sigma^I = \bigsqcup_{H\in \Sigma[I]}  V_H
\qquad \text{and} \qquad
U_H = \bigsqcup_{K\leq H} V_K
.\]
Because torus orbits are disjoint, it is useful to describe the (co)multiplication of permutohedral space in terms of torus orbits. 



\begin{lem}\label{lem}
	Given $(S,T)\in [I;2]$, the multiplication $\upmu_{S,T}: \bC\Sigma^S \times \bC\Sigma^T \to \bC\Sigma^I$ restricts to isomorphisms
	\[
	\upmu_{S,T}: V_{H} \times V_{K}   \xrightarrow{\sim}   V_{H ;  K}
	.\] 
\end{lem}
\begin{proof}
Recall the comultiplication of $\bC\textsf{\textbf{O}}$
\[
{\upmu^\ast_{S,T}}_{H ; K} : \bC[ \text{M}_{H ; K} ] \to  \bC[ \text{M}_{H} ] \otimes  \bC[ \text{M}_{K} ] 
\]
which maps the $\lambda_{S,T}$-face $\sigma^o_{(H | K)}$ to its decomposition as a product, and sends everything else to zero. We have that $\sigma^o_{ ( S_1|\dots |S_k|T_1 | \dots | T_l  ) }$ is contained within the $\lambda_{S,T}$-face, and so the restriction of ${\upmu^\ast_{S,T}}_{H ; K} $ to the subalgebra $\bC[\text{M}_{( S_1|\dots |S_k|T_1 | \dots | T_l  )}]$ is an isomorphism
\[   
\bC[\text{M}_{( S_1|\dots |S_k|T_1 | \dots | T_l  )}] \to   \bC[\text{M}_{( S_1|\dots |S_k)}] \otimes  \bC[\text{M}_{(T_1 | \dots | T_l  )}]
\]
which, via the identification
\[
V_H \xrightarrow{\sim} \Hom\! \big (    \bC[   \text{M}_{    (  S_1 | \dots | S_k )     } ] , \bC\big  )
,\qquad x \mapsto x|_{   \bC[   \text{M}_{    (  S_1 | \dots | S_k )     } ] }
\]
and similar identifications for $V_K$ and $ V_{H ;  K}$, is dual to $V_{H} \times V_{K}   \to  V_{H ;  K}$.

Alternatively, one can use the Losev-Manin realization to easily prove this result.
\end{proof}

\begin{lem}\label{lem2}
Given $(S,T)\in [I;2]$, the comultiplication $\Updelta_{S,T}: \bC\Sigma^I \to \bC\Sigma^S \times \bC\Sigma^T$ restricts to surjections
\[
\Updelta_{S,T}: V_{H} \to    V_{H|_S} \times   V_{H|_T}
.\] 
\end{lem}
\begin{proof}
First we show that $V_{H} \to    V_{H|_S} \times   V_{H|_T}$ is well-defined. Recall 
\[
\Updelta_{S,T}(x)  =   ( x|_{ \bC[M_{ H|_S}] }  ,   x|_{ \bC[M_{ H|_T}] }  )
.\]
Given $(\text{i}_1, \text{i}_2)\in H|_S$ such that $(\text{i}_2, \text{i}_1)\notin H|_S$, we have $(\text{i}_1 , \text{i}_2)\in H$ with $(\text{i}_2 , \text{i}_1)\notin H$. Therefore, for $x\in V_{H}$, we have 
\[
 x|_{ \bC[M_{ H|_S}]}(f^{  h_{\text{i}_1 \text{i}_2} })   = \underbrace{x(f^{  h_{\text{i}_1 \text{i}_2} })=0}_{\text{because $x\in V_{H}$}}
.\] 
Thus $x|_{ \bC[M_{ H|_S}] }\in  V_{H|_S} $. We have $x|_{ \bC[M_{ H|_T}] }\in  V_{H|_T} $ by the same argument. For surjectivity, we have
\[
\sigma^\circ_{\big (    S_1 \cap S | \dots | S_k \cap S   \big |    S_1 \cap T | \dots | S_k \cap T   \big)} \subset \sigma^\circ_{(  S_1 | \dots | S_k )}
,\]
and so every homomorphism in $ V_{H|_S} \times   V_{H|_T}$ can be extended to a homomorphism in $ V_{H}$.

Alternatively, one can use the Losev-Manin realization to easily prove this result.
\end{proof}

\begin{remark}
Thus, $\bC\Sigma$ is isomorphic to the species of \emph{closed torus orbits}
\[
F \mapsto \overline{V}_F:= \bigsqcup_{F\leq H} V_H
\] 
with multiplication given simply by inclusion of closed torus orbits. This way of doing things is nice because the key fact about permutohedral space being a space with factorization of its boundaries is formalized as the statement that the species $F \mapsto \overline{V}_F$ is strongly Joyal-theoretic.
\end{remark}



\section{Opens and Preposets}\label{sec:site}

We define the bimonoid in cospecies of Zariski open sets $\textsf{Op}^\text{op}$, and we realize the structure sheaf as an \hbox{$\textsf{Op}^\text{op}$-graded} bialgebra. This recovers $\bC\textbf{\textsf{O}}$ as the pullback of the grading along indexing certain opens by preposets.

\subsection{The Bimonoid of Opens}

For $F$ a composition, let $\textsf{Op}_F$ be the pointed category with objects Zariski open sets $U\subseteq \bC\Sigma^F$, morphisms inclusion of open sets $U\hookrightarrow V$, and distinguished object the empty set $\emptyset \subseteq \bC\Sigma^F$. We have the $\textsf{Cat}_\bullet$-valued cospecies of opens $\textsf{Op}^{\text{op}}$, given by 
\[
\textsf{Op}^{\text{op}}[F]: =  \textsf{Op}^{\text{op}}_F
,\]
\[
\textsf{Op}^{\text{op}}[\sigma,F](U):= \big(\bC\Sigma[\sigma,F]\big)^{-1}(U)
,\qquad
\textsf{Op}^{\text{op}}[F,\beta](U):=  \big(\bC\Sigma[F,\beta]\big)^{-1}(U)
.\]
This is a strict cospecies. We abbreviate $U'=\textsf{Op}^{\text{op}}[\sigma,F](U)$ and $\wt{U}=\textsf{Op}^{\text{op}}[F,\beta](U)$ as usual.

Let $\text{Op}^{\text{op}}$ denote the Joyalization of $\textsf{Op}^{\text{op}}$. We give $\textsf{Op}^{\text{op}}$ the structure of a weakly \hbox{Joyal-theoretic} cospecies by taking cartesian products of open sets,
\[
\text{prod}: \text{Op}^{\text{op}} \to  \textsf{Op}^{\text{op}}
,\qquad
(U_1,\dots, U_k)\,  \mapsto \,  U_1 \times \dots \times  U_k  :=    \pi_{S_1}^{-1} ( U_1) \cap \dots \cap  \pi_{S_k}^{-1} (U_k)
.\] 
Note this map commutes with the action of $\sigma$ and $\beta$ on the nose, so it is a strict morphism. 

We now give $\textsf{Op}^{\text{op}}$ the structure of a bimonoid. We have the functors the pullback of opens along the (co)multiplication of permutohedral space,
\[
 \Updelta^{-1}_{F,G} :\textsf{Op}^\text{op}_F \to \textsf{Op}^\text{op}_G
,\qquad
U \mapsto \Updelta^{-1}_{F,G}(U) 
\]
and
\[
\upmu^{-1}_{F,G} :\textsf{Op}^\text{op}_G \to \textsf{Op}^\text{op}_F
,\qquad
U \mapsto \upmu^{-1}_{F,G}(U) 
.\]

\begin{prop}
The functors $\Updelta^{-1}_{F,G}$ and $\upmu^{-1}_{F,G}$ give opens of permutohedral space $\textsf{Op}^{\text{op}}$ the structure of a bimonoid in \hbox{$\textsf{Cat}_\bullet$-valued} cospecies. This is a strict bimonoid. 
\end{prop}
\begin{proof}
This follows directly from the fact that $\bC\Sigma$ is a bimonoid, together with the fact that pullback of sets $X$ along functions $f$ is functorial in $f$, i.e.
\[
 (g\circ f)^{-1}(X)   =  f^{-1}\big ( g^{-1}(X)\big )      
.\qedhere \]
\end{proof}

\begin{remark}
Note the self-dual nature of the bimonoid axiom; it is the diagram $\text{B}_{G,F,A}$ of $\bC\Sigma$ which implies the diagram $\text{B}_{F,G,A}$ of $\textsf{Op}^{\text{op}}$.
\end{remark}

\subsection{The Structure Sheaf Bialgebra and the Restriction to Preposets}\label{Up}

We can model the structure sheaf of permutohedral space as the $\textsf{Op}^{\text{op}}$-graded \hbox{vector cospecies} $\bC\textbf{\textsf{Op}}^{\text{op}}$ given by taking regular functions, pulling them back along the corresponding maps of $\bC\Sigma$, and restricting them along inclusions $V\hookleftarrow U$,
\[
\bC\textbf{\textsf{Op}}^{\text{op}}_U[F]  := \big\{  \text{regular functions } f:U\to \bC       \big\}
\]
\[
\bC\textbf{\textsf{Op}}^{\text{op}}_U[\sigma,F](f) := f \circ \bC\Sigma[\sigma,F]
,\quad
\bC\textbf{\textsf{Op}}^{\text{op}}_U[F,\beta](f) :=   f \circ \bC\Sigma[F,\beta]
,\quad 
\bC\textbf{\textsf{Op}}^{\text{op}}_{V\hookleftarrow U}[F,F] (f) := f|_U
.\]
The functoriality of $\bC\textbf{\textsf{Op}}^{\text{op}}$ follows directly from the fact that functions commute with inclusion of subsets and the functoriality of $\textsf{Op}^{\text{op}}$, together with the fact that pullback of functions along functions $f$ is functorial in $f$.

We have the linear maps the pullback of regular functions along the (co)multiplication of permutohedral space,
\[
\Updelta_{F,G}^{\ast}   :      \bC\textbf{\textsf{Op}}^{\text{op}}_U[F] \to \bC\textbf{\textsf{Op}}^{\text{op}}_{ \Updelta_{F,G}^{-1}(U) }[G]
,\qquad
f \mapsto   \Updelta_{F,G}^{\ast} f :=      f\circ \Updelta_{F,G} 
\]
and
\[
\upmu_{F,G}^{\ast}  :  \bC\textbf{\textsf{Op}}^{\text{op}}_U[G] \to \bC\textbf{\textsf{Op}}^{\text{op}}_{ \upmu_{F,G}^{-1}(U) }[F]
,\qquad
f \mapsto    \upmu_{F,G}^{\ast} f :=     f\circ \upmu_{F,G} 
.\]

\begin{prop}
The maps $\Updelta_{F,G}^{\ast}$ and $\upmu_{F,G}^{\ast}$ give $\bC\textbf{\textsf{Op}}^{\text{op}}$ the structure of a \hbox{$\textsf{Op}^{\text{op}}$-graded} bialgebra in cospecies. 
\end{prop}
\begin{proof}
We have that (co)multiplication $\mathtt{f}$-naturality follows directly from the fact that functions commute with inclusion of subsets, together with the fact that pullback of functions along functions $f$ is functorial in $f$. Everything else follows directly from the fact that $\bC\Sigma$ is a bimonoid, again together with the fact that pullback of functions along functions is functorial.
\end{proof}

We can index certain opens of permutohedral space $\bC\Sigma^I$ with preposets; we have the \hbox{Joyal-theoretic} morphism of cospecies given by
\[
 U_{(-)}: \textsf{O}_\bullet \to \text{Op}^\text{op}
,\qquad 
p\mapsto U_p:=\bigcup_{F\leq p} U_F  = \bigsqcup_{F\leq p} V_F
.\]
In particular, we have $U_\bullet = \emptyset$. This is a strict morphism. 

\begin{remark}
Notice $U_p$ is naturally the toric variety of the subfan $\{\sigma_F \ | \ F \leq p \}$ of the braid fan. 
\end{remark}

\begin{thm}\label{indexbyopens}
The composition of morphisms
\[
\varphi: \textsf{O}_\bullet \xrightarrow{U_{(-)}}   \text{Op}^\text{op}   \xrightarrow{\text{prod}} \textsf{Op}^\text{op} 
,\qquad
(p_1, \dots, p_k)   \mapsto U_{p_1} \times \dots \times U_{p_k}
\]
is a strict homomorphism of bimonoids.
\end{thm}
\begin{proof}
For compositions $G\leq F$, let
\[
F=(S_1,\dots, S_k) \qquad \text{and} \qquad G=(T_1,\dots, T_l)
.\]
Let $1\leq \mathtt{i}_j \leq k$ be the integer such that $S_{\mathtt{i}_j}$ is the $\mathtt{i}$th lump of the restriction of $F$ to $T_j$, $1\leq j \leq l$. Thus,
\[
\mathtt{i}_j := k(1)+\, \cdots\,  + k(j-1) +\mathtt{i} 
\]
where $k(j):=l(F|_{T_j})$. 
	
Preservation of multiplication says that given an element $p=(p_1,\dots, p_k)\in \textsf{O}_\bullet[F]$, we should have
\begin{equation}\label{presmult}
\Updelta^{-1}_{F,G}(U_{p_1}\times \dots \times U_{p_k})= U_{(p_{1_1 }   | \dots | p_{k(1)_1} )}   \times \dots \times   U_{(p_{ 1_l }   | \dots | p_{k(l)_l} )} 
.
\end{equation}
However, the pullback of product opens along the comultiplication $\Updelta_{F,G}$ factorizes as follows,
\[
\Updelta^{-1}_{F,G}  (U_{p_1}\times \dots \times U_{p_k})   =( \Delta_{F|_{T_1}}      \times \dots \times  \Updelta_{F|_{T_l}})^{-1}(U_{p_1}\times \dots \times U_{p_k}) 
\]
\[
=
\Updelta^{-1}_{F|_{T_1}}  (U_{p_{1_1}}  \times \dots \times U_{p_{k(1)_1}})   \times \cdots\cdots  \times   \Updelta^{-1}_{F|_{T_l}} ( U_{p_{ 1_l }}   \times \dots \times U_{p_{k(l)_l}})
.\]	
Therefore it is enough to check \textcolor{blue}{(\refeqq{presmult})} in the case $G=(I)$, which is 
\[
\Updelta^{-1}_{F}(U_{p_1}\times \dots \times U_{p_k}) =
U_{(p_{1}   | \dots | p_{k} )}   
.\]
We have
\[
U_{p_1}\times \dots \times U_{p_k}  =    \bigsqcup_{ H_i\leq p_i} V_{H_1} \times\dots \times  V_{H_k}
\qquad \text{and} \qquad
U_{(p_1|\dots | p_k)}    = \bigsqcup_{H\leq (p_1|\dots | p_k)} V_H
.\]
Let $x\in V_H\subset \bC\Sigma[I]$. If $H\leq (p_1|\dots | p_k)$, then $H_i|_{S_i} \leq p_i$ for $1\leq i \leq k$, and so
\[
 \Updelta_{F}(x)  \in  V_{H|_{S_1}} \times \dots \times V_{H|_{S_k}} \subset    U_{p_1} \times  \dots \times U_{p_k}
.\]
Conversely, if $ \Updelta_{F}(x) \in U_{p_1}\times \dots U_{p_k}$, then
\[
( V_{H|_{S_1}} \times \dots \times V_{H|_{S_k}} )  \ \cap \   (U_{p_1} \times \dots \times U_{p_k}) \neq \emptyset
.\]
But torus orbits are disjoint. Therefore $H|_{S_i} \leq p_i$ for $1\leq i \leq k$, and so $H\leq (p_1|\dots |p_k)$.

The pullback of product opens along the multiplication $\upmu_{F,G}$ factorizes similarly, and so it is enough to check the preservation of comultiplication in the case $G=(I)$. This says given an element $p\in \textsf{O}_\bullet[I]$, we have
\[
\upmu^{-1}_{F}( U_p )=  
\begin{cases}
	U_{p\talloblong_{S_1}} \times \dots \times  	U_{p\talloblong_{S_k}} &\quad  \text{if}\  F \leq p      \\
	\emptyset &\quad \text{otherwise.}
\end{cases}
\]
The image of the multiplication $\upmu_{F}$ is all those torus orbits $V_H\subset \bC\Sigma[I]$ such that $H$ is obtained from a concatenation over $F$, that is
\[
\image (\upmu_{F}) =   \bigsqcup_{H_i \in \Sigma[S_i]}  V_{H_1  ; \cdots ; H_k}
.\] 
We have $V_{H_1  ; \cdots ; H_k} \subseteq U_p$ if and only if $H_1  ; \cdots ; H_k \leq p$, therefore by \autoref{lem} we have 
\[
\upmu^{-1}_{F}(U_p)=     \bigsqcup_{ \substack{H_i \in \Sigma[S_i] \\H_1  ; \cdots ; H_k \leq p }}  V_{H_1} \times \dots \times  V_{H_k}  
.\]
If $F\leq p$, then
\[
H_1  ; \cdots ; H_k \leq p   \quad  \iff \quad  H_i \leq   p\! \talloblong_{S_i} \quad \text{for all} \quad 1\leq i \leq k 
.\]
Therefore
\[
\upmu^{-1}_{F}(U_p)=     \bigsqcup_{ H_i \leq   p\talloblong_{S_i}}  V_{H_1} \times \dots \times  V_{H_k}  
=
\bigsqcup_{ H_1 \leq   p\talloblong_{S_1}}  V_{H_1}   \times  \dots \times      \bigsqcup_{ H_k \leq   p\talloblong_{S_k}}  V_{H_k} 
\ =\ 
	U_{p\talloblong_{S_1}} \times \dots \times  	U_{p\talloblong_{S_k}}  
\]
as required. Now suppose $F\nleq p$. However, we always have
\[
H \leq H_1 ; \cdots ; H_k  \qquad \text{and} \qquad  H_1 ; \cdots ; H_k \leq p
.\]
Therefore in this case there can be no such concatenations $H_1 ; \cdots ; H_k $, and so
\[
\upmu^{-1}_{F}(U_p) = \emptyset
\]
as required. 
\end{proof}

\begin{remark}
Notice the pullback $\textsf{O}_\bullet$-graded bialgebra along the homomorphism $\textsf{O}_\bullet\to  \textsf{Op}^{\text{op}}$, as constructed in \autoref{thm}, is (naturally isomorphic to) $\bC\textbf{\textsf{O}}$.
\end{remark}


\section{Sheaves of Modules} \label{sec:topos}
\subsection{The Cospecies of Sheaves} 

We recall some basic aspects of sheaves on permutohedral space, in particular the \'etale space perspective on sheaves.

For $F$ a composition, let $\textsf{Top}/\bC \Sigma^F$ denote the slice category of topological bundles over $\bC\Sigma^F$, where $\bC\Sigma^F$ is equipped with the Zariski topology. Explicitly, objects of $\textsf{Top}/\bC \Sigma^F$ are topological spaces $E$ equipped with continuous maps 
\[
\pi: E\to \bC\Sigma^F
\]
and morphisms are continuous maps $\varphi: E_1\to E_2$ such that the diagram
\[\begin{tikzcd}
	{E_1} && {E_2} \\
	& {\mathbb{C}\Sigma^F}
	\arrow["\varphi", from=1-1, to=1-3]
	\arrow["{\pi_1}"', from=1-1, to=2-2]
	\arrow["{\pi_2}", from=1-3, to=2-2]
\end{tikzcd}\]
commutes. An \emph{\'etale space} over $\bC\Sigma^F$ is a topological bundle $E\to \bC\Sigma^F$ which is also a local homeomorphism. For each point $x\in \bC\Sigma^F$, we call 
\[
E(x):= \pi^{-1}(x)
\] 
the \emph{stalk} of $E$ at $x$. For each open $U\subseteq \bC\Sigma^F$, we have the set $\Gamma_U(E)$ of \emph{sections} of $E$ over $U$, given by
\[
\Gamma_U(E):=  \big \{   \text{continuous maps}\ s:U\to E \ \big |\   \pi \circ  s = \text{id}_U   \big \}
.\]
We have the presheaf $\mathcal{F}_E$ on $\textsf{Op}_F$ associated to $E$, given by
\[
\mathcal{F}_E: \textsf{Op}_F^\text{\text{op}} \to \textsf{Set}
,\qquad 
U\mapsto  \Gamma_U(E)
,\quad (V\hookrightarrow U) \mapsto (s \mapsto s|_V) 
.\]
Moreover, $\mathcal{F}_E$ satisfies locality and gluing, and so is in fact a sheaf. 

Going the other way, consider a generic presheaf on $\textsf{Op}_F$,
\[
\mathcal{F}: \textsf{Op}_F^\text{\text{op}} \to \textsf{Set}
.\] 
We still call elements $s\in \mathcal{F}(U)$ \emph{sections}, and denote $s|_V:= \mathcal{F}(V\hookrightarrow U)(s)\in \mathcal{F}(V)$. For each (possibly non-open) subset $X\subseteq \bC\Sigma^F$, let
\[
\mathcal{F}(X):= \colim_{X\subseteq U}\mathcal{F}(U) = \bigsqcup_{X \subseteq U} \mathcal{F}(U)\ \   \Big/  \sim
\]
where the equivalence is given by 
\[
s_1\in \mathcal{F}(U_1) \sim s_2\in \mathcal{F}(U_2) \iff \text{there exists $W\in \textsf{Op}_F$ such that $X\subseteq W\subseteq U_1 \cap U_2$ and $s_1 |_W = s_2 |_W$}
.\]
For each point $x\in \bC\Sigma^F$, the set 
\[
\mathcal{F}(x):=\mathcal{F}\big (\{x\}\big )
\] 
is called the \emph{stalk} of $\mathcal{F}$ at $x$. The image of a section $s\in \mathcal{F}(U)$ in the stalk at a point $x\in U$ is called the \emph{germ} $s_x\in \mathcal{F}(x)$ of $s$ at $x$. We have the topological bundle $E_\mathcal{F}\to \bC\Sigma^F$ associated to $\mathcal{F}$, which has stalks the stalks of $\mathcal{F}$
\[
E_\mathcal{F} (x) :=    \mathcal{F}(x)
\]
and topology with basis
\[
\Big \{ (U,s):= \big \{    s_x \ \big |\ x\in U    \big  \} \ \Big | \ U\in \textsf{Op}_F,\    s\in \mathcal{F}(U)   \Big \}
.\]
These constructions extend to an idempotent adjunction between topological bundles and presheaves, which restricts to an equivalence between \'etale spaces and sheaves. Thus, when working with sheaves, we can treat them as certain presheaves, or as certain topological bundles. From now on, let us also refer to \'etale spaces as sheaves. 

Let $\textsf{Sh}_F$ denote the full subcategory of $\textsf{Top}/\bC \Sigma^F$ of sheaves over $\bC\Sigma^F$. We have the \hbox{$\textsf{Cat}$-valued} cospecies of sheaves $\textsf{Sh}$, given by pulling back bundles with respect to the corresponding maps of $\bC\Sigma$ (i.e. the so-called inverse image functor),
\[
\textsf{Sh} [F] :=  \textsf{Sh}_F
\]
\[
\textsf{Sh} [F,\sigma](E):= \sigma^{-1} E=  E  \times_{\bC \Sigma^{F}}  \bC \Sigma^{F'} = \lim (    E \xrightarrow{\pi} \bC\Sigma^F \xleftarrow{\sigma} \bC\Sigma^{F'}   )
\]
\[
\textsf{Sh}[\beta,F](E):= \beta^{-1} E = E \times_{\bC\Sigma^F} \bC \Sigma^G    =    \lim (    E \xrightarrow{\pi} \bC\Sigma^F \xleftarrow{\beta} \bC\Sigma^{G}   )
.\]
In fact, this can easily be formalized as a strict cospecies, since pullback sheaves have stalks equal to stalks of the original sheaf. We may give $\textsf{Sh}$ the structure of a bimonoid by pulling back sheaves with respect to the (co)multiplication of of $\bC\Sigma$,
\[
\Updelta^{-1}_{F,G}: \textsf{Sh}_F \to \textsf{Sh}_G
,\qquad 
\Updelta^{-1}_{F,G}  E= E\times_{\bC\Sigma^F} \bC\Sigma^G   =    \lim \big (  E  \xrightarrow{\pi}   \bC\Sigma^F \xleftarrow{\Updelta_{F,G}} \bC\Sigma^G \big)
\]
and
\[
\upmu^{-1}_{F,G}  : \textsf{Sh}_G \to \textsf{Sh}_F
,\qquad 
\upmu^{-1}_{F,G} E= E \times_{ \bC\Sigma^G	 } \bC\Sigma^F  =   \lim \big (   E  \xrightarrow{\pi}   \bC\Sigma^G \xleftarrow{\upmu_{F,G}} \bC\Sigma^F\big )
.\]
We will not use this bimonoid structure. We are interested in the bimonoid structure on sheaves of modules, and in particular invertible sheaves.

\subsection{The Cospecies of Sheaves of Modules}

The category of sheaves $\textsf{Sh}_F$ is cartesian monoidal by taking fibered products,
\[
E_1 \times E_2 :=E_1\times_{\bC \Sigma^F} E_2
.\]
If we model sheaves as certain presheaves, this monoidal product corresponds to the usual cartesian product of presheaves, given by
\[
\mathcal{F}\times \mathcal{G}(U) = \mathcal{F}(U)\times \mathcal{G}(U)
.\]

Let $\mathcal{R}\in \textsf{Sh}_F$ be an internal ring. In particular, this means stalks $\mathcal{R}(x)$ and sections $\Gamma_U(\mathcal{R})$ have the structure of usual rings. Let $M$ and $N$ be $\mathcal{R}$-modules. In particular, this means stalks $M(x)$ and sections $\Gamma_U(M)$ have the structure of $\mathcal{R}(x)$-modules and $\Gamma_U(\mathcal{R})$-modules respectively. The \emph{tensor product} of $\mathcal{R}$-modules $M\otimes N=M\otimes_{\mathcal{R}} N$ is the sheaf (=\'etale space) of the presheaf
\[
U \mapsto   \Gamma_U(M)\otimes_{\Gamma_U(\mathcal{R})} \Gamma_U(N)
.\]
Explicitly, the stalks of $M\otimes N$ are given by pointwise tensoring 
\[
M\otimes N (x):= M(x) \otimes   N(x) =  M(x) \otimes_{\mathcal{R}(x)}   N(x)   
\]
and the topology has basis
\[
\bigg\{\Big\{ \sum_i s^i_x\otimes t^i_x \ \Big |\ x\in U    \Big  \} \ \bigg | \ U\in \textsf{Op}_F,\  \sum_i s^i\otimes t^i \in  \Gamma_U(M)\otimes_{\Gamma_U(\mathcal{R})}  \Gamma_U(N)  \bigg\}
.\]
Here $\sum_i s^i\otimes t^i$ is denoting a generic element of $\Gamma_U(M)\otimes_{\Gamma_U(\mathcal{R})} \Gamma_U(N)$ as a sum of simple tensors. 

Let $\mathcal{K}_F$ denote the function field of rational functions on $\bC\Sigma^F$. The \emph{structure sheaf} $\mathcal{O}_F\in \textsf{Sh}_F$ is the internal ring with stalks given by
\[
\mathcal{O}_F(x) :=  \big  \{   g/f\in \mathcal{K}_F \ \big | \ f(x)\neq 0 \big  \}
\]
and topology with basis
\[
\Big \{ (U,g):= \big \{    \text{germ}_x(g) \ \big |\ x\in U    \big  \} \ \Big | \ U\in \textsf{Op}_F,\  g:U\to \bC \text{ a regular function}   \Big \}
.\]
We have
\[
\Gamma_U( \mathcal{O}_F )  = \big \{ \text{regular functions } g:U\to \bC   \big \}
.\]
Let $\textsf{Mod}_F$ denote the category of $\mathcal{O}_F$-modules. We have the \hbox{$\textsf{Cat}$-valued} cospecies of modules $\textsf{Mod}$, given by pulling back modules along the corresponding maps of $\bC\Sigma$,
\[
\textsf{Mod}[F] :=  \textsf{Mod}_F
\]
\[
\textsf{Mod} [\sigma,F](M):=      \bC\Sigma[\sigma,F]^\ast M   =   \sigma^{-1} M \otimes_{ \sigma^{-1} \mathcal{O}_{F} } \mathcal{O}_{F'}
\]
\[
\textsf{Mod}[F,\beta](M):=     \bC\Sigma[F,\beta]^\ast M = \beta^{-1} M \otimes_{\beta^{-1} \mathcal{O}_{F} } \mathcal{O}_{\wt{F}}
.\]
We abbreviate $M'=\textsf{Mod} [\sigma,F](M)$ and $\wt{M}=\textsf{Mod}[\beta,F](M)$ as usual. We still need to define three \hbox{$2$-cells} $(\sigma,\beta)$, $(\sigma,\tau)$ and $(\beta,\gamma)$. They are just the canonical isomorphisms which express how pullback of sheaves of modules along maps $f$ is functorial in $f$. We now make them explicit.

Given a bijection $\sigma:J\to I$ and a permutation $\beta \in \text{Sym}_k$, for each $\mathcal{O}_{F}$-module $M$, $F\in \Sigma[J]$, $l(F)=k$, we need a natural isomorphism of $\mathcal{O}_{\wt{F}'}$-modules
\[
(\sigma, \beta)_M :  \wt{(M')} \to  (\wt{M})'
.\]
On the stalk at $x\in \bC\Sigma^{\wt{F}'}$, $(\sigma, \beta)_M$ is given by
\[
{ (\sigma, \beta)_M }|_{x}  : 
M(\tilde{x}') \otimes \mathcal{O}_{F'}(\tilde{x}) \otimes   \mathcal{O}_{\wt{F}'}(x)       \to    
 M(\tilde{x}') \otimes \mathcal{O}_{\wt{F}}(x') \otimes   \mathcal{O}_{\wt{F}'}(x) 
\]
\[
\text{g}\otimes \text{f}_1 \otimes \text{f}_2 \mapsto \text{g}  \otimes 1 \otimes  ( \beta^\ast \text{f}_1 ) \text{f}_2
.\]
Here, $\beta^\ast \text{f}_1$ denotes the pullback of the germ $\text{f}_1$ along $\beta=\bC\Sigma[\wt{F}',\beta]$. The inverse of $(\sigma, \beta)$, which we need for the bimonoid axiom, is given by
\[
\text{g}\otimes \text{f}_1 \otimes \text{f}_2 \mapsto \text{g}  \otimes 1 \otimes  ( \sigma^\ast \text{f}_1 ) \text{f}_2
.\]
Similarly, $\sigma^\ast \text{f}_1$ denotes the pullback of the germ $\text{f}_1$ along $\sigma=\bC\Sigma[\sigma, \wt{F}']$. Given composable bijections $K\xrightarrow{\tau} J \xrightarrow{\sigma} I$, for each $\mathcal{O}_{F}$-module $M$, $F\in \Sigma[K]$, we need a natural isomorphism of $\mathcal{O}_{F''}$-modules
\[
(\sigma,\tau)_M :  
(M' )' \to  M''
.\]
On the stalk at $x\in \bC \Sigma[F'']$, $(\sigma,\tau)_M$ is given by
\[
{(\sigma,\tau)_M}|_{x}  :  M(x'') \otimes \mathcal{O}_{F'}(x')  \otimes   \mathcal{O}_{F''}(x)       \to    M(x'')  \otimes   \mathcal{O}_{F''}(x) 
\]
\[
\text{g}\otimes \text{f}_1 \otimes \text{f}_2 \mapsto \text{g} \otimes (\sigma^\ast\text{f}_1 ) \text{f}_2
.\]
Given permutations $\beta,\gamma\in \text{Sym}_k$, for each $\mathcal{O}_{F}$-module $M$, $l(F)=k$, we need a natural isomorphism of $\mathcal{O}_{\wt{\wt{F}}}$-modules
\[
(\beta,\gamma)_M : \wt{(\wt{M})} \to         \wt{\wt{M}}
.\]
On the stalk at $x\in \bC \Sigma^{\wt{\wt{F}}}$, $(\beta,\gamma)_M$ is given by
\[
{(\beta,\gamma)_M}|_{x}  : M(\tilde{\tilde{x}}) \otimes \mathcal{O}_{\wt{F}}(\tilde{x})  \otimes   \mathcal{O}_{\wt{\wt{F}}}(x)       
\to    
M(\tilde{\tilde{x}})  \otimes   \mathcal{O}_{\wt{\wt{F}}}(x) 
\]
\[
\text{g}\otimes \text{f}_1 \otimes \text{f}_2 
\mapsto 
\text{g} \otimes (\beta^\ast\text{f}_1 ) \text{f}_2
.\]
We now need to check the four coherence laws of cospecies.

\begin{prop}
Equipped with the $2$-cells $(\sigma,\beta)_M$, $(\sigma,\tau)_M$ and $(\beta,\gamma)_M$ as defined above, sheaves of modules on permutohedral space form a $\textsf{Cat}$-valued cospecies $\textsf{Mod}$.
\end{prop}
\begin{proof}
We have four coherence $3$-cells to check. Recall $(\sigma,\tau,\beta)$-coherence says
\[
{(\sigma \circ \tau, \beta)}_M  \circ  { \wt{         (\sigma,\tau)  }  }_{M}  =  {(\sigma, \tau)}_{\wt{M}} \circ  (\tau,\beta)'_{M}  \circ  {(\sigma, \beta)}_{M'}
.\]
On the stalk at $x\in \bC\Sigma^{\wt{F}''}$, the left-hand map is given by
\[
\text{g}\otimes \text{f}_1 \otimes \text{f}_2     \otimes  \text{f}_3 
\mapsto 
\text{g} \otimes (\sigma^\ast\text{f}_1 ) \text{f}_2 \otimes  \text{f}_3 
\mapsto
(g\otimes 1) \otimes    ( \beta^\ast\sigma^\ast\text{f}_1 ) (\beta^\ast \text{f}_2)            \text{f}_3 
\]
and the right-hand map is given by
\[
\text{g}\otimes \text{f}_1 \otimes \text{f}_2     \otimes  \text{f}_3 
\mapsto
(\text{g}\otimes \text{f}_1 \otimes 1 )  \otimes         (   \beta^\ast    \text{f}_2   )  \text{f}_3 
\mapsto (\text{g}\otimes 1) \otimes (\beta^\ast f_1) \cdot 1    \otimes         (   \beta^\ast    \text{f}_2   )  \text{f}_3 
\mapsto (\text{g}\otimes 1) \otimes  ( \sigma^\ast\beta^\ast \text{f}_1)   (   \beta^\ast    \text{f}_2   )  \text{f}_3
\] 
and $\sigma^\ast\beta^\ast \text{f}_1= \beta^\ast\sigma^\ast\text{f}_1$ by $(\sigma,\beta)$-functoriality of $\bC\Sigma$. Recall $(\beta,\gamma,\sigma)$-coherence says
\[
( \sigma, \gamma\circ \beta   )_M \circ   (\beta ,\gamma )_{M'} = (\beta ,\gamma )'_{M} \circ (\sigma,\beta)_{\wt{M}} \circ \wt{(\sigma,\gamma)}_{M}  
.\]
On the stalk at $x\in \bC\Sigma^{\wt{\wt{F'}}}$, the left-hand map is given by
\[
\text{g}\otimes \text{f}_1 \otimes \text{f}_2 \otimes \text{f}_3 
\mapsto 
\text{g} \otimes \text{f}_1 \otimes (\beta^\ast\text{f}_2 ) \text{f}_3
\mapsto
(g\otimes 1) \otimes (\beta^\ast  \gamma^\ast \text{f}_1 )  (\beta^\ast\text{f}_2 ) \text{f}_3
\]
and the right-hand map is given by
\[
\text{g}\otimes \text{f}_1 \otimes \text{f}_2 \otimes \text{f}_3  
\mapsto 
(g\otimes 1) \otimes  (\gamma^\ast\text{f}_1)\text{f}_2  \otimes \text{f}_3
\mapsto 
(g\otimes 1) \otimes 1 \otimes (\beta^\ast  \gamma^\ast\text{f}_1)  (\beta^\ast \text{f}_2 )   \text{f}_3
\mapsto 
(g\otimes 1)  \otimes (\beta^\ast  \gamma^\ast\text{f}_1)  (\beta^\ast \text{f}_2 )   \text{f}_3
.\]
Recall $(\sigma,\tau,\upsilon)$-coherence says
\[
( \sigma \circ \tau , \upsilon )_M \circ ( \sigma , \tau )_{M'} = (\sigma, \tau \circ \upsilon)_{M} \circ  (\tau , \upsilon)'_{M}
.\]
On the stalk at $x\in \bC\Sigma^{F'''}$, the left-hand map is given by
\[
g\otimes \text{f}_1 \otimes  \text{f}_2 \otimes \text{f}_3 
\mapsto
g\otimes \text{f}_1 \otimes  (\sigma^\ast\text{f}_2) \text{f}_3
\mapsto
g\otimes        (   \sigma^\ast \tau^\ast  \text{f}_1 )  (\sigma^\ast\text{f}_2) \text{f}_3
.\]
and the right-hand map is given by
\[
g\otimes \text{f}_1 \otimes  \text{f}_2 \otimes \text{f}_3 
\mapsto
g\otimes  (\tau^\ast\text{f}_1 )  \text{f}_2 \otimes \text{f}_3 
\mapsto
g\otimes  (\sigma^\ast\tau^\ast\text{f}_1 ) (\sigma^\ast \text{f}_2)  \text{f}_3 
.\]
Finally, recall $(\beta,\gamma,\delta)$-coherence says 
\[
( \gamma \circ \beta , \delta )_M \circ ( \beta , \gamma )_{\wt{M}} = (\beta, \delta \circ \gamma )_{M} \circ  \wt{(\gamma , \delta)}_{M}
.\]
On the stalk at $x\in \bC\Sigma^{\wt{\wt{\wt{F}}}}$, the left-hand map is given by 
\[
g\otimes \text{f}_1 \otimes  \text{f}_2 \otimes \text{f}_3 
\mapsto
g\otimes \text{f}_1 \otimes  (\beta^\ast\text{f}_2) \text{f}_3
\mapsto
g\otimes        (   \beta^\ast \gamma^\ast  \text{f}_1 )  (\beta^\ast\text{f}_2) \text{f}_3
.\]
and the right-hand map is given by
\[
g\otimes \text{f}_1 \otimes  \text{f}_2 \otimes \text{f}_3 
\mapsto
g\otimes  (\gamma^\ast\text{f}_1 )  \text{f}_2 \otimes \text{f}_3 
\mapsto
g\otimes  (\beta^\ast\gamma^\ast\text{f}_1 ) (\beta^\ast \text{f}_2)  \text{f}_3 
.\qedhere \]
\end{proof}

We now show that $\textsf{Mod}$ is weakly Joyal-theoretic by taking external tensor products of modules. Let $\text{Mod}$ denote the Joyalization of $\textsf{Mod}$. Given an $F$-tuple of modules \hbox{$M=(M_1,\dots, M_k)\in \text{Mod}[F]$}, we have the pullback modules $\pi^{\ast}_{S_i} M_{i}\in \textsf{Mod}_F$ given by
\[
\pi^{\ast}_{S_i} M_{i} =  \pi^{-1}_{S_i} M_{i} \otimes_{   \pi^{-1}_{S_i} \mathcal{O}_{S_i} } \mathcal{O}_{F} 
.\]  
Then, we have the morphism of cospecies the external tensor product of modules, which is given by
\[
\text{Mod} \to \textsf{Mod}
,\qquad
(M_{1} , \dots , M_{k}) \mapsto M_{1} \boxtimes \dots \boxtimes M_{k}:= \pi^{\ast}_{S_1}  M_{1}   \otimes  \dots \otimes  \pi^{\ast}_{S_k}  M_{k} 
.\]
This is not a strict morphism, and so we need to define $2$-cells $\sigma_F$ and $\beta_F$. In particular, for each bijection $\sigma: J\to I$ and $F$-tuple $M=(M_1,\dots, M_k)\in \text{Mod}[F]$, $F\in \Sigma[J]$, we require a map of $\mathcal{O}_{F'}$-modules
\[
\sigma_{( M_1 , \dots , M_k  )}  :    ( M_1 \boxtimes \dots \boxtimes M_k )' \to     M'_1 \boxtimes \dots \boxtimes M'_k 
.\]
On the stalk at $x\in \bC\Sigma^{F'}$, $\sigma_{( M_1 , \dots , M_k  )}$ is given by
\[
\bigotimes_{1\leq i \leq k }  \Big ( M_i \big (x'|_{S_i})  \otimes  \mathcal{O}_F(x')   \Big ) \otimes \mathcal{O}_{F'}(x)
\to 
\bigotimes_{1\leq i \leq k }  \bigg ( \Big (  M_i (x'|_{S_i}   ) \otimes \mathcal{O}_{S'_i}(x|_{S_i}) \Big )  \otimes \mathcal{O}_{F'}(x) \bigg)
\]
\[
\bigotimes_{1\leq i \leq k}  (\text{g}_i \otimes \text{f}_i )    \otimes  \text{f}
\ \mapsto\ 
\text{f} \cdot \bigotimes_{1\leq i \leq k}   \big (    (\text{g}_i \otimes 1 ) \otimes \sigma^\ast\text{f}_i  \big )
.\]
For each permutation $\beta\in \text{Sym}_k$ and $F$-tuple $M=(M_1,\dots, M_k)\in \text{Mod}[F]$, $l(F)=k$, we require a map of $\mathcal{O}_{\wt{F}}$-modules
\[
\beta_{( M_1 , \dots , M_k  )}  :    \wt{ M_1 \boxtimes \dots \boxtimes M_k}    \to     \wt{M}_1 \boxtimes \dots \boxtimes \wt{M}_k 
.\]
On the stalk at $x\in \bC\Sigma^{\wt{F}}$, $\beta_{( M_1 , \dots , M_k  )}$ is given by
\[
\bigotimes_{1\leq i \leq k }  \Big ( M_i \big (\tilde{x}|_{S_i})  \otimes  \mathcal{O}_F(\tilde{x})   \Big ) \otimes \mathcal{O}_{\wt{F}}(x)
\to 
\bigotimes_{1\leq i \leq k }  \bigg ( \Big (  M_i (\tilde{x}|_{S_i}   ) \otimes \mathcal{O}_{S_{\beta(i)}}(x|_{S_i}) \Big )  \otimes \mathcal{O}_{\wt{F}}(x) \bigg)
\]
\[
\bigotimes_{1\leq i \leq k}  (\text{g}_i \otimes \text{f}_i )    \otimes  \text{f}
\ \mapsto\ 
\text{f} \cdot \bigotimes_{1\leq i \leq k}   \big (    (\text{g}_i \otimes 1 ) \otimes \beta^\ast\text{f}_i  \big )
.\]

\begin{prop}\label{exten}
Equipped with the $2$-cells $\sigma_{( M_1 , \dots , M_k  )}$ and $\beta_{( M_1 , \dots , M_k  )}$ as defined above,	the external product of modules is a lax morphism of $\textsf{Cat}$-valued cospecies
	\[
	\text{Mod} \to \textsf{Mod}
.\]
\end{prop}
\begin{proof}
	There are three coherence $3$-cells to check. Recall $(\sigma,\beta)$-coherence says
	\[
	\underbrace{\eta_{\wt{F}'} \big ( (\sigma,\beta)_{( M_1 , \dots , M_k  )  } \big )}_{=\, \text{id}} \circ  \beta_{( M'_1 , \dots , M'_k  )}   \circ \wt{\sigma}_{( M_1 , \dots , M_k  )} 
	=
	{\sigma}_{(M_{\beta(1)} , \dots , M_{\beta(k)})}  \circ \beta'_{( M_1 , \dots , M_k  )}\circ  (\sigma,\beta)_{M_1 \boxtimes \dots \boxtimes M_k}
	.\]
	The left-hand map is given by
	\[
	\bigotimes_{1\leq i \leq k}  (\text{g}_i \otimes \text{f}_i )    \otimes  \text{f} \otimes \text{h}
	\ \mapsto\ 
	\text{f} \cdot \bigotimes_{1\leq i \leq k}   \big (    (\text{g}_i \otimes 1 ) \otimes \sigma^\ast\text{f}_i  \big ) \otimes \text{h}
	\ \mapsto \
	\text{h} (\beta^\ast\text{f}) \cdot \bigotimes_{1\leq i \leq k}   \big (    (\text{g}_i \otimes 1 \otimes 1 ) \otimes\beta^\ast  \sigma^\ast\text{f}_i  \big ) 
	\]
	and the right-hand map is given by
	\[
	\bigotimes_{1\leq i \leq k}  (\text{g}_i \otimes \text{f}_i )    \otimes  \text{f} \otimes \text{h}
	\ \mapsto \ 
	\bigotimes_{1\leq i \leq k}  (\text{g}_i \otimes \text{f}_i )  \otimes 1    \otimes  (\beta^\ast\text{f})  \text{h}
	\ \mapsto\ 
	\bigotimes_{1\leq i \leq k}  \big ((\text{g}_i \otimes 1  ) \otimes \beta^\ast \text{f}_i   \big )  \otimes  (\beta^\ast\text{f})  \text{h}
	\]
	\[
	\mapsto  \ 
	(\beta^\ast\text{f}) \text{h}  \cdot \bigotimes_{1\leq i \leq k}   \big (    (\text{g}_i \otimes 1 \otimes 1 ) \otimes  \sigma^\ast\beta^\ast\text{f}_i  \big ) 
	.\]
	We have $\beta^\ast\sigma^\ast\text{f}_i=\sigma^\ast\beta^\ast\text{f}_i$ by $(\sigma,\beta)$-functoriality of $\bC\Sigma$. Recall $\sigma$-coherence says 
	\[
	(\sigma \circ \tau)_{( M_1 , \dots , M_k  )} \circ (\sigma,\tau)_{M_1 \boxtimes \dots \boxtimes M_k  } =   \eta_{F''}\big ( (\sigma,\tau)_{( M_1 , \dots , M_k  )}\big)  \circ   \sigma_{( M'_1 , \dots , M'_k  )} \circ \tau'_{( M_1 , \dots , M_k  )} 
	.\]
	The left-hand map is given by
	\[
	\bigotimes_{1\leq i \leq k} ( \text{g}_i\otimes \text{f}_i)  \otimes \text{h}_1 \otimes \text{h}_2 
	\ \mapsto\  
	\bigotimes_{1\leq i \leq k} ( \text{g}_i\otimes \text{f}_i) \otimes (\sigma^\ast\text{h}_1 ) \text{h}_2
	\ \mapsto\  
	(\sigma^\ast\text{h}_1 ) \text{h}_2 \cdot \bigotimes_{1\leq i \leq k}\big ( ( \text{g}_i  \otimes 1 )\otimes \sigma^\ast \tau^\ast \text{f}_i  \big )
	\]
	and the right-hand map is given by
	\[
	\bigotimes_{1\leq i \leq k} ( \text{g}_i\otimes \text{f}_i)  \otimes \text{h}_1 \otimes \text{h}_2 
	\ \mapsto\  
	\text{h}_1  \cdot \bigotimes_{1\leq i \leq k} \big ( (\text{g}_i\otimes 1) \otimes \tau^\ast  \text{f}_i\big )  \otimes \text{h}_2 
	\]
	\[
	\mapsto \ 
	\text{h}_2(\sigma^\ast\text{h}_1 )  \cdot  \bigotimes_{1\leq i \leq k} \big ( (  (\text{g}_i\otimes 1) \otimes 1 ) \otimes \sigma^\ast\tau^\ast  \text{f}_i\big  )  
	\ \mapsto \ 
	\text{h}_2 (\sigma^\ast\text{h}_1 )\cdot \bigotimes_{1\leq i \leq k}\big ( ( \text{g}_i  \otimes 1 )\otimes \sigma^\ast \tau^\ast \text{f}_i  \big )
	.\]
	Then $\beta$-coherence follows similarly.
\end{proof}

\subsection{The Bimonoid of Sheaves of Modules}

We now give $\textsf{Mod}$ the structure of a bimonoid by pulling back modules with respect to the (co)multiplication of $\bC\Sigma$, that is the multiplication and comultiplication of $\textsf{Mod}$ are given by
\[
\Updelta^{\ast}_{F,G}  M=  \Updelta^{-1}_{F,G}M \otimes_{  \Updelta_{F,G}^{-1}\mathcal{O}_{F} }  \mathcal{O}_G
\qquad \text{and} \qquad 
\upmu^{\ast}_{F,G} M= \upmu^{-1}_{F,G}M \otimes_{  \upmu_{F,G}^{-1}\mathcal{O}_{G} }  \mathcal{O}_F
\]
respectively. This is not strict, so we need to define seven invertible $2$-cells
\[
\Updelta^\ast_{F,G,\sigma},\ \Updelta^\ast_{F,G,\beta},\   \Updelta^\ast_{F,G,E},\    \upmu^\ast_{F,G,\sigma},\ \upmu^\ast_{F,G,\beta},\     \upmu^\ast_{F,G,E},\   \text{B}_{F,G,A} 
.\]
As in the case of the $2$-cells for $\textsf{Mod}$ as a cospecies, they are just the canonical isomorphisms which express how pullback of sheaves of modules is functorial. We now make them explicit.

We need a natural isomorphism of $\mathcal{O}_{G'}$-modules
\[
{\Updelta^\ast_{F,G,\sigma}}_M   :   (\Updelta^\ast_{F,G} M)'  \to       \Updelta^\ast_{F',G'}  (M')
.\]
On the stalk at $x\in \bC\Sigma^{G'}$, ${\Updelta^\ast_{F,G,\sigma}}_M$ is given by
\[
        M\big (  \Updelta_{F,G}(x') \big )  \otimes \mathcal{O}_{G}( x' )    \otimes \mathcal{O}_{G'}(x) 
\to
M\big (  \Updelta_{F,G}(x') \big )\otimes  \mathcal{O}_{F'}\big (\Updelta_{F',G'}(x)\big )  \otimes \mathcal{O}_{G'}(x)  
\]
\[
\text{g} \otimes \text{f}_1 \otimes \text{f}_2 \mapsto \text{g} \otimes 1 \otimes    (\sigma^\ast \text{f}_1  ) \text{f}_2
.\]
We need a natural isomorphism of $\mathcal{O}_{\wt{G}}$-modules
\[
{\Updelta^\ast_{F,G,\beta}}_M   :    \wt{\Updelta^\ast_{F,G} M}  \to       \Updelta^\ast_{\wt{F},\wt{G}}  \wt{M}
.\]
On the stalk at $x\in \bC\Sigma^{\wt{G}}$, ${\Updelta^\ast_{F,G,\beta}}_M$ is given by
\[
          M\big (  \Delta_{F,G}(\tilde{x}) \big )  \otimes \mathcal{O}_{G}( \tilde{x} ) \otimes \mathcal{O}_{\wt{G}}(x) 
\to
M\big (  \Delta_{F,G}(\tilde{x}) \big )  \otimes    \mathcal{O}_{\wt{F}}\big (\Updelta_{\wt{F},\wt{G}}(x)\big ) \otimes \mathcal{O}_{\wt{G}}(x) 
\]
\[
\text{g} \otimes \text{f}_1 \otimes \text{f}_2 \mapsto \text{g} \otimes  1 \otimes    (\beta^\ast \text{f}_1  ) \text{f}_2
.\]
We need a natural isomorphism of $\mathcal{O}_E$-modules
\[
\Updelta^\ast_{F,G,E}  : \Updelta^\ast_{G,E}  (\Updelta^\ast_{F,G} M ) \to    \Updelta^\ast_{F,E} M
.\]
On the stalk at $x\in \bC\Sigma^E$, $\Updelta^\ast_{F,G,E}$ is given by
\[
  M \big (  \Updelta_{F,E}(x) \big )  \otimes  \mathcal{O}_G\big  ( \Delta_{G,E}(x) \big )  \otimes \mathcal{O}_E(x)
\to
M \big (  \Updelta_{F,E}(x) \big ) \otimes  \mathcal{O}_E(x)
\]
\[
\text{g} \otimes \text{f}_1 \otimes \text{f}_2 \mapsto \text{g} \otimes (\Updelta_{G,E}^\ast \text{f}_1 )\text{f}_2
.\]
We need a natural isomorphism of $\mathcal{O}_{F'}$-modules
\[
{\upmu^\ast_{F,G,\sigma}}_M   :    (\upmu^\ast_{F,G} M)'  \to       \upmu^\ast_{F',G'}  (M')
\]
On the stalk at $x\in \bC\Sigma^{F'}$, ${\upmu^\ast_{F,G,\sigma}}_M$ is given by
\[
      M\big (  \upmu_{F,G}(x') \big )  \otimes \mathcal{O}_{F}( x' )    \otimes \mathcal{O}_{F'}(x) 
\to
M\big (  \upmu_{F,G}(x') \big )\otimes  \mathcal{O}_{G'}\big (\upmu_{F',G'}(x)\big )  \otimes \mathcal{O}_{F'}(x)  
\]
\[
\text{g} \otimes \text{f}_1 \otimes \text{f}_2 \mapsto \text{g} \otimes 1 \otimes    (\sigma^\ast \text{f}_1  ) \text{f}_2
.\]
We need a natural isomorphism of $\mathcal{O}_{\wt{F}}$-modules
\[
{\upmu^\ast_{F,G,\beta}}_M   :    \wt{\upmu^\ast_{F,G} M}  \to       \upmu^\ast_{\wt{F},\wt{G}}  \wt{M}
.\]
On the stalk at $x\in \bC\Sigma^{\wt{F}}$, ${\upmu^\ast_{F,G,\beta}}_M$ is given by
\[
       M\big (  \Delta_{F,G}(\tilde{x}) \big )  \otimes \mathcal{O}_{F}( \tilde{x} ) \otimes \mathcal{O}_{\wt{F}}(x) 
\to
M\big (  \Delta_{F,G}(\tilde{x}) \big )  \otimes    \mathcal{O}_{\wt{G}}\big (\upmu_{\wt{F},\wt{G}}(x)\big ) \otimes \mathcal{O}_{\wt{F}}(x) 
\]
\[
\text{g} \otimes \text{f}_1 \otimes \text{f}_2 \mapsto \text{g} \otimes  1 \otimes    (\beta^\ast \text{f}_1  ) \text{f}_2
.\]
We need a natural isomorphism of $\mathcal{O}_E$-modules
\[
\upmu^\ast_{F,G,E}  : \upmu^\ast_{F,G}  (\upmu^\ast_{G,E} M ) \to    \upmu^\ast_{F,E} M
.\]
On the stalk at $x\in \bC\Sigma^F$, $\upmu^\ast_{F,G,E}$ is given by
\[
\upmu^\ast :   M \big (  \upmu_{F,E}(x) \big )  \otimes  \mathcal{O}_G\big  ( \Delta_{F,G}(x) \big )  \otimes \mathcal{O}_F(x)
\to
M \big (  \upmu_{F,E}(x) \big ) \otimes  \mathcal{O}_F(x)
\]
\[
\text{g} \otimes \text{f}_1 \otimes \text{f}_2 \mapsto \text{g} \otimes (\upmu_{F,G}^\ast \text{f}_1 )\text{f}_2
.\]
Finally, we need a natural isomorphism of $\mathcal{O}_G$-modules
\[
{\text{B}_{F,G,A}}_M   :    
\Updelta^\ast_{GF,G}  \beta^\ast  \upmu^\ast_{FG,F}  M   
\to
\upmu^\ast_{G,A}  \Updelta^\ast_{F,A} M
\]
On the stalk at $x\in \bC\Sigma^G$, ${\text{B}_{F,G,A}}_M$ is given by
\[
    M( \Updelta_{F,A} (\upmu_{G,A} (x)) ) \otimes  \mathcal{O}_{FG}( \beta(\Updelta_{GF,G}(x)))   \otimes   \mathcal{O}_{GF}(\Updelta_{GF,G}(x)) \otimes  \mathcal{O}_{G}(x)
\]
\[
\to
  M( \Updelta_{F,A} \upmu_{G,A} (x) )    \otimes \mathcal{O}_{A}(\upmu_{G,A}(x)) \otimes  \mathcal{O}_{G}(x)
\]
\[
\text{g} \otimes \text{f}_1 \otimes \text{f}_2  \otimes \text{f}_3
\mapsto
g \otimes 1 \otimes  ( \Updelta^\ast_{GF,G} \beta^\ast  \text{f}_1) ( \Updelta^\ast_{GF,G} \text{f}_2) \text{f}_3
.\]

\begin{thm}
The pullback of sheaves of modules along the (co)multiplication of permutohedral space, together with $2$-cells the canonical isomorphisms which express how pullback of sheaves of modules is functorial, gives $\textsf{Mod}$ the structure of a bimonoid in $\textsf{Cat}$-valued cospecies. 
\end{thm}
\begin{proof}
There are eleven coherence laws to check. Recall associativity $\sigma$-coherence says
\[
{\Updelta^\ast_{F',G',E'}}_{M'}  \circ   \Updelta^\ast_{G',E'} \big ({\Updelta^\ast_{F,G,\sigma}}_M \big )  \circ  {\Updelta^\ast_{G,E,\sigma}}_{\Updelta^\ast_{F,G}M}  
=
{\Updelta^\ast_{F,E,\sigma}}_{M} \circ ({\Updelta^\ast_{F,G,E}}_{M})'
.\]
The left-hand map is given by
\[
\text{g} \otimes \text{f}_1 \otimes \text{f}_2  \otimes \text{f}_3  
\mapsto   
\text{g}  \otimes \text{f}_1 \otimes 1 \otimes (\sigma^\ast\text{f}_2)  \text{f}_3
\mapsto 
\text{g}  \otimes 1 \otimes \sigma^\ast   \text{f}_1 \otimes (\sigma^\ast\text{f}_2)  \text{f}_3
\mapsto 
\text{g}  \otimes 1 \otimes   (\Updelta_{G',E'}^\ast  \sigma^\ast   \text{f}_1 ) (\sigma^\ast\text{f}_2)  \text{f}_3
\]
and the right-hand map is given by
\[
\text{g} \otimes \text{f}_1 \otimes \text{f}_2  \otimes \text{f}_3
\mapsto
\text{g} \otimes   (\Updelta_{G,E}^\ast \text{f}_1 ) \text{f}_2  \otimes      \text{f}_3
\mapsto
\text{g} \otimes 1 \otimes   (\sigma^\ast \Updelta_{G,E}^\ast \text{f}_1 )( \sigma^\ast \text{f}_2 )      \text{f}_3
\]
and $\Updelta_{G',E'}^\ast  \sigma^\ast   \text{f}_1= \sigma^\ast \Updelta_{G,E}^\ast \text{f}_1$ by comultiplication $\sigma$-naturality of permutohedral space $\bC\Sigma$. Then associativity $\beta$-coherence follows similarly by comultiplication $\beta$-naturality of $\bC\Sigma$. Also, coassociativity \hbox{$\sigma$-coherence} and coassociativity \hbox{$\beta$-coherence} follow similarly by multiplication \hbox{$\sigma$-naturality} and multiplication \hbox{$\beta$-naturality} of $\bC\Sigma$. Recall $\sigma$-functoriality multiplication coherence says
\[
\Updelta^\ast_{F'',G''}\big((\sigma,\tau)_M\big) \circ {\Updelta^\ast_{F',G',\sigma}}_{M'} \circ ({\Updelta^\ast_{F,G,\tau}}_{M})'    =    {\Updelta^\ast_{F,G,\sigma \circ \tau}}_M  \circ    (\sigma,\tau)_{\Updelta^\ast_{F,G} M}
.\]
The left-hand map is given by
\[
\text{g} \otimes \text{f}_1 \otimes \text{f}_2  \otimes \text{f}_3  
\mapsto 
\text{g} \otimes 1 \otimes  (\tau^\ast\text{f}_1 ) \text{f}_2                \otimes \text{f}_3  
\mapsto
\text{g} \otimes 1 \otimes 1 \otimes   (\sigma^\ast \tau^\ast\text{f}_1 ) (\sigma^\ast \text{f}_2)          \text{f}_3 
\mapsto
\text{g} \otimes 1 \otimes    (\sigma^\ast \tau^\ast\text{f}_1 ) (\sigma^\ast \text{f}_2)          \text{f}_3 
\]
and the right-hand map is given by
\[
\text{g} \otimes \text{f}_1 \otimes \text{f}_2  \otimes \text{f}_3
\mapsto
\text{g} \otimes \text{f}_1 \otimes   (\sigma^\ast \text{f}_2 ) \text{f}_3
\mapsto 
\text{g} \otimes 1 \otimes  (\sigma^\ast \tau^\ast \text{f}_1)    (\sigma^\ast \text{f}_2 ) \text{f}_3
.\]
Then $\beta$-functoriality multiplication coherence follows similarly. Also, $\sigma$-functoriality comultiplication coherence and $\beta$-functoriality comultiplication coherence follow similarly.

Recall associativity coherence says
\[
{\Updelta^\ast_{F,E,D}}_M \circ \Updelta^\ast_{E,D}( {{\Updelta^\ast_{F,G,E}}_M} )={\Updelta^\ast_{F,G,D}}_M \circ  {\Updelta^\ast_{G,E,D}}_{\Updelta^\ast_{F,G}M} 
.\]
The left-hand map is given by
\[
\text{g} \otimes \text{f}_1 \otimes  \text{f}_2 \otimes \text{f}_3 
\mapsto
\text{g} \otimes (\Updelta^\ast_{G,E}\text{f}_1 )  \text{f}_2 \otimes \text{f}_3 
\mapsto
\text{g} \otimes (\Updelta^\ast_{E,D} \Updelta^\ast_{G,E}\text{f}_1 ) (\Updelta^\ast_{E,D} \text{f}_2)  \text{f}_3 
\]
and the right-hand map is given by
\[
\text{g} \otimes \text{f}_1 \otimes  \text{f}_2 \otimes \text{f}_3 
\mapsto
\text{g} \otimes \text{f}_1 \otimes   (\Updelta^\ast_{E,D}\text{f}_2 ) \text{f}_3 
\mapsto
\text{g} \otimes   (\Updelta^\ast_{G,D} \text{f}_1 )   (\Updelta^\ast_{E,G}\text{f}_2 ) \text{f}_3 
\]
and
\[
\Updelta^\ast_{E,D} \Updelta^\ast_{G,E}\text{f}_1 =   \Updelta^\ast_{G,D} \text{f}_1 
\]
by coassociativity of $\bC\Sigma$. Then coassociativity coherence follows similarly by associativity of $\bC\Sigma$. Finally, recall bimonoid $\sigma$-coherence says
\[
\upmu_{G,A}({\Updelta^\ast_{F,A,\sigma }}_{M}) \circ  {\upmu_{G,A,\sigma}}_{\Updelta^\ast_{F,A} M} \circ {\text{B}'_{F,G,A}}_{M}
\]
\[
=
{\text{B}_{F',G',A'}}_{M'} \circ  \Delta^\ast_{GF,G} ( { \wt{\upmu_{FG,F,\sigma}}  }_{M} )\circ \Updelta^\ast_{GF,G} \big({(\sigma,\beta)}^{-1}_{\upmu^\ast_{FG,F} M}\big ) \circ  {\Updelta^\ast_{GF,G,\sigma}}_{\wt{\upmu^\ast_{FG,F} M}}
.\]
The left-hand map is given by
\[
\text{g} \otimes \text{f}_1 \otimes \text{f}_2  \otimes \text{f}_3  \otimes \text{f}_4
\mapsto 
g \otimes 1 \otimes    ( \Updelta^\ast_{GF,G} \beta^\ast  \text{f}_1) ( \Updelta^\ast_{GF,G} \text{f}_2) \text{f}_3 \otimes \text{f}_4
\]
\[
\mapsto
g \otimes 1  \otimes 1  \otimes  (\sigma^\ast \Updelta^\ast_{GF,G} \beta^\ast  \text{f}_1) (\sigma^\ast \Updelta^\ast_{GF,G} \text{f}_2)(\sigma^\ast \text{f}_3) \text{f}_4
\mapsto
g   \otimes  1 \otimes 1 \otimes (\sigma^\ast \Updelta^\ast_{GF,G} \beta^\ast  \text{f}_1) (\sigma^\ast \Updelta^\ast_{GF,G} \text{f}_2)(\sigma^\ast \text{f}_3) \text{f}_4
\]
and the right-hand map is given by
\[
\text{g} \otimes \text{f}_1 \otimes \text{f}_2  \otimes \text{f}_3  \otimes \text{f}_4
\mapsto 
\text{g} \otimes \text{f}_1 \otimes \text{f}_2      \otimes 1        \otimes (\sigma^\ast \text{f}_3) \text{f}_4
\mapsto 
\text{g} \otimes \text{f}_1 \otimes 1 \otimes \sigma^\ast \text{f}_2 \cdot 1 \otimes   (\sigma^\ast \text{f}_3) \text{f}_4
\]
\[
\mapsto
\text{g} \otimes 1 \otimes  \sigma^\ast\text{f}_1 \otimes  \sigma^\ast \text{f}_2 \otimes   (\sigma^\ast \text{f}_3) \text{f}_4
\mapsto
\text{g}\otimes 1 \otimes 1 \otimes (\Updelta^\ast_{G'F',G'} \beta^\ast \sigma^\ast\text{f}_1) (\Updelta^\ast_{G'F',G'}\sigma^\ast\text{f}_2)       (\sigma^\ast \text{f}_3) \text{f}_4
.\]
We have
\[
\sigma^\ast \Updelta^\ast_{GF,G} \beta^\ast  \text{f}_1 = \Updelta^\ast_{G'F',G'} \beta^\ast \sigma^\ast\text{f}_1
\qquad \text{and} \qquad
 \sigma^\ast \Updelta^\ast_{GF,G} \text{f}_2=\Updelta^\ast_{G'F',G'}\sigma^\ast\text{f}_2
\]
by multiplication $\sigma$-naturality and $(\sigma,\beta)$-functoriality of $\bC\Sigma$.
\end{proof}



\subsection{The {\normalfont $\textsf{Mod}$}-Graded Bialgebra of Global Sections} \label{Global Sections}

Given an $\mathcal{O}_F$-module $M$, a \emph{global section} $s\in M$ is a map of $\mathcal{O}_F$-modules of the form
\[
s:  \mathcal{O}_F \to M
.\]
We have the natural identification
\[
\Hom_{\textsf{Mod}_F}(  \mathcal{O}_F ,M) \xrightarrow{\sim}  \Gamma_{\bC\Sigma^F}(M)
,\qquad
s \ \mapsto\  \big ( x\mapsto s(1_{\mathcal{O}_F(x)}) \big  )
.\]
We have the $\textsf{Mod}$-graded vector cospecies $\bC\textbf{\textsf{Mod}}$ given by taking global sections, pulling them back along the corresponding maps of $\bC\Sigma$, and pushing them forward along maps of $\mathcal{O}_F$-modules,
\[
\bC\textbf{\textsf{Mod}}_M[F]:=  \Gamma_{\bC\Sigma^F}(M),
\]
\[
\bC\textbf{\textsf{Mod}}_M[\sigma, F](s) := \sigma^\ast s  = (\! \! \! \! \! \! \! \! \! \! \! \! \! \! \! \! \! \!  \! \! \! \! \! \! \! \! \!  \! \! \! \! \! \! \! \! \!   \! \! \! \!   \underbrace{\mathcal{O}_{F'}    \xrightarrow{\sim}  \mathcal{O}'_{F}}_{\text{`pullback of the structure sheaf is the structure sheaf'}}
\! \! \! \! \! \! \! \! \! \! \! \! \! \! \! \! \! \! \! \! \! \! \! \! \! \! \!   \! \! \! \! \! \! \! \! \!   \! \! \! \! 
\xrightarrow{s'}  M')
,\qquad
\bC\textbf{\textsf{Mod}}_M[F,\beta](s):= \beta^\ast s  = (\mathcal{O}_{\wt{F}}    \xrightarrow{\sim}  \wt{\mathcal{O}}_{F}  \xrightarrow{\tilde{s}}  \wt{M})
,\]
\[
\bC\textbf{\textsf{Mod}}_f[F,F](s):=f(s)= f\circ s
.\]
Here we have used the abbreviations $s'=\textsf{Mod}[\sigma,F](s)$ and $\tilde{s}=\textsf{Mod}[F,\beta](s)$. To avoid ambiguity, from now on let 
\[
\sigma^\ast (s) :=\textsf{Mod}[\sigma,F](s)
\qquad \text{and}
\qquad 
\sigma^\ast s:=\bC\textbf{\textsf{Mod}}_M[\sigma, F](s)
.\]
Similarly with the $\beta$'s. The verification that $\bC\textbf{\textsf{Mod}}$ is functorial, and so does indeed define a cospecies, is straightforward and follows along the lines of the proof of \autoref{prop:isbimonoid} below.

We can also pullback global sections with respect to the (co)multiplication of $\bC\Sigma$, giving maps of vector spaces
\[
{\Updelta^\ast_{F,G}}_M :\Gamma_{\bC\Sigma^F}(M) \to \Gamma_{\bC\Sigma^G}(\Updelta^\ast_{F,G}M)
,\qquad
s \mapsto  \Updelta^\ast_{F,G} s  =  (\mathcal{O}_G  \xrightarrow{\sim}  \Updelta^\ast_{F,G} \mathcal{O}_F   \xrightarrow{\Updelta^\ast_{F,G}(s)}  \Updelta^\ast_{F,G} M)
\]
and
\[
{\upmu^\ast_{F,G}}_M :\Gamma_{\bC\Sigma^G}(M) \to \Gamma_{\bC\Sigma^F}(\upmu^\ast_{F,G} M)
,\qquad
s \mapsto  \upmu^\ast_{F,G} s  =  (\mathcal{O}_F  \xrightarrow{\sim}  \upmu^\ast_{F,G} \mathcal{O}_G   \xrightarrow{\upmu^\ast_{F,G}(s)}  \upmu^\ast_{F,G} M)
.\]

\begin{prop}\label{prop:isbimonoid}
The pullback of global sections along the (co)multiplication of permutohedral $\bC\Sigma$ space equips $\bC\textbf{\textsf{Mod}}$ with the structure of a $\textsf{Mod}$-graded bialgebra in cospecies.  
\end{prop}
\begin{proof}
There are eleven diagrams to check. For multiplication \hbox{$\sigma$-naturality}, we need the map
\[
\Delta^\ast_{F,G,\sigma}\, \circ\,  \sigma^\ast \Delta^\ast_{F,G}  s  = \mathcal{O}_{G'}  \xrightarrow{\sim}    \mathcal{O}'_{G}  \xrightarrow{\sim}  (\Delta^\ast_{F,G} \mathcal{O}_F)' \xrightarrow{\sigma^\ast(\Delta^\ast_{F,G}(s))}  (\Delta^\ast_{F,G}  M)'   \xrightarrow{\Delta^\ast_{F,G,\sigma}} \Delta^\ast_{F',G'} M'
\]
to equal the map
\[
\Delta^\ast_{F',G'}  \sigma^\ast  s  =  \mathcal{O}_{G'}\xrightarrow{\sim}   \Delta^\ast_{F',G'} \mathcal{O}_{F'}  \xrightarrow{\sim}  \Delta^\ast_{F',G'}   \mathcal{O}'_F \xrightarrow{\Delta^\ast_{F',G'}(\sigma^\ast(s))} \Delta^\ast_{F',G'} M'
.\]
The first map is given on stalks by
\[
\text{f} \mapsto 1 \otimes \text{f} \mapsto 1 \otimes 1 \otimes \text{f} \mapsto   s(1) \otimes 1 \otimes \text{f}  \mapsto s(1) \otimes 1 \otimes \text{f} 
\]
and the second map by
\[
\text{f} \mapsto \text{f} \otimes 1 \mapsto   \text{f} \otimes 1 \otimes 1\mapsto  s(1) \otimes 1 \otimes \text{f} 
.\]
Comultiplication $\sigma$-naturality and (co)multiplication $\beta$-naturality follow similarly. For multiplication $\mathtt{f}$-naturality, we have
\[
\Delta^\ast_{F,G} (f) \big ( {\Delta^\ast_{F,G}}s \big ) =  \Delta^\ast_{F,G} (f)\circ  {\Delta^\ast_{F,G}}s 
=
\Delta^\ast_{F,G}  (f \circ s) \circ \Psi= \Delta^\ast_{F,G}  f(s)  
\]
where $\Psi$ denotes the canonical isomorphism $\Psi: \mathcal{O}_G \xrightarrow{\sim}  {\Delta^\ast_{F,G}} \mathcal{O}_F$. Comultiplication $\mathtt{f}$-naturality follows similarly. For associativity, we need the map
\[
\Delta^\ast_{F,E} s = \mathcal{O}_E  \xrightarrow{\sim}  \Delta^\ast_{F,E} \mathcal{O}_F \xrightarrow{\Delta^\ast_{F,E} (s)}\Delta^\ast_{F,E}  M
\]
to equal the map
\[
{\Delta^\ast_{F,G,E}}_M \circ\Delta^\ast_{G,E} \Delta^\ast_{F,G} s =\mathcal{O}_E \xrightarrow{\sim}  \Delta^\ast_{G,E} \mathcal{O}_G  \xrightarrow{\sim} \Delta^\ast_{G,E} \Delta^\ast_{F,G} \mathcal{O}_F \xrightarrow{\Delta^\ast_{G,E}( \Delta^\ast_{F,G} (s))}\Delta^\ast_{G,E} \Delta^\ast_{F,G}  M\xrightarrow{{\Delta^\ast_{F,G,E}}_M} \Delta^\ast_{F,E}  M
.\]
The first map is given on stalks by
\[
\text{f} \mapsto 1 \otimes \text{f} \mapsto  s(\text{1}) \otimes \text{f} 
\]
and the second map by
\[
\text{f} \mapsto 1 \otimes \text{f} \mapsto 1 \otimes \text{f}  \mapsto 1 \otimes 1 \otimes \text{f}  \mapsto  s(1) \otimes 1 \otimes \text{f}  \mapsto    s(1) \otimes  \text{f}
.\]
Coassociativity follows similarly. (Co)unitality is clear. For the bimonoid axiom, we need the map
\[
\mu^\ast_{G,A} \Delta^\ast_{F,A} s   =  \mathcal{O}_G \xrightarrow{\sim} \mu^\ast_{G,A} \mathcal{O}_A  \xrightarrow{\sim} \mu^\ast_{G,A}  \Delta^\ast_{F,A} \mathcal{O}_F \xrightarrow{\mu^\ast_{G,A}(\Delta^\ast_{F,A}(s))} \mu^\ast_{G,A}\Delta^\ast_{F,A} M
\]
to equal the map
\[
{\text{B}_{F,G,A}}_M \circ \Delta^\ast_{GF,G}\beta^\ast\upmu_{FG,F} s  
=
\mathcal{O}_G \xrightarrow{\sim} \Delta^\ast_{GF,G}\mathcal{O}_{GF}  \xrightarrow{\sim} \Delta^\ast_{GF,G}  \beta^\ast \mathcal{O}_{FG}  \xrightarrow{\sim} \Delta^\ast_{GF,G} \beta^\ast  \upmu^\ast_{FG,F}  \mathcal{O}_F 
\]
\[\xrightarrow{\Delta^\ast_{GF,G}  (\beta^\ast(\upmu^\ast_{FG,F}  (s)))} \Delta^\ast_{GF,G}  \beta^\ast \upmu^\ast_{FG,F}  M \xrightarrow{{\text{B}_{F,G,A}}_M}  \mu^\ast_{G,A}\Delta^\ast_{F,A} M
.\]
The first map is given on stalks by
\[
\text{f} \mapsto 1 \otimes \text{f} \mapsto  1 \otimes 1 \otimes \text{f} \mapsto   s(\text{1}) \otimes 1 \otimes \text{f} 
\]
and the second map by
\[
\text{f} \mapsto 1 \otimes \text{f} \mapsto  1 \otimes 1 \otimes \text{f} \mapsto 1 \otimes  1 \otimes 1 \otimes \text{f}\mapsto s(\text{1}) \otimes1 \otimes 1 \otimes \text{f}  \mapsto s(\text{1})  \otimes 1 \otimes \text{f} 
.\qedhere \]
\end{proof}

\section{Invertible Sheaves and Boolean Functions}\label{Invertible Sheaves and Boolean Functions}

We now study an indexing of sheaves of modules which has essential image the invertible sheaves. Invertible sheaves are often constructed as subsheaves of the sheaf of rational functions, and may be indexed by torus invariant Cartier divisors. However, we require a slightly different approach in order to expose the bimonoid structure of invertible sheaves.

\subsection{The Bimonoid of Boolean Functions}


Let $[ I,\text{2} ]$ denote the set of all subsets $A\subseteq I$. We have the Joyal-theoretic cospecies of \hbox{integer-valued} \emph{Boolean functions} $\textsf{BF}$, given by
\[
\textsf{BF}[I]:= \big \{  \text{functions}\ z: [ I,\text{2} ] \to \bZ  \ \big |\ z(\emptyset)=0 \big \}
,\qquad   \textsf{BF}[\sigma](z)\, (A) :=  z( A')
.\]
In particular, for each integer point $h=(h_\text{i})_{\text{i}\in I}\in \bZ I$ we have the Boolean function $z_h$ given by
\[
z_h(A) : =  \la h, \lambda_A \ra =  \sum_{\text{i}\in A} h_\text{i}
.\]
We let $\textsf{BF}[I]$ be a category by including a single morphism $z_1\to z_2$ whenever $z_2=z_1 + z_h$ for some $h\in \bZ I$. This is an equivalence relation.

We call $\text{hei}(z) := z(I)$ the \emph{height} of $z$. For $z=(z_1,\dots,z_k)\in \textsf{BF}[F]$, height is a $k$-tuple 
\[
\text{hei}(z_1,\dots,z_k):= \big(\text{hei}(z_1),\dots,\text{hei}(z_k)\big)
.\] 
We now give $\textsf{BF}$ the structure of a bimonoid, following \cite{aguiar2017hopf}. Given Boolean functions $z_1\in \textsf{BF}[S]$ and $z_2\in \textsf{BF}[T]$, let $(z_1|z_2)\in \textsf{BF}[I]$ be the Boolean function given by
\[
(z_1|z_2)(A):=z_1( A \cap S ) + z_2(  A \cap T)
.\]
Given a Boolean function $z\in \textsf{BF}[I]$, let $z\talloblong_S\in \textsf{BF}[S]$ and $z \! \!  \fatslash_{\, T}\in \textsf{BF}[T]$ be the Boolean functions given by
\[
z\! \talloblong_S(A):= z(A) 
\qquad \text{and} \qquad
z \! \! \!  \fatslash_{\, T}(A):=  z(A \sqcup S) -z(S)
.\]
Then $\textsf{BF}$ is a Joyal-theoretic bimonoid in $\textsf{Cat}$-valued cospecies, with multiplication and comultiplication given by
\[
\mu_{S,T}(z_1,z_2) :=(z_1|z_2) 
\qquad \text{and} \qquad 
\Delta_{S,T} (z) := (z\talloblong_S, z \! \!  \fatslash_{\, T})
\]
respectively. For the higher comultiplication, we denote e.g.
\[
(z\talloblong_S , z\talloblong_T , z\talloblong_U):= \Delta_{(S,T,U)} (z) 
.\]

\subsection{Invertible Sheaves From Cartier Divisors} \label{sec:cartdiv}

We also have the Joyal-theoretic cospecies $\bZ^{[ -;\text{2} ]}$, given by
\[
\bZ^{[ I;\text{2} ]}:= \big\{ \text{functions}\ w: [ I;\text{2} ] \to \bZ  \big\}
,\qquad
\sigma(w)\big ((S,T)\big ):= w\big( (S',T')  \big)
.\]
The functions $w\in \bZ^{[ I;\text{2} ]}$ are in one-to-one correspondence with torus invariant Cartier divisors $D_w$ on permutohedral space $\bC\Sigma^I$. Recall the invertible sheaf $\mathcal{O}_w\in \text{Mod}_I$ associated to $D_w$ is the submodule of the sheaf of rational functions given by
\[
\mathcal{O}_w(U) :=      \big \{g\in   \mathcal{K}_I      
\ \big |\    
g=0 \text{ or }    \text{div}g +D_w |_U \geq 0   \big \}
.\]
See e.g. \cite[Chapter 6]{toricvarieties}. The essential image of the indexing $w\mapsto \mathcal{O}_w$ is all the invertible sheaves. 

We have the inclusion of height zero Boolean functions
\[
\bZ^{[-;2]}\hookrightarrow \textsf{BF}
,\qquad
w\,  \mapsto\,   
\begin{cases}
	\big ( A\mapsto w(  A, B) \big  ), \quad \text{for $A\neq I,\emptyset$}     \\
	(I\mapsto 0) \\
	(\emptyset \mapsto 0)
\end{cases}	
\]
where $B=I-A$. The idea is to now generalize this construction of invertible sheaves to all of $\textsf{BF}$.

\subsection{Affine Semisimple Flats}

We give a geometric interpretation of the algebraic structure of $\textsf{BF}$, analogous to what we found for augmented preposets $\textsf{O}_\bullet$. First, we need an analog of the monoid $\text{T}^\vee$. 

For $a \in \bZ$, let $\text{T}^\vee_{I, a }$ denote the `sum equal to $a$' hyperplane of $\bR I$, given by
\[
\text{T}^\vee_{I, a } = 
\Big \{
h\in \bR I  \ \Big | \  \sum_{\text{i}\in I} h_\text{i} = a
\Big \}
.\]
For $F=(S_1,\dots, S_k)$ a composition of $I$ and a $k$-tuple of integers $(a_1 , \dots , a_k) \in \bZ^k$, we have the affine \emph{semisimple flat} $\textsf{T}^\vee_{F,(a_1 , \dots , a_k)}$, given by
\[
\textsf{T}^\vee_{F,(a_1 , \dots , a_k)} :=  \Big \{    h\in \bR I \ \Big | \   \sum_{\text{i}\in S_i} h_\text{i}   = a_i    \quad \text{for all} \quad 1\leq i \leq k    \Big \}
.\]
A key fact is that semisimple flats factorize; we have the bijective map $\eta_{F,(a_1 , \dots , a_k)}$ given by
\[
\eta_{F,(a_1 , \dots , a_k)}:
\text{T}^\vee_{S_1, a_1 }  \times   \dots \times \text{T}^\vee_{S_k, a_k } 
\xrightarrow{\sim}     \textsf{T}^\vee_{ F,(a_1 , \dots , a_k) }
,\qquad
\big ( (h_\text{i})_{\text{i}\in S_1} ,\dots ,(h_\text{i})_{\text{i} \in S_k} \big )  \mapsto  (h_\text{i})_{\text{i}\in I} 
.\]
For compositions $F, G$ with $G\leq F$, let
\[
F=(S_1,\dots,  S_k)\qquad \text{and} \qquad G=(T_1,\dots,  T_l)
\]
and let $1\leq \mathtt{i}_j \leq k$ be the integer such that $S_{\mathtt{i}_j}$ is the $\mathtt{i}$th lump of $F|_{T_j}$, $1\leq j \leq l$. Let $k(j)$ be the length of $F|_{T_j}$. Then we have the function ${\mu_{F,G}}_{(  a_1,\dots,a_k )}$ the inclusion of semisimple flats,
\[
{\mu_{F,G}}_{(  a_1,\dots,a_k )}   : \textsf{T}^\vee_{F,(  a_1,\dots,a_k )} \hookrightarrow \textsf{T}^\vee_{G,(  a_{1_1}+\,  \cdots\,  + a_{k(1)_{1}}  \,  , \dots \dots   ,  \,  a_{1_l}+\, \cdots \, + a_{k(l)_l} )}
,\qquad
h\mapsto h
.\]
We have the function the composition $\text{m}_{F,(a_1 , \dots , a_k)}:={\mu_{F,(I)}}_{(a_1 , \dots , a_k)} \circ {\eta_F}_{(a_1 , \dots , a_k)}$, which is given by
\[
\text{m}_{F,(a_1 , \dots , a_k)}:   
\text{T}^\vee_{S_1, a_1}  \times   \dots \times \text{T}^\vee_{S_k, a_k} 
\to  \text{T}^\vee_{I,a_1 + \dots + a_k}
,\qquad
\big ( (h_\text{i})_{\text{i}\in S_1} ,\dots ,(h_\text{i})_{\text{i} \in S_k} \big )  \mapsto  (h_\text{i})_{\text{i}\in I} 
.\]

\begin{remark}\label{rem:mon}
	In fact, $(F,a)\mapsto \textsf{T}^\vee_{F,a}$ defines a $\bZ$-graded set cospecies, where $\bZ$ denotes the constant cospecies $F\mapsto \bZ$. Moreover, this cospecies is a $\bZ$-graded monoid, with multiplication the inclusion of semisimple flats. Then $\text{m}_{F,(a_1 , \dots , a_k)}$ is the multiplication ${\mu_{F,(I)}}_{(a_1 , \dots , a_k)}$ of its Joyalization.
\end{remark}

\subsection{Plate Arrangements}

To each Boolean function $z\in \textsf{BF}[I]$ we associate the halfspace arrangement in $\text{T}_{I,\text{hei}(z)}^\vee$ which consists of the halfspaces 
\begin{align*}
	[[A,B]]_z  :=&\  \big \{     h\in \text{T}^\vee_{I, \text{hei}(z)}  \ \big | \       \la h, \lambda_{A}  \ra \leq z(A)  \big  \}  \\
	=&\  \big \{     h\in \text{T}^\vee_{I, \text{hei}(z)}  \ \big | \        \la h, \lambda_{B}  \ra \geq z(I)-z(A)   \big  \}
\end{align*}
where $(A,B)\in [I; 2]$. More generally, we have the associated plate arrangement, which consists of the \emph{plates}
\[
[[ H ]]_z:=  \bigcap_{(A,B) \leq H}  [[A,B]]_z   
\]
\[
=\ \big \{     h\in \text{T}^\vee_{I, \text{hei}(z)}       \ \big | \      \la h, \lambda_{\overline{A}_j}  \ra \leq z\big (\overline{A}_j\big ), \quad 0< j< \ell     \big \}
\]
\[ 
	=\ \big \{     h\in \text{T}^\vee_{I, \text{hei}(z)}       \ \big | \      \la h, \lambda_{\overline{A}_j}  \ra \leq z\big (\overline{A}_j\big ), \quad 0\leq j\leq \ell     \big \}
\]
where $H= (A_1,\dots, A_\ell) \in \Sigma[I]$ and 
\[
\overline{A}_j:=A_1\sqcup\dots \sqcup  A_j, \qquad  0 \leq j \leq \ell
.\]
Plates were introduced in \cite{oc17}, and are studied in \cite{early2017canonical}, \cite{ardilaplates}. A set of the form $\overline{A}_j$, $0\leq j \leq \ell$, is called an \emph{initial segment} of $H$. In particular, we have $\overline{A}_0=\emptyset$ and $\overline{A}_\ell=I$. Thus, as a region $[[ H ]]_z\subset \bR I$, the plate is given by the inequalities
\begin{equation}\label{eq:inequal}
  \la h, \lambda_{\overline{A}_j}  \ra \leq z\big (\overline{A}_j\big )
\end{equation}
for $0< j< \ell$, together with the equality for the ambient hyperplane
\begin{equation}\label{eq:equal}
\la h, \lambda_{I}  \ra =   z(I)
.
\end{equation}
If we let $j=0$ in \textcolor{blue}{(\refeqq{eq:inequal})} we obtain the trivial inequality
\[
 0= \la h, \lambda_{  \emptyset }  \ra \leq   z(\emptyset)=0
\]
and if we let $j=\ell$ in \textcolor{blue}{(\refeqq{eq:inequal})} we obtain  
\[
\la h, \lambda_{I}  \ra \leq   z(I)
\]
which is implied by \textcolor{blue}{(\refeqq{eq:equal})}. 

The plate $[[ H ]]_z$ is a module of the monoid $\sigma_H^o$, with the action given by addition of vectors in $\bR I$. If $z_2=z_1 + z_{h}$, we have the bijection
\[
[[H ]]_{z_1} \xrightarrow{\sim} [[H ]]_{z_2}
,\qquad
(-)\, \mapsto \,  (-)+h
.\]

\begin{prop}
The maximum affine subspace of the plate $[[H]]_z$ is the semisimple flat 
\[
\text{max aff} \big ( [[H]]_z \big ) = \textsf{T}^\vee_{H, \text{hei} (\Delta_H(z)  ) }
.\]
\end{prop}
\begin{proof}
The maximum affine subspace of the plate $	[[ H ]]_z$ is obtained by replacing inequalities with equalities in \textcolor{blue}{(\refeqq{eq:inequal})},
\[
\text{max aff} \big ( [[H]]_z \big ) =\big \{     h\in \text{T}^\vee_{I, \text{hei}(z)}       \ \big | \      \la h, \lambda_{\overline{A}_j}  \ra = z\big (\overline{A}_j\big ), \quad 0< j< \ell     \big \}.\]
Then notice this is just
\[
=  \textsf{T}^\vee_{H,(  z({A}_1),  z(\overline{{A}}_2- \overline{{A}}_1   ),      \dots,  z(\overline{{A}}_{\ell-1}- \overline{{A}}_{\ell-2}),  z(I- \overline{{A}}_{\ell-1}   )  )}
=  \textsf{T}^\vee_{H, \text{hei} (\Delta_H(z)  ) }
.\qedhere \]
\end{proof}

\begin{prop} \label{lem1}
For $F=(S_1,\dots, S_k)$ a composition of $I$, given an $F$-tuple of Boolean functions $z=(z_1,\dots, z_k) \in \textsf{BF}[F]$ and a composition $H\in \Sigma[I]$, the multiplication $\text{m}_{F,\text{hei}(z)}$ restricts to an embedding
\[
\text{m}_{F,\text{hei}(z)}:  
 [[    H|_{S_1}  ]]_{z_1}  \times \dots \times       [[    H|_{S_k}  ]]_{z_k}  \hookrightarrow   [[    H  ]]_{(z_1 | \dots | z_k)} 
\]
with image given by
\[
\bigcap_{(A,B)\leq (   H|_{S_1} | \dots |    H|_{S_k}   )}   [[    A,B  ]]_{(z_1 | \dots | z_k)}
.\]
\end{prop}
\begin{proof}

Let $H=(A_1,\dots, A_\ell)$, and let
\[
A_{j,i} =  S_i\cap A_j
,\qquad
\overline {A}_{j,i} =  A_{1,i}     \sqcup \dots \sqcup   A_{j,i}
 .\]
We have
\[
H|_{S_i}=     (   A_{1,i}, \dots , A_{\ell,i} )_+
.\]
The plate $[[    H|_{S_i}  ]]_{z_i}  \subset \bR S_i$ is the region satisfying the inequalities
\begin{equation}\label{eq:1}
\la h, \lambda_{\overline{A}_{j,i}}  \ra \leq z_i(\overline{A}_{j,i}) 
\end{equation}
for $0< j <  \ell$, together with the equality for the ambient space
\begin{equation}\label{eq:2}
\la h, \lambda_{ S_i }  \ra = z_i(S_i) 
.\end{equation}
The product $ [[    H|_{S_1}  ]]_{z_1}  \times \dots \times       [[    H|_{S_k}  ]]_{z_k}$ is the region satisfying these equalities for all $1\leq i \leq k$. On the other hand, the plate $\bigcap_{(A,B)\leq (   H|_{S_1} | \dots |    H|_{S_k}   )}   [[    A,B  ]]_{(z_1 | \dots | z_k)} \subset \bR I$ is the region satisfying the inequalities
\begin{equation}\label{eq:3}
 \la h, \lambda_{ A }  \ra  \leq  (z_1| \dots | z_k) (A ) =     z_1 ( A\cap S_1 ) + \dots +z_k ( A\cap S_k ) 
\end{equation}
where $(A,B)\leq ( H|_{S_1} | \dots |    H|_{S_k}   )$, together with the equality for the ambient space 
\begin{equation}\label{eq:4}
\la   h, \lambda_I \ra = (z_1| \dots | z_k)(I)= z_1(S_1) +\dots  + z_k(S_k)
.\end{equation}
Let $\mathcal{A}_i:= A\cap S_i$, so that $A=\mathcal{A}_1 \sqcup \dots \sqcup \mathcal{A}_k$ where $\mathcal{A}_i$ is an initial segment of $H|_{S_i}$. Then \textcolor{blue}{(\refeqq{eq:1})} \& \textcolor{blue}{(\refeqq{eq:2})}$\implies$\textcolor{blue}{(\refeqq{eq:3})} since
\[
\la h, \lambda_{A}  \ra =  \la h, \lambda_{\mathcal{A}_1}\ra + \dots + \la h, \lambda_{\mathcal{A}_k} \ra    \leq z_1(\mathcal{A}_1) +\dots + z_k(\mathcal{A}_k)  
.\]
Clearly \textcolor{blue}{(\refeqq{eq:2})}$\implies$\textcolor{blue}{(\refeqq{eq:4})}. We have \textcolor{blue}{(\refeqq{eq:3})}$\implies$\textcolor{blue}{(\refeqq{eq:1})} by setting $A=\overline{A}_{j,i}$ in \textcolor{blue}{(\refeqq{eq:3})}. We prove \textcolor{blue}{(\refeqq{eq:3})} \& \textcolor{blue}{(\refeqq{eq:4})} $\implies$\textcolor{blue}{(\refeqq{eq:2})} in two parts. First, set $A=S_i$ in \textcolor{blue}{(\refeqq{eq:3})} to give
\[
\la h, \lambda_{ S_i }  \ra \leq z_i(S_i) 
.\]
Second, set $A=I-S_i$ in \textcolor{blue}{(\refeqq{eq:3})} to give\footnote{\ where the hat $\hat{z}_i(S_i)$ means the term $z_i(S_i)$ is \emph{not} included in the sum}
\begin{equation}\label{eq:5}
\la h, \lambda_{ I-S_i }  \ra \leq z_1(S_1) + \dots +  \hat{z}_i(S_i)  +   \dots + z_k(S_k) 
.\end{equation}
Then
\[
  \la h, \lambda_{ S_i }  \ra  = 
  \lefteqn{\overbrace{\phantom{          \la h, \lambda_{ I}  \ra  - \la h, \lambda_{ I-S_i }  \ra    =  z_1(S_1) +   \dots + z_k(S_k)- \la h, \lambda_{ I-S_i } \ra                    }}^{\text{by \textcolor{blue}{(\refeqq{eq:4})}}}} 
  \la h, \lambda_{ I}  \ra  - \la h, \lambda_{ I-S_i }  \ra    =       \underbrace{     z_1(S_1) +   \dots + z_k(S_k)- \la h, \lambda_{ I-S_i } \ra  \geq   z_i(S_i)            }_{\text{by \textcolor{blue}{(\refeqq{eq:5})}}}
.\]
Thus, we have shown $\la h, \lambda_{ S_i }  \ra \leq z_i(S_i) $ and $\la h, \lambda_{ S_i }  \ra \geq z_i(S_i) $, and so $\la h, \lambda_{ S_i }  \ra = z_i(S_i) $. This completes the proof of the image. 

Finally, we need to show the image is contained in $  [[    H  ]]_{(z_1 | \dots | z_k)} $. We have
\[
H \leq (   H|_{S_1} | \dots |    H|_{S_k}   )
\]
and so 
\[
(A,B)\leq H \quad \implies \quad (A,B) \leq (   H|_{S_1} | \dots |    H|_{S_k}   ) 
.\]	
Therefore
\[
\bigcap_{(A,B)\leq (   H|_{S_1} | \dots |    H|_{S_k}   )}   [[    A,B  ]]_{(z_1 | \dots | z_k)} \ \subseteq \   [[    H  ]]_{(z_1 | \dots | z_k)} 
.\qedhere \]
\end{proof}

Given $(S,T)\in [I;2]$, the $(S,T)$-\emph{face} of the plate $[[H]]_z$ is the region of $[[H]]_z$ on which the function $\lambda_{S}$ is maximized. Given $F\in \Sigma[I]$, the $F$-\emph{face} of $[[H]]_z$ is the intersection of all the $(S,T)$-faces for $(S,T)\leq F$, i.e. it is the region on which $\lambda_{S}$ is maximized for all $(S,T)\leq F$.

\begin{prop}\label{prop:Fface}
If $F\leq H$, then the $F$-face of the plate $[[ H ]]_z  $ is given by the intersection
\[  	
[[H]]_z   \cap      \textsf{T}^\vee_{F, \text{hei}(\Delta_F(z))  }  
.\]
If $F\nleq H$, then $[[ H ]]_z  $ does not have an $F$-face.
\end{prop}
\begin{proof}
Let
\[
H= (A_1,\dots, A_\ell)
.\]
First consider the case $F=(S,T)$. If $(S,T)\leq H$, then $\lambda_{S}\leq z(S)$ on $[[ H ]]_z  $. But \hbox{$\lambda_{S}=z(S)$} on $\textsf{T}^\vee_{(S,T), \text{hei}(\Delta_{S,T}(z))  } $ by definition, and $[[ H ]]_z      \cap  \textsf{T}^\vee_{(S,T), \text{hei}(\Delta_{S,T}(z))  } $ is nonempty because $ \textsf{T}^\vee_{(S,T), \text{hei}(\Delta_{S,T}(z))  } $ contains the maximal affine subspace of $[[ H ]]_z $.

For generic $F$, the result then follows since
\[
  \textsf{T}^\vee_{F, \text{hei}(\Delta_F(z))  } = \text{max aff} \big ([[F]]_z\big )
  =\text{max aff} \big (   \bigcap_{(S,T)\leq F} [[S,T]]_z             \big )
  \]
  \[
  =    \bigcap_{(S,T)\leq F}   \text{max aff} \big ( [[S,T]]_z \big )            
  =    \bigcap_{(S,T)\leq F}   \textsf{T}^\vee_{(S,T), \text{hei}(\Delta_{S,T}(z))  }
.\]
If $(S,T)\nleq H$, then $(t,s)\in H$ for some $s\in S$ and $t\in T$. Therefore $h_{s t}\in \sigma^o_H$, and so $\lambda_{S}$ obtains arbitrary large values on the $\sigma^o_H$-module $[[H]]_z$. If $F\nleq H$, then there exists $(S,T)\leq F$ with $(S,T)\nleq H$.
\end{proof}

\begin{prop}\label{secondprop}
For $F=(S_1,\dots, S_k)$ a composition of $I$, given an $F$-tuple of compositions $K=(K_1,\dots, K_k)\in \mathsf{\Sigma}[F]$ and a Boolean function $z\in \textsf{BF}[I]$, the function $\text{m}_{F,\text{hei}(\Delta_F(z))}$ restricts to an embedding of cones
\[
[[ K_1    ]]_{z\talloblong_{S_1}} 
 \times \dots \times
  [[ K_k   ]]_{z\talloblong_{S_k}  }  \hookrightarrow   [[ K_1 ; \dots ; K_k ]]_z  
\]
with image the $F$-face of $[[ K_1 ; \dots ; K_k ]]_z$.
\end{prop}
\begin{proof}
In light of \autoref{prop:Fface}, this amount to showing that the image is given by the intersection
\[  	
[[ K_1 ; \cdots ; K_k]]_z   \cap      \textsf{T}^\vee_{F, \text{hei}(\Delta_F(z))  }  
.\] 
Let 
\[
 K_{i}   = (  B_{1,i},\dots, B_{l(i),i} )  
 \qquad \text{and} \qquad 
 \overline{B}_{j,i} :=   B_{1,i} \sqcup \dots \sqcup   B_{j,i} 
\]
for $1\leq j \leq l(i)$. Then the intersection $[[ K_1 ; \cdots ; K_k ]]_z\cap \textsf{T}^\vee_{F,\text{hei}(\Delta_F(z))} \subset \bR I$ is the region given by the plate inequalities
\begin{equation}\label{eq:11}
\la h, \lambda_{  \overline{S}_{i-1}\sqcup   \overline{B}_{j,i}       }  \ra \leq z(\overline{S}_{i-1}\sqcup \overline{B}_{j,i}), \qquad 0\leq j < l(i)
\end{equation}
for $1\leq i \leq k$ (although the equality for $i=1$, $j=0$ is trivial), together with the equalities
\begin{equation}\label{eq:33}
	\la h, \lambda_{S_i}  \ra =  \text{hei}(z\talloblong_{S_i})  =  z(\overline{S}_i) - z(\overline{S}_{i-1})    \qquad \text{for} \quad 1\leq i \leq k	 
	.
\end{equation}
On the other hand, $[[ K_i    ]]_{z\talloblong_{S_i}} \subset \bR S_i$ is the region given by the inequalities
\begin{equation}\label{eq:55}
\la h, \lambda_{\overline{B}_{j,i} }  \ra \leq   z\! \talloblong_{S_i}(\overline{B}_{j,i})  =   z(\overline{S}_{i-1} \sqcup    \overline{B}_{j,i}  )  -  z(\overline{S}_{i-1})
\end{equation}
for $0< j < l(i)$, together with the equality for the ambient space
\begin{equation}\label{eq:66}
\la h, \lambda_{ S_i      }  \ra   = z\talloblong_{S_i} (S_i)  = z(\overline{S}_i) -   z(\overline{S}_{i-1}) 
.
\end{equation}
The product $[[ K_1    ]]_{z\talloblong_{S_1}} \times \dots \times[[ K_k   ]]_{z\talloblong_{S_k}  }$ is the region satisfying these equalities for all $1\leq i \leq k$.

Clearly $\textcolor{blue}{(\refeqq{eq:33})}$ $\iff$ $\textcolor{blue}{(\refeqq{eq:66})}$. For $\textcolor{blue}{(\refeqq{eq:11})}$ \& $\textcolor{blue}{(\refeqq{eq:33})}$ $\implies$ $\textcolor{blue}{(\refeqq{eq:55})}$, we have
\[
z(\overline{S}_{i-1})+   
\la h, \lambda_{\overline{B}_{j,i}   }  \ra
=
\la h, \lambda_{\overline{S}_{i-1}}  \ra +   
\la h, \lambda_{\overline{B}_{j,i}   }  \ra
=\la h, \lambda_{\overline{S}_{i-1} \sqcup    \overline{B}_{j,i}   }  \ra \leq z(\overline{S}_{i-1} \sqcup    \overline{B}_{j,i}  ) 
\]
and so
\[
	\la h, \lambda_{\overline{B}_{j,i}   }  \ra       \leq    z(\overline{S}_{i-1} \sqcup    \overline{B}_{j,i}  )  -  z(\overline{S}_{i-1})
	.
\]
For $\textcolor{blue}{(\refeqq{eq:55})}$ \& $\textcolor{blue}{(\refeqq{eq:66})}$ $\implies$ $\textcolor{blue}{(\refeqq{eq:11})}$, we have
\[
\la h, \lambda_{  \overline{S}_{i-1}\sqcup   \overline{B}_{j,i}       }  \ra 
=
\la h, \lambda_{  S_{1}       }  \ra     +       \dots +   \la h, \lambda_{  S_{i-1}       }  \ra      +\la h, \lambda_{   \overline{B}_{j,i}       }  \ra
\]
\[ \leq          z(\overline{S}_1) + \big (z(\overline{S}_2) -  z(\overline{S}_1)  \big ) + \cdots\cdots  +     \big (z(\overline{S}_{i-1}) -  z(\overline{S}_{i-2})  \big )     +       z(\overline{S}_{i-1} \sqcup    \overline{B}_{j,i}  )  -  z(\overline{S}_{i-1})
\]
\[
=
z(\overline{S}_{i-1}\sqcup \overline{B}_{j,i})
.\qedhere \]
\end{proof}

\subsection{Invertible Sheaves From Boolean Functions}\label{A Homomorphism of Bimonoid1}



We now define a morphism of cospecies
\[
\mathcal{O}: \textsf{BF} \to \text{Mod}
,\qquad 
z\mapsto \mathcal{O}_z
\]
whose composition with the external product morphism $\text{Mod}\to \textsf{Mod}$ is a homomorphism of bimonoids $\varphi:\textsf{BF}\to  \textsf{Mod}$. The morphism $\mathcal{O}$ will be Joyal-theoretic, but not strict.

Given a composition $H\in \Sigma[I]$ and a Boolean function $z\in \textsf{BF}[I]$, we have the $\bC[\text{M}_H]$-module $\bC[\text{M}_{H,z}]$ given by
\[
\text{M}_{H,z} := [[H]]_z \cap \bZ I \qquad \text{and} \qquad  \bC[\text{M}_{H,z}]:=  \bigoplus_{h\in  \text{M}_{H,z}} \bC \cdot g^h
.\]
We define the invertible sheaf $\mathcal{O}_z$ of $z$ to be the coherent sheaf which, restricted to the affine open $U_H\subset \bC\Sigma^I$, is the coherent sheaf of the $\bC[\text{M}_H]$-module $\bC[\text{M}_{H,z}]$. We now make this explicit.

We make the identification
\[
\mathcal{K}_I=  \text{Frac} \big(   \bC[\text{T}^\vee_I \cap \bZ I]   \big  )   
\]
where recall $\mathcal{K}_I$ denotes the field of rational functions on permutohedral space $\bC\Sigma^I$. For each height $a\in \bZ$, we have the $\mathcal{K}_I$-module $\mathcal{K}_{I,a}$ given by
\[
\mathcal{K}_{I,a} :=  \text{Frac} \big(   \bC[\text{T}^\vee_{I,a} \cap \bZ I]      \big  )=  \bC[\text{T}^\vee_{I,a} \cap \bZ I]    \otimes_{ \bC[\text{T}^\vee_I \cap \bZ I]  }   \mathcal{K}_I
.\]

\begin{remark}
We can think of $\mathcal{K}_{I,a}$ as rational functions on permutohedral space which have an ambient degree of divergence equal to $a$. Importantly for the bimonoid structure, this extended notion of rational functions will be closed under pullback along the multiplication $\upmu_{F,G}$. 
\end{remark}

For $g\in \bC[\text{M}_{H,z}]\subseteq  \bC[\text{T}^\vee_{I,\text{hei}(z)} \cap \bZ I] $ and $f\in  \bC[\text{M}_H]$, we abbreviate
\[
g/f := g\otimes 1/f \in \mathcal{K}_{I,\text{hei}(z)}
.\]
For the principal open set $U_f\subseteq U_H$, put
\begin{equation}\label{eq:def}
\mathcal{O}_z (U_f) :=   \big \{ g/f^p  \ \big | \ g\in     \bC[\text{M}_{H,z}], \ p\in \bN   \big \} \subseteq  \mathcal{K}_{I,\text{hei}(z)}
.
\end{equation}
Assuming this is well-defined, this defines the sheaf $\mathcal{O}_z$ on a basis of $\bC\Sigma^I$, and so defines the sheaf entirely. The stalk of $\mathcal{O}_z$ at $x\in U_H$ is given by
\begin{equation}\label{eq:stalks}
\mathcal{O}_z(x) = \colim_{x\in U} \mathcal{O}_z (U) =     \big \{         g/f\in    \mathcal{K}_{I,\text{hei}(z)} \ \big | \  g\in     \bC[\text{M}_{H,z}], \  f(x)\neq 0     \big \}
.
\end{equation}
We need to check that \textcolor{blue}{(\refeqq{eq:def})} is well-defined, and that \textcolor{blue}{(\refeqq{eq:stalks})} really are the stalks. 

We show this by constructing isomorphisms with the subsheaves $\mathcal{O}_w$ of the sheaf of rational functions as follows. Suppose we have principal opens
\[
U_{f_1}\subseteq U_{H_1}  
\qquad \text{and} \qquad
U_{f_2}\subseteq U_{H_2}  
\qquad \text{with} \qquad
U_{f_1} = U_{f_2} =U
.\]
Pick some $d \in \text{T}^\vee_{I,-\text{hei}(z)}\cap \bZ I$. Then we have an isomorphism of vector spaces $\Phi_{\text{hei}(z),d}$ given by
\[
\Phi_{\text{hei}(z),d} : \mathcal{K}_{I,\text{hei}(z)} \xrightarrow{\sim}  \mathcal{K}_{I,\text{hei}(z)-\text{hei}(z)} =  \mathcal{K}_{I,0}=  \mathcal{K}_I 
,\qquad
g^h\otimes 1/f\  \mapsto \     g^{h+ d}\otimes 1/f
.\]
Let $w=z+z_d$, which is a height zero Boolean function with associated sheaf $\mathcal{O}_w$ described in \autoref{sec:cartdiv}. Given $f\in \bC[\text{M}_{H}]$, for the principal open $U_f\subseteq U_H$ we have
\[
\mathcal{O}_w (U_f) =   \big \{ g/f^p \in \mathcal{K}_{I}   \ \big | \ g\in     \bC[\text{M}_{H,w}] , \ p\in \bN   \big \} 
.\]
See e.g. toric section 4.3. Thus
\[
\Phi_{\text{hei}(z),d} (\mathcal{O}_z (U_{f_1})  )  =   \mathcal{O}_w (U_{f_1}) =   \mathcal{O}_w (U)  =\mathcal{O}_w (U_{f_2}) =  \Phi_{\text{hei}(z),d} (\mathcal{O}_z (U_{f_2})  ) 
.\]
Since $\Phi_{\text{hei}(z),d}$ is an isomorphism, we must have
\[
\mathcal{O}_z (U_{f_1}) =  \mathcal{O}_z (U_{f_2})
\]
as required. More generally, for any $d\in \bZ I$, we have an isomorphism of invertible sheaves
\[
\mathcal{O}_{z\to z+z_{d}} :  \mathcal{O}_z \xrightarrow{\sim} \mathcal{O}_{z+z_d}  
,\qquad 
g^h\otimes 1/f \  \mapsto  \    g^{h+ d}\otimes 1/f
.\]

Finally, we need to define a $2$-cell $\sigma_{I}$ and then check a single coherence law. In particular, we require a natural map of invertible sheaves $\sigma_{z}$ of the form
\[
\sigma_{z} :  \mathcal{O}'_z \to \mathcal{O}_{z'}
\]
for each Boolean function $z\in \textsf{BF}[J]$ and bijection $\sigma:J\to I$. Let $H\in \Sigma[J]$. Over $U_{H'}\subset \bC\Sigma^I$, the pullback $\mathcal{O}'_{z}$ is (naturally isomorphic to) the coherent sheaf of the $\bC[ \text{M}_{ H' } ]$-module 
\[
\bC[ \text{M}_{ H,z } ] \otimes     \bC[ \text{M}_{ H' } ]   :=\bC[ \text{M}_{ H,z } ] \otimes_{\bC[ \text{M}_{ H } ]}     \bC[ \text{M}_{ H' } ]
.\] 
Over $U_{H'}$, $\mathcal{O}_{z'}$ is by definition the coherent sheaf of the $\bC[ \text{M}_{ H' } ]$-module $ \bC[ \text{M}_{ H',z' } ]$. Over $U_{H'}$, we define $\sigma_{z}$ to be the isomorphism of coherent sheaves corresponding to the isomorphism of \hbox{$\bC[ \text{M}_{ H' } ]$-modules}
\[
\bC[ \text{M}_{ H,z } ] \otimes     \bC[ \text{M}_{ H' } ]  \to \bC[ \text{M}_{ H',z' } ] 
,\qquad 
g\otimes f \mapsto f \cdot  g'
.\]
Every element of the left-hand side is of the form $g\otimes 1$, and so this map is also given by $g\otimes 1 \mapsto g'$. We now make $\sigma_{z}$ explicit at the level of stalks. The stalk of the pullback $\mathcal{O}'_z$ at $x\in U_{H'}$ is given by
\[
\mathcal{O}'_z(x) = \mathcal{O}_z(x') \otimes  \mathcal{O}_I(x) \]
\[
=     \big \{         g/f\in    \mathcal{K}_{J,\text{hei}(z)} \, \big | \,  g\in     \bC[\text{M}_{H,z}],\,    f\in     \bC[\text{M}_{H}],  \,  f(x')\neq 0     \big \} 
\otimes
\big \{         f_1/f_2\in    \mathcal{K}_{I} \, \big | \,  f_1,f_2\in     \bC[\text{M}_{H'}], \,  f_1(x)\neq 0     \big \} 
.\]
The stalk of $\mathcal{O}_{z'}$ at $x\in U_{H'}$ is given by
\[
\mathcal{O}_{z'}(x) =   \big \{         g/f\in    \mathcal{K}_{I,\text{hei}(z')} \, \big | \,  g\in     \bC[\text{M}_{H',z'}], \,   f\in     \bC[\text{M}_{H'}],     \,  f(x)\neq 0     \big \}
.\]
Then the map $\sigma_{z}$ on the stalk at $x\in U_{H'}$ is given by
\[
g/f \otimes  f_1 / f_2   \,  \mapsto\,  (f_1 \cdot g')/(f_2\cdot f') 
.\]
This gives a globally well-defined $\sigma_{I}$ as we vary the affine chart $U_{H'}$ by $(\sigma,\mathtt{f})$-functoriality of $\bC\textbf{\textsf{O}}$. We write germs as e.g. $\text{g}=g/f$ and $\text{g}'=g'/f'$. Every element of $  \mathcal{O}'_z(x)$ is of the form $\text{g}\otimes 1$. Then this map may be written as 
\[
\text{g}\otimes 1 \mapsto \text{g}'
.\] 


\begin{prop}
We have a Joyal-theoretic morphism of cospecies
\[
\mathcal{O} : \textsf{BF} \to \text{Mod}
.\]
This is a strong morphism.
\end{prop}
\begin{proof}
Recall $\sigma$-coherence (for cospecies) says
\[
(\sigma \circ \tau)_z  \circ  (\sigma,\tau)_{\mathcal{O}_z} =\underbrace{   \eta_{F''}\big ( (\sigma,\tau)_{z}\big)}_{=\, \text{id}}  \circ \,  \sigma_{z'} \circ \tau'_{z}
.\]
At each stalk, the left-hand map is given by
\[
\text{g} \otimes 1 \otimes 1 
\mapsto 
\text{g} \otimes 1 
\mapsto 
\text{g}''
\]
and right-hand map is given by
\[
\text{g} \otimes 1 \otimes 1 \mapsto \text{g}' \otimes 1  \mapsto (\text{g}')'= \text{g}''
.\qedhere \]
\end{proof}


\begin{remark} 
	The restriction of the morphism $\textsf{BF} \to \text{Mod}$ to height zero Boolean functions $\bZ^{[-;2]}\hookrightarrow \textsf{BF}$ recovers the usual construction of invertible sheaves on toric varieties as subsheaves of the sheaf of rational functions. Thus, the essential image of $z\mapsto \mathcal{O}_z$ is all invertible sheaves. 
\end{remark}

\subsection{A Homomorphism of Bimonoids}\label{A Homomorphism of Bimonoids}

We have the composition of morphisms $\varphi:= \text{prod} \circ \mathcal{O}$, 
\[
\varphi: \textsf{BF} \to \textsf{Mod}
,\qquad 
( z_1,\dots, z_k )\mapsto \mathcal{O}_{z_1} \boxtimes \dots \boxtimes \mathcal{O}_{z_k}
.\] 
We now equip $\varphi$ with the structure of a strong homomorphism of bimonoids, which involves defining $2$-cells ${\mathtt{\Phi}_{F,G}}$ and ${\mathtt{\Psi}_{F,G}}$ and then checking nine coherence laws.

For compositions $G\leq F$, let
\[
F=(S_1,\dots,   S_k)\qquad \text{and} \qquad G=(T_1, \dots,   T_l)
.\]
Let $1\leq \mathtt{i}_j \leq k$ be the integer such that $S_{\mathtt{i}_j}$ is the $\mathtt{i}$th lump of $F|_{T_j}$, $1\leq j \leq l$. Let $k(j)$ denote the length of $F|_{T_j}$. 

Given an $F$-tuple of Boolean functions $z=(z_1,\dots,z_k)\in \textsf{BF}[F]$, we require a map of $\mathcal{O}_G$-modules ${\mathtt{\Phi}_{F,G}}_{(z_1,\dots,z_k)}$ of the form
\[
{\mathtt{\Phi}_{F,G}}_{(z_1,\dots,z_k)} :  
\Updelta^\ast_{F,G}  (  \mathcal{O}_{z_1}  \boxtimes\dots \boxtimes  \mathcal{O}_{z_k}   ) 
\to   
\mathcal{O}_{(z_{1_1} |\dots | z_{k(1)_1}  )} \boxtimes \cdots \boxtimes   \mathcal{O}_{(   z_{1_l} | \dots | z_{k(l)_l}  )}
.\]
Let $H=(H_1, \dots, H_l)\in \mathsf{\Sigma}[G]$. Over \hbox{$U_H\subset \bC\Sigma^G$}, the pullback $\Updelta^\ast_{F,G}\mathcal{O}_{(z_1 , \dots, z_k)}=\Updelta^\ast_{F,G}  (  \mathcal{O}_{z_1}  \boxtimes\dots \boxtimes  \mathcal{O}_{z_k}   )$ is (naturally isomorphic to) the coherent sheaf of the \hbox{$\bC[\text{M}_{H}]$-module}
\[
\bigotimes_{1\leq i \leq k}
\Big( 
\bC[\text{M}_{H|_{S_i},z_i}] \otimes   \bC[\text{M}_{ (H|_{S_1} ,\dots , H|_{S_k}) }] 
\Big) 
\otimes
\bC[\text{M}_{H}] \, :=
\]
\[ 
{\bigotimes_{1\leq i \leq k}}_{\, \bC[\text{M}_{ (H|_{S_1} ,\dots , H|_{S_k}) }]}  
\Big( 
\bC[\text{M}_{H|_{S_i},z_i}] \otimes_{ \bC[\text{M}_{H|_{S_i}}] }   \bC[\text{M}_{ (H|_{S_1} ,\dots , H|_{S_k}) }] 
\Big) 
\otimes_{\bC[\text{M}_{ (H|_{S_1} ,\dots , H|_{S_k}) }]}
\bC[\text{M}_{H}]
.\]
The external product $\mathcal{O}_{(z_{1_1} |\dots | z_{k(1)_1}  )} \boxtimes \cdots \boxtimes   \mathcal{O}_{(   z_{1_l} | \dots | z_{k(l)_l}  )}$ over $U_H$ is (naturally isomorphic to) the coherent sheaf of the \hbox{$\bC[\text{M}_{H}]$-module}
\[
\bigotimes_{1 \leq j\leq l}  \Big ( \bC[ \text{M}_{H_j,  ( z_{1_j} |\dots | z_{k(j)_j})  } ] \otimes \bC[\text{M}_{H}] \Big )
:= 
{\bigotimes_{1 \leq j\leq l}}_{\, \bC[\text{M}_{H}]} \bC[ \text{M}_{H_j,  ( z_{1_j} |\dots | z_{k(j)_j})  } ] \otimes_{\bC[\text{M}_{H_j}]} \bC[\text{M}_{H}]
.\]
Over $U_H$, we define ${\mathtt{\Phi}_{F,G}}_{(z_1,\dots,z_k)}$ to be the isomorphism of coherent sheaves which corresponds to the isomorphism of $\bC[\text{M}_{H}]$-modules
\[
\bigotimes_{1\leq i \leq k}
\Big( 
\bC[\text{M}_{H|_{S_i},z_i}] \otimes   \bC[\text{M}_{ (H|_{S_1} ,\dots , H|_{S_k}) }] 
\Big) 
\otimes
\bC[\text{M}_{H}] 
\to 
\bigotimes_{1 \leq j\leq l}  \Big ( \bC[ \text{M}_{H_j,  ( z_{1_j} |\dots | z_{k(j)_j})  } ] \otimes \bC[\text{M}_{H}]  \Big )
\]
\[
\underbrace{
	\bigotimes_{1\leq i \leq k}  ( g^{h_{i}} \otimes 1 ) \otimes 1
	=
	\bigotimes_{1\leq j \leq l}  \Big ( \bigotimes_{1_j \leq \mathtt{i}  \leq k(j)_j}  \big( g^{h_{\mathtt{i}_j}}\otimes 1 \big) \Big ) \otimes 1
}_{\text{collect into lumps of $G$}}
\, \mapsto \,    
\bigotimes_{1\leq j \leq l} g^{\text{m}_{F|_{T_j}, a_j}  (   \bigtimes_{1_j \leq \mathtt{i}  \leq k(j)_j}  h_{\mathtt{i}_j}   )} \otimes 1        
\]
where $a_j= (\text{hei}(z_{1_j} ) , \dots , \text{hei}(z_{k(j)_j} )  )$. Note not every element of the left-hand side is of the form $\bigotimes_{1\leq i \leq k} (g^{h_i}  \otimes 1 ) \otimes 1$, rather the map is given by extending this assignment \hbox{$\bC[\text{M}_{H}]$-linearly}. This map of modules is well-defined by \autoref{lem1}. It is an isomorphism because the map of \autoref{lem1} is injective.

We now make ${\mathtt{\Phi}_{F,G}}_{(z_1,\dots,z_k)}$ explicit at the level of stalks. The stalk of $\Updelta^\ast_{F,G}\mathcal{O}_{(z_1,\dots,z_k)}$ at $x\in U_{H}$ is given by
\[
\Updelta^\ast_{F,G}\mathcal{O}_{(z_1,\dots,z_k)}(x) =    
\bigotimes_{1\leq i \leq k} \Big ( \mathcal{O}_{z_i}( x|_{S_i} ) \otimes \mathcal{O}_{F} \big (    \Updelta_{F,G}(x) \big   ) \Big ) \otimes \mathcal{O}_G(x)
\]
\[
=\bigotimes_{1\leq i \leq k} \Big ( \big \{         g/f\in    \mathcal{K}_{S_i,\text{hei}(z_i)} \ \big | \  g\in     \bC[\text{M}_{H|_{S_i},z_i}],\,  f(x|_{S_i})\neq 0     \big \} \otimes \mathcal{O}_{F} \big (    \Updelta_{F,G}(x) \big   )  \Big )
 \otimes \mathcal{O}_G(x)
.\]
The stalk of $\mathcal{O}_{(z_{1_1} |\dots | z_{k(1)_1}  )} \boxtimes \cdots \boxtimes   \mathcal{O}_{(   z_{1_l} | \dots | z_{k(l)_l}  )}$ at $x\in U_{H}$ is given by
\[
\bigotimes_{1\leq j \leq l} \big ( \mathcal{O}_{(z_{1_j} |\dots | z_{k(j)_j}  )}  (x|_{T_j})  \otimes \mathcal{O}_G(x) \big )
\]
\[
=\bigotimes_{1\leq j \leq l} \Big ( \big \{         g/f\in    \mathcal{K}_{T_j,\text{hei}(z_{1_j} |\dots | z_{k(j)_j}  )} \ \big | \  g\in     \bC[\text{M}_{H|_{j},(z_{1_j} |\dots | z_{k(j)_j}  )}],\,  f(x|_{T_j})\neq 0     \big \}   
\otimes \mathcal{O}_G(x) \Big )
.\]
Then the map of stalks is given by extending the assignment
\[
\bigotimes_{1 \leq i\leq k} (g^{h_i}  \otimes 1 ) \otimes 1 
\ \mapsto\    
\bigotimes_{1 \leq j\leq l}  \big (g^{ \text{m}_{F|_{T_j}, a_j  }  (   \bigtimes_{1_j \leq \mathtt{i}_j \leq k(j)_j} h_{\mathtt{i}_j}   )} \otimes 1 \big )
\]
\hbox{$\mathcal{O}_G(x)$-linearly}.\footnote{\ explicitly, $\bigotimes_{i}  (g^{h_i}/f_i \otimes \text{f}_i ) \otimes \text{f}
\	\mapsto \   
	\bigslant{   (  \prod_{i} \Updelta^\ast_{F,G} \text{f}_i  )\text{f}}{ \prod_{i}  (\Updelta^\ast_{F,G}\pi^\ast_{S_i} f_i  )} 
	\bigotimes_{j}\big (g^{ \text{m}_{F|_{T_j}, a_j  }  (   \bigtimes_{\mathtt{i}_j}  h_{\mathtt{i}_j}   )} \otimes 1 \big )$} This gives a globally well-defined ${\mathtt{\Phi}_{F,G}}_{(z_1,\dots,z_k)}$ as we vary the affine chart $U_{H}$ by multiplication $\mathtt{f}$-naturality of $\bC\textbf{\textsf{O}}$.  When checking the coherence laws for a homomorphism of bimonoids in \autoref{thm:mainmain}, we shall abbreviate e.g. $h=(g^h \otimes1)\otimes 1$, so that the map gets written as
\begin{equation}\label{eq:Phi}
	\bigotimes_{1\leq i \leq k} h_i 
	\, \mapsto \,   
	\bigotimes_{1\leq j \leq l} \text{m}_{F|_{T_j}, a_j  } \Big (   \bigtimes_{\mathtt{i}}  h_{\mathtt{i}_j}    \Big )
	.
\end{equation}

Given a $G$-tuple of Boolean functions $z=(z_1, \dots, z_l)\in \textsf{BF}[G]$, we require a map of $\mathcal{O}_F$-modules ${\mathtt{\Psi}_{F,G}}_{( z_1 , \dots , z_l )}$ of the form
\[
{\mathtt{\Psi}_{F,G}}_{( z_1 , \dots , z_l )} :  
\upmu^\ast_{F,G} (\mathcal{O}_{  z_1} \boxtimes  \dots \boxtimes    \mathcal{O}_{  z_l}  ) 
 \to
   \mathcal{O}_{z_1\talloblong_{S_1}} \boxtimes \dots \boxtimes   \mathcal{O}_{z_l\talloblong_{S_k}}
.\]
Let $K=(K_1,\dots, K_k)\in \mathsf{\Sigma}[F]$. Over $U_K\subset \bC\Sigma^F$, the pullback $\upmu^\ast_{F,G} (\mathcal{O}_{  z_1} \boxtimes  \dots \boxtimes    \mathcal{O}_{  z_l}  ) $ is (naturality isomorphic to) the coherent sheaf of the \hbox{$\bC[\text{M}_K]$-module}
\[
\bigotimes_{ 1 \leq j \leq l }
\Big (
\bC[ \text{M}_{K_{1_j} ; \cdots ; K_{k(j)_j}, z_j } ] 
\otimes
\bC[\text{M}_{(K_{1_1} ; \cdots ; K_{k(1)_1},  \cdots, K_{1_l}; \cdots ; K_{k(l)_l})}]
\Big )
\otimes   \bC[ \text{M}_{K} ]
:=
\]
\[
{\bigotimes_{ 1 \leq j \leq l }}_{\bC[\text{M}_{(K_{1_1} ; \cdots ; K_{k(1)_1},  \cdots, K_{1_l}; \cdots ; K_{k(l)_l})}]} 
\Big (
\bC[ \text{M}_{K_{1_j} ; \cdots ; K_{k(j)_j}, z_j } ] 
\otimes_{ \bC[  M_{K_{1_j} ; \cdots ; K_{k(j)_j }}]  }  
\bC[\text{M}_{(K_{1_1} ; \cdots ; K_{k(1)_1},  \cdots, K_{1_l}; \cdots ; K_{k(l)_l})}]
\Big )
\]
\[
\otimes_{ \bC[\text{M}_{(K_{1_1} ; \cdots ; K_{k(1)_1},  \cdots, K_{1_l}; \cdots ; K_{k(l)_l})}] }   \bC[ \text{M}_{K} ]
.\]
Over $U_K$, the external product $\mathcal{O}_{z_1\talloblong_{S_1}} \boxtimes \dots \boxtimes   \mathcal{O}_{z_l\talloblong_{S_k}}$ is (naturality isomorphic to) the coherent sheaf of the $\bC[\text{M}_{K}]$-module
\[
{\bigotimes_{1\leq i\leq k}} 
\Big (\bC[ \text{M}_{ K_i, z\talloblong_{S_i} } ] \otimes  \bC[ \text{M}_{ K  } ]\Big )
:=  
{\bigotimes_{1\leq i\leq k}}_{\,  \bC[ \text{M}_{ K } ]}  
\Big (\bC[ \text{M}_{ K_i, z\talloblong_{S_i} } ] \otimes_{\bC[ \text{M}_{K_i}]} \bC[ \text{M}_{ K  } ]\Big )  
.\] 
Over $U_K$, we define ${\mathtt{\Psi}_{F,G}}_{( z_1 , \dots , z_l )}$ to be the isomorphism of coherent sheaves which corresponds to the isomorphism of $\bC[\text{M}_{K}]$-modules
\[
\bigotimes_{ 1 \leq j \leq l }
\Big (
\bC[ \text{M}_{K_{1_j} ; \cdots ; K_{k(j)_j}, z_j } ] 
\otimes
\bC[\text{M}_{(K_{1_1} ; \cdots ; K_{k(1)_1},  \cdots, K_{1_l}; \cdots ; K_{k(l)_l})}]
\Big )
\otimes   \bC[ \text{M}_{K} ]
\to
{\bigotimes_{1\leq i\leq k}} 
\Big (\bC[ \text{M}_{ K_i, z\talloblong_{S_i} } ] \otimes  \bC[ \text{M}_{ K  } ]\Big )  
\]
\[
\underbrace{
	\bigotimes_{1\leq j\leq l}  ( g^{h_{j}}  \otimes 1)  \otimes 1
=
\bigotimes_{1\leq j\leq l}  \big(    
g^{\text{m}_{	  F|_{T_j}   , a_j } 
 ( \bigtimes_{1_j\leq  \mathtt{i} \leq k(j)_j} d_{\mathtt{i}_j} )} \otimes 1 \big) \otimes 1
}_{\text{factorize each $h_{j}$}}
\  \mapsto \  
\bigotimes_{1\leq j\leq l}  \big ( \bigotimes_{1_j\leq  \mathtt{i} \leq k(j)_j}   g^{d_{\mathtt{i}_j}} \otimes 1 \big ) 
=
\bigotimes_{1\leq i \leq k} (g^{d_i} \otimes 1)
\]
where $a_j=( \text{hei}(z\talloblong_{S_{1_{j}}}),\dots,  \text{hei}(z\talloblong_{S_{ k(j)_j }} ) )$ and we have assumed each $h_j$ factorizes as 
\[
h_{j}= \text{m}_{	  F|_{T_j}   , a_j } \big (\bigtimes_{1_j\leq  \mathtt{i} \leq k(j)_j} d_{\mathtt{i}_j}\big )
,\] 
i.e. $h_j$ lies in the $ F|_{T_j}$-face of the plate $[[ K_{1_j} ; \dots ; K_{k(j)_j} ]]_{z_j}$. Notice that if the $h_j$ do not factorize in this way, then $\bigotimes_{1\leq j\leq l}  ( g^{h_{j}}  \otimes 1)  \otimes 1=0$. This map of modules is well-defined by \autoref{secondprop}. It is an isomorphism because the map of \autoref{secondprop} is injective. 

We now make ${\mathtt{\Psi}_{F,G}}_{( z_1 , \dots , z_l )}$ explicit at the level of stalks. The stalk of $\upmu^\ast_{F,G} (\mathcal{O}_{  z_1} \boxtimes  \dots \boxtimes    \mathcal{O}_{  z_l}  ) $ at $x\in U_K$ is given by
\[
\bigotimes_{1\leq j \leq l}  \Big (  \mathcal{O}_{z_j}\big (  \upmu_{F,G}(x)|_{T_j}\big )      \otimes \mathcal{O}_G\big(  \upmu_{F,G}(x)\big) \Big )        \otimes \mathcal{O}_F(x)
\]
\[
=\bigotimes_{1\leq j \leq l}  \Big ( \big \{         g/f\in    \mathcal{K}_{T_j,\text{hei}(z_j)} \ \big | \  g\in     
\bC[\text{M}_{K_{1_j} ; \cdots ; K_{k(j)_j}, z_j}]
,\,  f( \upmu_{F,G}(x)|_{T_j})\neq 0     \big \}       \otimes \mathcal{O}_G\big(  \upmu_{F,G}(x)\big) \Big )        \otimes \mathcal{O}_F(x)
.\]
The stalk of $\mathcal{O}_{z_1\talloblong_{S_1}} \boxtimes \dots \boxtimes   \mathcal{O}_{z_l\talloblong_{S_k}}$ at $x\in U_K$ is given by
\[
\bigotimes_{1\leq i \leq k} \big ( \mathcal{O}_{z\talloblong_{S_i}}(x|_{S_i}) \otimes \mathcal{O}_F(x) \big) 
\]
\[
=\bigotimes_{1\leq i \leq k} \big (\big \{         g/f\in    \mathcal{K}_{S_i,\text{hei}(z\talloblong_{S_i})} \ \big | \  g\in     \bC[\text{M}_{K_i,z\talloblong_{S_i}}],\,  f(x|_{S_i})\neq 0     \big \}   \otimes \mathcal{O}_F(x) \big) 
.\]
Then the map of stalks is given by extending the assignment
\[
	\bigotimes_{1\leq j\leq l}  ( g^{h_{j}}  \otimes 1)  \otimes 1
	=
	\bigotimes_{1\leq j\leq l}  \big(    
	g^{\text{m}_{	  F|_{T_j}   , a_j } 
		( \bigtimes_{1_j\leq  \mathtt{i} \leq k(j)_j} d_{\mathtt{i}_j} )} \otimes 1 \big) \otimes 1
\  \mapsto \  
\bigotimes_{1\leq j\leq l}  \big ( \bigotimes_{1_j\leq  \mathtt{i} \leq k(j)_j}   g^{d_{\mathtt{i}_j}} \otimes 1 \big ) 
=
\bigotimes_{1\leq i \leq k} (g^{d_i} \otimes 1)
\]
$\mathcal{O}_F(x)$-linearly. This gives a globally well-defined ${\mathtt{\Psi}_{F,G}}_{(z_1,\dots,z_k)}$ as we vary the affine chart $U_{K}$ by comultiplication $\mathtt{f}$-naturality of $\bC\textbf{\textsf{O}}$. When checking the coherence laws for a homomorphism of bimonoids in \autoref{thm:mainmain}, we shall abbreviate this map as
\begin{equation} \label{eq:Psi}
	\bigotimes_{1\leq j\leq l}  h_{j}    =	\bigotimes_{1\leq j\leq l}  \Big(    
	\text{m}_{	  F|_{T_j}   , a_j } 
	\big( \bigtimes_{\mathtt{i}} d_{\mathtt{i}_j}  \big)   \Big)  
\,  \mapsto \,  
\bigotimes_{1\leq j\leq l}  \bigotimes_{\mathtt{i}}    d_{\mathtt{i}_j}  
.
\end{equation}


\begin{thm}\label{thm:mainmain}
The indexing $z\mapsto \mathcal{O}_z$ of invertible sheaves on permutohedral space by Boolean functions, equipped with the $2$-cells ${\mathtt{\Phi}_{F,G}}$ and ${\mathtt{\Psi}_{F,G}}$ as defined above, is a strong homomorphism of bimonoids
\[
\varphi : \textsf{BF} \to \textsf{Mod}
.\]
\end{thm}
\begin{proof}
There are nine coherence laws to check. Recall multiplication $\sigma$-naturality coherence says
\[
\underbrace{\varphi_{G'} ( {\mu_{F,G,\sigma}}_z )}_{=\, \text{id}} \, \circ\,  {\sigma}_{\mu_{F,G}(z_1,\dots,z_k)}  \circ  {\mathtt{\Phi}'_{F,G}}_{(z_1,\dots,z_k)}  = {\mathtt{\Phi}_{F',G'}}_{(z_1,\dots,z_k)'} \circ \Updelta^\ast_{F',G'}({\sigma}_{(z_1,\dots,z_k)})
\circ {\mu_{F,G,\sigma}}_{   \mathcal{O}_{(z_1,\dots,z_k)}  }
.\]
The left-hand map is given by 
\[
 \bigotimes_{1\leq i \leq k}   h_{i}  
=
   \bigotimes_{1\leq j \leq l}   \bigotimes_{\mathtt{i}}   h_{\mathtt{i}_j}  
\mapsto    
 \bigotimes_{1\leq j \leq l}   \text{m}_{F|_{T_j},a_j} \Big (   \bigtimes_{\mathtt{i}}  h_{\mathtt{i}_j}  \Big  ) 
\mapsto     
\bigotimes_{1\leq j \leq l}  \text{m}_{F'|_{T'_j},a_j} \Big (   \bigtimes_{\mathtt{i}}  h'_{\mathtt{i}_j}  \Big  )  
\]
and the right-hand map is given by
\[
{\bigotimes_{1\leq i \leq k}} h_i  
\mapsto 
{\bigotimes_{1\leq i \leq k}} h_i 
\mapsto 
 \bigotimes_{1\leq i \leq k}  h'_i =  \bigotimes_{1\leq j \leq l}  \bigotimes_{   \mathtt{i}  }  h'_{\mathtt{i}_j}       
 \mapsto
\bigotimes_{1\leq j \leq l}  \text{m}_{F'|_{T'_j},a_j} \Big (   \bigtimes_{   \mathtt{i}  }  h'_{\mathtt{i}_j}    \Big )
.\]
Recall multiplication $\beta$-naturality coherence says
\[
\underbrace{\varphi_{\wt{G}} ( {\mu_{F,G,\beta}}_z )}_{=\, \text{id}}\,  \circ\,  {\beta_{G}}_{\mu_{F,G}( z_1,\dots,z_k )} \circ {\wt{\mathtt{\Phi}_{F,G}}}_{(z_1,\dots,z_k)} 
= 
{\mathtt{\Phi}_{\wt{G}F,\wt{G}}}_{\wt{(z_1,\dots,z_k)}}\circ  \Updelta^\ast_{\wt{G}F, \wt{G}}( {\hat{\beta}}_F {\! \,  }_{(z_1,\dots,z_k)}   ) \circ  \mu_{\wt{G}F, \wt{G}}( {\hat{\beta}}_\ta  ) {\circ \mu_{F,G,\beta}}_{\mathcal{O}_z}
.\]
The left-hand map is given by\footnote{\ where the tilde $\wt{F}$ refers to $\hat{\beta}$}
\[
\bigotimes_{1\leq i \leq k}
h_i 
\mapsto
\bigotimes_{1\leq j \leq l}   \text{m}_{F|_{T_j},a_j} \Big (   \bigtimes_{\mathtt{i} }  h_{\mathtt{i}_j}  \Big  )    
\mapsto
 \bigotimes_{1 \leq j\leq l}   \text{m}_{\wt{F}|_{T_{\beta(j)}},  a_j  } \Big (   \bigtimes_{ \mathtt{i}  } h_{\hat{\beta}(\mathtt{i}_j)}  \Big  )     
\]
and the right-hand map is given by
\[
{\bigotimes_{1\leq i \leq k}}  
h_i  
\mapsto 
{\bigotimes_{1\leq i \leq k}}  
 h_i  
\mapsto 
\bigotimes_{1\leq i \leq k}   h_{\hat{\beta}(i)}  
\mapsto
 \bigotimes_{1 \leq j\leq l}   \text{m}_{\wt{F}|_{T_{\beta(j)}},  a_j  } \Big (   \bigtimes_{ \mathtt{i}  } h_{\hat{\beta}(\mathtt{i}_j)}  \Big )
.\]
Comultiplication $\sigma$-naturality and $\beta$-naturality coherence follow similarly. Recall associativity coherence says 
\[
\underbrace{\varphi_E ({\mu_{F,G,E}}_{(z_1,\dots, z_k)})}_{=\, \text{id}} \circ  {\mathtt{\Phi}_{G,E}}_{\mu_{F,G}(z_1,\dots, z_k)}  \circ  \Updelta^\ast_{G,E} (   {\mathtt{\Phi}_{F,G}}_{(z_1,\dots, z_k)}  )  
=
{\mathtt{\Phi}_{F,E}}_{(z_1,\dots, z_k)} \circ {\mu_{F,G,E}}_{\mathcal{O}_{(z_1,\dots, z_k)}}
.\]
Let
\[
F=(S_1,\dots, S_i,\dots , S_k), \qquad G=(T_1,\dots, T_j,\dots, T_l),  \qquad E=(U_1, \dots, U_r, \dots, U_e)
.\]
Let $\mathtt{i}_j$ be as before. Let $1\leq \mathtt{j}_r \leq l$ be the integer such that $T_{\mathtt{j}_r}$ is the $\mathtt{j}$th lump of $G|_{U_r}$, and let $1\leq {\mathtt{i}_\mathtt{j}}_r\leq k$ such that $S_{{\mathtt{i}_\mathtt{j}}_r}$ is the $\mathtt{i}$th lump of $F|_{T_{\mathtt{j}_r}}$. Then the left-hand map is given by
\[
	\bigotimes_{1\leq i \leq k} h_{i}
	=
	\bigotimes_{1\leq j \leq l}  \bigotimes_{\mathtt{i}} h_{\mathtt{i}_j} 
 \mapsto    
\bigotimes_{1\leq j \leq l} \text{m}_{F|_{T_j},a_j} \Big (   \bigtimes_{\mathtt{i}}  h_{\mathtt{i}_j}  \Big  )        
=
\bigotimes_{1\leq r \leq e}  \bigotimes_{{\mathtt{j}}}  \text{m}_{F|_{T_{\mathtt{j}_r}}, a_j } \Big (   \bigtimes_{{\mathtt{i}}}  h_{ {\mathtt{i}_\mathtt{j}}_{r}   }  \Big  )   
\]
\[
\mapsto 
\bigotimes_{1\leq r \leq e}  \text{m}_{G|_{U_r}, a_r}   \bigg ( \bigtimes_{{\mathtt{j}}}  \text{m}_{F|_{T_{\mathtt{j}_r}}, a_{\mathtt{j}_r} } \Big (   \bigtimes_{{\mathtt{i}}}  h_{ {\mathtt{i}_\mathtt{j}}_{r}   }  \Big  )  \bigg) 
\]
and the right-hand map is given by
\[
	\bigotimes_{1\leq i \leq k} h_{i} \mapsto 	\bigotimes_{1\leq i \leq k} h_{i} \mapsto 
\bigotimes_{1\leq r \leq e}  \text{m}_{F|_{U_r}, a_r} \Big (  \bigtimes_{\mathtt{i}}  h_{ \mathtt{i}_{r}   }     \Big) 
.\]
However, we have
\[
\bigotimes_{1\leq r \leq e}  \text{m}_{G|_{U_r},a_r}   \bigg ( \bigtimes_{{\mathtt{j}}}  \text{m}_{F|_{T_{\mathtt{j}_r}},a_{\mathtt{j}_r}} \Big (   \bigtimes_{{\mathtt{i}}}  h_{ {\mathtt{i}_\mathtt{j}}_{r}   }  \Big  )  \bigg)  
 =
  \bigotimes_{1\leq r \leq e}  \text{m}_{F|_{U_r}, a_r} \Big (  \bigtimes_{\mathtt{i}}  h_{ \mathtt{i}_{r}   }     \Big) 
\]
directly from the definition of $\text{m}_{F,(a_1,\dots, a_k)}$ (the exact property we are using here is actually the associativity of the monoid mentioned in \autoref{rem:mon}). Coassociativity coherence says 
\[
\underbrace{\varphi_F ({\Delta_{F,G,E}}_{(z_1,\dots, z_{e})})}_{=\, \text{id}} \circ {\mathtt{\Psi}_{F,G}}_{\Delta_{G,E}(z_1,\dots, z_{e})}  \circ  \upmu^\ast_{F,G} (   {\mathtt{\Psi}_{G,E}}_{(z_1,\dots, z_{e})}  )  
=
{\mathtt{\Psi}_{F,E}}_{(z_1,\dots, z_{e})} \circ {\mu_{F,G,E}}_{\mathcal{O}_{(z_1,\dots, z_{e})}}
.\]
The left-hand map is given by
\[
	\bigotimes_{1\leq r\leq {e}}    h_{r} 
	=
	\bigotimes_{1\leq r\leq {e}}    \text{m}_{G|_{U_r},a_r} \Big( \bigtimes_{\mathtt{j}} d_{\mathtt{j}_r}  \Big)
\  \mapsto \  
\bigotimes_{1\leq r\leq e}  \bigotimes_{\mathtt{j}}    d_{\mathtt{j}_r} 
=
\bigotimes_{1\leq r\leq e}   \bigotimes_{\mathtt{j}}     \text{m}_{ F|_{T_{\mathtt{j}_r}}, a_{\mathtt{j}_r} } \Big (  \bigtimes_\mathtt{i}   c_{{\mathtt{i}_\mathtt{j}}_r} \Big )
\]
\[
\mapsto 
\bigotimes_{1\leq r\leq e}   \bigotimes_{\mathtt{j}}     \bigotimes_{{\mathtt{i}}}   c_{{\mathtt{i}_\mathtt{j}}_r}  =\bigotimes_{1\leq r\leq e}   \bigotimes_{\mathtt{i}}    c_{\mathtt{i}_r}
\]
and the right-hand map is given by
\[
\bigotimes_{1\leq r\leq e}  h_{r}    =	\bigotimes_{1\leq r\leq e}  \bigg(    
\text{m}_{	  F|_{U_r}   , a_r } 
\Big( \bigtimes_{\mathtt{i}} c_{\mathtt{i}_r}  \Big)   \bigg)  
\,  \mapsto \,  
\bigotimes_{1\leq r\leq e}   \bigotimes_{\mathtt{i}}    c_{\mathtt{i}_r} 
\]
(where we have again used the associativity of the monoid mentioned in \autoref{rem:mon}). Recall the bimonoid coherence law is
\[
\varphi_G({\text{B}_{F,G,A}}_{\ta}) 
\circ 
{\mathtt{\Phi}_{GF,G}}_{\wt{\Delta_{FG,F}(z)}} 
 \circ 
  {\beta_{FG}}_{ \Delta_{FG,F} (z)} \circ
\Updelta^\ast_{GF,G}(\wt{{\mathtt{\Psi}_{FG,F}}_z}) \]
\[
=  
{\mathtt{\Phi}_{GF,G}}_{\wt{\Delta_{FG,F}}(\ta)}  
 \circ
    {\mathtt{\Psi}_{G,A}}_{\mu_{F,A}(z)} 
    \circ 
     \upmu^\ast_{G,A}( {\mathtt{\Phi}_{F,A}}_{z}  ) \circ {\text{B}_{F,G,A}}_{\varphi_F(\ta)}
.\]
Let 
\[
F=(S_1,\dots, S_i, \dots, S_k), \quad G=(T_1,\dots, T_j, \dots, T_l),\]
\[ 
FG=(X_1,\dots, X_n,\dots, X_L),\quad GF=(Y_1,\dots,Y_n,\dots , Y_L), \quad  A=(A_1,\dots,A_m,\dots , A_{a})
.\]
Let $1\leq \mathtt{i}_m \leq k$ be the integer such that $S_{\mathtt{i}_m}$ is the $\mathtt{i}$th lump of $F|_{A_m}$, let $1\leq \mathtt{j}_m \leq j$ such that $T_{\mathtt{j}_m}$ is the $\mathtt{j}$th lump of $G|_{A_m}$, let $1\leq \mathtt{n}_i \leq L$ such that $X_{\mathtt{n}_i}$ is the $\mathtt{n}$th lump of $FG|_{S_i}$, and let $1\leq \mathtt{n}_j \leq L$ such that $Y_{\mathtt{n}_j}$ is the $\mathtt{n}$th lump of $GF|_{T_j}$. Let $1\leq \mathtt{n}_{\mathtt{i}_m}\leq L$ such that $X_{\mathtt{n}_{\mathtt{i}_m}}$ is the $\mathtt{n}$th lump of $FG|_{S_{\mathtt{i}_m}}$, and let $1\leq \mathtt{n}_{\mathtt{j}_m}\leq L$ such that $Y_{\mathtt{n}_{\mathtt{j}_m}}$ is the $\mathtt{n}$th lump of $GF|_{T_{\mathtt{j}_m}}$. 

Then the left-hand map is given by
\[
\bigotimes_{1\leq i\leq k} h_i = \bigotimes_{1\leq i\leq k}  \text{m}_{  FG|_{S_i} , a_i } \Big ( \bigtimes_{\mathtt{n}}   d_{\mathtt{n}_i} \Big ) 
\  \mapsto \  
		\bigotimes_{1\leq i\leq k}  \bigotimes_{\mathtt{n}}  d_{ \mathtt{n}_i  }
\]
\[ 
\mapsto\ 
		\underbrace{\bigotimes_{1\leq i\leq k}  \bigotimes_{\mathtt{n}}     d_{\beta(\mathtt{n}_i)} = 	\bigotimes_{1\leq j\leq l}  \bigotimes_{\mathtt{n}}  e_{\mathtt{n}_j}}_{\text{where $e_n:= d_{ \beta( n )}$}}
\ \mapsto \ 
\bigotimes_{1\leq j \leq l}   \text{m}_{ GF|_{T_j} , a_j } \Big (  \bigtimes_{\mathtt{n}} e_{ \mathtt{n}_j }   \Big   )   
\]
and the right-hand map is given by
\[
\bigotimes_{1\leq i\leq k} h_i 
=
\bigotimes_{1\leq m\leq a}   \bigotimes_{\mathtt{i}}     
h_{\mathtt{i}_m}
\]
\[
\  \mapsto \  
\bigotimes_{1\leq m\leq  a}  \text{m}_{ F|_{A_m} ,a_m} \Big  (  \bigtimes_{\mathtt{i}}  h_{\mathtt{i}_m} \Big  )    
= 
\bigotimes_{1\leq m\leq  a}  \text{m}_{ F|_{A_m} ,a_m}  \bigg ( \bigtimes_{\mathtt{i}} \text{m}_{FG|_{S_{\mathtt{i}_m}}, a_{\mathtt{i}_m}}\Big( \bigtimes_{\mathtt{n}}  d_{\mathtt{n}_{\mathtt{i}_m}} \Big  )      \bigg )   
\]
\[ 
= \bigotimes_{1\leq m\leq  a}  \text{m}_{ F|_{A_m} ,a_m} \bigg ( \bigtimes_{\mathtt{j}} \text{m}_{GF|_{T_{\mathtt{j}_m}} , a_{\mathtt{j}_m} }\Big( \bigtimes_{\mathtt{n}}  e_{\mathtt{n}_{\mathtt{j}_m}} \Big  )      \bigg ) 
\ \mapsto \
\bigotimes_{1\leq m\leq  a}   \bigotimes_{\mathtt{j}} \text{m}_{GF|_{T_{\mathtt{j}_m}} , a_{\mathtt{j}_m} } \Big( \bigtimes_{\mathtt{n}}  e_{\mathtt{n}_{\mathtt{j}_m}} \Big  )     
\]
\[
=
\bigotimes_{1\leq j \leq l}   \text{m}_{ GF|_{T_j} ,a_j} \Big (  \bigtimes_{\mathtt{n}} e_{ \mathtt{n}_j }   \Big   ) 
.\qedhere \]
\end{proof}

\begin{remark}
	The restriction of the homomorphism $\varphi : \textsf{BF} \to \textsf{Mod}$ to the bimonoid generalized permutohedra $\textsf{GP}\hookrightarrow \textsf{BF}$, as defined in \cite[Section 5.1 \& 12.3]{aguiar2017hopf}, is an indexing of the \emph{nef invertible sheaves}, as a consequence of e.g. \cite[Proposition 2.2.6]{eur}.
\end{remark}


\subsection{The Bialgebra of Global Sections of Invertible Sheaves}\label{The Joyal-Theoretic Bialgebra of Global Sections of Invertible Sheaves}

Let $\bC\textbf{\textsf{BF}}$ be the $\textsf{BF}$-graded bialgebra which is the pullback of the $\textsf{Mod}$-graded bialgebra $\bC\textbf{\textsf{Mod}}$ along the homomorphism $\varphi: \textsf{BF} \to \textsf{Mod}$, as constructed in \autoref{thm}. Thus
\[
\bC\textbf{\textsf{BF}}_{(z_1,\dots, z_k)}[F] =    \Gamma_{\bC\Sigma^{F}}(\mathcal{O}_{z_1} \boxtimes \dots \boxtimes  \mathcal{O}_{z_k})
.\]
Let $\bC\textbf{BF}$ be the Joyalization of $\bC\textbf{\textsf{BF}}$, thus 
\[
\bC\textbf{BF}_{(z_1,\dots, z_k)}[F] :=    \Gamma_{\bC\Sigma^{S_1}}(\mathcal{O}_{z_1}) \otimes  \dots \otimes   \Gamma_{\bC\Sigma^{S_k}}(\mathcal{O}_{z_k})
.\]
We now show that $\bC\textbf{\textsf{BF}}$ is a strongly Joyal-theoretic bialgebra. The external product of global sections defines a morphism of $\textsf{BF}$-graded vector cospecies $ \bC\textbf{BF}   \to \bC\textbf{\textsf{BF}}$ over the identity, given by
\[
\text{prod}: \bC\textbf{BF}   \to \bC\textbf{\textsf{BF}},\qquad    s_1 \otimes  \dots \otimes s_k\, \mapsto\,   s_{1} \boxtimes \dots \boxtimes  s_{k} := \pi_{S_1}^\ast s_1 \otimes  \dots \otimes  \pi_{S_k}^\ast s_k 
.\]




\begin{prop}
The external product of global sections
\[
\text{prod}: \bC\textbf{BF}   \to \bC\textbf{\textsf{BF}}
,\qquad
s_{1} \otimes \dots \otimes  s_{k}   \mapsto   s_{1} \boxtimes \dots \boxtimes  s_{k} 
\]
is an isomorphism of $\textsf{BF}$-graded vector cospecies. 
\end{prop}
\begin{proof} 
We have
\[
\Gamma_{\bC\Sigma^{S_1}}(\mathcal{O}_{z_1}) \otimes  \dots \otimes   \Gamma_{\bC\Sigma^{S_k}}(\mathcal{O}_{z_k})
= \bC\Big [ \bigcap_{H_1\in \Sigma[S_1]}   \text{M}_{H_1, z_1} \Big ] \otimes \dots \otimes   \bC\Big [ \bigcap_{H_k\in \Sigma[S_k]}   \text{M}_{H_k, z_k} \Big ]
.\]	
Consider an affine open of the form $U_H=U_{H_1}\times \dots \times U_{H_k}$ of $\bC\Sigma^{F}$. Tensoring coherent sheaves corresponds to tensoring their modules on affine charts, and so
\[
\Gamma_{U_H}(\mathcal{O}_{z_1} \boxtimes \dots \boxtimes  \mathcal{O}_{z_k}) 
= 
(\bC[ \text{M}_{H_1,z_1} ]  \otimes \bC[ \text{M}_{H} ] )\otimes  \dots \otimes  (\bC[ \text{M}_{H_k,z_k} ]\otimes \bC[ \text{M}_{H} ) 
.\]
We make the identification
\[
(\bC[ \text{M}_{H_1,z_1} ]  \otimes \bC[ \text{M}_{H} ] )\otimes  \dots \otimes  (\bC[ \text{M}_{H_k,z_k} ]\otimes \bC[ \text{M}_{H} ) 
=
 \bC[ \text{M}_{H_1,z_1}   \times \dots \times \text{M}_{H_k,z_k} ]
.\]
Then
\[
 \Gamma_{\bC\Sigma^{F}}(\mathcal{O}_{z_1} \boxtimes \dots \boxtimes  \mathcal{O}_{z_k}) = \bigcap_H \bC\Big[  \text{M}_{H_1,z_1}   \times \dots \times  \text{M}_{H_k,z_k} \Big ]
 = \bC\Big[ \bigcap_H \big ( \text{M}_{H_1,z_1}   \times \dots \times  \text{M}_{H_k,z_k}\big ) \Big ]
.\]
Then the external product of global sections is given by the isomorphism of vector spaces
\[
\bC\Big [ \bigcap_{H_1\in \Sigma[S_1]}   \text{M}_{H_1, z_1} \Big ] \otimes \dots \otimes   \bC\Big [ \bigcap_{H_k\in \Sigma[S_k]}   \text{M}_{H_k, z_k} \Big ]
\to
\bC\Big[ \bigcap_H \big ( \text{M}_{H_1,z_1}   \times \dots \times  \text{M}_{H_k,z_k}\big ) \Big ]
\]
\[
g^{h_1} \otimes \dots \otimes g^{h_k}  \mapsto  g^{(h_1,\dots,h_k)}.\qedhere
\]
\end{proof}

\begin{prop}
The $\textsf{BF}$-graded bialgebra $\bC\textbf{\textsf{BF}}$ is strongly Joyal-theoretic. 
\end{prop}

\begin{proof}
For the multiplication, we need to show factorization
\[
\Updelta^\ast_{F,G}\Big ( \bigotimes_{1\leq i \leq k} s_i  \Big )  =  \bigotimes_{1\leq j \leq l} \Updelta^\ast_{F|_{T_j}}\Big(  \bigotimes_\mathtt{i}  s_{\mathtt{i}_j}\Big )   
.\]
We have that the multiplication pulled-back to $\bC\textbf{\text{BF}}$ along $\text{prod}$ is given by
\[
\Updelta^\ast_{F,G} :  \bigotimes_{1\leq i \leq k} \Gamma_{\bC\Sigma^{S_i}}(\mathcal{O}_{z_i})   
\to
   \bigotimes_{1\leq j \leq l}    \Gamma_{\bC\Sigma^{T_j}}(\mathcal{O}_{  (  z_{1_j}  | \dots | z_{k(j)_j} )  })  
\]
\[
\bigotimes_{1\leq i \leq k}  s^{h_i} 
\mapsto  	
\bigotimes_{1\leq j \leq l}  s^{\text{m}_{F|_{T_j}, a_j  } \big (   \bigtimes_{\mathtt{i}_j}  h_{\mathtt{i}_j}    \big )}
\]
and indeed we have factorization because
\[
\bigotimes_{1\leq j \leq l}  s^{\text{m}_{F|_{T_j}, a_j  } \big (   \bigtimes_{\mathtt{i}_j}  h_{\mathtt{i}_j}    \big )}
=
\Updelta^\ast_{F|_{T_j}}\Big(  \bigotimes_\mathtt{i}  s^{h_{\mathtt{i}_j}} \Big )   
.\]
For the comultiplication, we need to show factorization
\[
\upmu^\ast_{F,G}\Big ( \bigotimes_{1\leq j \leq l} s_j  \Big )  =  \bigotimes_{1\leq j \leq l} \upmu^\ast_{F|_{T_j}}(  s_j  )   
.\]
We have that the comultiplication pulled-back to $\bC\textbf{\text{BF}}$ along $\text{prod}$ is given by
\[
\bigotimes_{1\leq j\leq l}      s^{h_{j}}    
=	
\bigotimes_{1\leq j\leq l}     s^{\text{m}_{	  F|_{T_j}   , a_j}  ( \bigtimes_{\mathtt{i}_j} h_{\mathtt{i}_j}  )}
\,  \mapsto \,  
 \bigotimes_{1\leq j\leq l}  \bigotimes_{\mathtt{i}}    s^{h_{\mathtt{i}_j}}  
\]
and indeed we have factorization because
\[
\bigotimes_{\mathtt{i}}    s^{h_{\mathtt{i}_j}}   =\upmu^\ast_{F|_{T_j}}(  s^{h_j}  )   
.\qedhere \]
\end{proof}

\bibliographystyle{alpha}
\bibliography{steinmann}

\end{document}